\documentclass[12pt]{amsart}
\usepackage{amsmath,amsthm,amsfonts,amssymb,eucal}%, hyperref}

\renewcommand {\a}{ \alpha }
\renewcommand{\b}{\beta}
\newcommand{\e}{\epsilon}
\newcommand{\vare}{\varepsilon}

\newcommand{\G}{\Gamma}
\newcommand{\U}{\Upsilon}
\renewcommand{\k}{\kappa}
\newcommand{\vark}{\varkappa}
\renewcommand{\d}{\delta}
\newcommand{\s}{\sigma}
\renewcommand{\l}{\lambda}
\renewcommand{\L}{\Lambda}
\newcommand{\z}{\zeta}

\newcommand{\p}{\partial}
\newcommand{\om}{\omega}
\newcommand{\Om}{\Omega}

\newcommand{\R}{ \mathbb R}

\newcommand {\gb}{\mathfrak b}

\newcommand{\gt}{\mathfrak t}

\newcommand {\GA}{\mathfrak A}

\newcommand {\GB}{\mathfrak B}

\newcommand {\GH}{\mathfrak H}

\newcommand{\GR}{\mathfrak R}
\newcommand {\GS}{\mathfrak S}

\newcommand {\GW}{\mathfrak W}

\newcommand {\GT}{\mathfrak T}

\newcommand {\bb}{\mathbf b}
\newcommand{\bv}{\mathbf v}

\newcommand {\BS}{\mathbf S}

\newcommand {\BO}{\mathbf O}

\newcommand {\bx}{\mathbf x}

\newcommand {\be}{\mathbf e}
\newcommand {\bh}{\mathbf h}
\newcommand {\bk}{\mathbf k}

\newcommand {\bm}{\mathbf m}

\newcommand {\bw}{\mathbf w}
\newcommand {\bz}{\mathbf z}
\newcommand {\by}{\mathbf y}
\newcommand {\bt}{\mathbf t}

\newcommand {\bn}{\mathbf n}

\newcommand{\bzero}{\mathbf 0}

\newcommand{\boldA}{\boldsymbol A}
\newcommand{\boldM}{\boldsymbol M}
\newcommand{\boldB}{\boldsymbol B}
\newcommand{\boldO}{\boldsymbol O}
\newcommand{\boldk}{\boldsymbol k}
\newcommand{\boldE}{\boldsymbol E}
\newcommand{\boldI}{\boldsymbol I}

\newcommand{\SL}{{\sf{\Lambda}}}

\newcommand{\SN}{{\sf{N}}}
\newcommand{\SX}{{\sf{X}}}

\newcommand {\bmu}{\boldsymbol\mu}

\newcommand {\boldeta}{\boldsymbol\eta}

\newcommand {\bxi}{\boldsymbol\xi}

\newcommand{\overc}{\overset{\circ}}

\newcommand{\lu}{\langle}
\newcommand{\ru}{\rangle}

\newcommand{\CR}{\mathcal R}

\newcommand{\CL}{\mathcal L}
\newcommand{\CT}{\mathcal T}
\newcommand{\CX}{\mathcal X}
\newcommand{\CY}{\mathcal Y}

\newcommand{\CM}{\mathcal M}

\newcommand{\CC}{\mathcal C}

\newcommand{\CW}{\mathcal W}

\newcommand{\plainC}[1]{\textup{{\textsf{C}}}^{#1}}

\newcommand{\plainS}{\textup{{\textsf{S}}}}

\newcommand{\plainL}[1]{\textup{{\textsf{L}}}^{#1}}

\DeclareMathOperator{\tr}{{tr}}

\newcommand{\1}
{{\,\vrule depth3pt height9pt}{\vrule depth3pt height9pt}
{\vrule depth3pt height9pt}{\vrule depth3pt height9pt}\,}

\DeclareMathOperator {\re} {{Re}}
\DeclareMathOperator {\dist} {{dist}}

\DeclareMathOperator{\op}{{Op}}

\DeclareMathOperator{\supp}{{supp}}

\DeclareMathOperator{\w}{w}

\newcommand{\Z}{\mathbb Z}

\newtheorem{thm}{Theorem}[section]
\newtheorem{cor}[thm]{Corollary}
\newtheorem{lem}[thm]{Lemma}
\newtheorem{prop}[thm]{Proposition}
\newtheorem{cond}[thm]{Condition}

\theoremstyle{definition}
\newtheorem{defn}[thm]{Definition}%[section]

\newtheorem{rem}[thm]{Remark}

\numberwithin{equation}{section}

\begin{document}
\hoffset -4pc

\title
[Quasi-classical asymptotics]
{{Quasi-classical asymptotics for pseudo-differential operators
with discontinuous symbols:\\ Widom's Conjecture}}
\author{A.V. Sobolev}
\address{Department of Mathematics\\ University College London\\
Gower Street\\ London\\ WC1E 6BT UK}
\email{asobolev@math.ucl.ac.uk}
\date{\today}

\begin{abstract}
Relying on the known two-term quasiclassical asymptotic  formula for
the trace of the function $f(A)$ of a Wiener-Hopf type operator $A$ in dimension one,
in 1982 H. Widom conjectured a multi-dimensional generalization of that formula for
a pseudo-differential operator
$A$ with a symbol $a(\bx, \bxi)$ having jump discontinuities in both variables.
In 1990 he proved the conjecture for the special case when the
jump in any of the two variables
occurs on a hyperplane.
The present paper gives a proof of Widom's Conjecture under the assumption that
the symbol has jumps in both variables on arbitrary smooth bounded surfaces.
\end{abstract}

\maketitle

%
%--------------------------------------------------------------------------------
%--------------------------------------------------------------------------------
%
%

\tableofcontents

%%%%%%%%%%%%%%%%%%%%%%%%%%%%%%%%%%%%%%%%%%%%%%%%%%%%%%%%%%%%%%%%%%%%%%%%%%%%%%%%%%%%%%%%%%%

\section{Introduction}

For two domains $\L, \Om\subset\R^d$, $d\ge 1$
consider in $\plainL2(\R^d)$  the operator defined by the formula
\begin{equation}\label{talpha:eq}
(\tilde T_\a(a) u)(\bx) = \biggl(\frac{\a}{2\pi}\biggr)^{d}
\chi_{\L}(\bx)\int_{\Om} \int_{\R^d} e^{i\a \bxi\cdot(\bx-\by)}
a(\bx, \bxi) \chi_{\L}(\by) u(\by) d\by d\bxi, \ \a >0,
\end{equation}
for any Schwartz class function $u$, where
$\chi_{\L}(\ \cdot\  )$ denotes the characteristic function
of $\L$,
and $a(\ \cdot\ , \ \cdot\ )$ is a smooth function,
with an appropriate decay in both variables.  Clearly,
$\tilde T_\a$ is a pseudo-differential operator with a symbol
having discontinuities in both variables.
We are interested in the asymptotics of the trace
$\tr g(\tilde T_\a)$ as $\a\to\infty$ with a smooth function
$g$ such that $g(0) = 0$.
In 1982 H. Widom in \cite{Widom1} conjectured the asymptotic formula
\begin{equation}\label{widom:eq}
\tr g\bigl(\tilde T_\a(a)\bigr)
= \a^d \ \GW_0\bigl(g(a); \L, \Om\bigr)
+ \a^{d-1} \log\a \ \GW_1\bigl(\GA(g; a); \p\L, \p\Om\bigr) + o(\a^{d-1}\log\a),
\end{equation}
with the following coefficients. For any symbol $b = b(\bx, \bxi)$, any
domains $\L, \Om$ and any $\plainC1$-surfaces $S, P$, let
\begin{equation}\label{w0:eq}
\GW_0(b) = \GW_0(b; \L, \Om)
= \frac{1}{(2\pi)^d} \int_{\L}\int_{\Om} b(\bx, \bxi) d\bxi d\bx,
\end{equation}
\begin{equation}\label{w1:eq}
\GW_1(b) = \GW_1(b ; S, P) = \frac{1}{(2\pi)^{d-1}}\int_{S} \int_{P} b(\bx, \bxi)
| \bn_{S}(\bx)\cdot\bn_{P}(\bxi) |  dS_{\bxi} dS_{\bx},
\end{equation}
where $\bn_{S}(\bx)$ and $\bn_P(\bxi)$ denote the exterior unit normals to
$S$ and $P$ at the points $\bx$ and $\bxi$ respectively, and
\begin{equation}\label{GA:eq}
\GA(g; b) = \frac{1}{(2\pi)^2}\int_0^1 \frac{g(bt) - t g(b)}{t(1-t)} dt,\
\GA(g): = \GA(g; 1).
\end{equation}
The main objective of the paper is to prove the formula
\eqref{widom:eq} for a large class of functions $g$ and bounded domains $\L, \Om$.

The interest in the pseudo-differential operators with discontinuous
symbols goes back to the classical Szeg\H o formula for
the determinant of a Toeplitz matrix, see \cite{Sz} and \cite{GrSz}.
There exists a vast body of literature devoted to various non-trivial generalizations of
the Szeg\H o formula in dimension $d=1$,
and it is not our intention to review them here. Instead, we refer
to the monographs by A. B\"ottcher-B. Silbermann \cite{BotSil},
and by N.K. Nikolski \cite{Nikol} for the background reading,
T. Ehrhardt's paper \cite{Ehr} for a review of the pre-2001 results, and
the recent paper by P. Deift, A. Its, I. Krasovsky \cite{DIK}, for the latest results
and references.
%We are concerned with the multidimensional version.
A multidimensional generalization of the \textsl{continuous} variant
of the Szeg\H o formula was obtained by I.J. Linnik \cite{Linnik} and
H. Widom \cite{Widom4}, \cite{Widom3}.
% under different restrictions on the domain $\L$.
%
%
%A multidimensional generalization of the \textsl{continuous} variant
%of the Szeg\H o formula was obtained in \cite{Linnik},  \cite{Widom4}, \cite{Widom3}.
In fact, paper \cite{Widom3} addressed a more general problem:
instead of the determinant, suitable analytic functions
of the operator were considered, and instead of the scalar symbol matrix-valued symbols
were allowed:
for $\Om = \R^d$ and $a(\bx, \bxi) = a(\bxi)$ it was shown that
\begin{equation*}
\tr g(\tilde T_\a(a)) = \a^d V_0 + \a^{d-1} V_1 + o(\a^{d-1}),\ \ %V_0 = \GW_0(g(a)),
\end{equation*}
%for suitable analytic functions $g$,
with some explicitly computable coefficients $V_0, V_1$, such that
$V_0 = \GW_0(g(a))$ for the scalar case.
%To be more precise, the paper \cite{Widom3} handled the matrix-valued symbols.
%The case of scalar symbols was treated in the earlier paper \cite{Widom4}
%for $g(t) = \log(1+t)$, but this paper
%required some additional conditions on the domain $\L$.
Under some mild extra smoothness assumptions on the boundary $\p\L$,
R. Roccaforte (see \cite{Roc}) found the term of order $\a^{d-2}$
in the above asymptotics of $\tr g(\tilde T_\a(a))$.

The situation changes if we assume that $\L\not = \R^d$ and $\Om\not=\R^d$, i.e.
that the symbol has jump discontinuities in both variables, $\bx$ and $\bxi$.
As conjectured by H. Widom,
in this case the second term should be of
order $\a^{d-1}\log\a$, see formula \eqref{widom:eq}.
For $d=1$ this formula was proved by H. Landau-H. Widom \cite{Land_Wid} and H. Widom
\cite{Widom1}.
%He proved this formula for $d = 1$ in \cite{Widom1}.
For higher dimensions, the asymptotics \eqref{widom:eq} was proved
in \cite{Widom2} under the assumptions that one of the domains
is a half-space, and that $g$  is analytic in a disk of a sufficiently
large radius.
After this paper there have been just a few publications
with partial results. Using
an abstract version of the Szeg\H o formula with a remainder estimate, found
by A. Laptev and Yu. Safarov (see \cite{LS}, \cite{LapSaf}),
D. Gioev (see \cite{G1, G2}) established a sharp bound
\begin{equation}\label{sharp:eq}
\tr g(\tilde T_\a(a)) - \a^d\GW_0(g(a))  = O(\a^{d-1}\log\a).
\end{equation}
In \cite{GiKl} D. Gioev and I.Klich observed a connection between the formula
\eqref{widom:eq} and the behaviour of the entanglement
entropy for free Fermions in the ground state.
As explained in \cite{GiKl}, the studied
entropy is obtained as $\tr h(\tilde T_\a)$
with some bounded domains $\L, \Om$, the symbol $a(\bx, \bxi) = 1$, and
the function
\begin{equation}\label{entropy:eq}
h(t) = -t\log t - (1-t) \log (1-t), t\in (0, 1).
\end{equation}
Since $h(0) = h(1) = 0$, the leading term,
i.e. $\GW_0(h(1))$, vanishes, and the conjecture \eqref{widom:eq} gives
the $\a^{d-1}\log\a$-asymptotics of the trace, which coincides with the
expected quasi-classical behaviour of the entropy. However, the formula
\eqref{widom:eq} is not justified for non-smooth functions, and in particular
for the function \eqref{entropy:eq}.  Instead,
in the  recent paper \cite{HLS}
R. Helling, H. Leschke and W. Spitzer
proved \eqref{widom:eq} for a quadratic $g$. With $g(t) = t - t^2$ this
gives the asymptotics of the particle number variance, which provides
a lower bound of correct order for the entanglement
entropy.

The operators of the form \eqref{talpha:eq} also play a
 role in Signal Processing. Although the main object there is
band-limited functions of one variable, in \cite{Slep} D. Slepian considered
some multi-dimensional generalizations.
In particular, he derived asymptotic formulas
for the eigenvalues and eigenfunctions of $\tilde T_\a(1)$ for the special case
when both $\L$ and $\Om$ are balls in $\R^d$. Some of those results
are used in \cite{Slep1}.
These results, however,
do not allow to study the trace $\tr g(\tilde T_\a(1))$.

The main results of the present paper
are Theorems \ref{main_anal:thm} and \ref{main_s:thm}. They
establish asymptotic formulas of the type \eqref{widom:eq} for
the operator $T(a) = T_\a(a)$, defined in \eqref{tnew:eq}, which is slightly different from
$\tilde T_\a(a)$, but as we shall see later,
the difference does not affect the first two terms of
the asymptotics \eqref{widom:eq}.
Theorem \ref{main_anal:thm}
proves formula \eqref{widom:eq} for functions $g$
analytic in a disk of sufficiently large radius. Theorem \ref{main_s:thm}
proves \eqref{widom:eq} for the real part of $T(a)$ with an arbitrary
$\plainC\infty$-function $g$.

The proof comprises the following main ingredients:
\begin{itemize}
\item[1.]
Trace class estimates for pseudo-differential operators with
discontinuous symbols,
\item[2.]
Analysis of the problem for $d=1$,
\item[3.]
A geometrical estimate,
\item[4.]
Reduction of the initial problem to the case $d=1$.
\end{itemize}
The most important and difficult is Step 4.
Here we divide the domain $\L$ into a \textsl{boundary layer}, which contributes to the first
and second terms in \eqref{widom:eq}, and the \textsl{inner part},
which matters only for the first term.
Then we construct a suitable partition of unity subordinate to this covering, which is given
by the functions $q^{\downarrow} = q^{\downarrow}(\bx)$ and $q^{\uparrow}
= q^{\uparrow}(\bx)$ respectively.
An interesting feature of the problem
is that the thickness of the boundary layer does not depend on $\a$.
The asymptotics of $\tr \bigl(q^{\uparrow} g(T_\a)\bigr), \a\to\infty,$
do not feel the boundary, and
relatively standard quasi-classical considerations lead to the formula
\begin{equation}\label{qup:eq}
\tr \bigl(q^{\uparrow} g(T_\a)\bigr)
= \a^d \GW_0(q^{\uparrow} g(a)) + O(\a^{d-1}), \a\to\infty.
\end{equation}
To handle the trace $\tr \bigl(q^{\downarrow} g(T_\a)\bigr)$
we construct an appropriate covering of the boundary layer
by open sets of a specific shape. For each of these sets the boundary $\p\L$ is approximated
by a hyperplane, which makes it possible to view the operator $T_\a$ as a PDO on the
boundary hyperplane,
whose symbol is an operator of the same type, but in dimension one. This reduction
brings us to Step  2 of the plan.
The $1$-dim situation was studied in H. Widom's paper \cite{Widom1}. Although its results
are not directly applicable, the method developed  there allows us to get the required asymptotics.
These ensure that
\begin{equation}\label{qd:eq}
\tr \bigl(q^{\downarrow} g(T_\a)\bigr)
= \a^d \GW_0(q^{\downarrow} g(a))
+ \a^{d-1}\log\a \ \GW_1\bigl(\GA(g; a)\bigr) + o(\a^{d-1}\log\a), \a\to\infty.
\end{equation}
Adding up \eqref{qup:eq} and \eqref{qd:eq}, gives \eqref{widom:eq}.

In the reduction to the $1$-dim case an important role is played by a result of geometrical
nature, which is listed above as the third ingredient. It is loosely described as follows.
Representing $\bxi\in\R^d$ as $\bxi = (\hat\bxi, t)$,
with $\hat\bxi = (\xi_1, \xi_2, \dots, \xi_{d-1})$, define for each $\hat\bxi\in\R^{d-1}$ the
set
\begin{equation*}
\Om(\hat\bxi) = \{t : (\hat\bxi, t)\in\Om\}\subset\R.
\end{equation*}
If it is non-empty, then it is at most countable
union of disjoint open intervals in $\R$ whose
length we denote by  $\tilde\rho_j$, $j= 1, 2, \dots$.
The important observation is that under appropriate restrictions on the smoothness of
the boundary $\p\Om$, the function
\begin{equation}\label{tildem:eq}
\tilde m_\d(\hat\bxi) = \sum_{j} \tilde \rho_j^{-\d},
\end{equation}
belongs to $\plainL1(\R^{d-1})$ for all $\d\in (0, 2)$. The precise formulation
of this result is given in Appendix 1.

 From the technical viewpoint Steps 2 and 4 are based on the trace class estimates derived
 at Step 1. In order to work with discontinuous symbols
 we also establish convenient estimates for smooth ones. The emphasis is on the estimates
 which allow one to control explicitly the dependence on the parameter $\a$, and on the
 scaling properties of the symbols.

At this point it is appropriate to compare our proof
with H. Widom's paper \cite{Widom2}, where
\eqref{widom:eq} was justified for the case when $\L$ (or $\Om$) was a half-space.
In fact, our four main steps are the same as in \cite{Widom2}. However, in \cite{Widom2}
the relative weight of these ingredients was different. If $\L$ is a half-space,
then the reduction to the $1$-dim case (i.e. Step 4) is almost
immediate whereas in the present paper, for general $\L$ such a reduction is a major issue.
 In \cite{Widom2} at Step 3 it was sufficient to have the geometric estimate
for $\d=1$. In the present paper it is crucial to have such an estimate
for $\d >1$. Moreover, trace class estimates were derived in \cite{Widom2}
under the assumption that $\L$ was a half-space, which
is clearly insufficient for our purposes.
  As far as the $1$-dim asymptotics are concerned (i.e. Step 2), our estimates
  are perhaps somewhat more detailed, since apart from the parameter $\a$,
  they allow one to monitor the dependence on the scaling parameters as well.

The detailed structure of the paper is described at the end of Section \ref{main:sect}.

\vskip 0.5cm

\noindent
\textbf{Some notational conventions.}
We conclude the Introduction by fixing some basic notations which will be used
throughout the paper.
For $\bx\in\R^d$ we denote $\lu \bx\ru = (1+|\bx|^2)^{\frac{1}{2}}$.
Very often we split $\bx=(x_1, x_2, \dots, x_d)$ in its components as follows:
\begin{equation*}
\bx = (\hat\bx, x_d),\ \hat\bx = (x_1, x_2, \dots, x_{d-1}),
\end{equation*}
and for some $l = 1, 2, \dots, d,$
\begin{equation*}
\overc\bx = (x_1, x_2, \dots, x_{l-1}, x_{l+1}, \dots, x_d).
\end{equation*}
The notation $B(\bx, R)$ is used for the open
ball in $\R^d$ of radius $R>0$, centered  at $\bx\in\R^d$.
For some $\rho >0$ let
\begin{equation}\label{cube:eq}
\CC^{(n)}_{\rho} = (-2\rho, 2\rho)^n
\end{equation}
be the $n$-dimensional open cube.

The characteristic function of the domain $\L\subset\R^d, d\ge 1$,
is denoted by $\chi_{\L}=\chi_{\L}(\bx)$. To avoid cumbersome
notation we write
$\chi_{\bz, \ell}(\bx) := \chi_{B(\bz, \ell)}(\bx)$.

For a function $u = u(\bx)$, $\bx\in\R^d$ its Fourier transform is defined as follows:
\[
\hat u(\bxi) = \frac{1}{(2\pi)^{\frac{d}{2}}}
\int e^{-i \bx\cdot\bxi} u(\bx) d\bx.
\]
The integrals without indication of the domain of integration are taken over the
entire Euclidean space $\R^d$.

The notation $\GS_p, 0 < p\le \infty$ is used for the standard Schatten-von Neumann
classes of compact operators in a separable Hilbert space, see e.g.
\cite{BS}, \cite{Simon}. In particular,
$\GS_1$ is the trace class, and $\GS_2$ is the Hilbert-Schmidt class.
 Unless otherwise stated the underlying Hilbert space is assumed to be
 $\plainL2(\R^d)$.

By $C, c$ (with or without indices) we denote various positive constants
whose precise value is of no importance.

\vskip 0.5cm

\noindent

\textbf{Acknowledgments.} Part of this paper was written during my stay at
the Erwin Schr\"odinger Institute, Vienna
in July 2009. I thank the organizers of the Programme
``Topics in Spectral Theory" for the invitation
and to the staff of the  Institute for their hospitality.

I am grateful to H. Leschke, R. Helling and W. Spitzer for useful discussions and their cordial
hospitality at Erlangen in September 2009.
Thanks go to A. Laptev, Yu. Safarov and M. Shubin for
stimulating discussions, and to A. B\"ottcher, P. Deift,  B. Helffer, A. Its, N. Lerner
and G. Rozenblum,
for pointing out to me some useful references. I am grateful to A. B\"ottcher and
J. Oldfield for reading parts of the text and correcting some inconsistencies.

This work was supported in part by EPSRC grants EP/F029721/1 and
EP/D00022X/2.

%%%%%%%%%%%%%%%%%%%%%%%%%%%%%%%%%%%%%%%%%%%%%%%%%%%%%%%%%%%%%%%%%%%%%%%%%%%%%%%%%

\section{Main result}\label{main:sect}

\subsection{Definitions and main results}
In this paper we need several types of pseudo-differential operators, depending on the
parameter $\a>0$.
For the  \textit{symbol} $a = a(\bx, \bxi)$,
\textit{amplitude} $p = p(\bx, \by, \bxi)$,
and any function $u$ from the Schwartz class on $\R^d$ we define
\begin{equation}\label{pdoampl:eq}
(\op^a_\a p) u(\bx) = \biggl(\frac{\a}{2\pi}\biggr)^d
\int\int e^{i\a(\bx-\by)\bxi} p(\bx, \by; \bxi) u(\by) d\bxi d\by,
\end{equation}
\begin{equation}\label{pdoleft:eq}
(\op^l_\a a) u(\bx) = \biggl(\frac{ \a}{2\pi}\biggr)^d
\int\int e^{i\a(\bx-\by)\bxi} a(\bx, \bxi) u(\by) d\bxi d\by,
\end{equation}
\begin{equation}\label{pdoright:eq}
(\op^r_\a a) u(\bx) = \biggl(\frac{\a}{2\pi}\biggr)^d
\int\int e^{i\a(\bx-\by)\bxi} a(\by, \bxi) u(\by) d\bxi d\by.
\end{equation}
If the function $a$ depends only on $\bxi$, then the operators
$\op_\a^l(a), \op_\a^r(a)$ and $\op_\a^a(a)$ coincide with each other,
and we simply write $\op_\a(a)$.
Later we formulate conditions on $a$ and $p$ which ensure
boundedness of the above operators uniformly in the parameter $\a\ge 1$.
Let $\L, \Om$ be two domains in $\R^d$, and let $\chi_{\L}(\bx)$, $\chi_{\Om}(\bxi)$
be their characteristic functions. We always use the notation
\begin{equation*}
P_{\Om,\a} = \op_\a(\chi_{\Om}).
\end{equation*}
We study the operator
\begin{equation}\label{tnew:eq}
T_\a(a) = T_\a(a; \L, \Om) = \chi_{\L} P_{\Om, \a}\op^l_\a(a) P_{\Om, \a}\chi_{\L},
\end{equation}
and its symmetrized version:
\begin{equation*}
S_\a(a) = S_\a(a; \L, \Om) = \chi_{\L} P_{\Om, \a}\ \re\op^l_\a(a)\  P_{\Om, \a}\chi_{\L}.
\end{equation*}
%If necessary, sometimes we also reflect the dependence on the parameter $\a$ and
%write $T_\a(a; \L, \Om)$, $S_\a(a; \L, \Om)$.
Note that $T_\a(a)$ differs from the
operator \eqref{talpha:eq} by the presence of
an extra projection $P_{\Om, \a}$ on the left of
$\op^l_\a(a)$. As we shall see later in Section \ref{nonsmooth1:sect}, this difference
does not affect the first two terms of the asymptotics \eqref{widom:eq}.

Let us now specify the class of symbols and amplitudes used throughout the paper.
We denote by $\BS^{(n_1, n_2, m)}$ the set of all (complex-valued)
functions $p = p(\bx, \by, \bxi)$, which are bounded together with their
partial derivatives up to order
$n_1$ w.r.t. $\bx$, $n_2$ w.r.t. $\by$ and $m$ w.r.t. $\bxi$.
It is convenient to define the norm in this class in the following way.
For arbitrary numbers $\ell>0$ and $\rho >0$ define
\begin{equation}\label{norm:eq}
\SN^{(n_1, n_2, m)}(p; \ell, \rho)
= \underset{\substack
{0\le n\le n_1\\
0\le k\le n_2\\
0\le r\le m}}
\max \ \underset{\bx, \by, \bxi}
\sup \ell^{n+k} \rho^{r}
|\nabla_{\bx}^{n}\nabla_{\by}^k\nabla_{\bxi}^r p(\bx, \by, \bxi)|.
\end{equation}
Here we use the notation
\begin{equation*}
|\nabla^l f(\bt)|^2 =
\sum_{j_1, j_2, \dots, j_l = 1}^r |\p_{j_1} \p_{j_2}\ \cdots\ \p_{j_l} f(\bt)|^2
\end{equation*}
for a function $f$ of the variable $\bt\in \R^r$.
The presence of the parameters $\ell$, $\rho$ allows one to consider
amplitudes with different scaling properties.

In the same way we introduce the classes $\BS^{(n, m)}$ (resp. $\BS^{(m)}$)
of all (complex-valued)
functions $a = a(\bx, \bxi)$ (resp. $a = a(\bxi)$),
which are bounded together with their
partial derivatives up to order
$n$ w.r.t. $\bx$, and $m$ w.r.t. $\bxi$. The norm $\SN^{(n, m)}(a; \ell, \rho)$
(resp. $\SN^{(m)}(a; \rho)$)
is defined in a way similar to \eqref{norm:eq}.
Note the straightforward inequality: if $a\in\BS^{(n, m)}, b\in \BS^{(n, m)}$, then
$ab\in\BS^{(n, m)}$ and
\[
\SN^{(m, n)}(ab; \ell, \rho)\le C_{m, n}
\SN^{(m, n)}(a; \ell, \rho)
\SN^{(m, n)}(b; \ell, \rho).
\]
As we show later, if the amplitude $p$ and/or symbol $a$
belong to an appropriate class $\BS$,
then the PDO's \eqref{pdoampl:eq}-- \eqref{pdoright:eq} are bounded.

Let $\boldM$ be a non-degenerate linear transformation, and let
$\boldk, \boldk_1\in\R^d$ be some vectors.
By $\boldA = (\boldM, \boldk)$ we denote the affine transformation
$\boldA\bx = \boldM + \boldk$.
A special role is played by the \textit{Euclidean isometries},
i.e. by the affine transformations of the form $\boldE = (\boldO, \boldk)$, where
$\boldO$ is an orthogonal transformation.  The set of all Eucledian
isometries on $\R^d$ is denoted by $E(d)$.
Let us point out some useful unitary equivalence for the operators
\eqref{pdoampl:eq} - \eqref{pdoright:eq}.
For the affine transformation $(\boldM, \boldk)$
define the unitary operator
$U = U_{\boldM, \boldk}: \plainL2(\R^d)\to\plainL2(\R^d)$ by
\begin{equation*}
(U_{\boldM, \boldk}u)(\bx) = \sqrt{|\det\boldM|} \
u(\boldM\bx+\boldk), u\in\plainL2(\R^d).
\end{equation*}
It is straightforward to check that for an arbitrary $\boldk_1\in\R^d$,
\begin{equation}\label{orthogonal1:eq}
\begin{cases}
U_{\boldM, \boldk} e^{-i\a\bx\cdot\boldk_1}
\op_\a^a(b) e^{i\a\bx\cdot\boldk_1}
U_{\boldM, \boldk}^{-1} =
\op_\a^a (b_{\boldM, \boldk, \boldk_1}),\\[0.2cm]
b_{\boldM, \boldk, \boldk_1}(\bx, \by, \bxi)
= b\bigl(\boldM\bx+\boldk, \boldM\by+\boldk, (\boldM^T)^{-1} \bxi + \boldk_1\bigr),
\end{cases}
\end{equation}
and
\begin{equation}\label{chilambda:eq}
\begin{cases}
U_{\boldM, \boldk} \chi_{\L} U_{\boldM, \boldk}^{-1} = \chi_{\L_{\boldM, \boldk}},\
U_{\boldM, \boldk} e^{ -i\a\bx\cdot\boldk_1}
P_{\Om, \a} e^{ i\a\bx\cdot\boldk_1} U_{\boldM, \boldk}^{-1}
= P_{\Om_{\boldM, \boldk_1}^T, \a},\\[0.2cm]
\L_{\boldM, \boldk} = \boldM^{-1}(\L-\boldk),\
\Om_{\boldM, \boldk_1}^T = \boldM^T(\Om-\boldk_1).
\end{cases}
\end{equation}
For the operator $T(a)$ this implies that
\begin{equation}\label{unitary_equivalence:eq}
U_{\boldM, \boldk} e^{ -i\a\bx\cdot\boldk_1}
T(a; \L, \Om)e^{i\a\bx\cdot\boldk_1} U_{\boldM, \boldk}^{-1}
= T(a_{\boldM, \boldk, \boldk_1}; \L_{\boldM, \boldk}, \Om_{\boldM, \boldk_1}^T).
\end{equation}
Note that the asymptotic coefficients
$\GW_0$ and $\GW_1$ are invariant with respect to
affine transformations, i.e.
\begin{equation}\label{gw_invariance:eq}
\begin{cases}
\GW_0(b; \L, \Om) =
\GW_0(b_{\boldM, \boldk, \boldk_1};
\L_{\boldM, \boldk}, \Om^T_{\boldM, \boldk_1}),\\[0.2cm]
\GW_1(b; \p\L, \p\Om) =
\GW_1(b_{\boldM, \boldk, \boldk_1}; \p\L_{\boldM, \boldk},
\p\Om^T_{\boldM, \boldk_1}).
\end{cases}
\end{equation}
The first of
the relations \eqref{gw_invariance:eq} is immediately checked by changing variables
under the integral \eqref{w0:eq}. The second one is proved in Appendix 4,
Lemma \ref{linear_invariance:lem}.

As one particular type of a linear transformation, it is useful to single out the case
when $\boldM = \ell\boldI$ for some $\ell>0$, and $\boldk = \bzero$,
i.e. $\boldM$ is a scaling transformation.
In this situation we denote
\begin{equation*}
\bigl(W_{\ell}u\bigr)(\bx)
= \bigl(U_{\boldM, \bzero} u\bigr)(\bx)
= \ell^{\frac{d}{2}} u(\ell \bx).
\end{equation*}
Then a straightforward calculation gives for any $\rho >0$:
\begin{equation}\label{unitary:eq}
W_{\ell}\op_\a^a (b) W_{\ell}^{-1} = \op^a_{\b}(b_{\ell, \rho}), \
b_{\ell, \rho}(\bx, \by, \bxi) = b(\ell\bx, \ell\by, \rho\bxi),\
\b = \a\ell\rho.
\end{equation}
In particular,
\begin{equation}\label{unitary_domain:eq}
W_{\ell} \chi_{\L} W_{\ell}^{-1} = \chi_{\L_{\ell}},\
\L_{\ell} = \ell^{-1}\L,\
\end{equation}
It is important that the  norm \eqref{norm:eq} is invariant
under certain linear transformations.
Note first of all that
\begin{equation}\label{scale:eq}
\SN^{(n_1, n_2, m)}(p; \ell, \rho)
= \SN^{(n_1, n_2, m)}(p_{\ell_1, \rho_1}; \ell\ell_1^{-1}, \rho\rho_1^{-1}),
\end{equation}
for arbitrary positive $\ell, \ell_1, \rho, \rho_1$.
Moreover, the norm is also invariant under Eucledian isometries:
\begin{equation}\label{orthogonal:eq}
\SN^{(n_1, n_2, m)}(p; \ell, \rho)
= \SN^{(n_1, n_2, m)}(p_{\BO, \boldk, \boldk_1}; \ell, \rho).
\end{equation}
Sometimes we refer to $\ell$ and $\rho$ as \textsl{scaling parameters}.

Now we can specify the classes of domains which we study.
We always assume that $\L$ and $\Om$ are domains with
smooth boundaries in the standard sense.  However, for the reference convenience and
to specify the precise conditions on the objects involved, we state our assumptions
explicitly.

\begin{defn}\label{domains:defn}
We say that a domain $\G\subset\R^d, d\ge 2$,
is \textsl{a $\plainC{m}$-graph-type domain}, with some $m\ge 1$,
if one can find a real-valued
function $\Phi\in\plainC{m}(\R^{d-1})$,
 with the properties
\begin{equation}\label{propertyphi:eq}
\begin{cases}
 \Phi(\hat{\mathbf{0}}) = 0,\\
\nabla\Phi\ \ \textup{is uniformly bounded on}\ \  \R^{d-1},\\
\nabla\Phi \ \ \textup{is uniformly continuous on}\ \  \R^{d-1},
\end{cases}
\end{equation}
and some transformation
$\boldE = (\boldO, \boldk)\in E(d)$ such that
\begin{equation*}
\boldE^{-1}\G = \{\bx: x_d > \Phi(\hat\bx)\},\
\hat\bx = (x_1, x_2, \dots, x_{d-1}).
\end{equation*}
In this case we write $\G = \G(\Phi; \boldO, \boldk)$ or
$\G = \G(\Phi)$, if the omission of the
dependence on $\boldE$ does not lead to confusion.
\end{defn}

We often use the notation
\begin{equation}\label{mphi:eq}
M_{\Phi} = \|\nabla \Phi\|_{\plainL\infty}.
\end{equation}
It is clear that  $\G(\Phi; \boldO, \boldk) = \boldE \G(\Phi; \boldI, \mathbf 0)$
with $\boldE = (\boldO, \boldk)$. The point $\boldk$ is on the boundary of the domain
$\G(\Phi; \boldO, \boldk)$.

If $\L = \G(\Phi; \boldO, \boldk)$, then the domain $\L_{\ell}$ (see definition
\eqref{unitary_domain:eq}) has the form
\begin{equation}\label{gamma_lambda:eq}
\L_{\ell} = \G(\Phi_\ell; \boldO, \ell^{-1}\boldk),\
\Phi_{\ell}(\hat\bx) = \ell^{-1} \Phi(\ell\hat\bx).
\end{equation}
Note that the value \eqref{mphi:eq} is invariant
under scaling:
\begin{equation}\label{scale_domain:eq}
\|\nabla\Phi\|_{\plainL\infty} = \|\nabla\Phi_\ell\|_{\plainL\infty}.
\end{equation}

In what follows we extensively use the relations
\eqref{orthogonal1:eq}, \eqref{chilambda:eq}
and \eqref{unitary_equivalence:eq} in order reduce the domains or symbols to a more
convenient form. For these purposes, let us make a note of the following
elementary property for the domain $\L = \G(\Phi; \boldI, \bzero)$. According to
\eqref{chilambda:eq}, for any $\boldk\in\p\L$ we have
\begin{equation}\label{trans:eq}
\L_{\boldI, \boldk} = \G(\Phi_{\boldk}; \boldI, \bzero),\
\textup{where}\ \L = \G(\Phi; \boldI, \bzero),\
\Phi_{\boldk}(\hat\bx) = \Phi(\hat\bx+\hat\boldk) - k_d.
\end{equation}
Clearly, $\Phi_{\boldk}(\hat\bzero) = 0$ and
$\|\nabla\Phi\|_{\plainL\infty} = \|\nabla\Phi_{\boldk}\|_{\plainL\infty}$.

In the next definition we introduce general $\plainC{m}$-domains.
Let $B(\bx, r) = \{\bv\in\R^d: |\bx-\bv|<r\}$ be
the ball of radius $r>0$ centered at $\bx$.

\begin{defn}\label{local_domains:defn}
Let $\L\subset\R^d$, $d\ge 2$
be a domain, $\bw\in\R^d$ be a vector and $R>0$ be a number.

\begin{enumerate}
\item
For a $\bw\in\p\L$ we say that in the ball $B(\bw, R)$ the domain $\L$
is represented by the $\plainC{m}$-graph-type
domain $\G = \G(\Phi; \boldO, \bw)$, $m\ge 1$,  if
there is a number $R = R_{\bw}>0$, such that
\begin{equation}\label{local_graph:eq}
\L\cap B(\bw, R) = \G\cap B(\bw, R).
\end{equation}
\item
For a $\bw\in \L$ we say that in the ball $B(\bw, R)$ the domain $\L$
is represented by $\R^d$, if
there is a number $R = R_{\bw}>0$, such that
\begin{equation*}
\L\cap B(\bw, R) = B(\bw, R).
\end{equation*}
\item
The domain $\L$, is said to be $\plainC{m}, m\ge 1$, if
for each point $\bw\in\p\L$ there is a number $R = R_{\bw}>0$,
such that in the ball $B(\bw, R)$ the domain $\L$ is represented by
a $\plainC{m}$-graph-type domain $\G(\Phi; \boldO, \bw)$ with some
 $\plainC{m}$-function
$\Phi=\Phi_{\bw}:\R^{d-1}\to\R$, satisfying \eqref{propertyphi:eq},
and some orthogonal transformation
$\boldO = \boldO_\bw$.
In this case we also say that the boundary
$\p\L$ is a $\plainC{m}$-surface.
\end{enumerate}
\end{defn}

\underline{The next two theorems represent the main results of the paper:}

\begin{thm}\label{main_anal:thm}
Let $\L, \Om\subset\R^d$, $d\ge 2$ be bounded
domains in $\R^d$ such that
$\L$ is $\plainC1$  and $\Om$ is $\plainC3$.
Let  $a = a(\bx, \bxi)$ be a symbol with the property
\begin{equation}\label{dersym:eq}
\max_{\substack{0\le n\le d+2\\ 0\le m\le d+2}}
\sup_{\bx, \bxi}
|\nabla_{\bx}^n \nabla_{\bxi}^m a(\bx, \bxi)|
< \infty,
\end{equation}
supported on the set $B(\bz, \ell)\times B(\bmu, \rho)$
with some $\bz, \bmu\in\R^d$ and $\ell, \rho >0$.
Let $g$ be a function analytic in the disk of radius
$R = C_1\SN^{(d+2, d+2)}(a; \ell, \rho)$, such that $g(0) = 0$.
If the constant $C_1$ is sufficiently large, then
\begin{align}\label{main_anal:eq}
\tr  g(T(a))
= &\ \a^d \GW_0(g(a); \L, \Om)\notag\\[0.2cm]
&\ + \a^{d-1} \log\a\  \GW_1(\GA(g; a); \p\L, \p\Om) + o(\a^{d-1}\log\a),
\end{align}
as $\a\to\infty$.
\end{thm}

For the self-adjoint operator $S(a)$ we have a wider choice of
functions $g$:

\begin{thm}\label{main_s:thm}
Let $\L, \Om\subset\R^d$, $d\ge 2$ be bounded
domains in $\R^d$ such that
$\L$ is $\plainC1$  and $\Om$ is $\plainC3$.
Let  $a = a(\bx, \bxi)$ be a symbol satisfying \eqref{dersym:eq}
with a compact support
in both variables.
Then for any function $g\in\plainC\infty(\R)$, such that
$g(0) = 0$, one has
\begin{align}\label{main_s:eq}
\tr  g(S(a))
= &\ \a^d \GW_0(g(\re a); \L, \Om)\notag\\[0.2cm]
&\ + \a^{d-1} \log\a\  \GW_1(\GA(g; \re a); \p\L, \p\Om) + o(\a^{d-1}\log\a),
\end{align}
as $\a\to\infty$.
\end{thm}

\begin{rem}
Let us make some comments on the integral \eqref{GA:eq}, which enters
the coefficient $\GW_1$ in the above theorems.
Note that the integral \eqref{GA:eq} is finite for any smooth function $g$
such that $g(0) = 0$ (see Lemmas \ref{GA_est_anal:lem}, \ref{GA_est_smooth:lem} for
the estimates).
Clearly, $\GA$ is linear in $g$, i.e.
$\GA(g +\tilde g; b) = \GA(g; b) + \GA(\tilde g; b)$.
For the function $g_p(t) = t^p$, $p=1, 2, \dots$, we have
$\GA(g_1; b) = 0$ and $\GA(g_p,; b) = b^p \GA(g_p)$.
\end{rem}

\begin{rem}
It would be natural to expect that the variables $\bx$ and $\bxi$ in the operator
$T(a)$ have ``equal rights". Indeed, it was shown in \cite{Widom2}, p. 173,  by an
elementary calculation, that the roles of $\bx, \bxi$ are interchangeable.
On the other hand, the conditions  on $\L$ and $\Om$ in the main theorems
above, are clearly asymmetric.
At present it is not clear how to rectify this drawback.
\end{rem}

\begin{rem}
Denote by $n(\l_1, \l_2; \a)$ with $\l_1\l_2 >0, \l_1 < \l_2$ the number of eigenvalues of
the operator $S(a)$ which are greater than $\l_1$ and less than $\l_2$.  In other words,
\[
n(\l_1, \l_2; \a) = \tr \chi_I(S(a)),\ \ I = (\l_1, \l_2).
\]
Since the interval $I$ does not contain the point $0$, this quantity is finite.
Theorem \ref{main_s:thm} can be used to find the leading term of the asymptotics of
the counting function $n(\l_1, \l_2; \a)$,
by approximating the characteristic function $\chi_I$ with smooth functions $g$.
Suppose for instance that $a_-< a(\bx, \bxi)<a_+$, $\bx\in\L, \bxi\in\Om$,
with some positive constants $a_-, a_+$, and that $[a_-, a_+]\subset I$.
Then it follows from Theorem \ref{main_s:thm} that
\begin{equation*}
n(\l_1, \l_2; \a) = \biggl(\frac{\a}{2\pi}\biggr)^d |\L| \ |\Om|
+ \a^{d-1} \log\a  \  \GW_1\bigl(\GA(\chi_I; a)\bigr)  + o(\a^{d-1}\log\a).
\end{equation*}
A straightforward calculation shows that
\[
\GA\bigl(\chi_I; a(\bx, \bxi)\bigr) = \frac{1}{(2\pi)^2}
\log \biggl(\frac{a(\bx, \bxi)}{\l_1} - 1\biggr).
\]
Another interesting case is when $[\l_1, \l_2]\subset (0, a_-)$.
This guarantees that $\GW_0(\chi_I(a)) = 0$, and the asymptotics of $n(\l_1, \l_2; \a)$
are described by the second term in \eqref{main_s:eq}:
\[
n(\l_1, \l_2; \a)
=  \a^{d-1}\log\a \ \GW_1\bigl(\GA(\chi_I; a)\bigr)  + o(\a^{d-1}\log\a),\  \a\to\infty.
\]
An elementary calculation gives:
\[
\GA\bigl(\chi_I; a(\bx, \bxi)\bigr) = \frac{1}{(2\pi)^2}
\log \frac{\l_2(a(\bx, \bxi)-\l_1)}{\l_1(a(\bx, \bxi)-\l_2)}.
\]
\end{rem}

\vskip 0.3cm

The proof of Theorems \ref{main_anal:thm} and
\ref{main_s:thm} splits in two unequal parts. The crucial
and the most difficult part is to justify the asymptotics for a polynomial $g$:

\begin{thm}\label{main:thm}
Let $\L, \Om\subset\R^d$, $d\ge 2$ be bounded
domains in $\R^d$ such that
$\L$ is $\plainC1$  and $\Om$ is $\plainC3$.
Let  $a = a(\bx, \bxi)$ be a symbol satisfying \eqref{dersym:eq}
with a compact support
in both variables. Then for $g_p(t) = t^p, p= 1, 2,  \dots, $
\begin{align}\label{main:eq}
\tr  g_p(T(a))
= &\ \a^d \GW_0(g_p(a); \L, \Om)\notag\\[0.2cm]
&\ + \a^{d-1} \log\a\  \GW_1(\GA(g_p; a); \p\L, \p\Om) + o(\a^{d-1}\log\a),
\end{align}
as $\a\to\infty$.
If $T(a)$ is replaced with $S(a)$, then the same formula holds
with the symbol $a$ replaced by $\re a$ on the right-hand side.
\end{thm}

Once this theorem is proved, the asymptotics can be closed
with the help of the sharp bounds \eqref{gioev0:eq} and \eqref{gioev:eq},
which were derived in \cite{G1}, \cite{G2} using the abstract version of the
Szeg\H o formula with a remainder estimate
obtained in \cite{LS} (see also \cite{LapSaf}).

\subsection{Asymptotic coefficient $\GA(g; b)$}
Here we provide some simple estimates for
the coefficient $\GA$ defined in \eqref{GA:eq}.

\begin{lem}\label{GA_est_anal:lem}
Suppose that the analytic function $g$ is given by the series
\begin{equation*}
g(z) = \sum_{m=0}^\infty \om_m z^m
\end{equation*}
with a radius of convergence $R >0$. Let
\begin{equation*}
g^{(1)}(t) =  \sum_{m=2}^\infty(m-1)|\om_m| t^{m-1},\ |t|\le R.
\end{equation*}
Then  for any $b, |b|<R$ the following estimate holds:
\begin{equation}\label{GA_est_anal:eq}
|\GA(g; b)|\le \frac{1}{(2\pi)^2} |b| g^{(1)}(|b|).
\end{equation}
\end{lem}

\begin{proof}
Consider first $g_m(z) = z^m, m\ge 2$, so that
\[
(2\pi)^2|\GA(g_m; b)| \le |b|^m\int_0^1 \frac{t-t^m}{t(1-t)} dt
\le (m-1)|b|^m.
\]
Thus
\[
|\GA(g; b)|\le \sum_{m=2}^\infty |\om_m||\GA(g_m; b)|\le
\frac{1}{(2\pi)^2} \sum_{m=2}^\infty (m-1)|\om_m| |b|^m,
\]
which leads to \eqref{GA_est_anal:eq}.
\end{proof}

\begin{lem}\label{GA_est_smooth:lem}
Suppose that  $g\in\plainC1(\R)$ and $g(0) = 0$.
Then for any $b\in\R$ the following estimate holds:
\begin{equation}\label{GA_est_smooth:eq}
|\GA(g; b)|\le \frac{1}{\pi^2} |b| \|g'\|_{\plainL\infty(-|b|, |b|)}.
\end{equation}
\end{lem}

\begin{proof} Denote $ w=\|g'\|_{\plainL\infty(-|b|, |b|)}$.
Since $|g(t)|\le w |t|$ for all $|t|\le|b|$, we have
\[
\biggl|\int_0^{\frac{1}{2}}\frac{g(bt) - tg(b)}{t(1-t)} dt\biggr|
\le 2 w|b|\int_0^{\frac{1}{2}} \frac{1}{1-t} dt\le 2 w |b|.
\]
Similarly,
\[
\biggl|\int_{\frac{1}{2}}^1\frac{g(bt) - tg(b)}{t(1-t)} dt\biggr|
\le 2 w|b|\int_{\frac{1}{2}}^1 \frac{1}{t} dt\le 2 w |b|.
\]
Adding up the two estimates, one gets \eqref{GA_est_smooth:eq}.
\end{proof}

\subsection{Plan of the paper}
We begin with estimates for norms and trace norms of the pseudo-differential
operators with smooth symbols, see Sect. \ref{smooth:sect}. Information about
various classes of compact operators can be found in \cite{GK}, \cite{BS}, \cite{Simon}.
The first trace norm estimate for PDO's was obtained in
\cite{ShT}, and later reproduced in \cite{Sh}, Proposition 27.3, and \cite{Robert},
Theorem II-49. The fundamental paper \cite{BS_U}
contains estimates in various compact operators classes for integral operators in terms of
smoothness of their \textsl{kernels}.
There are also publications focused on conditions
on the \textsl{symbol} which guarantee
that a PDO belongs to an appropriate Neumann-Schatten ideal, see e.g.
 \cite{Rond}, \cite{Arsu}, \cite{Toft}, \cite{ConToft} and references therein.
In spite of a relatively large number of available literature, these results
are not sufficient for our purposes. We need somewhat more detailed information
about trace norms. In particular, we derive an estimate for the
trace norm of a PDO  with weights, i.e. of $h_1 \op_\a^a(a) h_2$, where
the supports  of $h_1$ and $h_2$ are disjoint.
The most useful was paper \cite{Roz}, which served
as a basis for our approach. Although our estimates are quite elementary, and, probably
not optimal, they  provide bounds of correct orders in $\a$ and the scaling parameters.

Sections \ref{nonsmooth1:sect}, \ref{nonsmooth2:sect} are devoted to trace-class
estimates for PDO's with various jump discontinuities. The most basic
estimates are those for the commutators
$[\op_\a(a), \chi_{\L}]$, $[\op_\a(a), P_{\Om, \a}]$ where $\L$ and $\Om$
are graph-type domains,  see Lemmas
\ref{sandwich:lem} and \ref{sandwich_dual:lem}.
Similar bounds were derived in
\cite{Widom2} for the case when one of the domains $\L, \Om$
is a half-space. The new estimates which play a decisive role in the proof,
are collected in Sect. \ref{nonsmooth2:sect}.
Here the focus is on the PDO's with discontinuous symbols sandwiched between
weights having disjoint supports.
As in the
case of smooth symbols in Sect. \ref{smooth:sect}, it is important for us to
control the dependence of trace norms on the \textsl{distance}
between the supports.
The simplest result in Sect. \ref{nonsmooth2:sect},
illustrating this dependence is Lemma \ref{HS1:lem}.  The estimates
culminate in Lemma \ref{flat_boundary:lem} which bounds the error incurred when
replacing $\L$ by a half-space in the operator $T(a; \L, \Om)$.

The estimate \eqref{HSestim:eq} obtained in Sect. \ref{HS:sect} is used only
for closing the asymptotics in Sect. \ref{closure:sect}.
In contrast to the results
obtained in Sections \ref{nonsmooth1:sect}
and \ref{nonsmooth2:sect}, which estimate
norms of commutators of smooth symbols with $\chi_\L$ or $P_{\Om, \a}$,
the bound \eqref{HSestim:eq} is essentially
an estimate for the commutator $[\chi_{\L}, P_{\Om, \a}]$. Thus
it is not surprising
that apart from the standard term $\a^{d-1}$ the estimate acquires
the factor $\log\a$.

Sect. \ref{localisation:sect} shows that using appropriate cut-offs,
one can reduce the study of arbitrary smooth domains to graph-type domains.
This observation is conceptually straightforward, but technically important.

As explained in the Introduction, the asymptotics \eqref{main_anal:eq} and
\eqref{main_s:eq} are eventually derived from
the appropriate asymptotics in the one-dimensional
case. All the necessary information on the one-dimensional model problem is
assembled in Sect. \ref{1_dim:sect}. The analysis here is based on ideas
from \cite{Widom1} and \cite{Widom2}.

The proof of Theorem \ref{main:thm} starts in Sect. \ref{partition:sect}.
Here we
divide the domain $\L$ into  a  boundary layer and inner part.
 In order to implement the idea of approximating the domain $\L$ by a half-space,
we construct a partition of unity subordinate to
an appropriate covering of the boundary layer by open sets of specific shape,
and for each set perform the
approximation of $\L$ individually. A convenient covering is a diadic-type covering,
defined in Subsect. \ref{partitions:subsect}. Lemma \ref{brick:lem} describes
the sought approximation of $\L$.

From the technical point of view, Sect. \ref{asymptotics:sect} is the most demanding:
here we find the local asymptotics on each of
the covering sets described in Sect. \ref{partition:sect}. The calculations are based on
the asymptotics in the one-dimensional case obtained in Sect. \ref{1_dim:sect}.

All the intermediary results are put together in Sect. \ref{proof:sect} where
the proof of Theorem \ref{main:thm} is completed.

The asymptotics \eqref{widom:eq} is extended from polynomials to more general
functions in Sect. \ref{closure:sect}. The argument is based on
the sharp bounds of the type \eqref{sharp:eq} derived in \cite{G1}, \cite{G2}.
They lead to Theorems \ref{main_anal:thm} and \ref{main_s:thm}.

Appendices 1-4 contain some technical material. In particular, Appendix 1 contains
the geometrical lemma about the function \eqref{tildem:eq} mentioned in the Introduction.

%%%%%%%%%%%%%%%%%%%%%%%%%%%%%%%%%%%%%%%%%%%%%%%%%%%%%%%%%%%%%%%%%%%%%%%%%%%%%%%%%%%%

 \section{Estimates for PDO's with smooth symbols}\label{smooth:sect}

In this technical section we establish
estimates for the norms and trace norms of the operators
\eqref{pdoampl:eq} -- \eqref{pdoright:eq}.

In Subsections \ref{bound:subsect} and \ref{trace:subsect} we assume that $\a = 1$,
which does not restrict generality in view
of the scaling invariance \eqref{unitary:eq}.
On the other hand, in Subsection \ref{semi_class:subsect}
it is more convenient to allow arbitrary values $\a>0$.

\subsection{Boundedness}\label{bound:subsect}
To find convenient estimates for the norms we use the idea of
\cite{Hw}, Theorem 2 (see also \cite{Lerner}, Lemma 2.3.2
for a somewhat simplified version).

\begin{thm}\label{boundedness:thm}
Let $\om\ge 0$ be a number, and let $p(\bx, \by, \bxi)$ be an amplitude such that
for some $\om \ge 0$,
the function
\[
\lu\bx-\by\ru^{-\om}
|\nabla_{\bx}^{n_1} \nabla_{\by}^{n_2} \nabla_{\bxi}^m p(\bx, \by, \bxi)|
\]
is a bounded for all
\[
n_1, n_2\le r := \left[\frac{d}{2}\right]+1,\
\ \ \textup{and}\ \ m\le s := [d+\om] + 1.
\]
Then
$\op_1^a(p)$ is a bounded operator and
\begin{equation*}
\|\op_1^a(p)\|\le C \max_{\substack{n_1, n_2\le r,\\ m\le s}}
\sup_{\bx, \by, \bxi} \lu\bx-\by\ru^{-\om}
|\nabla_{\bx}^{n_1} \nabla_{\by}^{n_2} \nabla_{\bxi}^m p(\bx, \by, \bxi)|,
\end{equation*}
with a constant $C$ depending only on $d$ and $\om$.
\end{thm}

\begin{proof} Denote $\op_1^a(p) = \op(p)$.
Let us estimate
\begin{align*}
(\op(p) u, v) = &\ \frac{1}{(2\pi)^d}\int\int\int
e^{i\bxi\cdot(\bx-\by)} p(\bx, \by, \bxi)
u(\by)\overline {v(\bx)}   d\by d\bxi d \bx\\[0.2cm]
= &\ \frac{1}{(2\pi)^{2d}}
\int\int\int\int\int
e^{i (\bxi - \boldeta)\cdot\bx} e^{i(\bt - \bxi)\cdot\by}
p(\bx, \by, \bxi)
\hat u(\bt)\overline {\hat v(\boldeta)}   d\by d\bxi d \bx d\bt d\boldeta
\end{align*}
with arbitrary $u, v\in\plainC\infty_0(\R^d)$.
Thus we may assume that $p$ is in the Schwarz class on $\R^{3d}$.
Let $P^{(\bxi)}_{\bx}$ be the operator
\[
\frac{1 - i \bxi\cdot\nabla_{\bx}}{\lu \bxi\ru^2}.
\]
Since $P^{(\bxi)}_{\bx} e^{i\bxi\cdot\bx} = e^{i\bxi\cdot\bx}$,
after integration by parts $k$ times in $\bx$ and $\by$, we get
\begin{align*}
(2\pi)^{2d}&\ (\op(p) u, v)\\[0.2cm]
= &\
\int\int\int\int\int
e^{i (\bxi - \boldeta)\cdot\bx}
e^{i(\bt - \bxi)\cdot\by}
\biggl(\overline{P^{(\bxi-\boldeta)}_{\bx}}\biggr)^k
\biggl(\overline{P^{(\bt-\bxi)}_{\by}}\biggr)^k p(\bx, \by, \bxi)
\hat u(\bt)\overline {\hat v(\boldeta)}   d\by d\bxi d \bx d\bt d\boldeta.
\end{align*}
Since $P^{(\bx)}_{\bxi} e^{i\bxi\cdot\bx} = e^{i\bxi\cdot\bx}$,
we can also integrate $s$ times by parts in $\bxi$:
\begin{align*}
(2\pi)^{2d}(\op(p) u, v) = &\
\int\int\int\int\int
e^{i (\bxi - \boldeta)\cdot\bx} e^{i(\bt - \bxi)\cdot\by}\\[0.2cm]
&\ \biggl(\overline{P^{(\bx-\by)}_{\bxi}}\biggr)^s
\biggl[\biggl(\overline{P^{(\bxi-\boldeta)}_{\bx}}\biggr)^k
\biggl(\overline{P^{(\bt-\bxi)}_{\by}}\biggr)^k p(\bx, \by, \bxi)\biggr]
\hat u(\bt)\overline {\hat v(\boldeta)}   d\by d\bxi d \bx d\bt d\boldeta.
\end{align*}
This integral is a finite sum of terms of the form
\begin{align*}
\int\int\int\int\int
e^{i (\bxi - \boldeta)\cdot\bx} e^{i(\bt - \bxi)\cdot\by}
&\ g_{s m}(\bx-\by)
\phi_{k n_1}(\bxi-\boldeta) \psi_{k n_2}(\bt-\bxi)\\[0.2cm]
&\
p_{n_1 n_2 m}(\bx, \by, \bxi)
\hat u(\bt)\overline {\hat v(\boldeta)}   d\by d\bxi d \bx d\bt d\boldeta,
\end{align*}
where $m\le s$, $n_1, n_2\le k$,
$|p_{n_1 n_2 m}(\bx, \by, \bxi)| \le |\nabla_{\bx}^{n_1} \nabla_{\by}^{n_2}
\nabla_{\bxi}^{m} p(\bx, \by, \bxi)|$,
and
\begin{equation*}
|\phi_{k n_1}(\bmu)| + |\psi_{k n_2}(\bmu)| \le C\lu\bmu\ru^{-k},\
|g_{s m}(\bz)|\le C \lu\bz\ru^{-s}.
\end{equation*}
Rewrite the above integral in the form
\begin{equation*}
I:=(2\pi)^d\int\int\int e^{i\bxi\cdot(\bx-\by)} p_{n_1 n_2 m}(\bx, \by, \bxi)
g_{sm}(\bx-\by) \overline{\Phi(\bx, \bxi)} \Psi(\by, \bxi) d\bx d\by d\bxi,
\end{equation*}
with
\begin{align*}
\Phi(\bx, \bxi) = &\ \frac{1}{(2\pi)^{\frac{d}{2}}} \int e^{i\boldeta\cdot\bx} \
\overline{\phi_{kn_1}(\bxi-\boldeta)} \ \hat v(\boldeta)
d\boldeta,\\[0.2cm]
\Psi(\by, \bxi) = &\ \frac{1}{(2\pi)^{\frac{d}{2}}}
 \int e^{i\bt\cdot\by}
\psi_{k n_2}(\bt-\bxi) \hat u(\bt) d\bt.
\end{align*}
Both functions $\Phi$ and $\Psi$ are $\plainL2(\R^{2d})$ and
\[
\|\Psi\|_{\plainL2}^2\le C\|u\|_{\plainL2}^2,\
\|\Phi\|_{\plainL2}^2\le C \|v\|_{\plainL2}^2.
\]
Indeed, by Parseval's identity,
\[
\|\Phi\|_{\plainL2}^2
= \int\int  |\phi_{k n_1}(\bxi-\boldeta)|^2 |\hat v(\boldeta)|^2 d\bxi d\boldeta
= \|\phi_{k n_1}\|_{\plainL2}^2 \|v\|_{\plainL2}^2,
\]
and the norm of $\phi_{k n_1}$ is finite if we choose $k> d/2$.
Similarly for $\Psi$.
Thus the integral $I$ defined above can be estimated as follows:
\begin{align*}
|I|\le &\ C\sup_{\bx, \by, \bxi} \lu\bx-\by\ru^{-\om}
|p_{n_1 n_2 m}(\bx, \by, \bxi)|\\[0.2cm]
&\ \biggl[\int\int\int \lu\bx-\by\ru^{-s+\om} |\Phi(\bx, \bxi)|^2
d\bx d\by d\bxi\biggr]^{\frac{1}{2}}
\biggl[\int\int\int \lu\bx-\by\ru^{-s+\om}
|\Phi(\by, \bxi)|^2 d\bx d\by d\bxi\biggr]^{\frac{1}{2}}.
\end{align*}
Here we have used H\"older's inequality. The product of the last two integrals equals
\[
\|\Phi\|_{\plainL2}\|\Psi\|_{\plainL2}\ \ \int \lu\bz\ru^{-s+\om} d\bz,
\]
and the integral of $\lu\bz\ru^{-s+\om}$ is finite if we choose $s> \om + d$.
\end{proof}

One should say that there is a simpler looking test of boundedness for PDO's of the type
$\op_1^l(a)$ which requires no smoothness w.r.t. $\bxi$, but
imposes certain decay condition at infinity in the variable $\bx$, see \cite{H3}, Theorem 18.1.11'.
A similar result can be also obtained for the operator $\op_1^a(p)$, but it would
require from $p(\bx, \by, \bxi)$ a decay in both $\bx$ and $\by$,
which is not convenient for us.

\subsection{Trace class estimates}\label{trace:subsect}
Again, we assume that $\a = 1$.
We are going to use the ideas from \cite{Roz}.
In fact, our estimates are nothing but more precise quantitative
variants of Proposition 3.2 and Theorem 3.5 from \cite{Roz}.

Assuming that $p\in\plainL1(\R^d)$,  introduce the ``double" Fourier transform:
\begin{equation*}
\hat p(\boldeta, \bmu, \bxi)
= \frac{1}{(2\pi)^{ d }}
\int\int e^{-i\bx\cdot \boldeta-i\by\cdot\bmu} p(\bx, \by, \bxi)
d\bx d\by.
\end{equation*}

\begin{lem}\label{trace_amplitude1:lem}
Suppose that  $\hat p\in\plainL1(\R^d)$, $p\in \plainL1(\R^d)$.
Then for any $h_1, h_2\in\plainL2(\R^d)$ the operator
$h_1 \op^a_1(p) h_2$ is trace class
and
\begin{equation}\label{trace_amplitude1:eq}
\| h_1 \op^a_1(p) h_2\|_{\GS_1}\le  (2\pi)^{-2d}
\|h_1\|_{\plainL2} \| h_2\|_{\plainL2} \int \int\int
| \hat p(\boldeta, \bmu, \bxi)| d\boldeta d \bmu d\bxi .
\end{equation}
\end{lem}

\begin{proof}
Represent the amplitude $p$ as follows:
\begin{equation*}
p(\bx, \by, \bxi)
= \frac{1}{(2\pi)^{d}}
\int\int e^{i\bx\cdot \boldeta + i\by\cdot \bmu} \hat p(\boldeta, \bmu, \bxi)
d\boldeta d\bmu.
\end{equation*}
Let $g_j, m_j$ be two orthonormal sequences
in $\plainL2(\R^d)$. Then
\begin{align*}
(h_1\op^a_1(p) h_2g_j, m_j)
= &\ \biggl(\frac{1}{2\pi}\biggr)^{2d}
 \int\int\int
 \biggl[\int e^{i\bx\cdot\bxi + i\bx\cdot\boldeta} h_1(\bx) \overline{m_j(\bx)}
d\bx \\[0.2cm]
&\ \times\int e^{-i\by\cdot\bxi + i\by\cdot\bmu}
  h_2(\by) g_j(\by) d\by \biggr]
\hat p(\boldeta, \bmu, \bxi) d\bxi d\boldeta d\bmu,
\end{align*}
and by the Bessel inequality,
\begin{equation*}
\sum_j |(h_1\op^a_1(p) h_2g_j, m_j)|
\le \biggl(\frac{1}{2\pi}\biggr)^{2d}
\|h_1\|_{\plainL2} \|h_2\|_{\plainL2}\int \int \int
| \hat p(\boldeta, \bmu , \bxi)| d\boldeta d\bmu d\bxi.
\end{equation*}
By Theorems 11.2.3,4 from \cite{BS}, page 246,
the operator $h_1\op^a_1(p)h_2$ is trace class
and its trace norm satisfies the required bound.
\end{proof}

It is usually more convenient to write these estimates in terms of the
amplitudes themselves, and not their Fourier transforms.
In all the statements below we always assume that the amplitudes (symbols) have
the required partial derivatives and that the integrals involved, are finite.

\begin{cor} \label{symbol:cor}
Let $h_1, h_2$ be arbitrary $\plainL2$-functions. Then
\begin{equation*}
\|h_1 \op^a_1(p) h_2\|_{\GS_1}\le C\|h_1\|_{\plainL2}
\|h_2\|_{\plainL2}
\sum_{n_1, n_2 = 0}^{d+1}
 \int\int\int
|\nabla_{\bx}^{n_1} \nabla_{\by}^{n_2} p(\bx, \by, \bxi)|
d\bx d\by d\bxi.
\end{equation*}
\end{cor}

\begin{proof} Integrating by parts, we get:
\begin{equation*}
|\hat p(\boldeta, \bmu, \bxi)|
\le C (1+|\boldeta|)^{-d-1} (1+ |\bmu|)^{-d-1}
\sum_{n_1, n_2 = 0}^{d+1}
 \int\int
|\nabla_{\bx}^{n_1} \nabla_{\by}^{n_2} p(\bx, \by, \bxi)|
d\bx d\by.
\end{equation*}
Substituting this estimate in \eqref{trace_amplitude1:eq},
we get the required result.
\end{proof}

It is also useful to have bounds for operators with $h_1$ and $h_2$
having disjoint supports.
Below we denote by $\z\in\plainC\infty(\R)$ a function such that
\begin{equation}\label{zeta:eq}
\z(t) = 1\  \textup{if}\  |t|\ge 2, \ \ \ \textup{and} \
\z(t) = 0\
\textup{if} \   |t|\le 1.
\end{equation}
In all the subsequent estimates constants may depend on $\z$ and its derivatives,
but it is unimportant for our purposes.

\begin{cor} \label{symbol:cor1} Let $h_1, h_2$ be two $\plainL2$-functions.
Suppose that $p(\bx, \by, \bxi) = 0$ if $|\bx-\by|\le R$ with some $R\ge c$.
Then
\begin{align}\label{symbolcor11:eq}
\|h_1 \op^a_1(p) &\ h_2\|_{\GS_1}\notag \\[0.2cm]
\le &\ C_m\|h_1\|_{\plainL2}
\|h_2\|_{\plainL2}
\sum_{n_1, n_2=0}^{d+1}
 \int\int\int_{|\bx-\by|\ge R}
\frac{|\nabla_{\bx}^{n_1} \nabla_{\by}^{n_2}
\nabla_{\bxi}^m p(\bx, \by, \bxi)|}{|\bx-\by|^m}
d\bx d\by d\bxi,
\end{align}
for any $m= 0, 1, \dots$.

In particular, if $p(\bx, \by, \bxi)$
depends only on $\bx$ and $\bxi$, i.e. $p(\bx, \by, \bxi) = a(\bx, \bxi)$,
and the essential supports of
$h_1$ and $h_2$ are separated by a distance $R$, then
\begin{align}\label{symbolcor12:eq}
\|h_1 \op^l_1(a) h_2\|_{\GS_1} + &\ \|h_1 \op^r_1(a) h_2\|_{\GS_1}\notag\\[0.2cm]
\le &\ C_m \| h_1\|_{\plainL2} \
\| h_2\|_{\plainL2} R^{d-m}\sum_{n = 0}^{d+1}\int\int
|\nabla_{\bx}^n \nabla_{\bxi}^m a(\bx, \bxi)| d\bx d\bxi,
\end{align}
for any $m\ge d+1$.
\end{cor}

\begin{proof}
If the essential supports of $h_1$ and $h_2$ are separated
by a distance $R$, then $h_1 \op_\a^a(p) h_2$ can be rewritten as
\[
h_1 \op_1^a(\tilde p)  h_2, \tilde p(\bx, \by, \bxi)
= p(\bx, \by, \bxi) \z(2|\bx-\by|R^{-1}),
\]
with the function
 $\z\in\plainC\infty(\R)$ defined in \eqref{zeta:eq}.
Thus the bound \eqref{symbolcor12:eq} immediately follows
from \eqref{symbolcor11:eq}.

Proof of \eqref{symbolcor11:eq}.
Let $P = -i|\bx-\by|^{-2} (\bx-\by)\cdot\nabla_{\bxi}$. Clearly,
$Pe^{i\bxi\cdot(\bx-\by)} = e^{i\bxi\cdot(\bx-\by)}$, so,
integrating by parts $m$ times,
we get the following formula for the kernel of the operator $\op^a(p)$:
\begin{equation*}
\frac{1}{(2\pi)^d} \int e^{i\bxi\cdot(\bx - \by)} q(\bx, \by, \bxi) d\bxi,
\end{equation*}
with
\begin{equation*}
q(\bx, \by, \bxi) = i^m \frac{1}{|\bx - \by|^{2m}}
\bigl((\bx-\by)\cdot\nabla_{\bxi}\bigr)^m p(\bx, \by, \bxi).
\end{equation*}
It is straightforward to see that
\begin{align*}
\sum_{n_1, n_2 = 0}^{d+1}
|\nabla_{\bx}^{n_1} \nabla_{\by}^{n_2} q(\bx, \by, \bxi)|
\le C\frac{1}{|\bx-\by|^m}\sum_{n_1, n_2 = 0}^{d+1}
|\nabla_{\bx}^{n_1} \nabla_{\by}^{n_2} \nabla_{\bxi}^m p(\bx, \by, \bxi)|.
\end{align*}
By Corollary \ref{symbol:cor} this implies the proclaimed result.
\end{proof}

\begin{lem}\label{amplitude:lem}
 Let $h_1, h_2$ be arbitrary $\plainL2$-functions. Then
\begin{align}\label{amplitude:eq}
\|h_1 \op^a_1(p) h_2\|_{\GS_1}\le &\  C_Q\|h_1\|_{\plainL2}
\|h_2\|_{\plainL2}\notag\\[0.2cm]
&\ \times\sum_{n_1, n_2=0}^{d+1} \sum_{m=0}^Q
\int\int\int
\frac{|\nabla_{\bx}^{n_1} \nabla_{\by}^{n_2}
\nabla_{\bxi}^m p(\bx, \by, \bxi)|}{1+|\bx-\by|^Q}
d\bx d\by d\bxi,
\end{align}
for any $Q= 0, 1, \dots$.
\end{lem}

\begin{proof}
Let a function $\z$ be as defined in \eqref{zeta:eq}.
Denote
\begin{equation*}
p_1(\bx, \by, \bxi) = p(\bx, \by, \bxi) (1-\z(\bx-\by)),\ \
p_2(\bx, \by, \bxi) = p(\bx, \by, \bxi)\z(\bx-\by).
\end{equation*}
Estimate separately the trace norms
of the operators $h_1 \op^a_1(p_1) h_2$ and
$h_1 \op^a_1(p_2) h_2$.
Since $|\bx-\by|\le 2$ on the support of $p_1$, by Corollary
\ref{symbol:cor},
the trace  norm of $h_1 \op^a_1(p_1)h_2$ does not exceed the right hand side of
\eqref{amplitude:eq}.
For $p_2$ one uses Corollary \ref{symbol:cor1}, which also gives the required
bound.
\end{proof}

In the next Theorem we replace the $\plainL2$-norms of functions $h_1, h_2$
by much weaker ones.
Let $\CC = [0, 1)^d$ be the unit cube, and let $\CC_{\bz} = \CC+\bz, \bz\in\R^d$.
For $s\in (0, \infty]$ and any function $h\in\plainL2_{\textup{\tiny loc}}$
introduce the following quasi-norm:
\begin{equation}\label{brackh:eq}
\begin{cases}
\1 h\1_s =
\biggr[\sum_{\bz\in\Z^d}
\biggl(\int_{\CC_{\bz}} |h(\bx)|^2 d\bx\biggr)^{\frac{s}{2}}\biggr]^{\frac{1}{s}},\ \
0<s<\infty,\\[0.3cm]
\1 h\1_\infty =
\sup_{\bz\in\R^d}
\biggl(\int_{\CC_{\bz}} |h(\bx)|^2 d\bx\biggr)^{\frac{1}{2}},\ \
s=\infty.
\end{cases}
\end{equation}

\begin{thm}\label{full:thm}
 Let $h_1, h_2$ be arbitrary $\plainL2_{\textup{\tiny loc}}$-functions. Then
\begin{align}\label{amplitude1:eq}
\|h_1 \op^a_1(p) &\ h_2\|_{\GS_1}\notag\\[0.2cm]
\le &\ C_Q\1 h_1\1_{\infty}
\1 h_2\1_{\infty}
\sum_{n_1, n_2=0}^{d+1} \sum_{m=0}^Q
\int\int\int
\frac{|\nabla_{\bx}^{n_1} \nabla_{\by}^{n_2}
\nabla_{\bxi}^m p(\bx, \by, \bxi)|}{1+|\bx-\by|^Q}
d\bx d\by d\bxi,
\end{align}
for any $Q = 0, 1, \dots$.

In particular, if $p(\bx, \by, \bxi)$ depends only on $\bx$ and  $\bxi$,
i.e. $p(\bx, \by, \bxi) = a(\bx, \bxi)$, then
\begin{align}\label{symbol1:eq}
\|h_1 \op_1^l(a) h_2\|_{\GS_1}
+ &\ \|h_1 \op_1^r(a) h_2\|_{\GS_1}\notag\\[0.2cm]
\le &\ C \1 h_1\1_{\infty}
\1 h_2\1_{\infty}
\sum_{n=0}^{d+1} \sum_{m=0}^{d+1}
\int\int |\nabla_{\bx}^{n}
\nabla_{\bxi}^m a(\bx, \bxi)| d\bx d\bxi.
\end{align}

\end{thm}

\begin{proof} The estimate \eqref{symbol1:eq} follows immediately from
\eqref{amplitude1:eq} with $Q = d+1$. Let us prove \eqref{amplitude1:eq}.

Let $\z_j\in\plainC\infty_0(\R^d)$, $j\in\Z$, be a partition
of unity subordinate to a covering of $\R^d$ by unit cubes, such that the
number of intersecting cubes is uniformly bounded, and
\begin{equation}\label{zetaj:eq}
|\nabla_{\bx}^n\z_j(\bx)|\le C_n, n = 1, 2, \dots,
\end{equation}
 for all $\bx\in\R^d$ uniformly in $j$. By $\chi_j$ we denote
 the characteristic function of the cube labeled $j$.
In view of \eqref{zetaj:eq}, for the amplitude
\[
q_{j, s}(\bx, \by, \bxi) = \z_j(\bx) \z_s(\by) p(\bx, \by, \bxi),
\]
we obtain from Lemma \ref{amplitude:lem} that
\begin{align*}
\|h_1 \chi_j &\ \op^a_1(q_{j, s}) h_2\chi_s\|_{\GS_1}\\
\le &\ C\|h_1 \chi_j\|_{\plainL2}
\|h_2\chi_s\|_{\plainL2}
\sum_{n_1, n_2=0}^{d+1} \sum_{m=0}^Q
\int\int\int \chi_j(\bx) \chi_s(\by)
\frac{|\nabla_{\bx}^{n_1} \nabla_{\by}^{n_2}
\nabla_{\bxi}^m p(\bx, \by, \bxi)|}{1+|\bx-\by|^Q}
d\bx d\by d\bxi,
\end{align*}
The $\plainL2$-norms of $h_1\chi_j$ and $h_2 \chi_s$ are estimated by $\1 h_1\1_{\infty}$
and $\1 h_2\1_{\infty}$ respectively. Thus, remembering that
the number of intersecting cubes is uniformly bounded, we obtain
from the bound
\begin{equation*}
\| h_1 \op^a_1(p) h_2\|_{\GS_1}\le
\sum_{j, s} \|h_1 \op^a_1(q_{j, s}) h_2\|_{\GS_1},
\end{equation*}
the required estimate \eqref{amplitude1:eq}.
\end{proof}

In the same way one proves the ``local" variant of Corollary \ref{symbol:cor1}:

\begin{thm}\label{separate:thm}
Let $h_1, h_2$ be two $\plainL2$-functions.
Suppose that $p(\bx, \by, \bxi) = 0$ if $|\bx-\by|\le R$ with some $R\ge c$.
Then
\begin{equation*}
\|h_1 \op^a_1(p) h_2\|_{\GS_1}\le C\1 h_1\1_{\infty}\  \1 h_2\1_{\infty}
\sum_{n_1, n_2=0}^{d+1}
 \int\int\int_{|\bx-\by|\ge R}
\frac{|\nabla_{\bx}^{n_1} \nabla_{\by}^{n_2}
\nabla_{\bxi}^m p(\bx, \by, \bxi)|}{|\bx-\by|^m}
d\bx d\by d\bxi,
\end{equation*}
for any $m= 0, 1, \dots$.

In particular, if $p(\bx, \by, \bxi)$
depends only on $\bx$ and $\bxi$, i.e. $p(\bx, \by, \bxi) = a(\bx, \bxi)$,
and the essential supports of
$h_1$ and $h_2$ are separated by a distance $R$, then
\begin{align}\label{separate2:eq}
\|h_1 \op^l_1(a) h_2\|_{\GS_1} + &\ \|h_1 \op^r_1(a) h_2\|_{\GS_1}\notag\\[0.2cm]
\le &\ C_m \1 h_1\1_{\infty} \ \1 h_2\1_{\infty} R^{d-m}\sum_{n = 0}^{d+1}\int\int
|\nabla_{\bx}^n \nabla_{\bxi}^m a(\bx, \bxi)| d\bx d\bxi,
\end{align}
for any $m\ge d+1$.

\end{thm}

\subsection{Operators of a special form}
In addition to the general PDO's we often work with operators of the form
$h \op_1(a), a = a(\bxi)$,
which have been studied quite extensively. We need the following
estimate  which can be found
in \cite{BS_U}, Theorem 11.1 (see also \cite{BKS}, Section 5.8),
and for $s\in [1, 2)$ in \cite{Simon}, Theorem 4.5.

\begin{prop}\label{BS:prop}
Suppose that $h\in\plainL2_{\textup{\tiny loc}}(\R^d)$ and
$a\in\plainL2_{\textup{\tiny loc}}(\R^d)$
are functions such that $\1 h\1_s, \1 a\1_s<\infty$
with some $s\in (0, 2)$. Then $h \op_\a(a)\in\GS_s$ and
\begin{equation*}
\|h \op_1(a)\|_{\GS_s}\le C \1 h\1_s  \1 a\1_s.
\end{equation*}
\end{prop}

\subsection{Amplitudes from classes $\BS^{(n_1, n_2, m)}$:
semi-classical estimates}\label{semi_class:subsect}
Here we apply the trace class estimates obtained so far to
symbols and amplitudes from the classes $\BS^{(n_1, n_2, m)}$.
Now we are concerned with estimates
for $\a$-PDO, with an explicit control of dependence on $\a$. Moreover,
we shall explicitly monitor the dependence on scaling parameters in terms of norms
$\SN^{(n_1, n_2, m)}(p; \ell, \rho)$.

\begin{lem}\label{general_norm:lem}
Assume that $p\in \BS^{(k, k, d+1)}$
with
\begin{equation}\label{cald:eq}
k=\left[\frac{d}{2}\right]+1.
\end{equation}
Let $\ell>0, \rho>0$ be two parameters such that
$\a\ell\rho\ge c$. Then $\op^a_\a(p)$ is a bounded operator
and
\begin{equation}\label{general_norm:eq}
\|\op^a_{\a}(p)\|\le C \SN^{(k, k, d+1)}(p; \ell, \rho).
\end{equation}

\end{lem}

\begin{proof} Using \eqref{scale:eq}
with $\ell_1 = (\a\rho)^{-1}$, $\rho_1 = \rho$, and the unitary equivalence
\eqref{unitary:eq},
we conclude that it suffices to prove the sought
inequalities for $\a=\rho = 1$ and arbitrary $\ell \ge c$. Without loss of generality
suppose also that
$\SN^{(k, k, d+1)}(p; \ell, 1) = 1$, so that
\[
|\nabla_{\bx}^{n_1} \nabla_{\by}^{n_2}
\nabla_{\bxi}^m p(\bx, \by, \bxi)|\le \ell^{-n_1-n_2}\le C,
\]
for all $n_1, n_2 \le k, m\le d+1$. Now the required bound follows from
Theorem \ref{boundedness:thm} with $\om = 0$.
\end{proof}

\begin{lem}\label{quantisation_norm:lem}
Suppose that $p\in\BS^{(k, k, d+2)}$ with
$k$ defined in \eqref{cald:eq}, and that $\a\ell\rho\ge c$.
Denote
$a(\bx, \bxi) = p(\bx, \bx, \bxi)$. Then $a\in\BS^{(k, d+2)}$, and
\begin{equation}\label{quantisation_ampl_norm:eq}
\| \op_\a^a(p) - \op_\a^l(a)\|\le C (\a\ell\rho)^{-1}
\SN^{(k, k, d+2)}(p; \ell, \rho).
\end{equation}
Moreover, for any symbol $a\in\BS^{(k, d+2)}$, we have
\begin{equation}\label{quantisation_symbol_norm:eq}
\| \op_\a^r(a) - \op_\a^l(a)\|\le C (\a\ell\rho)^{-1}
\SN^{(k, d+2)}(a; \ell, \rho).
\end{equation}
\end{lem}

\begin{proof}
The bound \eqref{quantisation_symbol_norm:eq} follows from
\eqref{quantisation_ampl_norm:eq} with $p(\bx, \by, \bxi) = a(\by, \bxi)$,
so that $p\in \BS^{(k, k, d+2)}$ and
$\SN^{(n_1, n_2, m)}(p; \ell, \rho) = \SN^{(n_2, m)}(a; \ell, \rho)$ for
$n_1, n_2\le k$, and $m\le d+2$.

As in the proof of Lemma \ref{general_norm:lem}, in view of
\eqref{scale:eq} and \eqref{unitary:eq} we may assume that
$\a = 1, \rho = 1$ and $\ell\ge c$.
In order to apply Theorem \ref{boundedness:thm}
note that the amplitude
$b(\bx, \by, \bxi) = p(\bx, \by, \bxi) - p(\bx, \bx, \bxi)$
satisfies the bounds
\begin{equation*}
|\nabla_{\bx}^{n_1}\nabla_{\by}^{n_2} \nabla_{\bxi}^m
b(\bx, \by, \bxi)|
\le C\ell^{-n_1-n_2} \SN^{(k, k, d+2)}(p; \ell, 1),
\end{equation*}
for any $n_1, n_2\le k, m\le d+2$, and
\begin{equation*}
|\nabla_{\bxi}^m b(\bx, \by, \bxi)|
\le \ell^{-1} \SN^{(0, 1, d+2)}(p; \ell, 1)|\bx-\by|,
\end{equation*}
for any $m\le d+2$. Therefore,
\begin{equation*}
\lu \bx-\by\ru^{-1}|\nabla_{\bx}^{n_1}\nabla_{\by}^{n_2} \nabla_{\bxi}^m
b(\bx, \by, \bxi)|\le C\ell^{-1} \SN^{(k, k, d+2)}(p; \ell, 1),\
\ n_1, n_2\le k, m\le d+2.
\end{equation*}
Now by Theorem \ref{boundedness:thm} with $\om = 1$ we get
\[
\|\op_\a^a(b)\|\le
C\ell^{-1} \SN^{(k, k, d+2)}(p; \ell, 1),
\]
which leads to \eqref{quantisation_ampl_norm:eq}.
\end{proof}

Now we obtain appropriate trace class bounds. All bounds will be derived under
one of the following conditions. For the operator $\op_\a^a(p)$ we assume that
\begin{gather}
\textup{the support of the amplitude}\ \ p = p(\bx, \by, \bxi)\ \
\textup{is contained} \notag\\[0.2cm]
\textup{either in} \ \
B(\bz, \ell)\times \R^d\times B(\bmu, \rho)\label{p_support1:eq}\\[0.2cm]
\textup{or in}\ \
\R^d\times B(\bz, \ell)\times B(\bmu, \rho),\label{p_support2:eq}
\end{gather}
with some $\bz, \bmu\in\R^d$ and some $\ell>0, \rho>0$. For the  operators
$\op_\a^l(a), \op_\a^r(a)$ we assume that
\begin{equation}\label{a_support:eq}
\textup{the support of the symbol}\
\textup{is contained in} \ \
B(\bz, \ell)\times B(\bmu, \rho).
\end{equation}
The constants in the obtained estimates will be independent of
$\bz$, $\bmu$ and $\ell, \rho$.

\begin{lem}\label{general_trace:lem}
Let $p\in \BS^{(d+1, d+1, d+1)}$
be an amplitude satisfying either the condition \eqref{p_support1:eq}
or \eqref{p_support2:eq}, and let $a\in \BS^{(d+1, d+1)}$ be
a symbol satisfying the condition \eqref{a_support:eq}.
\begin{enumerate}
\item
If $\a\ell\rho\ge c$, then $\op_\a^{a}(p)\in\GS_1$
and $\op_\a^{l}(a)\in\GS_1$, $\op_\a^{r}(a)\in\GS_1$, and
\begin{equation}\label{trace_ampl:eq}
\|\op_\a^a(p)\|_{\GS_1}\le C (\a\ell\rho)^d
\SN^{(d+1, d+1, d+1)}(p; \ell, \rho),
\end{equation}
\begin{equation}\label{trace_symbol:eq}
\|\op_\a^l(a)\|_{\GS_1}\le C (\a\ell\rho)^d \SN^{(d+1 ,d+1)}(a; \ell, \rho).
\end{equation}
\item
If  $a\in\BS^{(d+1, m)}$ with some $m\ge d+1$, and $h_1$, $h_2$
are $\plainL2_{\textup{\tiny loc}}$-functions,
whose supports are separated by a distance $R>0$.
Suppose that $\a \rho R \ge c$ and $\a\ell\rho\ge c$.
Then $h_1\op_\a^l(a) h_2\in\GS_1$ and
\begin{align}\label{largea:eq}
\|h_1\op_\a^l(a) h_2\|_{\GS_1}
+ &\ \|h_1\op_\a^r(a) h_2\|_{\GS_1}\notag\\[0.2cm]
\le &\ C_m (\a\ell\rho)^d (\a R\rho)^{-m+d}
\1 h_1\1_{\infty} \ \1 h_2\1_{\infty} \SN^{(d+1, m)}(a; \ell, \rho),
\end{align}
where the norm $\1\ \cdot\ \1_s$  is defined in \eqref{brackh:eq}.
\end{enumerate}
\end{lem}

\begin{proof}
The estimate \eqref{trace_symbol:eq} is a special case of \eqref{trace_ampl:eq}.

In view of \eqref{orthogonal1:eq}
for both \eqref{trace_ampl:eq} and \eqref{largea:eq}
we may assume that $\bz = \bmu = \bzero$.
Furthermore,
using \eqref{scale:eq} and \eqref{unitary:eq}
with $\ell_1 = (\a\rho)^{-1}, \rho_1 = \rho$, we see that it suffices
to prove the sought inequalities for $\a = 1$, $\rho = 1$ and arbitrary $\ell\ge c$, and,
in the case of \eqref{largea:eq}, arbitrary $R>c$.
Assume without loss of generality that $\SN^{(d+1, d+1, d+1)}(p; \ell, 1) = 1$
and $\SN^{(d+1, m)}(a; \ell, 1) = 1$.

Proof of \eqref{trace_ampl:eq}. Suppose for definiteness
that the support of the amplitude $p$ satisfies \eqref{p_support1:eq}.
Since
\begin{equation*}
|\nabla_{\bx}^{n_1} \nabla_{\by}^{n_2} \nabla_{\bxi}^m p(\bx, \by, \bxi)|
\le \ell^{-n_1-n_2}\chi_{\bzero, \ell}(\bx)\chi_{\bzero, 1}(\bxi)
\le C \chi_{\bzero, \ell}(\bx)\chi_{\bzero, 1}(\bxi),
\end{equation*}
for all $n_1, n_2\le d+1$ and $m\le d+1$,
the bound \eqref{amplitude1:eq} with $Q = d+1$ and $h_1 = h_2 = 1$ gives that
\begin{equation*}
\|\op_1^1(p)\|_{\GS_1}\le C
\int_{|\bxi|\le 1}\int_{\R^d}
\int_{|\bx|\le \ell}
\frac{1}{1+|\bx-\by|^{d+1}}
d\bx d\by d\bxi\le C\ell^d,
\end{equation*}
which leads to \eqref{trace_ampl:eq}.

Proof of \eqref{largea:eq}. Since
\begin{equation*}
|\nabla_{\bx}^{n} \nabla_{\bxi}^s a(\bx, \bxi)|
\le \ell^{-n}\chi_{\bzero, \ell}(\bx)\chi_{\bzero, 1}(\bxi)
\le C \chi_{\bzero, \ell}(\bx)\chi_{\bzero, 1}(\bxi),
\end{equation*}
for all $n\le d+1$ and $s\le m$, the bound
\eqref{separate2:eq} with $h_1=h_2=1$ gives that
\begin{align*}
\|h_1\op_1^l(a)h_2\|_{\GS_1} + \|h_1\op_1^r(a)h_2\|_{\GS_1}\le &\ C
\1 h_1\1_{\infty} \ \1 h_2\1_{\infty} R^{d-m} \int_{|\bxi|\le 1}
\int_{|\bx|\le \ell}
d\bx  d\bxi\\[0.2cm]
\le &\ C \1 h_1\1_{\infty} \ \1 h_2\1_{\infty} R^{d-m} \ell^d,
\end{align*}
which leads to \eqref{largea:eq}.
\end{proof}

\begin{lem}\label{quantisation_trace:lem}
Let $p\in \BS^{(d+1, d+1, d+2)}$
be an amplitude satisfying either the condition \eqref{p_support1:eq}
or \eqref{p_support2:eq}, and let $a\in \BS^{(d+1, d+2)}$ be
a symbol satisfying the condition \eqref{a_support:eq}.
Suppose that $\a\ell\rho\ge c$. Denote
$b(\bx, \bxi) = p(\bx, \bx, \bxi)$. Then $b\in\BS^{(d+1, d+2)}$, it
satisfies \eqref{a_support:eq}, and
\begin{equation}\label{quantisation_ampl:eq}
\| \op_\a^a(p) - \op_\a^l(b)\|_{\GS_1}\le C(\a\ell\rho)^{d-1}
\SN^{(d+1, d+1, d+2)}(p; \ell, \rho).
\end{equation}
Moreover,
\begin{equation}\label{quantisation_symbol:eq}
\| \op_\a^r(a) - \op_\a^l(a)\|_{\GS_1}\le C(\a\ell\rho)^{d-1}
\SN^{(d+1, d+2)}(a; \ell, \rho).
\end{equation}
\end{lem}

\begin{proof}
The bound \eqref{quantisation_symbol:eq} follows from
\eqref{quantisation_ampl:eq} with $p(\bx, \by, \bxi) = a(\by, \bxi)$,
so that $p\in \BS^{(d+1, d+1, d+2)}$ and
$\SN^{(d+1, d+1, d+2)}(p; \ell, \rho) = \SN^{(d+1, d+2)}(a; \ell, \rho)$.

Proof of \eqref{quantisation_ampl:eq}.
As in the proof of Lemma \ref{general_trace:lem}, in view of
\eqref{orthogonal1:eq},
\eqref{scale:eq} and \eqref{unitary:eq} we may assume that
$\bz=\bmu = \bzero$ and $\a = 1, \rho = 1$ and $\ell\ge c$.
For definiteness suppose that the condition
\eqref{p_support1:eq} is satisfied. Without loss of generality assume that
$\SN^{(d+1, d+1, d+2)}(p; \ell, 1) = 1$.
In order to apply Theorem \ref{full:thm} note that
the amplitude
$g(\bx, \by, \bxi) = p(\bx, \by, \bxi) - p(\bx, \bx, \bxi)$
satisfies the bounds
\begin{align*}
|\nabla_{\bx}^{n_1}\nabla_{\by}^{n_2} \nabla_{\bxi}^m
g(\bx, \by, \bxi)|
\le  &\ C\ell^{-n_1-n_2}
\chi_{\bzero, \ell}(\bx)\chi_{\bzero, 1}(\bxi),\\[0.2cm]
|\nabla_{\bxi}^m g(\bx, \by, \bxi)|\le  &\ \ell^{-1}  |\bx-\by|
\chi_{\bzero, \ell}(\bx)\chi_{\bzero, 1}(\bxi),
\end{align*}
for any $n_1, n_2\le d+1$ and $m\le d+2$, so that
\begin{equation*}
|\nabla_{\bx}^{n_1}\nabla_{\by}^{n_2} \nabla_{\bxi}^m
g(\bx, \by, \bxi)|
\le C\ell^{-1} \lu\bx-\by\ru
\chi_{\bzero, \ell}(\bx)\chi_{\bzero, 1}(\bxi),\ \ n_1, n_2\le d+1,\ m\le d+2.
\end{equation*}
Thus by Theorem \ref{full:thm} with $Q = d+2$, we have
\begin{equation*}
\| \op_\a^a(g)\|_{\GS_1}
\le C \ell^{-1}
\underset{|\bxi|\le 1}\int \ \
\underset{|\bx|\le \ell}\int \ \
\int
\frac{1}
{1+|\bx-\by|^{d+1}} d\by d\bx d\bxi \le C\ell^{d-1},
\end{equation*}
which is the required bound.
\end{proof}

\begin{cor}\label{product:cor}
Let $a, b\in\BS^{(d+1, d+2)}$ be two symbols
such that either $a$ or $b$ satisfies \eqref{a_support:eq}.
Suppose that $\a\ell\rho\ge c$.
Then
\begin{equation}\label{product_smooth:eq}
\|\op_\a^l(a)\op_\a^l(b) - \op_\a^l(ab)\|_{\GS_1}\le
C(\a\ell\rho)^{d-1} \SN^{(d+1, d+2)}(a; \ell, \rho)
\SN^{(d+1, d+2)}(b; \ell, \rho).
\end{equation}
\end{cor}

\begin{proof}
Suppose that $a$ satisfies \eqref{a_support:eq}.   By
\eqref{trace_symbol:eq} and \eqref{quantisation_symbol_norm:eq},
\begin{align}\label{lltolr:eq}
\|\op_\a^l(a)\op_\a^l(b) - \op_\a^l(a)\op_\a^r(b)\|_{\GS_1}
\le &\ \|\op_\a^l(a)\|_{\GS_1} \|\op_\a^l(b) - \op_\a^r(b)\|\notag\\[0.2cm]
\le &\ C(\a\ell\rho)^{d-1}\SN^{(d+1, d+1)}(a; \ell, \rho)
\SN^{(k, d+2)}(b; \ell, \rho),
\end{align}
where $k$ is defined in \eqref{cald:eq}.
The operator $\op_\a^l(a) \op_\a^r(b)$ has the form $\op_\a^a(p)$
with $p(\bx, \by, \bxi) = a(\bx, \bxi) b(\by, \bxi)$.
Clearly, $p$ satisfies the condition
\eqref{p_support1:eq}, $p\in \BS^{(d+1, d+1, d+2)}$ and
\[
\SN^{(d+1, d+1, d+2)} (p; \ell, \rho)
\le C\SN^{(d+1, d+2)}(a; \ell,\rho) \SN^{(d+1, d+2)}(b, \ell, \rho).
\]
The symbol of $\op_\a^l(ab)$ is $p(\bx, \bx, \bxi)$. Thus by
\eqref{quantisation_ampl:eq},
\[
\|\op_\a^l(a) \op_\a^r(b) - \op_\a^l(ab)\|_{\GS_1}
\le C(\a\ell\rho)^{d-1}
\SN^{(d+1, d+2)}(a; \ell,\rho) \SN^{(d+1, d+2)}(b, \ell, \rho).
\]
Together with \eqref{lltolr:eq} this gives \eqref{product_smooth:eq}.

The case when $b$ satisfies \eqref{a_support:eq}
is done in a similar way, and we omit the proof.
\end{proof}

We need also an estimate for the
trace of the difference $(\op_\a^l(a))^m - \op_\a^l(a^m)$
with an explicit control of the dependence on $m$.

\begin{lem}\label{multiple:lem}
Let $a\in\BS^{(d+2, d+2)}$ satisfy \eqref{a_support:eq}.
Suppose that $\a\ell\rho\ge 1$. Then
\begin{equation}\label{multiple:eq}
\|\bigl(\op_\a^l(a)\bigr)^m - \op_\a^l(a^m)\|_{\GS_1}
\le D^{m-1} (m-1)^{2(d+2)}
(\a\ell\rho)^{d-1} \bigl(\SN^{(d+1, d+2)}(a; \ell, \rho)\bigr)^m,
\end{equation}
for any $m = 1, 2, \dots$, with the constant
\[
D = D_{tr} + D_{norm},
\]
where $D_{norm}, D_{tr}$ are the constants in the bounds \eqref{general_norm:eq}
and \eqref{product_smooth:eq} respectively under the condition $\a\ell\rho\ge 1$.
\end{lem}

\begin{proof}
Suppose without loss of generality that $\ell=\rho = 1$ and
$\SN^{(d+1, d+2)}(a; 1, 1) = 1$,
so that by Lemma \ref{general_norm:lem},
\[
\|\op_\a^l(a)\|\le D_{norm},\ \textup{if}\ \a\ge 1.
\]
Note that
\begin{equation}\label{multiple1:eq}
\SN^{(d+1, d+2)}(a^p; 1, 1)\le p^{2(d+2)},
\end{equation}
for any $p = 1, 2, \dots$.

We prove the estimate \eqref{multiple:eq}
by induction.
For $m=2$ it follows from \eqref{product_smooth:eq} that
\begin{equation*}
\|\bigl(\op_\a^l(a)\bigr)^2 - \op_\a^l(a^2)\|_{\GS_1}
\le D_{tr} \a^{d-1}\le D \a^{d-1}.
\end{equation*}
Suppose that the sought estimate holds for $m = p$, i.e.
\begin{equation}\label{multiple_p:eq}
\|\bigl(\op_\a^l(a)\bigr)^p - \op_\a^l(a^p)\|_{\GS_1}
\le D^{p-1} (p-1)^{2(d+2)} \a^{d-1},
\end{equation}
and let us show that it holds for $m = p+1$. Rewrite:
\begin{align*}
\bigl(\op_\a^l(a)\bigr)^{p+1} - \op_\a^l(a^{p+1})
= &\ \bigl[\bigl(\op_\a^l(a)\bigr)^p
- \op_\a^l(a^p) \bigr]\op_\a^l(a)\\[0.2cm]
&\ + \bigl[\op_\a^l(a^p) \op_\a^l(a)  - \op_\a^l(a^{p+1})\bigr].
\end{align*}
By \eqref{multiple_p:eq}, \eqref{multiple1:eq} and
\eqref{product_smooth:eq},
\begin{align*}
\|\bigl(\op_\a^l(a)\bigr)^{p+1} - \op_\a^l(a^{p+1})\|_{\GS_1}
\le &\ D^{p-1} (p-1)^{2(d+2)} D_{norm}\a^{d-1}  + D_{tr} p^{2(d+2)}\a^{d-1}\\[0.2cm]
\le &\ p^{2(d+2)}\a^{d-1} \bigl(D^{p-1} D_{norm} + D_{tr}
\bigr).
\end{align*}
Since $D_{norm}\ge 1$, the right-hand side of the above bound does not
exceed $D^p p^{2(d+2)} \a^{d-1}$, as required.
\end{proof}

%%%%%%%%%%%%%%%%%%%%%%%%%%%%%%%%%%%%%%%%%%%%%%%%%%%%%%%%%%%%%%%%%%%%%%%%%%%%%%%%%%%%

\section{Trace-class estimates for operators with non-smooth symbols}
\label{nonsmooth1:sect}

Here we obtain trace class estimates for operators
with symbols having jump discontinuities. More precisely, we consider symbols containing
characteristic functions $\chi_{\L}$ and/or $P_{\Om, \a}$.
For $d\ge 2$ both domains $\L$ and $\Om$
are supposed to be $\plainC1$-graph-type domain.
We intend to consider the cases $d\ge 2$ and $d=1$ simultaneously. For $d=1$ we
assume as a rule that $\L$ and $\Om$ are half-infinite open intervals.
For the reference convenience we state these conditions explicitly:

\begin{cond}\label{graph:cond}
If $d = 1$, then $\L = \{x\in\R: x\ge x_0\}$ and
$\Om = \{\xi\in\R: \xi\ge \xi_0\}$ with some $x_0, \xi_0\in\R$.

If $d\ge 2$, then both $\L$ and $\Om$ are $\plainC1$-graph-type domains
in the sense of Definition \ref{domains:defn}, i.e.
$\L = \G(\Phi; \boldO_{\L}, \boldk_{\L})$,
$\Om = \G(\Psi; \boldO_{\Om}, \boldk_{\Om})$
with some Euclidean isometries
$(\boldO_{\L}, \boldk_{\L})$,
$(\boldO_{\Om}, \boldk_{\Om})$ and some $\plainC1$-functions $\Phi$, $\Psi$,
satisfying the conditions \eqref{propertyphi:eq}.
\end{cond}

The precise values of the constants
$M_{\Phi} = \|\nabla\Phi\|_{\plainL\infty}$
and $M_{\Psi} = \|\nabla\Psi\|_{\plainL\infty}$ do not play any role in this section.
In fact, our results will be uniform in the functions
$\Phi$, $\Psi$, satisfying the condition
\begin{equation}\label{gradient:eq}
\max(M_{\Phi}, M_{\Psi}) \le M,
\end{equation}
with some constant $M$.
Referring to the unitary equivalence \eqref{chilambda:eq}, we
often assume that either $\boldO_{\L} = \boldI, \boldk_{\L} = \mathbf 0$,
or $\boldO_{\Om} = \boldI, \boldk_{\Om} = \mathbf 0$.

\subsection{A partition of unity}
For any $\plainC1$-graph-type domain $\L$
we construct a specific partition of unity on $\R^d$.
The following remark on the domain $\L = \G(\Phi; \boldI, \bzero)$ will be useful
for this construction.
In view of the condition \eqref{gradient:eq},
\begin{equation}\label{linside:eq}
|x_d - \Phi(\hat\bx) - \bigl(y_d - \Phi(\hat\by)\bigr)|
\le \lu M\ru\ |\bx - \by|, \ \lu M\ru := \sqrt {1+M^2}
\end{equation}
for all $\bx, \by\in\R^d$.
 In particular, if $\bx\in \L$ and $\by\notin\L$,
 i.e. if $x_d > \Phi(\hat\bx)$ and $y_d \le \Phi(\hat\by)$, we have
 $x_d - \Phi(\hat\bx)\le \lu M\ru\ |\bx-\by|$, and hence
\begin{equation}\label{dist:eq}
\dist \{\bx, \complement\L\}\ge \frac{1}{\lu M\ru}\bigl(
x_d - \Phi(\hat\bx)\bigr), \ \ \bx\in \L = \G(\Phi; \boldI, \bzero),
\end{equation}
where $\complement\L$ is the standard notation for the complement of $\L$.

\begin{lem} \label{boundary_partition:lem}
Let $\L = \G(\Phi; \boldO, \boldk)$
be a graph type domain with $\Phi$ satisfying the condition
\eqref{gradient:eq},
and let
\begin{equation*}
\L^{(t)} = \G(\Phi+t; \boldO, \boldk),\ t\in\R.
\end{equation*}
Then
for any $\d>0$ there exist two non-negative functions
$\z_1 = \z_1^{(\d)}, \z_2 = \z_2^{(\d)}\in\plainC\infty(\R^d)$ such that
$\z_1+\z_2 = 1$,
\begin{equation*}
\z_1(\bx) =
\begin{cases}
0,\ \  \textup{ if}\ \  \bx\notin\L^{(\d)},\\
1,\ \  \textup{ if}\ \  \bx\in\L^{(3\d\lu M\ru)},
\end{cases}
\end{equation*}
and
\begin{equation}\label{boundary_partition:eq}
|\nabla^s\z_1(\bx)| + |\nabla^s\z_2(\bx)|\le C_s\d^{-s}, \ s = 1, 2, \dots,
\end{equation}
uniformly in $\bx$, with constants $C_s$, which may depend on $M$, but
are independent of the function $\Phi$, satisfying
\eqref{gradient:eq}, on the transformation $\boldO$ or on the vector $\boldk$.
\end{lem}

\begin{proof} Since the left-hand side of \eqref{boundary_partition:eq}
is invariant with respect to
the orthogonal transformations and translations,
without loss of generality we may assume that $\boldO = \boldI$ and $\boldk = \bzero$.
Let $\psi_j\in\plainC\infty_0, j = 1, 2, \dots,$ be a partition of unity of $\R^d$
subordinate to the covering
of $\R^d$ by balls of radius $2$ centred at the points of the integer lattice $\Z^d$.
Clearly,
\[
|\nabla^s\psi_j(\bx)|\le C_s,
\]
uniformly in $j$. Denote
\[
\phi_j(\bx) = \psi_j(\bx\d^{-1}),
\]
so that the functions $\phi_j\in\plainC\infty_0, j= 1, 2, \dots, $ form a partition of
unity subordinate to the covering of $\R^d$ by balls of radius $2\d$, centred
at the points of the lattice $(\d\Z)^d$. Moreover, $\phi_j$ have the property that
\begin{equation}\label{diff_phi:eq}
|\nabla^s\phi_j(\bx)|\le C_s\d^{-s},
\end{equation}
uniformly in $\bk$. Define
\begin{equation*}
\z_1(\bx) = \sum_{j} \phi_j(\bx),
\end{equation*}
where the summation is taken over all indices $j$ such that
\[
\phi_j(\bx) \chi_{\L^{(\d)}}(\bx) = \phi_j(\bx),
\]
so that $\z_1(\bx) = 0$ for $\bx\notin\L^{(\d)}$.
By \eqref{dist:eq} the distance from any point $\bx\in\L^{(\e)}, \e > \d$
to $\complement\L^{(\d)}$ is bounded from below by $\lu M\ru^{-1}(\e-\d)$, so
we conclude that
\[
\dist\{\overline{\L^{(R)}}, \complement\L^{(\d)}\}\ge 2\d,\ \ R = 3\d\lu M\ru.
\]
Consequently $\z_1(\bx) = 1$ for all $\bx\in\L^{(R)}$, as required.
Define $\z_2 = 1- \z_1$. Then \eqref{boundary_partition:eq} is satisfied in
view of \eqref{diff_phi:eq}. This completes the proof.
\end{proof}

For the functions $\z_1, \z_2$ constructed in the above lemma we write sometimes
$\z_1^{(\d)}(\bx; \L)$ and $\z_2^{(\d)}(\bx; \L)$.

\subsection{Trace class estimates} Assume as before, that Condition \ref{graph:cond}
is satisfied. If $d\ge 2$, the in all subsequent estimates the constants
are independent of the transformations $(\boldO_{\L}, \boldk_{\L})$ and
$(\boldO_{\Om}, \boldk_{\Om})$, and are uniform in
the functions $\Phi$ and $\Psi$ satisfying \eqref{gradient:eq}, but may depend on the
constant $M$ in \eqref{gradient:eq}. If $d = 1$, then the estimates are uniform in
the numbers $x_0, \xi_0$, which enter the definitions of $\L$ and $\Om$.

As in the previous section we assume as a rule that the symbols
are compactly supported and satisfy the condition \eqref{a_support:eq}.
The constants in the obtained estimates will be independent of
$\bz$, $\bmu$ and $\ell, \rho$.

\begin{lem}\label{sandwich:lem}
Suppose that the symbol $a\in\BS^{(d+1, d+2)}$
satisfies \eqref{a_support:eq} if $d\ge 2$, and that it is supported
on $\R\times B(\mu, \rho)$ if $d = 1$, with some $\mu\in\R$.
Assume that $\a\ell\rho\ge c$.
Let $\op_\a(a)$ denotes any of the operators
$\op^l_\a(a)$ or $\op^r_\a(a)$.  Then
\begin{equation}\label{sandwich1:eq}
\|\chi_{\L} \op_\a(a) (1-\chi_\L)\|_{\GS_1}\le C
(\a\ell\rho)^{d-1}
 \SN^{(d+1, d+2)}(a; \ell, \rho),
\end{equation}
and
\begin{equation}\label{sandwich2:eq}
\|[\op_\a(a), \chi_\L]\|_{\GS_1}\le C  (\a\ell\rho)^{d-1}\
\SN^{(d+1, d+2)}(a; \ell, \rho).
\end{equation}

\end{lem}

\begin{proof}
For definiteness we prove the above estimates for the operator
$\op_\a^l(a)$.

The estimate \eqref{sandwich2:eq} follows from
from \eqref{sandwich1:eq} due to the identity
\[
[\op^l_\a(a), \chi_\L] = (1-\chi_\L)\op^l_\a(a) \chi_\L
-  \chi_\L \op^l_\a(a) (1-\chi_\L).
\]
Let us prove \eqref{sandwich1:eq}.
In view of \eqref{orthogonal1:eq}, \eqref{orthogonal:eq}
and \eqref{chilambda:eq} we may assume that $\boldO_{\L} = \boldI, \boldk = \mathbf 0$,
i.e. $\L = \G(\Phi; \boldI, \bzero)$, if $d\ge 2$, and
$\L = (0, \infty)$, if $d = 1$.

Assume that $d\ge 2$.
Using \eqref{scale:eq}, \eqref{unitary_domain:eq} and \eqref{unitary:eq}
with $\ell_1 = (\a\rho)^{-1}$, $\rho_1 = \rho$, one reduces
the estimate \eqref{sandwich1:eq} to
\begin{equation*}
\|\chi_{\L_{\ell_1}} \op^l_1(a_{\ell_1, \rho_1})
(1-\chi_{\L_{\ell_1}})\|_{\GS_1}\le C (\a\ell\rho)^{d-1} \
\SN^{(d+1, d+2)}
(a_{\ell_1, \rho_1}; \a\rho\ell, 1),
\end{equation*}
where we have used the notation $\L_{\ell_1} = \G(\Phi_{\ell_1}; \boldI, \bzero)$,
$\Phi_{\ell_1}(\hat\bx) = \ell_1^{-1}\Phi(\ell_1\hat\bx)$.
Thus in view of \eqref{scale_domain:eq} it suffices to prove \eqref{sandwich1:eq}
for $\a = \rho = 1$, and arbitrary $\ell\ge c$.

The next step is to
replace the characteristic functions with their smoothed-out versions.
Let $\z_1$ and $\z_2$ be the functions constructed in Lemma
\eqref{boundary_partition:lem} for $\d = 1$. It is clear that
\begin{equation*}
\chi_\L(\bx)\le \z_1(\bx; \L^{(-3\lu M\ru)}),\\
1-\chi_{\L}(\bx)\le \z_2(\bx; \L).
\end{equation*}
Therefore it suffices to estimate
the trace-norm of the operator
\begin{equation*}
\z_1\op_1^l(a) \z_2.
\end{equation*}
Denote
\begin{equation}\label{ptoq:eq}
p(\bx, \by, \bxi) =
\z_1(\bx; \L^{(-3\lu M\ru)})a(\bx, \bxi)\z_2(\by; \L),
\end{equation}
so that by \eqref{boundary_partition:eq},
\begin{equation*}
|\nabla_\bx^{n_1} \nabla_{\by}^{n_2} \nabla_{\bxi}^m p(\bx, \by, \bxi)|
\le C \SN^{(n_1, m)} (a; \ell, 1)
\chi_{\bz, \ell}(\bx) \chi_{\bmu, 1}(\bxi),
\end{equation*}
for all $n_1\le d+1, m\le d+2$ and all $n_2$.
Here we have used the fact that $\ell\ge c$.
For the sake of brevity assume, without loss of generality, that
\[
\SN^{(d+1, d+2)} (a; \ell, 1) = 1.
\]
By Theorem \ref{full:thm} with $Q = d+2$,
\begin{align*}
\|\op_1^a(p)\|_{\GS_1}
\le &\ C \underset{
x_d\ge \Phi(\hat\bx) - 3 \lu M\ru }
\int\ \ \
\underset{y_d\le \Phi(\hat\by)+3\lu M\ru}
\int\int\frac{
\chi_{\bz, \ell}(\bx)\ \chi_{\bmu, 1}(\bxi)
}
{1+|\bx-\by|^{d+2}} d\bxi d\by d \bx\\[0.2cm]
\le &\ C\bigl( I_1(\ell) + I_2(\ell) \bigr),
\end{align*}
where
\begin{align*}
I_1(\ell) = &\ \underset{
x_d\ge \Phi(\hat\bx) + 3\lu M\ru  }
\int\ \
\underset{y_d\le \Phi(\hat\by)+3\lu M\ru}
\int\frac{
\chi_{\bz, \ell}(\bx)
}
{1+|\bx-\by|^{d+2}} d\by d\bx, \\[0.2cm]
I_2(\ell) = &\ \underset{
|x_d- \Phi(\hat\bx)|\le 3\lu M\ru }
\int\ \ \ \ \
\int\frac{
\chi_{\bz, \ell}(\bx)}
{1+|\bx-\by|^{d+2}}  d\by d\bx.
\end{align*}
By \eqref{dist:eq}, in the integral $I_1(\ell)$ we have
\begin{equation*}
|\bx-\by|\ge \frac{1}{\lu M\ru}(x_d - \Phi(\hat\bx) - 3\lu M\ru).
\end{equation*}
Thus we get
\begin{align*}
I_1(\ell)\le &\ C
\underset{\hat\bx\in B(\hat\bz, \ell)}
\int\ \ \underset
{x_d \ge \Phi(\hat\bx)+ 3\lu M\ru}\int\ \
\bigl(1+|x_d - \Phi(\hat\bx)-3\lu M\ru|^2\bigr)^{-1}  d x_d d\hat\bx \\[0.2cm]
\le &\ \tilde C
\underset{\hat\bx\in B(\hat\bz, \ell)}\int
d\hat\bx \le C' \ell^{d-1}.
\end{align*}
The integral $I_2(\ell)$ is estimated in a more straightforward way:
\[
I_2(\ell)\le C\
\underset{\hat\bx\in B(\hat\bz, \ell)}
\int\ \ \underset{
|x_d - \Phi(\hat\bx)|\le 3\lu M\ru}
\int  dx_d d\hat\bx
\le \tilde C \underset{\hat\bx\in B(\hat\bz, \ell)}
\int d\hat\bx
\le C' \ \ell^{d-1},
\]
so that $I_1(\ell) + I_2(\ell)\le C \ell^{d-1}$,
which entails \eqref{sandwich1:eq}.

Suppose now that $d = 1$.
As in the case $d\ge 2$ it suffices to prove the estimate for $\a=\rho = 1$
and arbitrary $\ell\ge c$.
 Let $p(x, y, \xi) = \z_1(x) a(x, \xi) \z_2(y)$
with $\z_1, \z_2\in\plainC\infty$ such that
\begin{equation*}
\z_1(x) =
\begin{cases}
1, x>0,\\
0, x<-1,
\end{cases}
\z_2(x) =
\begin{cases}
1, x<0,\\
0, x>1.
\end{cases}
\end{equation*}
Then
\begin{equation*}
|\p_x^{n_1}\p_y^{n_2}\p_{\xi}^{m} p(x, y, \xi)|
\le C \SN^{(n_1, m)}(a; \ell, 1)   \chi_{\mu, 1}(\xi).
\end{equation*}
Therefore Theorem
\ref{full:thm} with $Q = 3$ gives
\begin{equation*}
\| \op_1^a(p)\|_{\GS_1}\le C \SN^{(2, 3)}(a; \ell, 1)
\underset{x\ge  - 1}\int \ \ \underset{y\le +1}\int \int
\frac{\chi_{\mu, 1}(\xi)}{1+|x-y|^3} d\xi dx dy
\le C'\SN^{(2, 3)}(a; \ell, 1),
\end{equation*}
which leads to \eqref{sandwich1:eq}
\end{proof}

\begin{cor}
Let $d = 1$, and let $a\in \BS^{(3)}$ be a function supported on
$(\mu-\rho, \mu + \rho)$ with some $\mu\in\R, \rho >0$.
Suppose that $\L = (x_0, \infty)$ with some $x_0 \in\R$. Then
\begin{equation}\label{sandwich_111:eq}
\|\chi_{\L} \op_\a(a) (1-\chi_\L)\|_{\GS_1}\le C
\SN^{(3)}(a; \rho),
\end{equation}
and
\begin{equation}\label{sandwich_121:eq}
\|[\op_\a(a), \chi_\L]\|_{\GS_1}\le C
\SN^{(3)}(a; \rho),
\end{equation}
for all $\a>0$ and $\rho>0$, uniformly in $x_0, \mu\in\R$.
\end{cor}

\begin{proof}
Consider $a=a(\xi)$ as a function of two variables $x, \xi$, so that
\[
\SN^{(2, 3)} (a; \ell, \rho)=\SN^{(3)}(a; \rho)
\]
for any $\ell >0$. Thus, using \eqref{sandwich1:eq} and \eqref{sandwich2:eq}
with $\ell = (\a\rho)^{-1}$ we get \eqref{sandwich_111:eq} and \eqref{sandwich_121:eq}
for arbitrary $\a>0$ and $\rho > 0$.
\end{proof}

We need a similar result for the operator $P_{\Om, \a}$ instead of $\chi_\L$:

\begin{lem}\label{sandwich_dual:lem}
Suppose that the symbol $a\in\BS^{(d+2, d+1)}$ satisfies \eqref{a_support:eq},
and that it is supported on
$(z-\ell, z+\ell)\times\R$ with some $z\in\R$  if $d = 1$.
Assume that $\a\ell\rho\ge c$.
Let $\op_\a(a)$ denotes any of the operators
$\op^l_\a(a)$ or $\op^r_\a(a)$.  Then
\begin{equation}\label{sandwich1_dual:eq}
\|P_{\Om, \a} \op_\a(a) (1-P_{\Om, \a})\|_{\GS_1}\le C (\a\ell\rho)^{d-1}
\SN^{(d+2, d+1)}(a; \ell, \rho),
\end{equation}
and
\begin{equation}\label{sandwich2_dual:eq}
\|[\op_\a(a), P_{\Om, \a}]\|_{\GS_1}\le C  (\a\ell\rho)^{d-1}
\SN^{(d+2, d+1)}(a; \ell, \rho).
\end{equation}
\end{lem}

\begin{proof} Inverting the roles of the variables $\bx$ and $\bxi$,
one obtains the above estimates directly from Lemma \ref{sandwich:lem}.
\end{proof}

Let us prove the analogue of Corollary \ref{product:cor} for the operators
of the form $T(a; \L, \Om)$.

\begin{cor}\label{product1:cor}
Let $d\ge 1$ and $\a\ell\rho\ge c$.
Let $a, b\in\BS^{(d+2, d+2)}$ be symbols, such that
either $a$ or $b$
satisfies the condition \eqref{a_support:eq}.
Then
\begin{gather}
\|T(a) T(b) - T(1)T(ab)\|_{\GS_1}
+ \|T(b) T(a) - T(ab)T(1)\|_{\GS_1}
+ \|[T(a), T(b)]\|_{\GS_1}\notag\\[0.2cm]
\le
C  (\a\ell\rho)^{d-1}
\SN^{(d+2, d+2)}(a; \ell, \rho)
\SN^{(d+2, d+2)}(b; \ell, \rho).\label{product:eq}
\end{gather}
\end{cor}

\begin{proof} For definiteness assume that
$a$ satisfies \eqref{a_support:eq}.
For brevity we write $\op$ instead of $\op_\a^l$. Then
\begin{align*}
\|T(a) T(b) - T(1)T(ab)\|_{\GS_1}
\le &\ 2\|[\op(a), P_{\Om, \a}]\|_{\GS_1} \|\op(b)\|\\[0.2cm]
+ &\ \|[\op(a), \chi_\L]\|_{\GS_1} \|\op(b)\|
+ \|\op(a)\op(b) - \op(ab)\|_{\GS_1}.
\end{align*}
The claimed estimate for $T(a)T(b) - T(1) T(ab)$
follows from \eqref{sandwich2:eq}, \eqref{sandwich2_dual:eq}
Lemma \ref{general_norm:lem}
and Corollary \ref{product:cor}. In the same way one proves the required bound
for the trace norm of $T(b) T(a) - T(ab) T(1)$.

Using any of these two bounds for the symbols $a' = ab$
and $b' = 1$ we get the same
estimate for the trace norm of $T(ab)T(1) - T(1)T(ab)$, which
leads to \eqref{product:eq} for the commutator $[T(a), T(b)]$.
\end{proof}

We need a version of the above lemma for the one-dimensional case
with a specific choice of $a$ and $b$:

\begin{lem}\label{product_1d:lem}
Let $d = 1$, and let $a=a(\xi), b= b(\xi)$ be functions from $\BS^{(3)}$
such that either $a$ or $b$ is supported on
$(\mu-\rho, \mu+\rho)$ with some $\mu \in\R$ and $\rho>0$.
Assume that $\L = (x_0, \infty)$
with some $x_0\in\R$, and let $\Om$
be an arbitrary open subset of $\R$.
Then for all $\a>0$, $\rho>0$ we have
\begin{align}\label{product_1d:eq}
\|T(a) T(b) - &\ T(ab)T(1)\|_{\GS_1}
+ \|T(a) T(b) - T(1)T(ab)\|_{\GS_1}\notag \\[0.2cm]
+&\  \|[T(a), T(b)]\|_{\GS_1}
\le C\SN^{(3)}(a; \rho)\  \SN^{(3)} (b; \rho),
\end{align}
with a constant independent of $x_0, \mu$ and $\Om$.
\end{lem}

\begin{proof} For definiteness assume that
$b$ is supported on $(\mu-\rho, \mu +\rho)$.
For brevity we write
$\op$ instead of $\op_\a^l$.
Since $[\op(b), P_{\Om, \a}] = 0$ and
$\op(a)\op(b) = \op(ab)$, we can write
\begin{align*}
\|T(a) T(b) - &\ T(ab)T(1)\|_{\GS_1}
+ \|T(b) T(a) - T(1)T(ab)\|_{\GS_1}\\[0.2cm]
\le &\ 2\|[\op(b), \chi_\L]\|_{\GS_1} \|\op(a)\|
\le C\|a\|_{\plainL\infty} \SN^{(3)} (b; \rho).
\end{align*}
At the last step we have used \eqref{sandwich_121:eq}.
Using either of the above estimates for $b' = ab, a' = 1$,
we get the  estimate
\eqref{product_1d:eq} for the trace norm of $T(ab)T(1) - T(1)T(ab)$.
This leads to
the same bound for the commutator $[T(a), T(b)]$.
\end{proof}

%%%%%%%%%%%%%%%%%%%%%%%%%%%%%%%%%%%%%%%%%%%%%%%%%%%%%%%%%%%%%%%%%%%%%%%%%%%%%%%%%%%%

\section{Further trace-class estimates for operators with non-smooth symbols}
\label{nonsmooth2:sect}

\subsection{Weights with disjoint supports}
Here we obtain more special bounds for trace norms of PDO's with discontinuous symbols.

As in Section \ref{nonsmooth1:sect}, \underline{we always assume that both domains
$\L$ and $\Om$ satisfy Condition
\ref{graph:cond}}. To avoid
cumbersome proofs in this section we always assume that $d\ge 2$, although many of the estimates
easily generalize to $d=1$.
As before, in all subsequent estimates the constants are independent
of the functions $\Phi$ and $\Psi$ defining $\L$ and $\Om$,
but may depend on the constant $M$ in \eqref{gradient:eq}.

It is technically convenient to introduce
smoothed-out versions of the characteristic functions of the balls.
Let $h,\eta\in\plainC\infty_0(\R^d)$ be two non-negative functions such that
$0\le h\le 1$, $0\le \eta\le 1$,
$h(\bx) = 1$ for $|\bx|\le 1$, and $h(\bx) = 0$ for $|\bx|\ge 5/4$;
$\eta(\bxi) = 1$ for $|\bxi|\le 1$, and $\eta(\bxi) = 0$ for $|\bxi|\ge 5/4$.
Denote
\begin{equation}\label{h_eta:eq}
\begin{cases}
h_{\bz, \ell}(\bx) = h\bigl((\bx-\bz)\ell^{-1}\bigr);\\[0.2cm]
\eta_{\bmu, \rho}(\bxi) = \eta\bigl((\bxi-\bmu)\rho^{-1}\bigr),\ \
\Xi_{\bmu, \rho} = \op_{\a} (\eta_{\bmu, \rho}).
\end{cases}
\end{equation}
Constants in all estimates will be independent
of $\bz, \bmu$ and $\rho, \ell$.

\begin{lem}\label{HS1:lem}
Let   $\a\ell\rho\ge c$. Then for any $r \ge 8/5$,
 \begin{equation}\label{HS1:eq}
 \|h_{\bz, \ell} \Xi_{\bmu, \rho} P_{\Om, \a}
 \bigl(1-h_{\bz,  r\ell}\bigr)\|_{\GS_1}
+ \|\bigl(1-h_{\bz,  r\ell}\bigr)\Xi_{\bmu, \rho} P_{\Om, \a}
h_{\bz, \ell}   \|_{\GS_1}
\le C (\a\ell\rho)^{d-1} r^{-\om},
\end{equation}
with arbitrary $\om < 1/2$,
uniformly in $\bz, \bmu\in\R^d$. The constant $C$ is independent  of the parameters
$\a, \ell, \rho$.
\end{lem}

\begin{proof}
For brevity we write $\op_\a$ instead of $\op^l_\a$.
Furthermore, we make some assumptions which do not restrict generality.
\begin{enumerate}
\item
By virtue of \eqref{orthogonal1:eq} and \eqref{chilambda:eq} we may assume that
the coordinates are chosen in such a way that $\Om = \G(\Psi; \boldI, \mathbf 0)$;
\item
Using \eqref{scale:eq}, \eqref{scale_domain:eq}
and \eqref{unitary:eq} with $\ell_1 = \ell$ and
$\rho_1 = (\a\ell)^{-1}$, we conclude that it suffices to prove the
estimate for $\ell = \a = 1$ and $\rho\ge c$.
\end{enumerate}

\vskip 0.1cm

\noindent
It is sufficient to establish the sought bound for the first term
on the left-hand side of \eqref{HS1:eq}, since the estimated
operators are mutually adjoint.

Split the symbol $\eta_{\bmu, \rho}\chi_{\Om}$
into two parts: smooth and non-smooth in the following way.
Let $\z_1^{(\d)}(\bxi) = \z_1^{(\d)}(\bxi; \Om)$
and $\z_2^{(\d)}(\bxi) = \z_2^{(\d)}(\bxi; \Om)$ be the functions
constructed in Lemma \ref{boundary_partition:lem} for $\d>0$. Define
\[
\eta_{\bmu, \rho}\chi_{\Om} = \psi_1 + \psi_2\chi_{\Om},\
\psi_j = \eta_{\bmu, \rho}\z_j^{(\d)}, j = 1, 2.
\]
To handle $\psi_1$ note that in view of \eqref{boundary_partition:eq},
\begin{equation}\label{sloy:eq}
|\nabla_{\bxi}^m \psi_1(\bxi)|
\le
\begin{cases}
C\rho^{-m} ,\ \xi_d - \Psi(\hat\bxi)\ge 3\lu M\ru \d,\\[0.2cm]
C\d^{-m},\ |\xi_d -\Psi(\hat\bxi)|\le 3\lu M\ru \d,
\end{cases}
\end{equation}
for all $ \bxi\in B(\bmu, 5\rho/4)$, and all $m =0,1,\dots$.
As $r\ge 8/5$, the distance between the supports of $h = h_{\bz, 1}$ and
$\tilde h = 1 - h_{\bz, r}$ is at least $7r/32$.
Thus, to estimate the trace-class norm, we can use Theorem \ref{separate:thm}. Precisely,
using \eqref{separate2:eq}, one obtains:
\begin{align}\label{tildep+:eq}
\|h&\ \op_1 (\psi_1) \tilde h\|_{\GS_1}\notag\\
\le &\ C r^{d-m}\biggl[\rho^{-m}
\underset{\substack{|\bxi-\bmu|\le 5\rho/4\\
\xi_d - \Psi(\hat\bxi)
\ge 3\lu M\ru\d}}
\int \ \ \ \underset{|\bx-\bz|\le 2}\int
  d\bx d\bxi
+ \d^{-m}
\underset{\substack{|\bxi-\bmu|\le 5\rho/4\\
|\xi_d - \Psi(\hat\bxi)|
\le 3\lu M\ru\d}}\int \ \
\underset{|\bx-\bz|\le 2}\int   d\bx d\bxi
\biggr]\notag\\[0.2cm]
 \le  &\ C\bigl((\rho r)^{d-m}
+ \rho^{d-1}r^{d-m} \d^{1-m}\bigr).
\end{align}
To handle $\op_1(\psi_2\chi_{\Om}) $ we
recall, that by definition of $\z_2^{(\d)}$, the symbol $\psi_2\chi_{\Om}$
is supported in
\[
\{\bxi: |\bxi-\bmu|\le 5\rho/4, |\xi_d - \Psi(\hat\bxi)|\le 3\lu M\ru\d\},
\]
so that $\1\psi\chi_{\Om}\1_{1}\le C\d^{\frac{1}{2}} \rho^{d-1}$,
see \eqref{brackh:eq} for definition of $\1\ \cdot\ \1_s$.
Therefore, by Proposition \ref{BS:prop},
\begin{equation*}
\| h \op_1(\psi_2)P_{\Om, \a} \tilde h\|_{\GS_1}
\le \| h \op_1(\psi_2\chi_{\Om})\|_{\GS_1}\le C\d^{\frac{1}{2}}\rho^{d-1}.
\end{equation*}
Together with \eqref{tildep+:eq} this gives
\[
\| h \Xi_{\bmu, \rho}P_{\Om, \a}\tilde h\|_{\GS_1} \le
C\bigl[(\rho r)^{d-m} + \rho^{d-1} \bigl(r^{d-m}\d^{1-m} + \d^{\frac{1}{2}}\bigr)\bigr].
\]
Taking $\d = r^{-2\om}$ with $\om <1/2$, and making $m$ sufficiently large,
we estimate the right hand side by
\[
C\rho^{d-1} r^{-\om},
\]
which leads to the required estimate \eqref{HS1:eq}.
\end{proof}

\begin{lem}\label{balls:lem}
Let  $\a\ell\rho\ge c$ and
\begin{equation}\label{varkappa:eq}
\vark_1 = \frac{1}{2(2d-1)}.
\end{equation}
Then for any $\vark\in (0, \vark_1)$ there exists a number $r_1>0$, depending
only on $\vark$ and the dimension $d$, such that
\begin{equation}\label{balls:eq}
\|\Xi_{\bmu, \rho} h_{\bz, \ell}\ \chi_{\L} P_{\Om, \a}
\bigl(1- h_{\bz, r\ell}\bigr)\|_{\GS_1}
+ \|\Xi_{\bmu, \rho} h_{\bz, \ell} \ \chi_{\L} P_{\Om, \a}
\bigl(1- \Xi_{\bmu, r\rho}\bigr)\|_{\GS_1}
\le C (\a\ell\rho)^{d-1} r^{-\vark},
\end{equation}
for all $r \ge r_1$,
uniformly in $\bz, \bmu\in\R^d$. The constant
$C$ is independent of the parameters $\a, \ell,\rho, r$.
\end{lem}

\begin{proof}
Denote
\[
R =  \Xi_{\bmu, \rho} h_{\bz, \ell}\ \chi_{\L} P_{\Om, \a} G,
\]
where
\[
G =  1- h_{\bz, r\ell} \ \ \textup{or}\ \ \  G = 1- \Xi_{\bmu, r\rho}.
\]
Represent:
\begin{align*}
R = &\  X_1 X_2 + Y_1 Y_2,\\
 X_1 = &\  \Xi_{\bmu, \rho} h_{\bz, \ell}\
 \chi_{\L}\bigl(1-\Xi_{\bmu, t\rho}\bigr),\
 X_2 = P_{\Om, \a} G,\\
 Y_1 = &\ \Xi_{\bmu, \rho} \chi_{\L},\
 Y_2 =  h_{\bz, \ell}\  \Xi_{\bmu, t\rho}\  P_{\Om, \a} G.
\end{align*}
Clearly,
\[
\|X_2\|, \|Y_1\|\le 1,
\]
so it remains to estimate the trace norms of $X_1, Y_2$. Inverting the roles
of the variables $\bx,\bxi$, we get from Lemma \ref{HS1:lem} that
\begin{equation*}
\|X_1\|_{\GS_1} \le C(\a\ell\rho)^{d-1} t^{-\om},
\end{equation*}
for all $t \ge 8/5$ and arbitrary $\om <1/2$.
As far as $Y_2$ is concerned,
if $G = 1 - \Xi_{\bmu, r\rho}$, then by choosing $t=  4r/5$ we guarantee
that $\Xi_{\bmu, t\rho}G = 0$, so that
\begin{equation*}
\|R\|_{\GS_1} \le \|X_1\|_{\GS_1}\le C(\a\ell\rho)^{d-1} r^{-\om},
\end{equation*}
for all $r\ge 2$. Since $\vark_1 \le  1/2$, this leads to \eqref{balls:eq}
for all $r\ge 2$.

For $G = 1- h_{\bz, r\ell}$, Lemma \ref{HS1:lem} gives:
\begin{equation*}
\|Y_2\|_{\GS_1}\le  C(\a\ell t\rho)^{d-1} r^{-\om},
\end{equation*}
for all $r\ge 8/5$ and $\om > 1/2$.
Thus,
\begin{equation}\label{R:eq}
\|R\|_{\GS_1}\le \|X_1\|_{\GS_1} + \|Y_2\|_{\GS_1}\le C(\a\ell\rho)^{d-1}
\bigl( t^{d-1}r^{-\om} + t^{-\om}\bigr),\ \  t \ge 8/5, r\ge 8/5.
\end{equation}
The minimum of the right hand side is attained
at
\begin{equation}\label{omtokappa:eq}
t = r^{\frac{\om}{d-1+\om}} = r^{\frac{\vark}{\om}}, \ \vark = \frac{\om^2}{d-1+\om}.
\end{equation}
For sufficiently large $r_1$, under the condition $r \ge r_1$ we have $t\ge 2$,
so that \eqref{R:eq} is applicable, and hence
\[
\|R\|_{\GS_1}\le C(\a\ell\rho)^{d-1} r^{-\vark}.
\]
The formula \eqref{omtokappa:eq} maps $\om\in (0, 1/2)$ into $\vark\in (0, \vark_1)$.
This implies \eqref{balls:eq}.
\end{proof}

Now we obtain a more elaborate version of Lemma \ref{balls:lem}.

\begin{lem}\label{HS1_mnogo:lem}
Let  $\a\ell\rho\ge c$.
Let $\vark_j, j = 1, 2, \dots$, be the sequence of positive numbers such that
$\vark_1$ is defined by \eqref{varkappa:eq} and
\begin{equation}\label{vark_sequence:eq}
\vark_{j+1} = \frac{\vark_j \vark_1}{2(d-1)+\vark_1+\vark_j},\ j = 1, 2, \dots.
\end{equation}
Then for any $ p = 1, 2, \dots$, and any $\vark\in (0, \vark_p)$
there exists a number $r_p = r_p(\vark, d) >0$, such that
\begin{align}\label{mnogo1:eq}
 \|\Xi_{\bmu, \rho}  h_{\bz, \ell}  (\chi_{\L} P_{\Om, \a})^p
\bigl(1- h_{\bz, r\ell}\bigr)\|_{\GS_1}
+ &\
\|\Xi_{\bmu, \rho}  h_{\bz, \ell}  (\chi_{\L} P_{\Om, \a})^p
\bigl(1- \Xi_{\bmu, r\rho}\bigr)\|_{\GS_1}\notag\\[0.2cm]
\le &\ C (\a\ell\rho)^{d-1} r^{-\vark},
\end{align}
and
\begin{align}\label{mnogo11:eq}
 \|\bigl(1- h_{\bz, r\ell}\bigr)
  (\chi_{\L} P_{\Om, \a})^p \Xi_{\bmu, \rho}  h_{\bz, \ell} \|_{\GS_1}
+ &\ \|\bigl(1- \Xi_{\bmu, r\rho}\bigr)
  (\chi_{\L} P_{\Om, \a})^p \Xi_{\bmu, \rho}  h_{\bz, \ell} \|_{\GS_1}
   \notag\\[0.2cm]
\le &\ C (\a\ell\rho)^{d-1} r^{-\vark},
\end{align}
for all $r\ge r_p$,
uniformly in $\bz, \bmu\in\R^d$. The constant
$C$ is independent of the parameters $\a, \ell,\rho, r$.
\end{lem}

\begin{proof} We begin by proving \eqref{mnogo1:eq}.
Denote
\begin{align*}
R^{(s)}_1(\ell, \rho; r) = &\ \Xi_{\bmu, \rho}  h_{\bz, \ell}  (\chi_{\L} P_{\Om, \a})^s
\bigl(1- h_{\bz, r\ell}\bigr),\\[0.2cm]
R^{(s)}_2(\ell, \rho; r) =  &\ \Xi_{\bmu, \rho}  h_{\bz, \ell}  (\chi_{\L} P_{\Om, \a})^s
\bigl(1- \Xi_{\bmu, r\rho}\bigr).
\end{align*}
 If $p=1$, then the sought bound holds due to Lemma \ref{balls:lem}.
  For $p\ge 2$ we proceed by induction. Suppose that
for some $s<p, s\ge 1$, we have
 \begin{equation}\label{induction1:eq}
 \| R^{(s)}_1\|_{\GS_1} + \|R^{(s)}_2\|_{\GS_1}
\le C  \ (\a\ell\rho)^{d-1} r^{-\s},
\end{equation}
for arbitrary $\s< \vark_s$ and all $r\ge r_s$.
 Let us show that this implies the same bound for $p = s+1$,  arbitrary
 $\s < \vark_{s+1}$, and $r\ge r_{s+1}$ with some $r_{s+1}$. To this end rewrite:
 \begin{align*}
 R^{(s+1)}_1(\ell,\rho; r) =  &\ S_1 + S_2 + S_3,\\
S_1 =   &\ \Xi_{\bmu, \rho}  h_{\bz, \ell}  (\chi_{\L} P_{\Om, \a})^s
\bigl(1-h_{\bz, t\ell}\bigr) \chi_{\L}P_{\Om, \a}
\bigl(1- h_{\bz, r\ell}\bigr)\\
= &\ R^{(s)}_1(\ell, \rho; t) \chi_{\L}P_{\Om, \a}
\bigl(1- h_{\bz, r\ell}\bigr), \\
S_2 =  &\ \Xi_{\bmu, \rho}  h_{\bz, \ell}    (\chi_{\L} P_{\Om, \a})^s
(1- \Xi_{\bmu, t\rho}) h_{\bz, t\ell} \chi_{\L}P_{\Om, \a}
 \bigl(1- h_{\bz, r\ell}\bigr)\\
 = &\ R^{(s)}_2(\ell, \rho; t) h_{\bz, t\ell} \chi_{\L}P_{\Om, \a}
 \bigl(1- h_{\bz, r\ell}\bigr),\\
 S_3 = &\ \Xi_{\bmu, \rho}  h_{\bz, \ell}   (\chi_{\L} P_{\Om, \a})^s
\Xi_{\bmu, t\rho} h_{\bz, t\ell}  \chi_{\L}P_{\Om, \a}
\bigl(1- h_{\bz, r\ell}\bigr)\\
= &\ \Xi_{\bmu, \rho}  h_{\bz, \ell}   (\chi_{\L} P_{\Om, \a})^s
R^{(1)}_1(t\ell, t\rho; rt^{-1}),
 \end{align*}
with an arbitrary $t>0$.
 It follows from \eqref{induction1:eq} and from Lemma \ref{balls:lem} that
\begin{align*}
\|S_1\|_{\GS_1} + \|S_2\|_{\GS_1}\le &\
C  (\a\ell\rho)^{d-1} t^{-\s_s},\\
\|S_3\|_{\GS_1}\le &\ C (\a\ell t^2\rho)^{d-1}
(rt^{-1})^{-\s_1},
\end{align*}
for any $\s_s < \vark_s$, $\s_1 < \vark_1$,
and all $t \ge r_s$ and  $r t^{-1} \ge  r_1$. Thus
\begin{equation*}
\|R^{(s+1)}_1(\ell, \rho; r)\|_{\GS_1}
\le C (\a\ell\rho)^{d-1}
\bigl(t^{2(d-1)+\s_1} r^{-\s_1} + t^{-\s_s}\bigr).
\end{equation*}
The right hand side is minimized when
\begin{equation}\label{stos+1:eq}
t^{-\s_s} = t^{2(d-1)+\s_1} r^{-\s_1},\  \
\textup{i.e.}\ \ t^{\s_s} = r^{\s_{s+1}},\ \
\s_{s+1} = \frac{\s_1\s_{s}}{2(d-1)+ \s_1+\s_s}.
\end{equation}
This choice of $t$ satisfies $t\ge r_s$ if $r\ge r_{s+1}$ with a sufficiently
large $r_{s+1}$. In addition,
since $\s_{s+1} < \s_{s}$, increasing $r_{s+1}$  if necessary,
one guarantees that $r t^{-1}\ge r_1$.
Moreover, by  definition \eqref{vark_sequence:eq}, the formula
\eqref{stos+1:eq} maps $\s_s\in (0, \vark_s)$ to $\s_{s+1}\in (0, \vark_{s+1})$.
This completes the proof of \eqref{induction1:eq} for $p = s+1$,
and hence, by induction leads to \eqref{mnogo1:eq}.

The bounds \eqref{mnogo11:eq} follow directly from \eqref{mnogo1:eq}. For example,
the first trace norm on the left hand side of \eqref{mnogo11:eq} equals
the trace norm of the adjoint operator, i.e.
\[
\|
h_{\bz, \ell}  \Xi_{\bmu, \rho}  (P_{\Om, \a}\chi_{\L} )^p
 \bigl(1- h_{\bz, r\ell}\bigr) \|_{\GS_1}.
\]
Now, we get the required bound for this trace norm after exchanging the roles
of the variables $\bx$ and $\bxi$ and using
\eqref{mnogo1:eq}. Similarly for the second term on the
left hand side of \eqref{mnogo11:eq}.
\end{proof}

Let $\Upsilon_\d(\bz, \be), \d>0, \bz\in\R^d$,
$\be\in\mathbb S^{d-1}$,
be the ``layer" defined by
\begin{equation}\label{Ups:eq}
\Upsilon_{\d}(\bz, \be) = \{\bx\in\R^d: |(\bx-\bz)\cdot\be| < \d\}.
\end{equation}

\begin{lem}\label{HS2_mnogo:lem}
Suppose that $\d\rho\a\ge c$.
Let the sequence $\vark_j$, j = $1, 2, \dots$ be
as defined in \eqref{vark_sequence:eq}, and
let $r_p = r_p(\vark, d)>0$ be the numbers found in Lemma \ref{HS1_mnogo:lem}
for arbitrary $\vark\in (0, \vark_p)$, $p = 1, 2, \dots$.
Then
\begin{align}\label{mnogo2:eq}
 \|  \Xi_{\bmu, \rho}  h_{\bz, \ell}
 \bigl(\chi_{\L}P_{\Om, \a}\bigr)^p
 \chi_{\Upsilon_{\d}(\bw, \be)}&\ (1-\chi_{B(\bz, 4\ell)})
 \chi_{B(\bz, r\ell)}\|_{\GS_1}\notag\\[0.2cm]
\le &\ C
(\a\ell\rho)^{d-1}
\bigl[
r^{-\vark}
+  r^{2(d-1)}(\ell\d^{-1})^{-\varkappa}
\bigr],
\end{align}
for all $r\ge r_p$, $\ell\ge r_p \d$, uniformly in $\bw, \bz, \bmu\in\R^d$.
The constant
$C$ is independent of  the parameters $\a, \ell,\rho, r, \d$.
\end{lem}

\begin{proof} Let us fix a $\vark \in (0, \vark_p)$.
Replace the operator on the left-hand side by
\[
\Xi_{\bmu, \rho}  h_{\bz, \ell}
 \bigl(\chi_{\L}P_{\Om, \a}\bigr)^p \Xi_{\bmu, r\rho}
 \chi_{\Upsilon_{\d}(\bw, \be)}(1-\chi_{B(\bz, 4\ell)})
 h_{\bz, r\ell}
\]
with $r\ge r_p$. By Lemma
\ref{HS1_mnogo:lem} this operation may change the trace norm
at most by
\[
 C (\a\ell\rho)^{d-1}  r^{-\vark }.
\]
Let $\{B(\bx_l, \d)\}, l = 1, 2, \dots, N$ be a collection of balls covering
$\U_{\d}(\bw, \be)\cap B(\bz, r\ell)$ such that
$B(\bx_l, \d)\subset \U_{2\d}(\bw, \be)$.
Then
\begin{equation*}
 \|h_{\bz, \ell} (\chi_{\L} P_{\Om, \a})^p \ \Xi_{\bmu, r\rho}
 \bigl(1-\chi_{B(\bz, 4\ell)}\bigr)
 \chi_{B(\bx_l, \d)}\|_{\GS_1}
\le
\|\bigl(1- h_{\bx_l, \ell}\bigr)  (\chi_{\L} P_{\Om, \a})^p\  \Xi_{\bmu, r\rho}
\  h_{\bx_l, \d}\|_{\GS_1}.
\end{equation*}
Since $\ell\ge r_p \d$,
it follows from \eqref{mnogo11:eq} that the right hand side
 does not exceed
 \[
 C (\a \d r\rho)^{d-1} (\ell\d^{-1})^{-\varkappa},
 \]
 uniformly in $l = 1, 2, \dots, N$.
Thus
\begin{align*}
\|\Xi_{\bmu, \rho}  h_{\bz, \ell}
 \bigl(\chi_{\L}P_{\Om, \a}\bigr)^p \ \Xi_{\bmu, r\rho}
 \chi_{\U_{\d}(\bw, \be)}& (1-\chi_{B(\bz, 4\ell)})
 \chi_{B(\bz, r\ell)}\|_{\GS_1}\\
\le &\ \sum_{l=1}^N \|\bigl(1- h_{\bx_l, \ell}\bigr)
 (\chi_{\L} P_{\Om, \a})^p\  \Xi_{\bmu, r\rho}
\  h_{\bx_l, \d}\|_{\GS_1}\\
\le &\  C  N(\a \d r \rho)^{d-1} (\ell\d^{-1})^{-\varkappa}.
\end{align*}
Choose the covering in such a way that
the number of balls $N$ does not exceed $C(r \ell\d^{-1})^{d-1}$, and hence
\begin{align*}
\|\Xi_{\bmu, \rho}  h_{\bz, \ell}
 \bigl(\chi_{\L}P_{\Om, \a}\bigr)^p \ \Xi_{\bmu, r\rho}
 \chi_{\U_{\d}(\bw, \be)}& (1-\chi_{B(\bz, 4\ell)})
\chi_{B(\bz, r\ell)}\|_{\GS_1}\\
\le &\  C  (\a \ell  r^2 \rho)^{d-1} (\ell\d^{-1})^{-\varkappa}.
\end{align*}
Returning to the initial operator, we obtain:
\begin{align*}
\|\Xi_{\bmu, \rho}  h_{\bz, \ell}
 \bigl(\chi_{\L}P_{\Om, \a}\bigr)^p \
 \chi_{\U_{\d}(\bw, \be)}& (1-\chi_{B(\bz, 4\ell)})
 \chi_{B(\bz, r\ell)}\|_{\GS_1}\\
\le &\  C (\a \ell \rho)^{d-1}
\bigl[ r^{-\vark}
 + r^{2(d-1)}(\ell\d^{-1})^{-\vark}
\bigr],
\end{align*}
as claimed.
\end{proof}

\subsection{Reduction to the flat boundary}
The above lemmas provide useful tools for the study of the
operator $T(1; \L, \Om)$. Our objective in this subsection
is to show that under suitable conditions one can replace $\L$ by a half-space.
Since the domain $\Om$ remains unchanged, we omit $\Om$ from the notation of $T$
and simply write $T(1; \L)$.

As before we assume that $\L$, $\Om$ satisfy Condition \ref{graph:cond}.
In addition, we assume that $\L = \G(\Phi; \boldI, \mathbf 0)$,
i.e. the domain is given by
\[
\L = \{\bx: x_d > \Phi(\hat\bx)\}.
\]
Define
\begin{equation}\label{pizhat:eq}
\Pi =  \{\bx: x_d > \nabla\Phi(\hat{\mathbf 0})\cdot\hat\bx\}.
\end{equation}
Since
$\nabla\Phi$ is uniformly continuous,
we have
\begin{equation}\label{nabla_mod:eq}
\max_{|\hat\bx-\hat\bz|\le s}|\nabla\Phi(\hat\bx) - \nabla\Phi(\hat{\bz})|
=: \vare(s) \to 0, \ s\to 0,
\end{equation}
so that
\begin{equation}\label{modulus:eq}
\max_{|\hat\bx|\le s}
|\Phi(\hat\bx) - \Phi(\hat{\mathbf 0})
- \nabla\Phi(\hat{\mathbf 0})\cdot\hat\bx|\le \varepsilon(s) s.
\end{equation}
In the next lemma we use the parameters $\vark_k$ defined
in \eqref{vark_sequence:eq}, and also the notation
\begin{equation}\label{rtilde:eq}
\tilde r_p(\vark, d) = \max_{1\le k\le p} r_k(\vark, d),\ 0< \vark < \vark_p,
\end{equation}
where $r_k(\vark, d)$ are the numbers introduced for
$\vark \in (0,  \vark_k)$ in Lemma \ref{HS1_mnogo:lem}.
Since $\vark_k$ is a decreasing sequence, the numbers
$\tilde r_p(\vark, d)$ are well-defined.

\begin{lem}\label{flat_boundary:lem}
Let $\L = \G(\Phi; \boldI, \mathbf 0)$ and $\Pi$ be as defined above.
Let the point $\bz\in\R^d$
and parameters $\ell, t>0$ be such that
\begin{equation}\label{distance:eq}
B(\bz, 4\ell)\subset \L\cap \Pi\cap B(\mathbf 0, t).
\end{equation}
Suppose that $\a\ell\rho\ge c$.
Let the sequence $\vark_j$, j = $1, 2, \dots$ be
as defined in \eqref{vark_sequence:eq}, and
let $\tilde r_p = \tilde r_p(\vark, d)>0$ be the numbers defined
in \eqref{rtilde:eq}. Then for any $p = 1, 2, \dots$, and any $\vark\in (0, \vark_p)$,
under the conditions
\begin{equation}\label{ells:eq}
\tilde r_p\ell\le t,\
2 \tilde r_p\bigl(4t\ \varepsilon(4t)\bigr)\le \ell,
\end{equation}
one has
\begin{equation}\label{flat_boundary:eq}
\| \Xi_{\bmu, \rho} h_{\bz, \ell} \bigl(g_p(T(1; \L))
- g_p(T(1; \Pi))\bigr)\|_{\GS_1}
\le   C
(\a\ell\rho)^{d-1}\CR_\vark (\a; \ell, \rho, t),
\end{equation}
with
\begin{equation}\label{remainder:eq}
\CR_\vark (\a; \ell, \rho, t)= (t\ell^{-1})^{-\vark} +
(t\ell^{-1})^{2(d-1)}(\a\ell\rho)^{-\vark}
+ (t\ell^{-1})^{2(d-1)}\bigl(t \ell^{-1}\varepsilon(4t)\bigr)^{\vark},
\end{equation}
and a constant $C$ independent of $\bz, \bmu$.
\end{lem}

 \begin{proof}
Rewrite the difference of the operators on the left-hand side as
\begin{equation*}
g_p(T(1; \L)) - g_p(T(1; \Pi))
= \sum_{k=0}^{p-1} g_k(T(1; \Pi)) (T(1; \L)-T(1; \Pi))
g_{p-1-k}(T(1; \L)),
\end{equation*}
and estimate every term in this sum individually. Since
\begin{equation*}
T(1; \L) - T(1; \Pi)
= (\chi_{\L}-\chi_{\Pi}) P_{\Om, \a} \chi_{\L}
+  \chi_{\Pi} P_{\Om, \a} (\chi_{\L}-\chi_{\Pi}),
\end{equation*}
the $k$th term in the sum takes the form
\begin{align*}
g_k(T(1; \Pi))(\chi_{\L}-\chi_{\Pi}) P_{\Om, \a}
&\ g_{p-1-k}(T(1; \L)) \\[0.2cm]
+ &\ g_k(T(1; \Pi)) P_{\Om, \a}
(\chi_{\L}-\chi_{\Pi})
g_{p-1-k}(T(1; \L))\\[0.2cm]
= &\ X_1 X_2 + Y_1 Y_2,
 \end{align*}
with
\begin{align*}
X_1 = &\ (\chi_{\Pi}P_{\Om, \a})^k (\chi_{\L} - \chi_{\Pi}),\
X_2 = \chi_{\Pi} P_{\Om, \a} g_{p-k-1}(T(1; \L)),\\[0.2cm]
Y_1 = &\  (\chi_{\Pi}P_{\Om, \a})^{k+1}
(\chi_{\L} - \chi_{\Pi}), \
Y_2 = g_{p-1-k}(T(1; \L)).
\end{align*}
Clearly, the norms of $X_2$ and $Y_2$ do not exceed $1$. Let us estimate the trace norms of
$X_1$ and $Y_1$.
Represent $X_1$ as
\begin{align*}
X_1 = &\ X_{11} + X_{12},\\
X_{11} = &\ (\chi_{\Pi}P_{\Om, \a})^k (\chi_{\L} - \chi_{\Pi}) h_{\mathbf 0, 3t},\\
X_{12} = &\ (\chi_{\Pi}P_{\Om, \a})^k  (\chi_{\L} - \chi_{\Pi})
\bigl(1- h_{\mathbf 0, 3t}\bigr).
\end{align*}
Due to the condition \eqref{distance:eq}, $B(\bz, 5t/4)\subset B(\mathbf 0, 3t)$,
so that
\[
1- h_{\bzero, 3t} = (1- h_{\bzero, 3t})(1-h_{\bz, t}).
\]
Pick an arbitrary $\vark < \vark_p$, and assume that
\[
t\ell^{-1}\ge \max_{1\le k\le p} r_k(\vark, d),
\]
where $r_k$ are the numbers from Lemma \ref{HS1_mnogo:lem}.
Thus it follows from Lemma \ref{HS1_mnogo:lem} that
\begin{align}\label{x12:eq}
\| \Xi_{\bmu, \rho} h_{\bz, \ell} X_{12}\|_{\GS_1}
\le &\ \|\Xi_{\bmu, \rho} h_{\bz, \ell}  (\chi_{\Pi} P_{\Om, \a})^k
\bigl(1-h_{\bz, t}\bigr)\|_{\GS_1}\notag\\[0.2cm]
\le &\ C (\a\ell\rho)^{d-1} (t \ell^{-1})^{-\vark}.
\end{align}
In order to estimate $\|X_{11}\|_{\GS_1}$, we note that the difference
$\chi_{\L} - \chi_{\Pi}$ is supported on the set $\L\triangle\Pi$, i.e. on
\begin{equation*}
\{\bx: \Phi(\hat\bx) < x_d < \nabla\Phi(\hat\bzero)\cdot\hat\bx\}
\bigcup
\{\bx: \nabla\Phi(\hat\bzero)\cdot\hat\bx < x_d < \Phi(\hat\bx)\}.
\end{equation*}
By \eqref{modulus:eq},
\begin{equation*}
\max_{\hat\bx: |\hat\bx|\le 4t}
|\Phi(\hat\bx) - \nabla\Phi(\hat\bzero)\cdot\hat\bx|
\le 4t\ \varepsilon(4t) .
\end{equation*}
Thus
\[
\L\triangle\Pi
\cap B(\bzero, 4 t)
\subset \Upsilon_{\nu}(\bzero, \be), \nu = \varepsilon(4t) 4t,
\]
(see \eqref{Ups:eq} for definition of $\Upsilon_{\nu}$), where
\[
\be = \frac{\bigl(-\nabla\Phi(\hat\bzero), 1\bigr)}
{\sqrt{|\nabla\Phi(\hat\bzero)|^2+1}}
\]
is the unit normal to the hyperplane $\Pi$.
 In order to use Lemma
\ref{HS2_mnogo:lem} we must ensure that the thickness $\nu$
satisfies the condition $\a\nu\rho \ge c$, and
hence it is more convenient to
consider a thicker layer $\Upsilon_{\d}(\bzero, \be)$ with
\begin{equation*}
\d = c_0(\a\rho)^{-1}+ 4t\varepsilon(4t),
\end{equation*}
with some $c_0>0$.
Note that due to \eqref{distance:eq},
$h_{\bzero, 3t} = h_{\bzero, 3t} \ \chi_{B(\bz, 8 t)}$  and
\[
\chi_{\L\triangle\Pi} = \chi_{\L\triangle\Pi}\bigl(1-\chi_{B(\bz, 4\ell)}\bigr).
\]
Therefore
\begin{equation}\label{x11:eq}
\| \Xi_{\bmu, \rho} h_{\bz, \ell} X_{11}\|_{\GS_1}
\le \|\Xi_{\bmu, \rho} h_{\bz, \ell}
(\chi_{\Pi} P_{\Om, \a})^k \chi_{\Upsilon_{\d}}(\bzero, \be)
(1-\chi_{B(\bz, 4\ell)}) \chi_{B(\bz, 8 t)}\|_{\GS_1}.
\end{equation}
Now, in order to use Lemma \ref{HS2_mnogo:lem},
let us check that its conditions are satisfied.
It follows from \eqref{ells:eq} and the condition $\a\ell\rho\ge c$ that for a
sufficiently small $c_0$ we get $r_p \d \le \ell$.
Furthermore, in view of the first condition in \eqref{ells:eq}, we have
\begin{equation*}
r:=\frac{8t}{\ell}\ge 8 \tilde r_p.
\end{equation*}
Applying Lemma  \ref{HS2_mnogo:lem} to the right hand side of \eqref{x11:eq}, we get
\begin{align*}
\| \Xi_{\bmu, \rho} h_{\bz, \ell} &\ X_{11}\|_{\GS_1}
\le  C(\a\ell\rho)^{d-1}\bigl[ r^{-\vark}
+ r^{2(d-1)  } (\ell\d^{-1})^{-\vark}
\bigr]\\[0.2cm]
\le &\ C(\a\ell\rho)^{d-1}
\bigl [
(t\ell^{-1})^{-\vark}
+ (t\ell^{-1})^{2(d-1)}
(\a\ell\rho)^{-\vark} + (t\ell^{-1})^{2(d-1)} (t\ell^{-1} \vare(4t))^{\vark}
\bigr].
\end{align*}
Together with \eqref{x12:eq}, this bound leads to the estimate of the form
\eqref{flat_boundary:eq} for the operator $X_1 X_2$.
Similarly, one obtains an estimate for the trace norm of $Y_1$,
which leads to the required bound for  $\|Y_1 Y_2\|_{\GS_1}$. Together, they ensure
\eqref{flat_boundary:eq}.
\end{proof}

In the next lemma we replace the product of the two test functions
$h_{\bz, \ell}(\bx) $ and $\eta_{\bmu, \rho}(\bxi)$ with an arbitrary
compactly supported symbol $b(\bx, \bxi)$.

\begin{lem}\label{flat_boundary_b:lem}
Let $\L$ and $\Pi$ be as in Lemma \ref{flat_boundary:lem}  and let $\a\ell\rho\ge c$.
Assume that for some point $\bz\in\R^d$, some
 $t>0$, and some $\vark \in (0, \vark_p)$
 the conditions \eqref{distance:eq} and \eqref{ells:eq} are satisfied.
 Let $b\in \BS^{(d+1, d+2)}$ be a symbol supported in
$B(\bz, \ell)\times B(\bmu, \rho)$ with some $\bmu\in\R^d$.
Then
\begin{align}\label{flat_boundary_b:eq}
\| \op^l_\a(b) \bigl(g_p(T(1, \L))
- &\ g_p(T(1, \Pi))\bigr)\|_{\GS_1}\notag\\[0.2cm]
\le  &\ C
(\a\ell\rho)^{d-1}\CR_{\vark} (\a; \ell, \rho, t)
\SN^{(d+1, d+2)}(b; \ell, \rho),
\end{align}
with the factor $\CR_{\vark} (\a; \ell, \rho, t)$
defined in \eqref{remainder:eq}.
\end{lem}

\begin{proof}
We write $\op$ instead of $\op_\a^l$.
Rewrite $\op(b)$ as follows:
\begin{equation}\label{b_split:eq}
\op(b) = \op(b) \Xi_{\bmu, \rho} h_{\bz, \ell}
+ \op(b) (1-h_{\bz, \ell}).
 \end{equation}
In view of the bound
\begin{align*}
\| \op^l_\a(b) \Xi_{\bmu, \rho} h_{\bz, \ell}\bigl(g_p(T(1, \L))
- &\ g_p(T(1, \Pi))\bigr)\|_{\GS_1}\\
\le &\ \| \op^l_\a(b) \|\ \|\Xi_{\bmu, \rho} h_{\bz, \ell}\bigl(g_p(T(1, \L))
- g_p(T(1, \Pi))\bigr)\|_{\GS_1},
\end{align*}
the bound \eqref{flat_boundary_b:eq} for the first term on the right hand side of
\eqref{b_split:eq} follows from Lemmas \ref{flat_boundary:lem}
and \ref{general_norm:lem}.
To estimate the second term assume that $h(\bx) = 1$ for all $\bx\in B(0, 9/8)$,
so that the  supports of $b(\ \cdot\ , \bxi)$ and $1-h_{\bz, \ell}$
are separated by the distance of at least $\ell/8$. Thus the bound \eqref{largea:eq}
with $Q = d+2$ gives
\[
\|\op(b) (1-h_{\bz, \ell})\|_{\GS_1}\le C (\a\ell\rho)^{d-2}
\SN^{(d+1, d+2)}(b; \ell, \rho).
\]
Since $\vark<\vark_p <1$,
together with the estimate for the first term this bound produces
\eqref{flat_boundary_b:eq}.
\end{proof}

%%%%%%%%%%%%%%%%%%%%%%%%%%%%%%%%%%%%%%%%%%%%%%%%%%%%%%%%%%%%%%%%%%%%%%%%%%%%%%%%%%%%

\section{A Hilbert-Schmidt class estimate}\label{HS:sect}

Although we mostly use trace class estimates, in the proof of the
asymptotics \eqref{main_s:eq} for arbitrary smooth functions we need one estimate in the
Hilbert-Schmidt class.
As in Sect. \ref{nonsmooth2:sect} we assume that
the domains $\L$ and $\Om$ satisfy Condition
\ref{graph:cond} and that $d\ge 2$. The constants in all
subsequent estimates are independent on the functions
$\Phi, \Psi$, or the parameters $\boldO, \boldk$, but may depend
on the constant $M$ in \eqref{gradient:eq}.

\begin{lem}\label{HSestim:lem}
Let $\L$ and $\Om$ be two domains satisfying Condition
\ref{graph:cond}.
Let $a=a(\bx, \bxi)$ be a symbol from $\BS^{(d+2, d+2)}$,
supported in $B(\bz, \ell)\times B(\bmu, \rho)$ with some
$\bz, \bmu\in\R^d$ and $\ell, \rho >0$.
Let $\op_\a(a)$ denote any of the operators $\op_\a^l(a)$ or $\op_\a^r(a)$,
and let $\a\ell\rho\ge 2$.
Then
\begin{equation}\label{HSestim:eq}
\|\chi_{\L} P_{\Om, \a}\op_\a (a) P_{\Om, \a} (1-\chi_{\L})\|_{\GS_2}^2
\le C(\a\ell\rho)^{d-1}
\log(\a\ell\rho) \bigl(\SN^{(d+2, d+2)}(a; \ell, \rho)\bigr)^2,
\end{equation}
uniformly in $\bz, \bmu\in\R^d$.
\end{lem}

This lemma was first proved in \cite{G1}
(see Theorem 3.2.2, p. 129) and can be also
found in \cite{G2}, Theorem 2.1. Our proof is a minor variation
of that from \cite{G1} and \cite{G2}. It begins with a lemma which
appeared  in \cite{BCT}, Lemma 2.10, and \cite{G1}, Lemma 3.4.1. For the sake
of completeness we provide our proof of this lemma, which is somewhat more elementary than
that of \cite{G1}.

\begin{lem}\label{modulus:lem}
Let $u\in\plainL2(\R^d)$ be a function satisfying the bound
\begin{equation*}
\int_{\R^d} |u(\bx+\bh) - u(\bx)|^2 d\bx\le \vark |\bh|^\b,
\end{equation*}
for any $\bh\in\R^d$, with some $\b\ge 0$ and some constant $\vark>0$.
Then for all $r > 0$,
\begin{equation*}
\int_{|\bxi|\ge r} |\hat u(\bxi)|^2 d\bxi\le C \vark r^{-\b},
\end{equation*}
with a universal constant $C$.
\end{lem}

\begin{proof}
By Parceval's identity,
\begin{align}\label{planch:eq}
\vark|\bh|^\b\ge \int_{\R^d} |u(\bx+\bh) - u(\bx)|^2 d\bx
= &\ \int_{\R^d} |e^{i\bh\cdot\bxi}-1|^2 |\hat u(\bxi)|^2 d\bxi\notag\\[0.2cm]
\ge &\ \int_{|\bxi|\ge r} |e^{i\bh\cdot\bxi}-1|^2 |\hat u(\bxi)|^2 d\bxi.
\end{align}
Integrating in $\bh$ over the ball $|\bh|\le h_0$, we get
\begin{align*}
\int_{|\bh|\le h_0} |e^{i\bh\cdot\bxi}-1|^2 d\bh
= &\ \int_{|\bh|\le h_0} \bigl(2-2\cos(\bh\cdot\bxi)\bigr) d\bh\\[0.2cm]
= &\ 2\w_d h_0^d - 2(2\pi)^{\frac{d}{2}}|\bxi|^{-\frac{d}{2}}
h_0^{\frac{d}{2}} J_{\frac{d}{2}}(h_0|\bxi|),
\end{align*}
where $\w_d$ is the volume of the unit ball in $\R^d$, and
$J_{d/2}$ is the standard Bessel function of order $d/2$. Assume that $h_0 r\ge 1$
so that
\[
J_{\frac{d}{2}}(h_0|\bxi|)\le \frac{C}{\sqrt{h_0|\bxi|}},\ |\bxi|\ge r.
\]
Therefore
\begin{align*}
\int_{|\bh|\le h_0} |e^{i\bh\cdot\bxi}-1|^2 d\bh
\ge 2 h_0^d \biggl(\w_d - \frac{C}{(h_0 r)^{\frac{d+1}{2}}}\biggr).
\end{align*}
Taking $h_0 = C_1 r^{-1}$ with a sufficiently large $C_1$,
we ensure that the right-hand side
is bounded from below by $h_0^d \w_d$, and hence \eqref{planch:eq}
leads to the bound
\[
\w_d h_0^d \int_{|\bxi|\ge r} |\hat u(\bxi)|^2 d\bxi
\le \int_{|\bh|\le h_0}\int_{|\bxi|\ge r} |e^{i\bh\cdot\bxi}-1|^2
|\hat u(\bxi)|^2 d\bxi d\bh
\le C \vark h_0^{\b + d},
\]
whence
\[
\int_{|\bxi|\ge r} |\hat u(\bxi)|^2 d\bxi\le C \vark h_0^\b
= \tilde C \vark r^{-\b},
\]
as claimed.
\end{proof}

Now we use this lemma for a specific choice of the function $u$:

\begin{lem} \label{l2cont:lem}
Let $\Om = \G(\Psi; \boldO, \boldk)\subset\R^d$
be a $\plainC1$-graph-type-domain, and let $\eta_{\bmu, \rho}$,
$\bmu\in\R^d, \rho >0$,
be as defined  in \eqref{h_eta:eq}.
Then
\begin{equation}\label{shift:eq}
\int_{\R^d} |\eta_{\bmu, \rho}(\bxi+\bh)\chi_{\Om}(\bxi+\bh)
- \eta_{\bmu, \rho}(\bxi)\chi_{\Om}(\bxi)|^2 d\bxi
\le C\rho^{d-1} \lu M_{\Psi}\ru |\bh|,
\end{equation}
for all $\bh\in\R^d$, uniformly in $\bmu\in\R^d$ and $\rho >0$.
\end{lem}

\begin{proof} Without loss of generality assume
that $\boldO = \boldI$, $\boldk = \bzero$.
Also, in view
of \eqref{scale_domain:eq} we may assume that $\rho = 1$.

It suffices to prove the estimate \eqref{shift:eq} for $|\bh|<1$.
Estimate the integrand by
\[
|\eta_{\bmu, 1}(\bxi+\bh) - \eta_{\bmu, 1}(\bxi)|
+ \eta_{\bmu, 1}^2 (\bxi)|\chi_{\Om}(\bxi+\bh) - \chi_{\Om}(\bxi)|^2.
\]
Thus the integral on the right-hand side of \eqref{shift:eq} up to a constant
does not exceed $I_1+I_2$ with
\begin{equation*}
I_1 = \int |\eta_{\bmu, 1}(\bxi+\bh) - \eta_{\bmu, 1}(\bxi)| d\bxi,\ \
I_2 = \int \eta_{\bmu, 1}^2 (\bxi)|\chi_{\Om}(\bxi+\bh) - \chi_{\Om}(\bxi)|^2d\bxi.
\end{equation*}
It is straightforward to see that
\[
I_1\le C |\bh| \int_{|\bmu-\bxi|\le 3 } d\bxi
\le C |\bh|.
\]
In $I_2$ the integration is restricted to the set $(\Om-\bh)\triangle\Om$, i.e.
to the set
\begin{equation*}
\{\bxi: \Psi(\hat\bxi + \hat\bh) - h_d \le \xi_d \le \Psi(\hat\bxi)\}
\cup \{\bxi: \Psi(\hat\bxi)\le\xi_d\le\Psi(\hat\bxi + \hat\bh) - h_d \}.
\end{equation*}
On this set we have
\[
|\xi_d - \Psi(\hat\bxi)|
\le |\Psi(\hat\bxi+\hat\bh) - \Psi(\hat\bxi)|+ |\bh|
\le M_{\Psi}|\bh| + |\bh|\le 2\lu M_{\Psi}\ru |\bh|.
\]
Therefore,
\[
I_2\le \underset{|\hat\bxi-\hat\bmu|\le 2}
\int
\ \ \underset{|\xi_d - \Psi(\hat\bxi)|\le 2\lu M_{\Psi}\ru |\bh|}\int d\xi_d d\hat\bxi
\le C \lu M_{\Psi}\ru |\bh|.
\]
This leads to \eqref{shift:eq}.
\end{proof}

\begin{lem}\label{HSestim_prelim:lem}
Let $\L$ and $\Om$ be two domains satisfying Condition
\ref{graph:cond}.
Let $h_{\bz, \ell}, \eta_{\bmu, \rho}$
be functions defined in \eqref{h_eta:eq}.
Suppose that $\a\ell\rho\ge 2$.
Then
\[
\|\chi_{\L} h_{\bz, \ell} \Xi_{\bmu, \rho} P_{\Om, \a}
(1-\chi_{\L})\|_{\GS_2}^2\le C(\a\ell\rho)^{d-1} \log(\a\ell\rho),
\]
uniformly in $\bmu, \bz\in\R^d$.
\end{lem}

\begin{proof} Due to \eqref{unitary:eq}, \eqref{unitary_domain:eq}
and \eqref{scale_domain:eq}
we may assume that $\ell = \rho = 1$, and $\a\ge 2$. Denote $h = h_{\bz, 1}$,
$\eta = \eta_{\bmu, 1}$.
Denote $b(\bxi) = \eta(\bxi) \chi_{\Om}(\bxi)$, and
\[
\check b(\bt)
= \frac{1}{(2\pi)^{\frac{d}{2}}}\int_{\R^d}
e^{i\bt\cdot\bxi} b(\bxi) d\bxi.
\]
Thus
\[
\|\chi_{\L}h \op_{\a}(\eta) P_{\Om, \a} (1-\chi_{\L})\|_{\GS_2}^2
= \frac{\a^{2d}}{(2\pi)^d}
\int\int  |h(\bx)|^2 (1-\chi_{\L}(\by))\chi_{\L}(\bx)
|\check b\bigl(\a(\bx-\by)\bigr) |^2 d\bx d\by.
\]
Split the integral in two: for $|x_d-\Phi(\hat\bx)|\le \a^{-1}$,
which we denote by $I_1$, and for $|x_d-\Phi(\hat\bx)|>\a^{-1}$,
which we denote by $I_2$.
To handle $I_1$ it suffices to use the Parceval's identity:
\[
\int_{\R^d} |\check b(\a(\bx-\by))|^2 d\by = \a^{-d} \int_{\R^d}   |b(\bxi)|^2 d\bxi
\le C \a^{-d},
\]
so that
\begin{equation*}
I_1\le C\a^d
\underset{|x_d-\Phi(\hat\bx)|< \a^{-1}}
\int |h(\bx)|^2 d\bx
\le C\a^{d-1}.
\end{equation*}
For $I_2$ observe that according to \eqref{dist:eq},
\[
|\bx-\by|\ge \frac{1}{\sqrt{\lu M_{\Phi}\ru}}|x_d - \Phi(\hat\bx)|.
\]
By Lemma \ref{modulus:lem},
\[
\underset{|\bx-\by|\ge \frac{1}{\sqrt{M_{\Phi}}}|x_d - \Phi(\hat\bx)|}
\int |\check b(\a(\bx-\by))|^2 d\by
\le C\a^{-1-d} \lu M_{\Psi}\ru \sqrt{\lu M_\Phi\ru}\  |x_d - \Phi(\hat\bx)|^{-1}.
\]
Consequently,
\begin{align*}
I_2\le &\ C\a^{d-1} \underset{|x_d-\Phi(\hat\bx)|\ge \a^{-1}} \int
|h(\bx)|^2 |x_d - \Phi(\hat\bx)|^{-1} d\bx\\[0.2cm]
\le &\ C\a^{d-1}\underset{|\hat\bx-\hat\bz|\le 5/4}\int\
\biggl[\int_{\a^{-1}}^{\a} t^{-1} dt
+ \a^{-1}\underset{|t-z_d|\le 5/4}\int dt\biggr]\ d\hat\bx
= \tilde C \a^{d-1}\log \a.
\end{align*}
 The proof is complete.
\end{proof}

\begin{proof}[Proof of Lemma \ref{HSestim:lem}]
For definiteness we prove the above estimate
for $\op_\a(a) = \op_\a^l(a)$.
Furthermore, it suffices to prove the sought estimate for
$\ell = \rho =1, \a\ge 2$ and $\SN^{(d+2, d+2)}(a; 1, 1) = 1$.
We write $G_1\approx G_2$
with two operators $G_1, G_2$, if
$$
\|G_1 - G_2\|_{\GS_1}\le C\a^{d-1}.
$$
By \eqref{sandwich2_dual:eq},
\[
P_{\Om, \a} \op^l_\a(a) P_{\Om, \a}\approx \op^l_\a(a) P_{\Om, \a}.
\]
Denote $h = h_{\bz, 1}$, $\eta = \eta_{\bmu, 1}$.
Since
$a(\bx, \bxi) = a(\bx, \bxi) h(\bx) \eta(\bxi)$,
according to \eqref{product_smooth:eq},
\[
\op_\a(a)\approx \op_\a(a) \op_\a (h\eta).
\]
Moreover, in view of \eqref{sandwich2:eq},
\[
\chi_{\L} \op_\a(a) \approx \op_\a(a) \chi_{\L}.
\]
Thus
\[
\chi_{\L}\op_\a(a)P_{\Om, \a}(1-\chi_{\L})
\approx \op_\a(a)\chi_{\L} h \op_\a(\eta) P_{\Om, \a}(1-\chi_{\L}) .
\]
Since for any trace class operator, $\|S\|_{\GS_2}^2\le \|S\|_{\GS_1} \|S\|$,
and $\op_\a(a)$ is bounded uniformly in $\a$ (see Lemma \ref{general_norm:lem}),
we have
\[
\|\chi_{\L} P_{\Om, \a}\op_{\a}(a) P_{\Om, \a} (1-\chi_{\L})\|_{\GS_2}^2
\le C\|\chi_{\L} h \op_\a(\eta) P_{\Om, \a}(1-\chi_{\L})\|_{\GS_2}^2 + C\a^{d-1}.
\]
Now the required estimate follows from Lemma
\ref{HSestim_prelim:lem}.
\end{proof}

%%%%%%%%%%%%%%%%%%%%%%%%%%%%%%%%%%%%%%%%%%%%%%%%%%%%%%%%%%%%%%%%%%%%%%%%%%%%%%%%%%%%

 \section{Localisation}\label{localisation:sect}

Now it is time to replace the global assumptions on the domain by the local ones.
Now we do not need to assume that the domains $\L$ and $\Om$ are of graph type,
see Definition \ref{domains:defn}. Instead
we assume that inside a ball of fixed radius,
both $\L$ and $\Om$ are representable either by $\plainC1$-graph-type domains
or by $\R^d$,
in the sense of Definition \ref{local_domains:defn}.
For the reference convenience we state this assumption explicitly:

\begin{cond}\label{global:cond}
The domain $\L$ (resp. $\Om$) satisfies one of the following two conditions:
\begin{enumerate}
\item\label{graph:subcond}
If $d=1$, then for some numbers $x_0\in\R$(resp. $\xi_0\in\R$) and $R>0$,
we have $\L\cap (x_0-R,x_0 + R) = (x_0, x_0+R)$
(resp. $\Om\cap (\xi_0-R,\xi_0 + R) = (\xi_0, \xi_0+R)$ ).

If $d\ge 2$, then
for some point $\bw\in\p\L$ (resp. $\boldeta\in\p\Om$) and some number $R >0$,
in the ball $B(\bw, R)$ (resp. $B(\boldeta, R)$)
the domain $\L$ (resp. $\Om$) is represented by a graph-type
domain $\G(\Phi; \boldO_{\L}, \bw)$
(resp. $\G(\Psi; \boldO_{\Om}, \boldeta)$)
with some $\plainC1$-function $\Phi$ (resp. $\Psi$), satisfying
\eqref{propertyphi:eq},
and an orthogonal transformation $\boldO_{\L}$
(resp. $\boldO_{\Om}$).

\item\label{space:subcond}
If $d=1$, then for some numbers $w\in\R$ (resp. $\eta\in\R$) and $R>0$,
we have $\L\cap (w-R,w + R) = (w-R, w+R)$
(resp. $\Om\cap (\eta-R,\eta + R) = (\eta-R, \eta+R)$).

If $d\ge 2$, then for some point $\bw\in\L$ (resp. $\boldeta\in\Om$) and some number $R >0$,
in the ball $B(\bw, R)$ (resp. $B(\boldeta, R)$)
the domain $\L$ (resp. $\Om$) is represented by  the entire Euclidean space $\R^d$.
\end{enumerate}
\end{cond}

As before, our estimates will be uniform in the functions
$\Phi, \Psi$ satisfying the bound \eqref{gradient:eq} with some constant $M$,
but may depend on the value of $M$.

For $d=1$ we use the notation $\L_0 := (x_0, \infty)$ (resp. $\Om_0: = (\xi_0, \infty)$)
if $\L$ (resp. $\Om$) satisfies
Condition \ref{global:cond}(\ref{graph:subcond}) and
$\L_0 = \R$ (resp. $\Om_0 = \R$) if $\L$ (resp. $\Om$)
satisfies Condition \ref{global:cond}(\ref{space:subcond}).

For $d\ge 2$ we use the notation $\L_0 := \G(\Phi)$ (resp. $\Om_0: = \G(\Psi)$)
if $\L$ (resp. $\Om$) satisfies
Condition \ref{global:cond}(\ref{graph:subcond}) and
$\L_0 = \R^d$ (resp. $\Om_0 = \R^d$) if $\L$ (resp. $\Om$)
satisfies Condition \ref{global:cond}(\ref{space:subcond}).

For brevity we also denote
\begin{equation*}
T(a) = T(a; \L, \Om),\ \ T^{(0)}(a) = T(a; \L_0, \Om_0).
\end{equation*}
In this section we study the trace norms of the operators of the type
\begin{equation*}
\op_\a(b) g_p(T(a; \L, \Om)),\
\end{equation*}
where $g_p(t) = t^p,  p=1, 2, \dots$,  and the symbols $a = a(\bx, \bxi)$,
$b = b(\bx, \bxi)$ satisfy the following conditions:
\begin{equation}\label{ab:eq}
\textup{$a, b\in\BS^{(d+2, d+2)}$,  and $b$ is supported in
$B(\bz, \ell)\times B(\bmu, \rho)\subset B(\bw, R)\times B(\boldeta, R)$,}
\end{equation}
with some $\bz, \bmu\in\R^d$ and $\ell, \rho >0$.
Sometimes in the proofs for brevity
we use the notation $\op(b)$ instead of $\op_\a^l(b)$.
If $\L = \L_0, \Om = \Om_0$, then one can take $R=\infty$.

\begin{lem}\label{omtoom0:lem}
Let $\L, \Om\subset\R^d$ and $\L_0, \Om_0\subset\R^d, d\ge 1,$
be a specified above. Suppose that the symbol $b$ satisfies \eqref{ab:eq}.
Denote by $\op_\a(b)$ any of the operators $\op_\a^l(b)$ or $\op_\a^r(b)$.
Then
\begin{equation}\label{half:eq}
\| \op_\a(b) P_{\Om, \a}\chi_{\L}
- P_{\Om_0, \a} \chi_{\L_0}\op_\a(b)\|_{\GS_1}
\le  C (\a\ell\rho)^{d-1}\SN^{(d+2, d+2)}(b; \ell, \rho),
\end{equation}
and
\begin{equation}\label{half1:eq}
\| \op_\a(b) \chi_{\L} P_{\Om, \a}
- \chi_{\L_0} P_{\Om_0, \a} \op_\a(b)\|_{\GS_1}
\le  C (\a\ell\rho)^{d-1}\SN^{(d+2, d+2)}(b; \ell, \rho).
\end{equation}
\end{lem}

\begin{proof} Without loss of generality assume
\begin{equation*}
\SN^{(d+2, d+2)}(b; \ell, \rho) = 1.
\end{equation*}
In view of \eqref{quantisation_symbol:eq} any of the above inequalities
for $\op^l_\a(b)$ immediately implies the same inequality for $\op^r_\a(b)$.
We prove \eqref{half:eq} for the operator $\op^l_\a(b)$.
The inequality \eqref{half1:eq} is proved in the same way.
Write:
\begin{align*}
\op^l_\a(b) P_{\Om, \a}\chi_{\L}
= &\ \op^l_\a(b) P_{\Om_0, \a}\chi_{\L}\\[0.2cm]
= &\ [\op^l_\a(b), P_{\Om_0, \a}]\chi_{\L} + P_{\Om_0, \a} (\op^l_\a(b) - \op^r_\a(b))\chi_{\L}
+ P_{\Om_0, \a} [\op^r_\a(b), \chi_{\L_0}]\\[0.2cm]
&\ + P_{\Om_0, \a} \chi_{\L_0} (\op^r_\a(b) - \op^l_\a(b))
+ P_{\Om_0, \a} \chi_{\L_0}  \op^l_\a(b).
\end{align*}
Now \eqref{half:eq} follows
by virtue of Lemmas \ref{quantisation_trace:lem}, \ref{sandwich:lem} and
\ref{sandwich_dual:lem}. Similarly one proves \eqref{half1:eq}.
\end{proof}

\begin{lem}\label{ltol0:lem}
Let $T(a)$ and $T^{(0)}(a)$ be as described above.
Let $d\ge 1$, and let $a, b$ be some symbols satisfying condition \eqref{ab:eq}.
Then under the assumption $\a\ell\rho\ge c$ one has
\begin{equation}\label{ltol0_1:eq}
\| T(b) - T^{(0)}(b)\|_{\GS_1}
\le C
(\a\ell\rho)^{d-1} \SN^{(d+2, d+2)}(b; \ell, \rho),
\end{equation}
\begin{equation}\label{ltol0_4:eq}
\|\op^l_\a(b) T(a) - T^{(0)}(ab)\|_{\GS_1}\le C(\a\ell\rho)^{d-1}
\SN^{(d+2, d+2)}(b; \ell, \rho)\SN^{(d+2, d+2)}(a; \ell, \rho),
\end{equation}
and
\begin{align}\label{ltol0_2:eq}
\| \op^l_\a(b)\bigl[
T(a)& - T^{(0)}(a)
\bigr]\|_{\GS_1}
 +  \| T^{(0)}(b)\bigl[
T(a) - T^{(0)}(a)
\bigr]\|_{\GS_1}
\notag\\[0.2cm]
&\ \ \ \ \le C
(\a\ell\rho)^{d-1} \SN^{(d+2, d+2)}(b; \ell, \rho)
\SN^{(d+2, d+2)}(a; \ell, \rho).
\end{align}

\end{lem}

\begin{proof}
Without loss of generality assume that
\begin{equation}\label{normalise:eq}
\SN^{(d+2, d+2)}(b; \ell, \rho)
= \SN^{(d+2, d+2)}(a; \ell, \rho) = 1.
\end{equation}
Denote $\op(b) = \op^l_\a(b)$.
Note that \eqref{half:eq} immediately leads to
\begin{equation}
\| \op(b) (P_{\Om, \a}\chi_{\L} - P_{\Om_0, \a}\chi_{\L_0})\|_{\GS_1}
+ \|(\chi_{\L} P_{\Om, \a} - \chi_{\L_0}P_{\Om_0, \a}) \op(b) \|_{\GS_1}
\le  C(\a\ell\rho)^{d-1}.\label{onesided:eq}
\end{equation}
To prove \eqref{ltol0_1:eq}, write
\begin{align*}
T(b) - T^{(0)}(1) \op(b)
= &\ (\chi_{\L}P_{\Om, \a} - \chi_{\L_0}P_{\Om_0, \a})\op(b) P_{\Om}\chi_{\L}\\[0.2cm]
&\ + \chi_{\L_0} P_{\Om_0, \a}
\bigl(\op(b) P_{\Om, \a}\chi_{\L} - P_{\Om_0, \a} \chi_{\L_0}\op(b)\bigr),
\end{align*}
so that by \eqref{onesided:eq} and \eqref{half:eq},
\begin{align*}
\|T(b) - T^{(0)}(1) \op(b)\|_{\GS_1}
\le &\ \|(\chi_{\L}P_{\Om, \a} - \chi_{\L_0}P_{\Om_0, \a})\op(b)\|_{\GS_1}\\[0.2cm]
&\ + \|\op(b) P_{\Om, \a}\chi_{\L} - P_{\Om_0, \a} \chi_{\L_0}\op(b)\|_{\GS_1}
\le C(\a\ell\rho)^{d-1},
\end{align*}
which leads to
\begin{equation}\label{ltol0_3:eq}
\|T(b) - T^{(0)}(1)\op^l_\a(b)\|_{\GS_1}
+ \|T^{(0)}(b) - T^{(0)}(1)\op^l_\a(b)\|_{\GS_1}
\le C (\a\ell\rho)^{d-1}.
\end{equation}
Thus \eqref{ltol0_1:eq} follows.

In order to prove \eqref{ltol0_4:eq}  rewrite:
\begin{align*}
 \op(b) T(a)
 =  &\ \bigl(\op(b) \chi_{\L}  P_{\Om}
 - \chi_{\L_0} P_{\Om_0, \a} \op(b)\bigr) \op(a) P_{\Om, \a} \chi_{\L} \\[0.2cm]
&\  + \chi_{\L_0}P_{\Om_0, \a} \bigl(\op(b)\op(a)
- \op(ab)\bigr) P_{\Om, \a} \chi_{\L} \\[0.2cm]
 &\ + \chi_{\L_0}P_{\Om_0, \a} \op(ab)
 \bigl(P_{\Om, \a} \chi_{\L}  - P_{\Om_0, \a} \chi_{\L_0} \bigr)
 + T^{(0)}(ab).
\end{align*}
Now \eqref{ltol0_4:eq} follows from \eqref{general_norm:eq},
\eqref{half1:eq}, \eqref{product_smooth:eq} and \eqref{onesided:eq}.

The inequality \eqref{ltol0_4:eq} immediately implies
\eqref{ltol0_2:eq}
for the first trace norm in \eqref{ltol0_2:eq}.
 Moreover, it follows from
\eqref{ltol0_3:eq} that $T^{(0)}(b)$ in the second term in \eqref{ltol0_2:eq}
can be replaced by $T^{(0)}(1)\op^l_\a(b)$.
Now the required bound for the second term
follows from the bound for the first trace norm in \eqref{ltol0_2:eq}.
\end{proof}

\begin{lem}\label{localisation1:lem}
Let $d\ge 1$, and let $a, b$ be some symbols
satisfying condition \eqref{ab:eq}.
Assume that $\a\ell\rho\ge c$. Then
\begin{align}\label{localisation0:eq}
\|\op_\a^l(b) \bigl(g_p(T(a)) - &\ g_p(T^{(0)}(a)) \bigr)\|_1\notag\\[0.2cm]
\le &\ C  (\a\ell\rho)^{d-1}
\SN^{(d+2, d+2)} (b; \ell, \rho)
\bigl(\SN^{(d+2, d+2)}(a; \ell, \rho)\bigr)^p.
\end{align}

\begin{align}\label{localisation:eq}
\|\op_\a^l(  b^p) g_p(T(a)) - &\ g_p(T^{(0)}(ab)) \|_1\notag\\[0.2cm]
\le &\ C  (\a\ell\rho)^{d-1}
\bigl(\SN^{(d+2, d+2)}(b; \ell, \rho) \bigr)^p
\bigl(\SN^{(d+2, d+2)}(a; \ell, \rho)\bigr)^p.
\end{align}
\end{lem}

\begin{proof}
We write for brevity $\op$ instead of $\op_\a^l$,
and assume without loss of generality that
\eqref{normalise:eq} is satisfied.

\underline{Step 1.} Let us show first that
\begin{equation}\label{bthrough:eq}
\|\op(b) g_p( T(a))
- g_{p-1} (T^{(0)}(a)) T^{(0)}(ab)\|_{\GS_1}
\le  C  (\a\ell\rho)^{d-1}.
\end{equation}
We do it by induction. If $p = 1$, then
\eqref{bthrough:eq} is exactly \eqref{ltol0_4:eq}.
 Suppose \eqref{bthrough:eq} holds for some $p =m$, and let us deduce \eqref{bthrough:eq}
for $p = m+1$. Write:
\begin{align*}
\op(b) g_{m+1}( T(a)) - g_m (T^{(0)}(a))&\  T^{(0)}(ab)\\
= &\ \biggl(\op(b) g_m(T(a)) - g_{m-1}(T^{(0)}(a)) T^{(0)}(ab)\biggr)T(a)\\
+ &\ g_{m-1}(T^{(0)} (a))\biggl(T^{(0)} (ab) T(a) - T^{(0)}(a) T^{(0)}(ab)\biggr).
\end{align*}
The claimed estimate for the first term on the right-hand side
follows from \eqref{bthrough:eq} for $p = m$, and Lemma \ref{general_norm:lem}.

By \eqref{ltol0_2:eq}, in the second term we can replace $T(a)$
with $T^{(0)}(a)$. It remains to use
Corollary \ref{product1:cor} and Lemma \ref{general_norm:lem}.
Thus by induction \eqref{bthrough:eq}
holds for all $p = 1, 2, \dots$.  The estimate \eqref{localisation0:eq}
is an immediate consequence
of \eqref{bthrough:eq}.

\underline{Step 2.}
Let us show that
\begin{equation}\label{bthrough01:eq}
\|(\op(b))^p g_p( T(a)) - g_{p}(T^{(0)}(ab))\|_{\GS_1}
\le C  (\a\ell\rho)^{d-1}.
\end{equation}
Again we use induction.
For $p=1$ this inequality repeats
\eqref{bthrough:eq}. Suppose it holds for some $p=m$.
In order to prove it for $p = m+1$ write
\begin{align*}
(\op(b))^{m+1} g_{m+1}(T(a)) -&  g_{m+1}(T^{(0)}(ab))\\
= &\ (\op(b))^{m}\biggl(\op(b) g_{m+1}(T(a)) - g_m(T^{(0)}(a))T^{(0)}(ab)\biggr) \\
+ &\ \biggl((\op(b))^m g_{m}(T^{(0)}(a)) - g_m(T^{(0)}(ab))\biggr)T^{(0)}(ab).
\end{align*}
Now we use \eqref{bthrough01:eq} for $p=m$, \eqref{bthrough:eq}
for $p = m+1$, and Lemma \ref{general_norm:lem}.

\underline{Step 3.}
It follows from Corollary \ref{product:cor} that
\[
\|\op(b^p) - (\op(b))^p\|_{\GS_1}\le C (\a\ell\rho)^{d-1},
\]
and hence
\[
\|\bigl(\op(b^p) - (\op(b))^p\bigr) g_p( T(a)) \|_{\GS_1}\le C (\a\ell\rho)^{d-1},
\]
in view of Lemma \ref{general_norm:lem}.
To complete the proof put together Steps 1, 2 and 3.
\end{proof}

For the one-dimensional case we need a more involved version of the above  lemma:

\begin{lem}\label{localisation_1d:lem}
Let $d = 1$, and let $a=a(\xi), b = b(\xi)$
be functions from $\BS^{(3)}$,
such that $b$ is supported on the interval
$(\mu - \rho, \mu + \rho)$ with some $\mu\in\R$ and $\rho>0$.
Assume that $\L = (x_0, \infty)$ with some $x_0\in\R$ and let $\Om$
be an arbitrary subset of $\R$.
Then for all $\a>0, \rho>0$, we have
 \begin{equation*}
\|\op_\a(b^p) g_p(T(a)) - g_p(T(a b))\|_{\GS_1}
\le C \bigl(\SN^{(3)}(a; \rho)\bigr)^p
\bigl(\SN^{(3)} (b; \rho)\bigr)^p.
\end{equation*}
\end{lem}

\begin{proof}
Lemma \ref{localisation1:lem} is not directly applicable, although the proof
is.  Assume without loss of generality that
\[
\SN^{(3)}(b; \rho) = \SN^{(3)}(a; \rho) = 1.
\]

\underline{Step 1.} Let us show first that
\begin{equation}\label{bthrough11:eq}
\|\op(b) g_p(T(a)) - g_{p-1}(T(a)) T(ab)\|_{\GS_1}\le C.
\end{equation}
We do it by induction. If $p = 1$, then
\begin{equation*}
\op(b) T(a) - T(ab) = [\op(b), \chi_\L] \op(a) P_{\Om, \a} \chi_\L,
\end{equation*}
and \eqref{bthrough11:eq} follows
from \eqref{sandwich_121:eq} and the bound $\|\op(a)\|\le \| a\|_{\plainL\infty}$.
Suppose \eqref{bthrough11:eq} holds for some $p =m$, and
let us deduce \eqref{bthrough11:eq}
for $p = m+1$. Write:
\begin{align*}
\op(b) g_{m+1}(T(a)) - &\ g_m (T(a)) T(ab)\\
= &\ \biggl(\op(b) g_m(T(a)) - g_{m-1}(T(a)) T(ab)\biggr)T(a)\\
+ &\ g_{m-1}(T(a))\biggl(T(ab) T(a) - T(a) T(ab)\biggr).
\end{align*}
The claimed estimate follows from \eqref{bthrough11:eq} for $p = m$,
 and Lemma \ref{product_1d:lem}, and more precisely,
 from the bound
 \begin{equation*}
\| T(ab) T(a) - T(a) T(ab)\|_{\GS_1}
\le  C.
\end{equation*}

\underline{Step 2.}
Let us show that
\begin{equation}\label{bthrough12:eq}
\|(\op(b)^p g_p(T(a)) - g_{p} (T(ab))\|_{\GS_1}\le  C.
\end{equation}
Again we use induction.
For $p=1$ this inequality repeats \eqref{bthrough11:eq}. Suppose it holds for some $p=m$.
In order to prove it for $p = m+1$ write
\begin{align*}
(\op(b))^{m+1} g_{m+1}(T(a)) -&  g_{m+1}(T(ab))\\
= &\ (\op(b))^{m}\biggl(\op(b) g_{m+1}(T(a)) - g_m(T(a))T(ab)\biggr) \\
+ &\ \biggl((\op(b))^m g_{m}(T(a)) - g_m(T(ab))\biggr)T(ab).
\end{align*}
Now one uses \eqref{bthrough11:eq} for $p = m+1$,
and \eqref{bthrough12:eq} for $p=m$ to show that
the right hand side does not exceed
\begin{equation*}
C\|b\|_{\plainL\infty}^m
+ C \|a\|_{\plainL\infty}\|b\|_{\plainL\infty}\le C'.
\end{equation*}
The proof is complete.
\end{proof}

\begin{lem}\label{commute_through1:lem}
 Let $d\ge 1$, and let $a, b$ be some symbols satisfying
condition \eqref{ab:eq}.
Suppose that $\a\ell\rho\ge c$. Then
\begin{align*}
\|\op_\a^l(b^p) g_p(T(a))
- &\ \op_\a^l( (b a)^p) g_p(T^{(0)}(1))\|_{\GS_1}\\[0.2cm]
\le &\ C  (\a\ell\rho)^{d-1}
\bigl(\SN^{(d+2, d+2)}(b; \ell, \rho) \bigr)^p
\bigl(\SN^{(d+2, d+2)}(a; \ell, \rho)\bigr)^p.
\end{align*}
\end{lem}

\begin{proof} Estimate:
\begin{align*}
\|\op_\a^l(b^p) g_p(T(a))
- \op_\a^l( (b a)^p) g_p(T^{(0)}(1))\|_{\GS_1}
&\ \le
\| g_p(T^{(0)}(ab))
- \op_\a^l(b^p) g_p(T(a))\|_{\GS_1}\\[0.2cm]
+ &\ \| g_p(T^{(0)}(ab))
- \op_\a^l( (b a)^p) g_p(T^{(0)}(1))\|_{\GS_1}.
\end{align*}
It remains to use Lemma \ref{localisation1:lem} twice.
\end{proof}

\begin{lem}\label{commute_through2:lem}
Let $d\ge 1$, and let $a, b$ be some symbols satisfying
condition \eqref{ab:eq}.
 Then under the assumption $\a\ell\rho\ge c$, one has
\begin{align}\label{through1:eq}
\|\op_\a^l(b) g_p(T(a))
- &\ \op_\a^l(b a^p) g_p(T^{(0)}(1))\|_{\GS_1}\notag\\[0.2cm]
 \le &\ C (\a\ell\rho)^{d-1}
 \SN^{(d+2, d+2)}(b; \ell, \rho)
\biggl( \SN^{(d+2, d+2)}(a; \ell, \rho)\biggr)^p.
\end{align}

\end{lem}

\begin{proof}
Assume without loss of generality that \eqref{normalise:eq} is satisfied.
Also, due to \eqref{localisation0:eq} we may assume that
$T(a) = T^{(0)}(a)$.
Representing $b = b q^p$,
$q(\bx, \bxi) = h_{\bz, \ell}(\bx) \eta_{\bmu, \rho}(\bxi)$ we can write:
\begin{align*}
\op(b) g_p(T^{(0)}(a)) - & \ \op(b a^p) g_p(T^{(0)}(1))
= Z_1 + Z_2 + Z_3,\\[0.2cm]
Z_1 = &\
\bigl[\op(b) - \op(b) \op(q^p)\bigr] g_p(T^{(0)}(a)),\\[0.2cm]
Z_2 = &\ \op(b)\bigl[
\op(q^p) g_p(T^{(0)}(a)) - \op((qa)^p) g_p(T^{(0)}(1))\bigr],\\[0.2cm]
Z_3 =  &\ \bigl[
\op(b)\op((qa)^p) - \op(b a^p) \bigr] g_p(T^{(0)}(1)).
\end{align*}
By Corollary \ref{product:cor} and Lemma \ref{general_norm:lem},
\begin{align*}
\|Z_1\|_{\GS_1}\le &\
\|\op_\a^l(b) - \op_\a^l(b) \op_\a^l\bigl(q^p\bigr)\|_{\GS_1}
\| \op(a)\|^p
\le C(\a\ell\rho)^{d-1},\\
\|Z_3\|_{\GS_1}\le &\ C(\a\ell\rho)^{d-1}.
\end{align*}
Furthermore, by Lemma \ref{general_norm:lem},
and Lemma \ref{commute_through1:lem} used with $R = \infty$,
\begin{align*}
\|Z_2\|_{\GS_1}\le &\ \|\op(b)\|\ \|\op_\a^l(q^p) g_p(T^{(0)}(a))
- \op^l_\a((q a)^p)g_p(T^{(0)}(1))\|_{\GS_1}\\
\le &\ C (\a\ell\rho)^{d-1}.
\end{align*}
Put together, these bounds produce \eqref{through1:eq}.
\end{proof}

If $\L = \R^d$ or $\Om = \R^d$,
Lemma \ref{commute_through2:lem} leads to the following asymptotics:

\begin{thm}\label{localasymptotics:thm}
Let $d\ge 1$, and let $a, b$ be some symbols satisfying
condition \eqref{ab:eq}.
Suppose that $\a\ell\rho\ge c$.
In addition assume that either $\L$ or $\Om$ satisfies
Condition \ref{global:cond}(\ref{space:subcond}), i.e.
either $\L_0 = \R^d$ or $\Om_0 = \R^d$.
Then
\begin{align}\label{asymptotics1:eq}
|\tr\bigl(\op_\a^l(b) g_p(T(a))\bigr)
- &\  \a^d \GW_0(b g_p(a); \L, \Om)|\notag\\[0.2cm]
\le &\ C (\a\ell\rho)^{d-1}
 \SN^{(d+2, d+2)}(b; \ell, \rho)
\biggl( \SN^{(d+2, d+2)}(a; \ell, \rho)\biggr)^p.
\end{align}
For $a = 1$ and the function $g(t) = t-t^p$ one has
\begin{equation}\label{asymptotics2:eq}
\|\op_\a^l(b) g(T(1))\|_{\GS_1} \le C (\a\ell\rho)^{d-1}
 \SN^{(d+2, d+2)}(b; \ell, \rho).
\end{equation}

\end{thm}

\begin{proof} Without loss of generality assume \eqref{normalise:eq}.
If $\L_0 = \R^d$, then
 $T^{(0)}(1) = P_{\Om_0, \a}$, and it follows from Lemma \ref{commute_through2:lem}
that
\begin{equation}\label{asymptotics3:eq}
\|\op(b) g_p(T(a)) - \op(b a^p)P_{\Om_0, \a}\|_{\GS_1}
\le C (\a\ell\rho)^{d-1}.
\end{equation}
Take the trace:
\begin{equation*}
\tr \bigl(\op(b a^p)P_{\Om_0, \a} \bigr)= \biggl(\frac{\a}{2\pi}\biggr)^d
\int_{\R^d} \int_{\Om_0} b(\bx, \bxi) g_p(a(\bx, \bxi) )
d\bxi d\bx = \a^d \GW_0(b g_p(a); \L, \Om).
\end{equation*}
This gives \eqref{asymptotics1:eq}. Similarly for $\Om_0 = \R^d$.

To obtain \eqref{asymptotics2:eq} we use \eqref{asymptotics3:eq} with $a=1$.
\end{proof}

%%%%%%%%%%%%%%%%%%%%%%%%%%%%%%%%%%%%%%%%%%%%%%%%%%%%%%%%%%%%%%%%%%%%%%%%%%%%%%%%%%%%

\section{Model problem in dimension one}\label{1_dim:sect}

One of the pivotal points of the proof is the reduction to a model operator
for $d=1$.  This section is entirely devoted to the study of this problem.

\subsection{Model problem: reduction to multiplication}
The pair of the model operators on $\plainL2(\R_+)$ is defined as follows:
\begin{equation*}
T_{\pm} := T_\a(1; \R_+, \R_{\pm})=T_1(1; \R_+, \R_{\pm}).
\end{equation*}
To simplify simultaneous considerations of $T_+$ and $T_-$, we change
$\xi \to -\xi$ in $T_-$ we arrive at
\begin{equation*}
(T_\pm u)(x) = \frac{1}{2\pi} \int_0^\infty\int_{0}^\infty
e^{\pm i\xi(x-y)}u(y) dy d\xi,\ u\in\plainS(\R).
\end{equation*}
As in \cite{Widom1}, using the Mellin
transform $\CM:\plainL2(\R_+)\to\plainL2(\R)$:
\begin{equation*}
\tilde u(s) = \frac{1}{\sqrt{2\pi}} \int_0^\infty x^{-\frac{1}{2} + is} u(x) dx,
\end{equation*}
one can easily show that the operator $T_{\pm}$ is
unitarily equivalent to the multiplication by the function
\[
\frac{1}{1+e^{\pm 2\pi z}}
\]
in $\plainL2(\R)$. Thus $\CM g(T_{\pm})\CM^*$ is also multiplication by a function.
If $g$ is $\plainC1$, and $g(0) = g(1) = 0$,
this function is integrable, and
hence one can write the kernel of $g(T_{\pm})$:
\[
K_{\pm}(x, y) = \frac{1}{2\pi} (xy)^{-\frac{1}{2}}
\int_{\R} \biggl(\frac{y}{x}\biggr)^{\pm is} g\bigl((1+e^{2\pi s})^{-1}\bigr) ds.
\]
Note that the function $K_{\pm}(x, y)$ is homogeneous of degree $-1$:
\begin{equation}\label{homogeneous:eq}
K_{\pm}(t x, t y) = t^{-1} K(x, y),
\end{equation}
for any $t >0$.
By a straightforward change of variables, one sees that
\begin{equation}\label{k11:eq}
K_{\pm}(1, 1) = \frac{1}{2\pi} \int g\bigl((1+e^{2\pi s})^{-1}\bigr) ds = \GA(g),
\end{equation}
where $\GA(g)$ is defined in \eqref{GA:eq}.

\subsection{Model problem: asymptotics}

We are computing the asymptotics of the trace
\begin{equation*}
I_{\pm}(\a) = \tr\bigl(\op_\a^l(b) g(T_{\pm})\bigr),
\end{equation*}
with a suitable symbol $b$, for arbitrary function $g$, satisfying the condition
\begin{equation}\label{gcondition:eq}
g\in\plainC1(\R),\ \ g(0) = g(1) = 0.
\end{equation}
Rewrite this trace as follows:
\begin{equation*}
I_{\pm}(\a; b) = \frac{\a}{2\pi}\int_{\R_+}\int_{\R_+}
\int_{\R}b(z, \xi) e^{i\a\xi(z-x)} K_{\pm}(x,z) d\xi dz dx.
\end{equation*}
The symbol $b=b(x, \xi)$ is assumed to be of the form
\begin{equation}\label{widoma:eq}
b(x, \xi) = \psi(x) a(x, \xi),\ \
\textup{the symbol}\ \
a\in \BS^{(2, 2)} \ \ \textup{is supported on}  \ \
\R\times (-\rho, \rho),
\end{equation}
where $\psi\in\plainC\infty_0(\R)$ is a non-negative function
such that $\psi(x)\le 1$ for all $x\in\R$ and
\begin{equation}\label{psi_defn:eq}
\psi (x)
=
\begin{cases}
0,\ x\notin (1, 4r),\\
1, \ x\in (4, r),
\end{cases}
\end{equation}
with a parameter $r > 4$.

We begin with some estimates:

\begin{lem}\label{1dim_bound:lem}
 Let $\L$ and $\Om$ be arbitrary open subsets of $\R$.
Let the symbol $a$ and function $\psi$ be as above.
Then for any continuous function $g$ any $L \ge r$ and $\a r \rho\ge c$, we have
\begin{equation}\label{1dim_bound:eq}
\| \op_{\a} (\psi a) g(T_\pm(1))\|_{\GS_1}
\le C \a\rho r \SN^{(2, 2)}(a; L, \rho).
\end{equation}
\end{lem}

\begin{proof} Let $\tilde h:= h_{0, 2r}(x)$ be the function defined
in \eqref{h_eta:eq}.
The estimated trace norm
does not exceed $C \| \op_{\a}(\tilde h a)\|_{\GS_1}$.
By Lemma \ref{general_trace:lem},
\[
\|\op_{\a}(\tilde h a)\|_{\GS_1}\le C \a\rho r \SN^{(2, 2)}(a; r, \rho).
\]
Since the right-hand side is monotone increasing in $r$ and $L\ge r$,
the claimed estimate follows.
\end{proof}

Now we compute the asymptotics of $I_{\pm}(\a; b)$ as $\a\to\infty$.
%The next lemma is a modified version of a similar result
%from \cite{Widom1}.

\begin{thm}\label{widom1:thm}
 Let the symbol $b$ be as in \eqref{widoma:eq}, and let
the function $g$ be as in \eqref{gcondition:eq}.
Then for any $\d >0$, $L > 0$ and $\rho >0$, such that $\a\rho\ge c$,  we have
\begin{equation}\label{widom1:eq}
|I_{\pm}(\a; b) - \GA(g)  a(0, 0) \log r |
\le
C[1+ r L^{-1} + (\a\rho)^{-\d}\log r] \SN^{(1,2)}(a; L, \rho),
\end{equation}
uniformly in $r\ge 5$, with a constant $C = C(\d)$.
\end{thm}

This lemma is a modified version of a similar result
from \cite{Widom1} which gives the asymptotics of the trace $I_{\pm}(\a; b)$
as $\a\to\infty$, when $\psi(x) = 1$, and the parameters $L, \rho$ are fixed.
On the contrary,  the formula \eqref{widom1:eq}
explicitly depends on four parameters $L, r, \rho, \a$. Later we use \eqref{widom1:eq}
with large $r$ and such $\rho$ and $L$ that the right-hand side is uniformly bounded.

\begin{proof} Without loss of generality we may assume that $\rho = 1$
and $\SN^{(1, 2)}(a; L, 1) = 1$,
so that
\begin{equation}\label{derivatives:eq}
|\p_{\xi}^s a(x, \xi)| \le 1,\ s = 0, 1, 2,\
\ |\p_x a(x, \xi)|\le L^{-1},
\end{equation}
for all $x, \xi$.

\underline{Step I.} We conduct the proof for the sign ``$+$" only.
The case of ``$-$" is done in the same way. For brevity the subscript ``$+$"
is omitted from the notation.

Let $\phi\in\plainC\infty_0(\R)$
%
%
% be a function which vanishes outside
%a neighbourhood of $t = 1$
%and such that $|\phi|\le 1$, $\phi(1) = 1$.
and consider the integral
\begin{equation*}
J'(\a) = \frac{\a}{2\pi} \int_0^\infty\int_{0}^\infty
 \int_{-1}^1
 K(x, z) (1-\phi(xz^{-1}))e^{i\a\xi(z-x)} b(z, \xi) d\xi dz dx.
\end{equation*}
Integrating by parts in $\xi$ twice, by \eqref{derivatives:eq} we get
 \[
 \biggl|\int_{-1}^1
 e^{i\a\xi(z-x)} b(z, \xi) d\xi\biggr|
 \le \frac{C}{(1+\a|z-x|)^2}.
\]
In view of \eqref{homogeneous:eq} we can estimate:
\begin{align*}
|J'(\a)|\le &\ C\a \int_0^\infty\int_{0}^\infty
 \frac{|K(x, z)| |1-\phi(xz^{-1})|}{(1+\a|x-z|)^2} dx dz\\
 = &\ C\a \int_0^\infty\int_{0}^\infty
 \frac{|K(x, 1)| |1-\phi(x)|}{(1+\a z|x-1|)^2} dx dz
 = C\int_0^\infty
 \frac{|K(x, 1)| |1-\phi(x)|}{|x-1|} dx.
\end{align*}
This integral converges at infinity since $K(x, 1) = O(x^{-1/2})$.
If $K(1, 1) = 0$, then, then it converges also at $x = 1$ for any
function $\phi\in\plainC\infty_0$, and in particular for $\phi = 0$.
In this case the integral $J'(\a) = I(\a)$ is uniformly bounded, so
\eqref{widom1:eq} is proved.

If $K(1, 1) \not = 0$, then the above integral converges under the assumption
$\phi(1) = 1$. Thus it remains to find the asymptotics of
 the integral
\begin{equation*}
J(\a) = \frac{\a}{2\pi} \int_0^\infty
\int_0^\infty\int_{\R}
K(x, z) \phi(xz^{-1}) e^{i\a\xi(z-x) }b(z, \xi) d\xi dz dx.
\end{equation*}
%If $K(1, 1) = 0$, then one can take $\phi = 0$, so that
%$I(\a) = J'(\a)$, and hence the integral $I(\a)$ is
%uniformly bounded.
%Assume that $K(1, 1)\not = 0$.
To choose a suitable function $\phi$, let $\nu>0$  be a number such that
$K(x, 1)$ is separated from zero for all $x\in [1-\nu, 1+\nu]$. Take a function
$\om\in\plainC\infty_0(\R)$ supported on $[1-\nu, 1+\nu]$, such that
\[
\om(x) = 1,\ 1-\nu/2\le x\le 1+ \nu/2,
\]
and define
\begin{equation*}
\phi(x) = \om(x) \frac{K(1, 1)}{K(x, 1)},
\end{equation*}
so that
\begin{equation*}
K(x, z) \phi(xz^{-1}) = z^{-1} K(1, 1) \om(xz^{-1}).
\end{equation*}
Now $J(\a)$ can be rewritten as follows:
\begin{align*}
J(\a) = &\ \a \frac{K(1, 1)}{\sqrt{2\pi}}
\int_0^\infty
\int_{\R}
\biggl[\frac{1}{\sqrt{2\pi}}\int_{\R}  \om(xz^{-1})
e^{i\a\xi z(1-xz^{-1})}\frac{dx}{z}\biggr]
b(z, \xi) d\xi dz\\[0.2cm]
= &\
\a\frac{K(1, 1)}{\sqrt{2\pi}}
\int_0^\infty
\int_{\R}
e^{i\a\xi z} \hat \om(\a\xi z)
b(z, \xi) d\xi dz
\end{align*}
Rewriting it again, using \eqref{k11:eq}, we get
\begin{equation*}
J(\a) = \a\frac{\GA(g)}{\sqrt{2\pi}}
\int_{\R} e^{i\a\xi} \hat \om(\a\xi) \tau(\xi) d\xi,
\end{equation*}
with
\begin{equation}\label{tau:eq}
\tau(\xi) = \int_0^\infty b\bigl(z, \xi z^{-1}\bigr) \frac{dz}{z}
= \int_{1}^{2r} b\bigl(z, \xi z^{-1}\bigr) \frac{dz}{z}.
\end{equation}
Split $J(\a)$ in two integrals:
\begin{align*}
J(\a) = &\ J_1(\a) + J_2(\a),\\[0.2cm]
J_1(\a) = &\ \a\frac{\GA(g)}{\sqrt{2\pi}}
\int_{|\xi|<1} e^{i\a\xi} \hat \om(\a\xi) \tau(\xi) d\xi,\\[0.2cm]
J_2(\a) = &\ \a\frac{\GA(g)}{\sqrt{2\pi}}
\int_{|\xi|\ge 1} e^{i\a\xi} \hat \om(\a\xi) \tau(\xi) d\xi,
\end{align*}
Since $|\hat\om(\eta)|\le C_\d|\eta|^{-\d-1}$ for any $\d>0$,
we have
\begin{equation}\label{bigxi:eq}
\a\int_{|\xi|\ge 1} |\hat \om(\a\xi)| d\xi\le C\a^{-\d}.
\end{equation}
Thus, estimating
\[
|\tau(\xi)|\le \int_{1}^{2r} z^{-1} dz = \log (2r),
\]
we obtain
\begin{equation}\label{J2:eq}
J_2(\a) \le C\a^{-\d}\log r.
\end{equation}
The contribution of $J_1(\a)$ requires more careful calculations:

\underline{Step II. Study of $J_1(\a)$.}
We compare \eqref{tau:eq} with the number
\begin{equation}\label{tau0:eq}
\tau_0 = a(0, 0)\int_{1}^{4r} \frac{dz}{z} = a(0, 0) (\log r + \log 4).
\end{equation}
Estimate the difference:
\begin{align*}
|\psi(z) a(z, \xi z^{-1}) -  a(0, 0)|
\le &\ |\psi(z) \bigl(a(z, \xi z^{-1}) -  a(0, \xi z^{-1})\bigr)| \\
+ &\ |\psi(z)\bigl(a(0, \xi z^{-1}) -  a(0, 0)\bigr)|+ |\psi(z) - 1| |a(0, 0)|.
\end{align*}
Thus
\begin{align*}
|\tau(\xi) - \tau_0|
\le &\ \max_{z, t} |\p_z a(z, t)|\int_{1}^{4r} dz
+ \max_{z, t} |\p_t a(z, t)| |\xi|
\int_{1}^{4r} z^{-2}dz\\
&\ + |a(0, 0)|\biggl(\int_{1}^4 \frac{dz}{z}
+ \int_r^{4r} \frac{dz}{z}\biggr)\\
\le &\
C\bigl(\max_{z, t} |\p_z a(z, t)| r
+ \max_{z, t} |\p_t a(z, t)|
+ |a(0, 0)|\bigr)\\
\le &\ C(1+ r L^{-1}),
\end{align*}
for all $|\xi|<1$, where we have used \eqref{derivatives:eq}.
In view of \eqref{tau0:eq} this entails the bound
\begin{align*}
\biggl|J_1(\a) - \a \GA(g) a(0, 0) \log r\frac{1}{\sqrt{2\pi}}
\int_{|\xi|<1} e^{i\a\xi} \hat\om(\a\xi) d\xi \biggr|
\le & \ C\a (1+rL^{-1})\int_{\R} |\hat\om(\a\xi)|d\xi\\
 = &\ C'(1+rL^{-1}).
\end{align*}
By definition of the function $\om$ and by \eqref{bigxi:eq},
\[
\frac{\a}{\sqrt{2\pi}}
\int_{|\xi|<1} e^{i\a\xi} \hat\om(\a\xi) d\xi
= 1 + O(\a^{-\d}),
\]
and hence, from the boundedness of $J'(\a)$ and from \eqref{J2:eq}
it follows that
\begin{align*}
\bigl|I(\a; b) - \GA(g) a(0, 0) \log r\bigr|
\le &\ |J'(\a)| + |J_2(\a)| + |J_1(\a) - \GA(g) a(0, 0) \log r |\\[0.2cm]
\le &\ C'(1+rL^{-1} + \a^{-\d}\log r).
\end{align*}
This leads to \eqref{widom1:eq}.
\end{proof}

\subsection{Multiple intervals}\label{multiple:subsect}
Here instead of the model operator $T_\pm$ we consider
the operator with a more general set $\Om$. We begin with introducing some
quantities related to finite subsets of the real line.

Let $\rho >0$, and let $\SX\subset(-2\rho, 2\rho)$ be a finite set. Now
we introduce a number describing the spacing of the elements of $\SX$.
If $\SX = \varnothing$, then for any $\d\ge 0$ we define
\begin{equation}\label{mempty:eq}
m_{\d}(\SX) = (4\rho)^{-\d}, \SX = \varnothing.
\end{equation}
If $\#\SX = N\ge 1$, then we label the points $\xi\in\SX$ in increasing order:
$\xi_1, \xi_2, \dots, \xi_N$, and define
\begin{equation}\label{rhok:eq}
m_{\d}(\SX) = \sum_{j=1}^N \rho_j^{-\d},\
\begin{cases}
\rho_j = \dist\{\xi_j, \SX\setminus\{\xi_j\}\},\ N\ge 2, \\[0.2cm]
\rho_1 = 4\rho,\ N = 1.
\end{cases}
\end{equation}
Note that
\begin{equation}\label{mosch:eq}
\#\SX\le 2^\d\rho^\d m_\d(\SX), \ \textup{for any}\ \ \d\ge 0,
\end{equation}
and, more generally,
\begin{equation}\label{moments:eq}
m_\nu(\SX)\le (2\rho)^{\d-\nu} m_{\d}(\SX),\ \textup{for any}\ \ 0\le \nu \le \d.
\end{equation}
Note also the natural ``monotonicity" of $m_\d$ in the set $\SX$: if
$\SX_1\subset\SX_2\subset (-2\rho, 2\rho)$ are two finite sets, then
\begin{equation}\label{sxmonotone:eq}
m_\d(\SX_1)\le m_\d(\SX_2),
\end{equation}
for any $\d\ge 0$.

Now let $\Om$ be a subset of real line such that for some $\rho>0$
\begin{equation}\label{multipleom:eq}
\Om\bigcap(-2\rho, 2\rho) = \biggl(\bigcup_j I_j\biggr)\bigcap (-2\rho, 2\rho),
\end{equation}
where $\{I_j\}, j=1, 2, \dots$ is a finite collection
 open intervals such that
$\bar I_j\cap \bar I_s = \varnothing$ for $j\not = s$.
Let $\SX$ be the set of their endpoints inside $(-2\rho, 2\rho)$.
We are interested in the asymptotics of the trace
\begin{equation}\label{newm:eq}
\tr\bigl(\op^l_\a(b) g(T(1; \R_+, \Om))\bigr),\ g(t) = t^p - t,
  \ \textup{with some}\ \ p\in\mathbb N,
\end{equation}
where $b=b(x, \xi)$ is as in \eqref{widoma:eq}.

We begin with the simplest case, $\SX = \varnothing$.
Let  $\z\in\plainC\infty_0(\R)$ be a non-negative function, such that
\begin{equation}\label{zetan:eq}
\z(x) =
\begin{cases}
0, |x|\ge 1/4,\\
1, |x|< 1/8.
\end{cases}
\end{equation}

\begin{lem}\label{varnothing:lem}
Let $g$ be as in \eqref{newm:eq}, and let the symbol $a$ be as in \eqref{widoma:eq}.
If $\SX = \varnothing$, then under the condition $\a L\rho\ge c$ we have
\begin{equation}\label{varnothing:eq}
\| \op_\a^l(a) g(T(1; \R_+, \Om))\|_{\GS_1}\le C \SN^{(1,2)}(a; L, \rho).
\end{equation}

\end{lem}

\begin{proof}
Also, without loss of generality assume that
$\SN^{(1,2)}(a; L, \rho) = 1$, so that
\begin{equation}\label{anorma:eq}
\|\op_\a^l(a)\|\le C,
\end{equation}
by Lemma \ref{general_norm:lem}.

Since $a(x, \xi) = 0$ for $|\xi|\ge \rho$,
we have
\[
a(x, \xi) = a( x, \xi) (\tilde\z(\xi))^p, \tilde\z(\xi)
= \z\bigl(\xi(8\rho)^{-1}\bigr),
\]
for any $p\in\mathbb N$.
By Lemma \ref{localisation_1d:lem},
\[
\|\op_\a(\tilde\z^p) g_p(T(1; \R_+, \Om))
- g_p(  T(\tilde\z; \R_+, \Om))\|_{\GS_1}\le
  C'\bigl(\SN^{(3)}(\tilde\z; \rho)\bigr)^p\le C,
\]
for all $\a>0$. Since $\tilde\z(\xi) = 0$ for $|\xi|\ge 2\rho$,
we have either $\tilde\z\chi_{\Om} = 0$ or $\tilde\z\chi_{\Om} = \tilde\z$.
If $\tilde\z\chi_{\Om} = 0$, then $T(\tilde\z; R_+, \Om) = 0$, and the
above estimate together with \eqref{anorma:eq} imply that
\begin{align*}
\|\op_\a^l(a) &\ g(T(1; \R_+, \Om))\|_{\GS_1}\\[0.2cm]
\le &\ \| \op_\a^l(a)\| \bigl(\|\op_\a(\tilde\z^p) g_p(T(1; \R_+, \Om))\|_{\GS_1}
+ \|\op_\a(\tilde\z) g_1(T(1; \R_+, \Om))\|_{\GS_1}\bigr)\le C,
\end{align*}
as claimed.

If $\tilde\z\chi_{\Om} = \tilde\z$, then
$T(\tilde\z; \R_+, \Om) = T(\tilde\z; \R_+, \R)$, and  using
Lemma \ref{localisation_1d:lem} backwards, we get
\[
\|\op_\a(\tilde\z^p) \bigl(g_p(T(1; \R_+, \Om))
- g_p(\chi_{\R_+})\bigr)\|_{\GS_1}\le C,
\]
for any $p = 1, 2, \dots$, which leads to
\[
\|\op_\a^l(a) \bigl(g_p(T(1; \R_+, \Om))
- g_p(\chi_{\R_+})\bigr)\|_{\GS_1}\le C,
\]
in view of \eqref{anorma:eq}. Therefore
\[
\|\op_\a^l(a) \bigl(g(T(1; \R_+, \Om)) - g(\chi_{\R_+})\bigr)\|_{\GS_1}\le C.
\]
Since $g(\chi_{\R_+}) = 0$, this entails \eqref{varnothing:eq}.
The proof is complete.
\end{proof}

If $\SX\not=\varnothing$, then we find the asymptotics of the trace \eqref{newm:eq}
using Theorem \ref{widom1:thm}. To this end we need to build a suitable partition of unity,
which will reduce the problem to the case $\Om = \R_+$ or $\Om = \R_-$.
We attach to each point $\xi_j\in\SX$, $j=1, 2, \dots, N:=\#\SX$, the cut-off function
\begin{equation*}
u_j(\xi) = \z\biggl(\frac{\xi -\xi_j}{\rho_j}\biggr),
\end{equation*}
where $\rho_j, \z$ are defined in \eqref{rhok:eq} and \eqref{zetan:eq} respectively.
Clearly, $u_j\in\BS^{(\infty)}$ and
\[
\SN^{(m)}(u_j; \rho_j)< C_m,
\]
with a constant $C_m$ depending only on $m$.
By construction the supports of the functions $u_j$ associated with distinct points
of the set $\SX$, do not overlap.
To complete our definition of the partition of unity we set
\[
u(\xi) = 1 - \underset{1\le j\le N}\sum u_j(\xi).
\]
Insert now this partition of unity into the trace
\eqref{newm:eq} and establish the asymptotics of
the constituent traces individually.

\begin{lem}\label{node:lem}
Let the symbols
$a$ and $b$ be as in \eqref{widoma:eq}, and
let $g$ be as in \eqref{newm:eq}. Suppose that
$N:=\#\SX\ge 1$.
Let $u_j\in\plainC\infty_0(\R)$, $j = 1, 2, \dots, N$,  be as above.
Assume that
\[
\a L \rho\ge c,\ \ \a\rho_j\ge c.
\]
Then for any
$\d >0$  we have
\begin{align*}
\bigl|\tr\bigl( \op_\a^l(b u_j)g(T(1; \R_+, \Om))\bigr)
- &\ \GA(g) a(0, \xi_j)\log r\bigr|\\
\le &\ C \bigl(1+ rL^{-1} + (\a\rho_j)^{-\d}\log r\bigr)
\SN^{(1, 2)}(a; L, \rho),
\end{align*}
with a constant $C$ independent of $j$.
\end{lem}

\begin{proof}
Translating by $\xi_j$, we may assume that $\xi_j = 0$.
Without loss of generality assume $\SN^{(1, 2)}(a; L, \rho) = 1$.
Denote $\tilde u_j(\xi) = u_j(\xi/2)$, so that $\tilde u_j u_j = u_j$.
By Lemma \ref{localisation_1d:lem}   we have
\[
\|\op_\a(\tilde u_j^p) g_p(T(1; \R_+, \Om))
- g_p\bigl(T(\tilde u_j; \R_+, \Om)\bigr)\|_{\GS_1}\le C,
\]
for all $\a>0$. If $\xi_j$ is the left endpoint of
one of the intervals $\{ I_j\}$ in \eqref{multipleom:eq}, then
$\tilde u_j\chi_{\Om} = \tilde u_j \chi_{\R_+}$. If $\xi_j$ is the right endpoint,
then $\tilde u_j\chi_{\Om} = \tilde u_j \chi_{\R_-}$. Thus
by Lemma \ref{localisation_1d:lem} again,
\begin{equation}\label{node:eq}
\|\op_\a(u_j) g(T(1; \R_+, \Om))  - \op_\a(u_j) g(T_{\pm})\|_{\GS_1}\le C,
\end{equation}
where the sign ``$+$" (resp. ``$-$") is used for the left (resp. right) endpoint.
As $\a L \rho\ge c$, Lemma \ref{general_norm:lem}
gives
\[
\|\op_\a^l(\psi a)\|\le C.
\]
Thus in combination with \eqref{node:eq} we obtain
\begin{equation*}
\| \op_\a^l(\psi a u_j)\bigl(g(T(1; \R_+, \Om)) - g(T_{\pm})\bigr)\|_{\GS_1}
\le C ,
\end{equation*}
for $\a L\rho\ge c$.
To find the trace of $\op_\a^l(\psi a u_j) g(T_{\pm})$
note that $a u_j\in\BS^{(2, 2)}$,
$u_j(\xi) = 0$ for $|\xi|\ge \rho_j/4$, and
\begin{equation*}
\SN^{(1, 2)}(au_j; L, \rho_j)\le C \SN^{(1, 2)}(a; L, \rho),
\end{equation*}
since $\rho\ge \rho_j$.
Since $\a\rho_j\ge c$, now the result follows from Theorem \ref{widom1:thm}.
\end{proof}

\begin{lem}\label{outside:lem}
Let the symbol $a$ be as in \eqref{widoma:eq},
and let $g$ be as in \eqref{newm:eq}.
Suppose that $N:=\#\SX\ge 1$.
Then under the condition  $\a L\rho\ge c$ we have for any $\d \ge 1$:
\begin{equation*}
\| \op_\a^l(a u) g(T(1; \R_+, \Om))\|_{\GS_1}\le C_\d \rho^\d\
m_{\d}(\SX) \ \SN^{(1, 2)}(a; L, \rho).
\end{equation*}
\end{lem}

\begin{proof}
Again, without loss of generality assume that $\SN^{(1,2)}(a; L, \rho) = 1$.

By definition of $u_j$ and $u$, we have $u(\xi) = 0$ whenever
$|\xi-\xi_j|\le \rho_j/8$ for any $\xi_j\in\SX$. Let $l$ be a number such that
$\rho_l = \min_j \rho_j$. Without loss of generality assume that $l = 1$, so that
$u\in\BS^{(\infty)}$ with
\[
\SN^{(m)}(u; \rho_1)\le C_m.
\]
Let $\varphi\in\plainC\infty_0(\R)$ be a function such that
\begin{equation*}
\varphi(t) =
\begin{cases}
1, |t|\le 1/4;\\
0, |t|\ge 3/4,
\end{cases}
\end{equation*}
and that the collection
$\varphi(t-n)$, $n\in\Z$, forms a partition of unity on $\R$. Denote
\begin{equation*}
\phi_n(\xi) = \z\bigl(\xi(8\rho)^{-1}\bigr) u(\xi)
\varphi(12\xi \rho_1^{-1} - n), \ n\in\Z,
\end{equation*}
where $\z$ is defined in \eqref{zetan:eq},
so that $\phi_n(\xi) = 0$ for $|\xi|\ge 2\rho$.
Denote by $J\le C\rho\rho_1^{-1}$ the smallest natural
number such that
\[
\sum_{|n|\le J} \phi_n(\xi)
= \z\bigl(\xi(8\rho)^{-1}\bigr)u(\xi).
\]
Let us show that
\begin{equation}\label{phij:eq}
\|\op_\a(\phi_n) g(T(1; \R_+, \Om))\|_{\GS_1}\le C,
\end{equation}
for all $\a>0$ uniformly in $n, |n|\le J$. Indeed, let
\[
\tilde\varphi_n (\xi) =
\varphi\bigl(4\rho_1^{-1}(\xi - n \rho_1/12)\bigr),
\]
so that $\phi_n = \phi_n \tilde\varphi_n$.
Consider only those $n$, for which $\phi_n\not \not\equiv 0$.
Since
$u(\xi) = 0$ for all $\xi\in (\xi_j - \rho_j, \xi_j +\rho_j)$, $j=1, 2, \dots, N$,
the latter requirement implies that
\[
\textup{either}\ \ \ \tilde\varphi_n\chi_{\Om} = 0
\ \ \textup{or}\ \  \tilde\varphi_n\chi_{\Om}
= \tilde\varphi_n.
\]
By Lemma \ref{localisation_1d:lem},
\[
\|\op_\a(\tilde \varphi_n^p) g_p(T(1; \R_+, \Om))
- g_p(  T(\tilde\varphi_n; \R_+, \Om)\|_{\GS_1}\le
  C'\bigl(\SN^{(3)}(\tilde\varphi_n; \rho_1)\bigr)^p\le C,
\]
for all $\a>0$.
In the case $\tilde\varphi_n\chi_{\Om} = 0$, this immediately implies
\eqref{phij:eq} with $g(t) = t^p-t$, since $\phi_n = \phi_n \tilde\varphi_n^p$
for any $p=  1, 2, \dots$. Suppose now that
$\tilde\varphi_n\chi_{\Om}= \tilde\varphi_n$, so that
\[
g_p(  T(\tilde\varphi_n; \R_+, \Om) = g_p(  T(\tilde\varphi_n; \R_+, \R).
\]
Using Lemma \ref{localisation_1d:lem} "backwards", we get
\[
\|\op_\a(\tilde\varphi_n^p)
\bigl(g_p(T(1; \R_+, \Om)) - g_p(\chi_{\R_+})\bigr)\|_{\GS_1}\le C,
\]
which leads to
\[
\|\op_\a(\phi_n) \bigl(g(T(1; \R_+, \Om)) - g(\chi_{\R_+})\bigr)\|_{\GS_1}\le C.
\]
Since $g(\chi_{\R_+}) = 0$, this entails \eqref{phij:eq}.
Summing over all $n$'s, we get
\begin{align*}
\|\op^l_\a(au)g(T(1; \R_+, \Om))\|_{\GS_1}
\le &\ \|\op^l_\a(a)\| \
\sum_{|n|\le J} \|\op_\a(\phi_n) g(T(1; \R_+, \Om))\|_{\GS_1}\\[0.2cm]
\le &\ C' J\le C\rho\rho_1^{-1}\le C\rho \ m_1(\SX),
\end{align*}
where we have used Lemma \ref{general_norm:lem} to estimate the norm of $\op^l_\a(a)$
under the condition $\a L\rho\ge c$, and the bound
\[
\rho_1^{-1}\le m_1(\SX),
\]
which follows from the definition \eqref{rhok:eq}. Using
\eqref{moments:eq} we complete the proof.
\end{proof}

\begin{thm} \label{cross_section:thm}
Let $g$ be as in \eqref{newm:eq}.
Let the symbols $a, b$ and the function $\psi$ be as in \eqref{widoma:eq}.
Assume that $\a\rho\ge c$,\ $L\ge r \ge 1$.
Then
\begin{align}\label{cross:section:eq}
\biggl|\tr\bigl( \op_\a^l(\psi a)g(T(1; \R_+, \Om))\bigr)
- &\  \GA(g) \underset{\xi\in\SX}\sum  a(0, \xi)\log r\biggr|\notag\\
\le &\ C_\d m_\d(\SX) \bigl[\rho^\d + \a^{1-\d} \rho r + \a^{-\d} \log r\bigr]
\SN^{(2, 2)}(a; L, \rho),
\end{align}
for any $\d\ge 1$,
where the sum on the left hand side equals zero if $\SX=\varnothing$.
\end{thm}

\begin{proof}
Assume that $\SN^{(2, 2)}(a; L, \rho) = 1$.
If $\SX = \varnothing$, then the estimate immediately follows
from Lemma \ref{varnothing:lem} and the definition \eqref{mempty:eq}.

Suppose that $\SX$ is non-empty.
Represent the trace
\[
\CT =
\tr\bigl( \op_\a^l(\psi a)g(T(1; \R_+, \Om))\bigr)
\]
on the left-hand side of \eqref{cross:section:eq} as the sum
\begin{align*}
\CT = &\ \CT_1+\CT_2+\CT_3,\\[0.2cm]
\CT_1 = &\ \underset{1\le j\le N; \a\rho_j\ge 1}\sum
\tr\bigl( \op_\a^l(\psi a u_j )g(T(1; \R_+, \Om))\bigr),\\[0.2cm]
\CT_2 = &\ \underset{1\le j\le N; \a\rho_j< 1}\sum
\tr\bigl( \op_\a^l(\psi a u_j )g(T(1; \R_+, \Om))\bigr),\\[0.2cm]
\CT_3 = &\ \tr\bigl( \op_\a^l(\psi a u)g(T(1; \R_+, \Om))\bigr).
\end{align*}
To every term with $\a\rho_j \ge 1$ we apply Lemma \ref{node:lem},
so that
\begin{align}\label{ct1:eq}
\biggl|
\CT_1 - \GA(g) \underset{\xi_j\in\SX: \a\rho_j\ge 1}\sum  a(0, \xi_j)\log r\biggr|
\le &\   C\biggl[N (1+ rL^{-1}) + \log\ r\ \a^{-\d}\sum_{j=1}^N\rho_j^{-\d}\biggr]
\notag\\
\le &\  C_\d m_\d(\SX)\bigl[\rho^\d + \a^{-\d}\log\ r \bigr] , \d \ge 0,
\end{align}
where we have used \eqref{mosch:eq}, and the inequality $rL^{-1}\le 1$.
If $\a\rho_j<1$ we use the estimate \eqref{1dim_bound:eq}, which we rewrite
as follows:
\begin{equation*}
|\tr\bigl( \op_\a^l(\psi a u_j )g(T(1; \R_+, \Om))\bigr)
- \log \ r \ \GA(g) a(0, \xi_j)|\le C\a\rho r,
\end{equation*}
which leads to
\begin{equation*}
\biggl|
\CT_2 - \GA(g) \underset{\xi_j\in\SX: \a\rho_j< 1}
\sum  a(0, \xi_j)\log r\biggr|
\le   C \tilde N \a\rho r,
\end{equation*}
where $\tilde N$ is the number of points $\xi_j$ such that $\a\rho_j<1$. Similarly to
\eqref{mosch:eq} one can estimate:
\[
\tilde N\le \a^{-\d}\underset{j: \a\rho_j < 1}\sum\rho_j^{-\d}
\le \a^{-\d} m_{\d}(\SX),\
\]
for any $\d \ge 0$, so that
\begin{equation}\label{ct2:eq}
\biggl|
\CT_2 - \GA(g) \underset{\xi_j\in\SX: \a\rho_j< 1}
\sum  a(0, \xi_j)\log r\biggr|
\le   C \a^{1-\d} m_{\d}(\SX)\rho r.
\end{equation}
To estimate $\CT_3$ use Lemma \ref{outside:lem}:
$
|\CT_3|\le C_\d \rho^\d m_{\d}(\SX),
$
for any $\d \ge 1$. Together with \eqref{ct1:eq} and \eqref{ct2:eq} this leads to
\eqref{cross:section:eq}.
\end{proof}

%%%%%%%%%%%%%%%%%%%%%%%%%%%%%%%%%%%%%%%%%%%%%%%%%%%%%%%%%%%%%%%%%%%%%%%%%%%%%%%%%%%%

\section{Partitions of unity, and a reduction to the flat boundary}
\label{partition:sect}

Now we are in a position to evaluate the asymptotics of the trace
 \begin{equation}\label{b_trace:eq}
\GT_\a(b; \L, \Om; g) := \tr \bigl(\op_\a^l(b) g(T(1; \L, \Om))\bigr)
\end{equation}
for a compactly supported symbol $b$, with the function $g(t) = t^p-t$,
$p = 1, 2, \dots $. Theorem \ref{localasymptotics:thm} gives the required asymptotics
for the case when $b$ is supported on a ball contained inside $\L$, and the
answer in this case does not include any information about the boundary of $\L$.
 Now we are ready to tackle the more difficult case,
 when the support of $b$ has a non-trivial intersection
 with the boundary $\p\L$.  We assume that the domains $\L$ and $\Om$ satisfy Condition
\ref{graph:cond}, i.e. , $\L=\G(\Phi; \boldO_{\L}, \boldk_\L)$,
$\Om = \G(\Psi; \boldO_{\Om}, \boldk_{\Om})$,
with some orthogonal transformations $\boldO_{\L}, \boldO_{\Om}$, some
$\boldk_{\L}, \boldk_{\Om}\in\R^d$, and some
functions $\Phi,\Psi\in\plainC1(\R^{d-1})$, which satisfy \eqref{gradient:eq}.
Without loss of generality we
suppose that $\boldO_\L = \boldI$, $\boldk_{\L}=\mathbf 0$, i.e.
\begin{equation*}
\L = \G(\Phi; \boldI, \mathbf 0) = \{\bx\in \R^d:  x_d > \Phi(\hat\bx)\}.
\end{equation*}
About the symbol $b$ we assume that either
\begin{equation}\label{bsupport_inf:eq}
b\in\BS^{(d+2, d+2)},\ \supp b\subset \R^d\times B(\mathbf 0, \rho),
\end{equation}
or
\begin{equation}\label{bsupport:eq}
b\in\BS^{(d+2, d+2)},\ \supp b\subset B(\mathbf 0, 1)\times B(\mathbf 0, \rho),
\end{equation}
with some $\rho >0$. For some intermediate results
we need weaker smoothness assumptions on $b$.
Sometimes the dependence of the trace \eqref{b_trace:eq}
on some of the arguments is omitted from the notation and
then we write $\GT_\a(b; \L; g)$ or $\GT_\a(b; \L)$ etc.

\subsection{Two partitions of unity}\label{partitions:subsect}
Our strategy is to ``approximate" $\L$ by a half-space and then use the approach
developed in \cite{Widom2}.
The first step is to
divide the domain $\L$ into two subsets: a ``boundary layer",
which will eventually contribute to the second term of the asymptotics \eqref{main:eq},
 and the ``inner part", which affects only the first term of the asymptotics.
The study of these two domains require two different partitions of unity.
We begin with the partition of unity for the boundary layer.

Let $v_1, v_2\in\plainC\infty(\R)$ be two non-negative functions such that
$v_1(t) + v_2(t) = 1$ for all $t\in\R$ and
\begin{equation}\label{v:eq}
v_1(t) =
\begin{cases}
0,\ t\le 1,\\
1,\ t\ge 2.
\end{cases}
\end{equation}
Define a partition of unity subordinate to the covering
of the half-axis by the intervals
\[
\Delta_{-1} = (-2, 3),
\Delta_k = (2r^k, 3 r^{k+1}), \ \ k = 0, 1, \dots,
\]
where $r > 2$. Denote
\begin{equation*}
v_1^{(k)}(t) = v_1\biggl(\frac{t-r^k}{r^k}\biggr),\
v_2^{(k)}(t) = v_2\biggl(\frac{t-r^{k+1}}{r^{k+1}}\biggr), k =0, 1, \dots,
\end{equation*}
and define
\begin{equation*}
\z_{-1}(t) = v_2(t-1) v_1(t+2),\ \ \
\z_k(t) = v_1^{(k)}(t) v_2^{(k)}(t), \ k = 0, 1, 2, \dots.
\end{equation*}
It is clear that $\z_k(t) = 0$ if $t\notin \Delta_k$, $k = -1, 0, \dots$.
It follows from the definition of $v_1, v_2$ that
for $r >2$ and any $K$,
\[
\sum_{k=-1}^K \z_k(t)=  v_2^{(K)}(t),\ \ \ \
\sum_{k=-1}^\infty \z_k(t) = 1, \ t\ge 0.
\]
Define two cut-off functions on $\R^d$:
\begin{equation}\label{arrows:eq}
q^{\downarrow}(\bx) = v_2^{(K)}\bigl(\a(x_d - \Phi(\hat\bx))\bigr),\
q^{\uparrow}(\bx) = v_1^{(K)}\bigl(\a(x_d - \Phi(\hat\bx))\bigr).
\end{equation}
To find the trace asymptotics of $\GT_\a(b)$  we study the
traces $\GT_\a(q^{\downarrow} b)$ and $\GT_\a(q^{\uparrow} b)$
for the following value of the parameter $K$:
\begin{equation}\label{k_upper_bound:eq}
K = K(\a; r, A) = \biggl[\frac{\log \a - A}{\log r}\biggr],
\end{equation}
with some number $A>0$.
Here $[\dots]$ denotes the integer part. This value of $K$ is chosen thus to ensure that
$q^{\downarrow}$ is supported in a thin ``boundary layer" whose width does not depend on
$\a$, whereas $q^{\uparrow}$
is supported ``well inside" the domain $\L$. More precisely,
\begin{equation}\label{q_up_down:eq}
\begin{cases}
q^{\downarrow}(\bx) = 0,\ x_d - \Phi(\hat\bx)\ge 3re^{-A}, \\[0.2cm]
q^{\uparrow}(\bx) = 0,\ x_d - \Phi(\hat\bx)\le 2 r^{-1} e^{-A}.
\end{cases}
\end{equation}
As we show later on, the asymptotics of the
trace $\GT_\a(q^{\downarrow}b)$ (see \eqref{b_trace:eq})
 ``feels" the boundary $\p\L$, whereas   $\GT_\a(q^{\uparrow}b)$ can be handled as
 if $\L$ were the entire space $\R^d$.
 The trace with $q^{\downarrow}$ requires a more careful analysis.
 In particular, we need a further partition of unity in variable $\hat\bx$.
Cover $\R^{d-1}$ by cubes
of the form $Q_{\bm} = Q_{\mathbf 0} + \bm$,
$\bm\in\Z^{d-1}$, where
\[
Q_{\mathbf 0} = (-1, 1)^{d-1}.
\]
Let $\s_{\bm} = \s_{\bm}(\hat\bx)$
 be a partition of unity, associated with this covering, such that
\begin{equation}\label{hyperpart:eq}
 \s_{\bm}(\hat\bx) = \s_{\bzero}(\hat\bx - \bm),
\end{equation}
which guarantees that
 \[
\ |\nabla_{\hat\bx}^j \s_{\bm}(\hat\bx)|\le C_j ,
 \]
for all $j$ uniformly in $\bm\in\Z^{d-1}$.
For each $k = 0, 1, \dots$ we use the partition of unity
\begin{equation}\label{sigmakm:eq}
\s_{k, \bm}(\hat\bx) = \s_{\bm}(\a \hat\bx r^{-k-1}),\ \bm\in\Z^{d-1}.
\end{equation}
Now define for all $\bx\in\R^d$
\begin{equation}\label{qkm:eq}
\begin{cases}
q_{-1}(\bx) = \z_{-1}\bigl(\a\bigl(x_d - \Phi(\hat\bx)\bigr)\bigr),\\[0.2cm]
q_{k, \bm} (\bx)
= \z_k\bigl(\a\bigl(x_d - \Phi(\hat\bx)\bigr)\bigr) \s_{k, \bm}(\a\hat\bx r^{-k-1}),\
k = 0, 1, \dots, \bm\in\Z^{d-1},
\end{cases}
\end{equation}
so that
\begin{equation}\label{qdown:eq}
q^{\downarrow}(\bx) = q_{-1}(\bx) + \sum_{k=0}^K \sum_{\bm\in\Z^{d-1}}
q_{k, \bm}(\bx).
\end{equation}
The contributions of the boundary layer and the inner region are found
in Lemmas \ref{alphaa_down:lem} and \ref{alphaa_up:lem} respectively.
For Lemma \ref{alphaa_up:lem} we need another, more standard partition of unity.
To define it we follow the scheme described in \cite{H}, Ch.1.
Let us state the required result in the
form convenient for our purposes:

\begin{prop}\label{hor:prop}
 Let $\ell\in\plainC1(\R^d)$ be a function such that
\begin{equation}\label{slow:eq}
|\ell(\bx) - \ell(\by)|\le \varrho |\bx-\by|,\
\end{equation}
for all $\bx, \by\in\R^d$ with some $\varrho\in [0, 1)$. Then there exists
a set $\bx_j\in\R^d$, $j\in\mathbb N$ such that the balls
$B(\bx_j, \ell(\bx_j))$ form a covering of $\R^d$ with the finite intersection property,
i.e. each ball intersects no more than $N = N(\varrho)<\infty$ other balls.
Furthermore, there exist non-negative functions $\psi_j\in\plainC\infty_0(\R^d)$,
$j\in\mathbb N$, supported in $B(\bx_j, \ell(\bx_j))$ such that
\begin{equation*}
\sum_j \psi_j(\bx) = 1,
\end{equation*}
and
\begin{equation*}
|\nabla^m\psi_j(\bx)|\le C_m \ell(\bx)^{-m},
\end{equation*}
for all $m$ uniformly in $j$.
\end{prop}

For our purposes the convenient choice of $\ell(\bx)$ for all $\bx\in\R^d$ is
\begin{equation}\label{dist_ell:eq}
\ell(\bx) =
\begin{cases}
\dfrac{1}{32\lu M\ru}\bigl(x_d-\Phi(\hat\bx)\bigr),\
x_d > \Phi(\hat\bx)+\a^{-1},\\[0.2cm]
\dfrac{1}{32\a\lu M\ru},\
x_d \le \Phi(\hat\bx)+\a^{-1}.
\end{cases}
\end{equation}
Here $M>0$ is the constant from condition
\eqref{gradient:eq}. Since $|\nabla\ell|\le 1/32$ a.e.,
the condition \eqref{slow:eq} is satisfied
with $\varrho=1/32$.

\begin{lem}\label{in_ball:lem}
If $x_d - \Phi(\hat\bx)> \a^{-1}$, then
$B(\bx, 32\ell)\subset \L$.
\end{lem}

\begin{proof} By \eqref{linside:eq},
under the condition $x_d - \Phi(\hat\bx)> \a^{-1}$ we have
for any $\by\in B(\bx, 32\ell)$:
\begin{align*}
y_d - \Phi(\hat\by)
\ge &\ x_d - \Phi(\hat\bx) - \lu M\ru|\bx-\by|\\
= &\ 32\lu M\ru\ell(\bx) - \lu M\ru|\bx-\by|>0.
\end{align*}
This proves that $\by\in \L$, as claimed.
\end{proof}

Incidentally, the partition of unity just introduced is also useful for
the boundary layer. Indeed,
our study of the traces $\GT_\a(q_{k, \bm} b)$, $k = 0, 2, \dots,$ $\bm\in\Z^{d-1},$
is based on the estimates obtained in the previous sections for
symbols from the classes $\BS^{(m, n)}$. Observe, however, that
those estimates are not directly applicable to symbols of the type $q_{k, \bm}$,
since these functions do not have isotropic scaling properties
in-built in the norms \eqref{norm:eq}. Thus in Lemma \ref{brick:lem} below
we study individually operators with $\psi_j q_{k, \bm}$ and
then sum up in $j$.

\subsection{Reduction to the flat boundary}

The first step in finding the asymptotics of $\GT_\a(q_{k, \bm}b; \L))$ is to
replace $\L$ by a half-space, as in Lemma \ref{flat_boundary_b:eq}. We emphasize
that the choice of this half-space depends on $k = 0, 1, \dots$ and $\bm\in\mathbb Z^{d-1}$.

For a fixed $\hat\bz\in\R^{d-1}$ define
\begin{equation*}
\Phi_0(\hat\by) = \Phi_0(\by; \hat\bz) =
 \Phi(\hat\bz)+ \nabla\Phi(\hat\bz)\cdot(\hat\by-\hat\bz).
\end{equation*}
Then by \eqref{nabla_mod:eq},
\begin{equation}\label{modulus2:eq}
\underset{\hat\bz\in\R^{d-1}}\sup\ \underset{\hat\by:|\hat\bz-\hat\by|\le s} \max
|\Phi(\hat\by) - \Phi_0(\hat\by, \hat\bz)|\le s\vare(s),
\end{equation}
For each $k = 0, 1 \dots$, $\bm\in\Z^{d-1}$
define the approximating half-space by
\begin{equation}\label{pim:eq}
\Pi_{k, \bm} = \{\bx: x_d > \Phi_0( \hat\bx;\ \hat\bx_{k, \bm})\},
\ \ \hat\bx_{k, \bm} = \a^{-1} r^{k+1}\bm.
\end{equation}
For our analysis of various asymptotics it is convenient to introduce the notion of a
$\CW$-sequence.

\begin{defn}\label{w:defn}
Let $w_k=w_k(r, A)$, $k\ge 0$, be a sequence of non-negative numbers,
depending on the parameters $r, A >0$, and let
$K = K(\a; r, A)$ be as defined in \eqref{k_upper_bound:eq}.
 We say that $w_k$ is a \textsl{$\CW$-sequence} if
\begin{equation}\label{limitzero:eq}
\lim_{r\to\infty}\limsup_{A\to\infty}
\limsup_{\a\to\infty} \frac{1}{\log\a}
\sum_{k\le K(\a; r, A)} w_k(r, A) = 0.
\end{equation}
If a $\CW$-sequence depends on some other \textsl{fixed} parameter, for instance $M$ from
\eqref{gradient:eq},
then it is not reflected in the notation and we still write simply $w_k(r, A)$.
\end{defn}

\begin{lem}\label{brick:lem}
Let $\L$ and $\Om$ be some domains satisfying Condition
\ref{graph:cond} with $\boldO_\L=\boldI, \boldk_{\L}=\mathbf 0$.
Suppose that the symbol $b$ satisfies \eqref{bsupport_inf:eq}.
Then for any $k = 0, 1, \dots, K = K(\a; r, A)$,
\begin{align}\label{brick:eq}
\| q_{k, \bm}\op_\a(b))\bigl(
g_p(T(1; \L))
- &\ g_p(T(1; \Pi_{k, \bm}))\bigr)\|_{\GS_1}\notag\\[0.2cm]
\le &\  r^{(k+1)(d-1)} \rho^{d-1} w_k(r, A)
\SN^{(d+1, d+2)}(b; 1, \rho),
\end{align}
with some $\CW$-sequence $w_k(r, A)$ independent of the symbol $b$,
uniformly in $\bm\in\Z^{d-1}$, $\rho\ge c$. Moreover,
$w_k(r, A)$ does not depend on the
functions $\Phi, \Psi$, and on the transformation $(\boldO_{\Om}, \boldk_{\Om})$, but
may depend on the parameter $M$.
\end{lem}

\begin{proof} Denote $\Pi = \Pi_{k, \bm}$. Using \eqref{trans:eq}
we may assume that $\bm = \hat\bzero$ and $\Phi(\hat\bzero) = 0$,
so that $\Phi_{\bzero}(\hat\bx) = \nabla\Phi(\hat{\bzero})\cdot \hat\bx$.
For $\mu >0$ denote
\[
D_k(\mu) = \{\bx:  2\mu^{-1}\a^{-1} r^k < x_d - \Phi(\hat\bx) < 3\mu\a^{-1} r^{k+1},\
 |\hat\bx|< \mu\a^{-1} r^{k+1}\}.
\]
It is clear that $\supp q_{k, \mathbf 0}\subset D_k(1)$.
Also, since $k\ge 0$, we have
$x_d - \Phi(\hat\bx)> \a^{-1}$ for $\bx\in D_k(2)$.

\underline{Step 1:} Let us show first that
\begin{equation}\label{dk2:eq}
D_k(2)\subset \L\cap\Pi\cap B\bigl(\mathbf 0, 13\lu M\ru \a^{-1} r^{k+1}\bigr),
\end{equation}
for sufficiently large $A$.
By construction $D_k(2)\subset\L$.
For an arbitrary $\bx\in D_k(2)$ write, using \eqref{modulus2:eq}:
\begin{align}\label{malost:eq}
x_d - \Phi_0(\hat\bx) = &\ x_d - \Phi(\hat\bx)+ (\Phi(\hat\bx) - \Phi_0(\hat\bx))\notag\\
\ge &\ x_d - \Phi(\hat\bx) -  2\a^{-1}r^{k+1}\varepsilon(2\a^{-1}r^{k+1})  \notag\\
\ge &\ \a^{-1} r^k -
2\a^{-1} r^k   r\varepsilon(2\a^{-1}r^{k+1})\notag\\
= &\ \a^{-1} r^k \bigl[1 - 2 r \varepsilon(2\a^{-1}r^{k+1})\bigr].
\end{align}
Estimating
\[
\a^{-1}r^{k+1}\le r e^{-A}
\]
for all $k\le K$, and taking a sufficiently large $A$, we can guarantee
that the right hand side of \eqref{malost:eq} is positive
which proves that $D_k(2)\subset \Pi$. Furthermore, since $\Phi(\hat{\mathbf 0})=0$
and $|\nabla\Phi(\hat\bx)|\le M$, for any $\bx\in D_k(2)$ we have
\begin{equation*}
|x_d|\le |x_d - \Phi(\hat\bx)|
+ |\Phi(\hat\bx)|\le 6\a^{-1} r^{k+1} + M|\hat\bx|
\le 12\lu M\ru \a^{-1} r^{k+1},
\end{equation*}
so that
\[
|\bx|\le \sqrt{x_d^2 + |\hat\bx|^2}\le 13\lu M\ru \a^{-1} r^{k+1},
\]
and hence $\bx\in B(\mathbf 0, 13\lu M\ru \a^{-1} r^{k+1})$, as claimed.
Thus \eqref{dk2:eq} is proved.

\underline{Step 2.} Let  $B(\bx_j, \ell_j), \ell_j: =\ell(\bx_j)$ be the balls
from Proposition \ref{hor:prop}, which form a covering of $\R^d$
with the finite intersection property. Let $J\subset\Z$ be the set of all indices
$j$ such that $B(\bx_j, 4\ell_j)\subset D_k(2)$.  Let us prove that
\begin{equation}\label{pokr:eq}
D_k(1)\subset \bigcup_{j\in J} B(\bx_j, \ell_j).
\end{equation}
It suffices to show for arbitrary $\bx\in D_k(2)$,
that if $\bw\in B(\bx, 4\ell(\bx))$, but $\bw\notin D_k(2)$,
then $B(\bx, \ell(\bx))\cap D_k(1) = \varnothing$.
In view of \eqref{dist_ell:eq},
$\ell(\bx) = (32\lu M\ru)^{-1}(x_d - \Phi(\hat\bx))$.

We consider separately three cases.

\textit{Case 1. }Suppose that $\bw\in B(\bx, 4\ell(\bx))$
and $w_d - \Phi(\hat\bw)< \a^{-1} r^k$. Then by \eqref{linside:eq},
\begin{align*}
x_d - \Phi(\hat\bx) \le &\ w_d -\Phi(\hat\bw) + \lu M \ru |\bw - \bx|\\
< &\  \a^{-1} r^k + \frac{1}{8}(x_d - \Phi(x_d)).
\end{align*}
Thus $x_d - \Phi(\hat\bx)< 8(7\a)^{-1} r^k$, so that
for any $\by\in B(\bx, \ell(\bx))$ we have by \eqref{linside:eq}
again that
\begin{equation*}
y_d - \Phi(\hat\by)\le  x_d - \Phi(\hat\bx) + \lu M\ru |\bx-\by|
\le \frac{33}{32}(x_d - \Phi(\hat\bx)) < 2 \a^{-1} r^k,
\end{equation*}
and hence $B(\bx, \ell(\bx))\cap D_k(1) = \varnothing$.

\textit{Case 2.}
Suppose that $\bw\in B(\bx, 4\ell(\bx))$
and $w_d - \Phi(\hat\bw)> 6\a^{-1} r^{k+1}$.
 Then by \eqref{linside:eq},
\begin{align*}
x_d - \Phi(\hat\bx) \ge &\ w_d - \Phi(\hat\bw) - \lu M\ru |\bw-\bx|\\
> &\ 6\a^{-1} r^{k+1} - \frac{1}{8}(x_d - \Phi(x_d)).
\end{align*}
Thus $x_d - \Phi(\hat\bx)> 48(9\a)^{-1} r^{k+1}$, so that
for any $\by\in B(\bx, \ell(\bx))$ we have
\begin{equation*}
y_d - \Phi(\hat\by)\ge  x_d - \Phi(\hat\bx) - \lu M \ru |\bx-\by|
> \frac{31}{32} (x_d - \Phi(\hat\bx)) >  3 \a^{-1} r^{k+1},
\end{equation*}
and hence $B(\bx, \ell(\bx))\cap D_k(1) = \varnothing$ again.

\textit{Case 3.}
Suppose now that $\bw\in B(\bx, 4\ell(\bx))$, but $|\hat\bw|\ge 2\a^{-1} r^{k+1}$.
By virtue of \textit{Case 2} above, we may assume
that $x_d-\Phi(\hat\bx)\le 6\a^{-1}r^{k+1}$.
Consequently,
\[
|\hat\bx| \ge |\hat\bw| - |\bx - \bw|
\ge 2\a^{-1}r^{k+1} - \frac{1}{8\lu M\ru} (x_d - \Phi(\hat\bx))
\ge \frac{5}{4}\a^{-1} r^{k+1}.
\]
Thus for any $\by\in B(\bx, \ell(\bx))$,
\[
|\hat\by|\ge |\hat\bx| - |\bx-\by|\ge |\hat\bx|
- \frac{1}{32\lu M\ru}(x_d - \Phi(\hat\bx))
\ge \frac{5}{4}\a^{-1} r^{k+1} - \frac{1}{5\lu M \ru} \a^{-1} r^{k+1}
> \a^{-1} r^{k+1},
\]
so that $B(\bx, \ell(\bx))\cap D_k(1) = \varnothing$. This proves
\eqref{pokr:eq}.

\underline{Step 3.} Let $\psi_j$ be the partition of unity from Proposition
\ref{hor:prop}.
According to \eqref{pokr:eq}
 \begin{equation*}
q_{k, \mathbf 0}(\bx) = \sum_{j\in J} \phi_j(\bx),\ \
\phi_j(\bx) =   q_{k, \mathbf 0}(\bx)\psi_j(\bx).
\end{equation*}
Thus the trace norm
\begin{equation*}
Z = \|q_{k, \mathbf 0}\op_\a^l(b)\bigl(g_p(T(1; \L)) - g_p(T(1; \Pi))\bigr)\|_{\GS_1}
\end{equation*}
is estimated as follows:
\begin{equation}\label{zbyz:eq}
Z\le\sum_{j\in J}Z_j,\ \
Z_j = \| \op^l_\a( \phi_j b) \bigl(g_p(T(1; \L)) - g_p(T(1; \Pi))\bigr)\|_{\GS_1}.
\end{equation}
For each $Z_j$ we use Lemma \ref{flat_boundary_b:lem} with
some $\vark \in (0, \vark_p)$.
Since $\bx_j\in D_k(2)$, we have
\begin{equation}\label{ellj:eq}
\frac{1}{32\lu M\ru}\a^{-1} r^k\le \ell_j \le \frac{6}{32\lu M\ru} \a^{-1} r^{k+1},
\end{equation}
and as $k\le K$, this means that $\ell_j\le r e^{-A}$. For sufficiently
large $A$, we have $\ell_j\le 1$, so that
\begin{equation*}
\SN^{(d+1, d+2)}(b\phi_j; \ell_j, \rho)
\le C\SN^{(d+1, d+2)}(b; 1, \rho).
\end{equation*}
By \eqref{dk2:eq} and by definition
of the set $J$, the ball $B(\bx_j, 4\ell_j), j\in J$ satisfies
the condition \eqref{distance:eq} with
\[
t = \a^{-1}r^{k+2}.
\]
for all $r\ge 13\lu M\ru$. Furthermore, in view of \eqref{ellj:eq},
\[
5\lu M \ru r < t\ell_j^{-1} <  32\lu M\ru r^2,\ j\in J.
\]
Thus for sufficiently large $r$ the first condition in
\eqref{ells:eq} is satisfied uniformly in $k$.
Now estimate for $k\le K$:
\[
t\ell_j^{-1}\vare(4t)< 32\lu M\ru r^2 \vare(r^2 e^{-A}).
\]
For sufficiently large $A$ this quantity is arbitrarily small, and hence
the second condition in \eqref{ells:eq} is also satisfied.
Finally, $\a\ell_j \rho\ge cr^k\rho\ge c'$. Thus by Lemma \ref{flat_boundary_b:lem},
\begin{equation*}
Z_j\le C  (\a\ell_j \rho)^{d-1} \CR_{\vark}(\a; \ell_j, \rho, t)
\SN^{(d+1, d+2)}(b; 1, \rho),
\end{equation*}
uniformly in $j\in J$. By definition \eqref{remainder:eq},
\begin{align}\label{Wpk:eq}
\CR_{\vark}(\a; \ell_j, \rho, t)
\le &\ C \bigl[ r^{-\vark} + r^{4(d-1)} r^{-\vark k} + r^{4(d-1)}
\bigl(r^2\varepsilon(r^{2}e^{-A})\bigr)^{\vark}\bigr]\notag\\
=: &\ \tilde w_k(r, A),
\end{align}
for all $\rho \ge c$, uniformly in $j\in J$.
By \eqref{zbyz:eq},
\begin{equation*}
Z\le \sum_{j\in J} Z_j\le C
(\a\rho)^{d-1}  \SN^{(d+1, d+2)}(b; 1, \rho)   \tilde w_k(r, A)
\sum_{j\in J} \bigl(\ell(\bx_j)\bigr)^{-1}.
\end{equation*}
Due to the finite intersection property (see Proposition \ref{hor:prop}),
and to the property
$B(\bx_j, 4\ell(\bx_j))\subset D_k(2)$ for all $j\in J$, the sum on the
right-hand side is estimated by
the integral
\begin{equation*}
\int_{D_k(2)} \ell(\bx)^{-1} d\bx
 = \underset{|\hat\bx|\le 2\a^{-1} r^{k+1}}\int
 \int_{\a^{-1} r^k}^{6\a^{-1} r^{k+1}} t^{-1} dt d\hat\bx
 \le C\a^{1-d} r^{(k+1)(d-1)} \log r,
\end{equation*}
and hence
\[
Z\le C r^{(k+1)(d-1)} \rho^{d-1} w_k(r, A)
\SN^{(d+1, d+2)}(b; 1, \rho),\ \
w_k(r, A) =  \log r\ \tilde w_k(r, A).
\]

\underline{Step 4.}
To complete the proof of \eqref{brick:lem} it remains
to show that $w_k(r, A)$ is a $\CW$-sequence, i.e. it
satisfies the property \eqref{limitzero:eq}.
To this end consider each term in the definition
of $w_k= \log r \ \tilde w_k$ (see \eqref{Wpk:eq}) separately. Clearly,
\[
\sup_A \sup_\a\frac{\log r}{\log\a}
\sum_{k\le K} r^{-\vark}\le r^{-\vark}\to 0, \ r\to\infty,
\]
so this sum satisfies \eqref{limitzero:eq}.
Consider now
\begin{equation*}
\frac{\log r}{\log\a}
\sum_{k\le K} r^{4(d-1)} r^{-\vark k}\le \frac{\log r}{\log\a}\
r^{d(d-1)} \frac{1}{1-r^{-\vark}}\to 0, \ \a\to\infty,
\end{equation*}
so this sum also satisfies \eqref{limitzero:eq}.
For the last term in the definition \eqref{Wpk:eq}, note that
\[
r^{k+2}\a^{-1}\le r^2 e^{-A},
\]
and hence
\begin{equation*}
\sup_\a \frac{\log r}{\log\a}
\sum_{k\le K} r^{4(d-1)} (r^2 \vare(r^{k+2}\a^{-1}))^{\vark}
\le r^{4(d-1)} (r^2 \vare(r^2 e^{-A}))^{\vark}\to 0, \ A\to\infty,
\end{equation*}
which follows from the fact that
 $\vare(s)\to 0$ as $s\to 0$.
Therefore $w_k$ satisfies \eqref{limitzero:eq},
and the proof of the Lemma is now complete.
\end{proof}

%%%%%%%%%%%%%%%%%%%%%%%%%%%%%%%%%%%%%%%%%%%%%%%%%%%%%%%%%%%%%%%%%%%%%%%%%%%%%%%%%%%%

\section{Asymptotics of the trace \eqref{b_trace:eq}}\label{asymptotics:sect}

Now we compute the asymptotics of the trace
$\GT_\a(q_{k, \bm} b; \Pi_{k, \bm}, \Om; g)$
(see \eqref{b_trace:eq} for definition of $\GT_\a$, and
\eqref{pim:eq} for definition of $\Pi_{k, \bm}$),
where $k = 0, 1, \dots$,  $\bm\in\Z^{d-1}$,
\begin{equation}\label{gp:eq}
g(t) = t^p-t,\ \  \textup{with some}\ \  p\in\mathbb N.
\end{equation}
As before we assume that $\L$ and $\Om$ are graph-type domains, but we also need
stronger smoothness conditions on $\Psi$,
and some specific restrictions on the orthogonal
transformations $\boldO_{\L}, \boldO_{\Om}$.

\begin{cond}\label{graph_type1:cond}
Let $\Phi\in\plainC1(\R^{d-1})$,
$\Psi\in\plainC3(\R^{d-1})$ be some real-valued functions satisfying
\eqref{gradient:eq} and
\begin{equation}\label{gradient11:eq}
\|\nabla^2\Psi\|_{\plainL\infty}
+ \|\nabla^3\Psi\|_{\plainL\infty}\le \tilde M,
\end{equation}
with some $\tilde M>0$.
Assume that $\Phi(\hat{\bold0}) = \Psi(\hat{\bold0}) = 0$.
The domain $\L$ is defined as
$\L = \G(\Phi; \boldI, \mathbf 0)$, and there is an index $l = 1, 2, \dots, d$ such that
the domain $\Om$ is defined by
\begin{equation}\label{omcircle:eq}
\Om =
\begin{cases}
\Om^{(+)}= \{\bxi: \xi_l > \Psi(\overc\bxi)\},\\[0.2cm]
\textup{or}\ \ \Om^{(-)} = \{\bxi: \xi_l < \Psi(\overc\bxi)\},
\end{cases}
\overc\bxi = (\xi_1, \dots, \xi_{l-1}, \xi_{l+1}, \dots, \xi_d).
\end{equation}
\end{cond}

One can easily see that the above definition of the domain $\Om$ in fact
describes $\Om$ as a graph-type domain $\G(\tilde\Psi; \boldO, \mathbf 0)$
with some easily identifiable orthogonal transformation $\boldO$ and function
$\tilde\Psi$. For example,
if $l = d-1$, then for the domain $\Om^{(-)}$
the entries $O_{js}$
of the matrix $\boldO$ are given by
\begin{equation*}
O_{js} =
\begin{cases}
\d_{j,s},\ 1\le j\le d-2,\\
-\d_{j+1, s},\ j = d-1,\\
\d_{j-1, s},\ j = d,
\end{cases}
\end{equation*}
and the function $\tilde\Psi$ by $\tilde\Psi(\hat\bxi)= -\Psi(\hat\bxi)$.

In what follows one of the main players will be the set
\begin{equation}\label{ompm:eq}
\Om^{(\pm)}(\hat\bxi) = \{t\in \R: (\hat\bxi, t)\in\Om^{(\pm)}\} =
\{\xi_d: \pm \xi_l> \pm \Psi(\overc \bxi)\}.
\end{equation}
We are interested in the structure of the set
$\Om^{(\pm)}(\hat\bxi)\cap (-2\rho, 2\rho)$ for
$\hat\bxi\in\CC_\rho^{(d-1)}$ with an appropriate $\rho >0$, see \eqref{cube:eq} for the
definition of the cube $\CC_\rho^{(n)}$.
The set  $\Om^{(\pm)}(\hat\bxi)\cap (-2\rho, 2\rho)$
is either empty, or it
is an open set, i.e. it is a countable union of disjoint open intervals whose
endpoints either coincide with $\pm 2\rho$, or lie inside of the interval $(-2\rho, 2\rho)$.
Denote
by $\SX^{(\pm)}(\hat\bxi)$ the set of those endpoints which lie strictly inside
$(-2\rho, 2\rho)$. Define also
\begin{equation*}
\SX(\hat\bxi) = \{\xi_d\in(-2\rho, 2\rho): \Psi(\overc \bxi) = \xi_l\}.
\end{equation*}
Obviously, $\SX^{(\pm)}(\hat\bxi)\subset \SX(\hat\bxi)$. In case
the set $\SX(\hat\bxi)$ is finite we use the quantity $m_\d(\SX(\hat\bxi))$
introduced in \eqref{rhok:eq}.
The next lemma is of primary importance:

\begin{lem}\label{bigdim:lem}
Let $\Psi\in\plainC3(\R^{d-1})$ be a function satisfying
Condition \ref{graph_type1:cond}, and let
the sets $\SX^{(\pm)}(\hat\bxi), \SX(\hat\bxi)\subset (-2\rho, 2\rho)$ be
as defined above. Then
\begin{enumerate}
\item
For almost all $\hat\bxi\in \CC^{(d-1)}_{\rho}$:
\begin{enumerate}
\item
the set $\SX(\hat\bxi)$ is finite,
\item
$\SX(\hat\bxi) = \SX^{(\pm)}(\hat\bxi)$,
\item
the disjoint open intervals forming $\Om^{\pm}(\hat\bxi)$ have \underline{distinct}
endpoints.
\end{enumerate}
\item
For any $\d \in (0, 2)$
the function $m_{\d}\bigl(\SX(\hat\bxi)\bigr)$ satisfies the
bound
\begin{equation}\label{reciprocal_om:eq}
\int_{\hat\bxi\in \CC^{(d-1)}_{\rho}} m_\d(\SX(\hat\bxi)) d\hat\bxi
\le C \rho^{d-1-\d}
\bigl(1+(\rho + \rho^2)\tilde M\bigr).
\end{equation}
The constant
$C$ depends only on $\d$ and dimension $d$.
\end{enumerate}

\end{lem}

This lemma follows immediately from Theorem \ref{bigdim:thm}.

First we establish the asymptotics for the case $\nabla\Phi(\hat\bx_{k, \bm}) = \mathbf 0$.
Throughout this section all $\CW$-sequences
$w_k(r, A)$ do not depend on the
functions $\Phi\in\plainC1, \Psi\in\plainC3$, but
may depend on the constants $M$ in \eqref{gradient:eq} and
$\tilde M$ in \eqref{gradient11:eq}.

\begin{lem}\label{No1:lem}
Assume that
\begin{enumerate}
\item
$g$ is given by \eqref{gp:eq},
\item
$\L, \Om$
are two graph-type domains satisfying Condition \ref{graph_type1:cond},
\item
$0\le k\le K$ with $K$ defined in \eqref{k_upper_bound:eq},
\item
$b$ is a symbol satisfying \eqref{bsupport_inf:eq},
\item
$\nabla\Phi(\hat\bx_{k, \bm}) = 0$.
\end{enumerate}
Then
\begin{align*}
\biggl|  \GT_\a(q_{k, \bm} b; \Pi_{k, \bm}, \Om; g)
-  \a^{d-1}\log r\ \GA(g) &\ \GW_1(\s_{k, \bm} b; \p\L, \p\Om)\biggr| \\[0.2cm]
\le &\ r^{(k+1)(d-1)}w_k(r, A)
\SN^{(2, 2)}(b, 1, \rho),
\end{align*}
with some $\CW$-sequence $w_k(r, A)$ independent of the symbol $b$
and of the point $\bm\in\Z^{d-1}$, uniformly
in $\rho\in [c, C]$ with arbitrary positive constants
$c, C$ such that $c < C$.
\end{lem}

\begin{proof}
Using \eqref{trans:eq}
we may assume that $\Phi(\hat\bx_{k, \bm}) = 0$.
We also assume throughout that $\a \ge 2$, $r\ge 5$, $A\ge 2$, and $c\le \rho\le C$.

\underline{Step 1: a reduction to the one-dimensional case.}
Let $\Om^{(\pm)}(\hat\bxi)$ be as defined in \eqref{ompm:eq}.
Since $\Pi_{k, \bm} = \{\bx: x_d >0\}$ (see \eqref{pim:eq}),
a straightforward calculation shows that $g_p\bigl(T(1; \Pi_{k, \bm}, \Om)\bigr)$
is a PDO in $\plainL2(\R^{d-1}, \GH), \GH = \plainL2(\R),$
with the operator-valued symbol
\[
g_p\bigl(T(1; \R_+, \Om^{(\pm)}(\hat\bxi))\bigr).
\]
Therefore the operator $X = q_{k, \bm}\op_\a(b) g\bigl(T(1; \Pi_{k, \bm}, \Om)\bigr)$
can be viewed as a PDO with the operator-valued symbol
\[
\CX_\a(\hat\bx, \hat\bxi)
= q_{k, \bm}(\hat\bx,\ \cdot\ )
\op_\a\bigl(b(\hat\bx,\ \cdot\ ; \hat\bxi, \ \cdot\ )\bigr)
g\bigl(T(1; \R_+, \Om^{(\pm)}(\hat\bxi))\bigr),
\]
 i.e.
\[
(X_\a u)(\bx) = \biggl(\frac{\a}{2\pi}\biggr)^{d-1}
\underset{\R^{d-1}}\int
\underset{\R^{d-1}}\int e^{i\a\hat\bxi\cdot(\hat\bx-\hat\by)}
\bigl(\CX_\a(\hat\bx, \hat\bxi) u(\hat\by, \ \cdot\ )\bigr)
d\hat\by  d\hat\bxi,
\]
 for any $u$ from the Schwartz class on $\R^d$.
 In order to use Lemma \ref{trace:lem} we need to find the asymptotics of the
 trace of the operator $\CX(\hat\bx, \hat\bxi)$. By definition \eqref{qkm:eq},
\[
\CX_\a(\hat\bx, \hat\bxi)
= \s_{k, \bm}(\hat\bx) \tilde\CY_\a(\hat\bx, \hat\bxi),
\]
with the operator-valued symbol
\begin{equation*}
\tilde\CY_\a(\hat\bx, \hat\bxi)
= \tilde\psi \op_\a\bigl(b(\hat\bx,\ \cdot \ ; \hat\bxi,\ \cdot\ )\bigr)
g\bigl(T(1; \R_+, \Om^{(\pm)}(\hat\bxi))\bigr), \
\tilde\psi(x_d) = \z_k(\a(x_d - \Phi(\hat\bx))).
\end{equation*}
 After the change of the variable
 \[
 x_d = r^k\a^{-1} t,
 \]
 the operator $\tilde\CY_\a(\hat\bx, \hat\bxi) $ becomes unitarily equivalent to
 \begin{equation}\label{beta:eq}
 \CY_\b(\hat\bx, \hat\bxi)
 = \psi\op_{\b}(a) g(T(1; \R_+, \Om^{(\pm)}(\hat\bxi))),\ \b = r^k,
 \end{equation}
 where
 \begin{align*}
 \psi(t) = &\ \z_k(r^k t -\a\Phi(\hat\bx))
 = v_1\bigl(t-1 - \a r^{-k}\Phi(\hat\bx)\bigr)
 v_2\bigl(tr^{-1} - 1 - \a r^{-k-1} \Phi(\hat\bx)\bigr),\\[0.2cm]
 a(t, \xi) = &\ b(\hat\bx, r^k\a^{-1} t; \hat\bxi, \xi).
 \end{align*}
 The asymptotics of $\tr \CY_\b(\hat\bx, \hat\bxi)$ is found with the help of
 Theorem \ref{cross_section:thm}. Let us check that
 the symbol $a$, the function $\psi$, and the set $\Om^{(\pm)}(\hat\bxi)$ satisfy
 the conditions of this Theorem.

 Note that $a\in \BS^{(2, 2)}$ and
\begin{equation}\label{bigL:eq}
\SN^{(2, 2)}(a; L, \rho)\le \SN^{(2, 2)}(b; 1, \rho),\
L := \a r^{-k}\ge e^A,
\end{equation}
where we have used that $k\le K$(see \eqref{k_upper_bound:eq}).
Moreover, $a(t, \xi) = 0$ for $|\xi|\ge \rho$.

Recall that $\a r^{-(k+1)} \hat\bx\in Q_{\bm}$,
so that in view of \eqref{nabla_mod:eq} and \eqref{k_upper_bound:eq},
\begin{equation}\label{phiplane:eq}
|\Phi(\hat\bx)| = |\Phi(\hat\bx) - \Phi(\hat\bx_{\bm})|
\le \vare(r^{k+1}\a^{-1}) r^{k+1} \a^{-1}\le \vare (r e^{-A}) \a^{-1} r^{k+1} e^{-A},
\end{equation}
and hence
 \begin{equation*}
 \a r^{-k} |\Phi(\hat\bx)|\le r\vare(r e^{-A}).
 \end{equation*}
 For sufficiently large $A$, we can guarantee that
 $ r\vare( r e^{-A}) \le 1/4$. Thus by the definition \eqref{v:eq}, the  function
$\psi(t)$ satisfies the requirements
\eqref{psi_defn:eq}.

From now on we always assume that $\hat\bxi$ is in 
the subset of full measure in $\CC_\rho^{(d-1)}$
such that 
%for all $\hat\bxi$ in this subset 
the properties (1)(a)-(c) from lemma
\ref{bigdim:lem} hold.  This guarantees that $\Om^{(\pm)}(\hat\bxi)$ is of the form required
in Theorem \ref{cross_section:thm}.

\underline{Step 2: asymptotics of $\tr \CY_\b(\hat\bx, \hat\bxi)$.}
Recall that $r\ge 2$, and note also that for sufficiently large  $A$ we have
$L\ge r$, see \eqref{bigL:eq} for definition of $L$. Moreover,
$\b \rho = r^k \rho\ge c $, see \eqref{beta:eq}  for
definition of $\b$.
Thus by Theorem \ref{cross_section:thm},
\begin{align*}
\biggl| \tr \CY_\b(\hat\bx, \hat\bxi) - &\ \GA(g) \log r
\sum_{\xi\in \SX(\hat\bxi)} a(0, \xi)\biggr|\\
\le &\ Cm_\d(\SX(\hat\bxi)) \bigl(1 + \b^{1-\d} r\bigr)
\SN^{(2, 2)}(a; L, \rho),
\end{align*}
for any $\d\ge 1$.
Rewrite, remembering the definition of $a(x, \xi)$:
\begin{align*}
\biggl| \tr \CY_\b(\hat\bx, \hat\bxi) - &\ \GA(g) \log r
\sum_{\xi_d\in \SX(\hat\bxi)}
b\bigl(\hat\bx, 0; \overc\bxi, \Psi(\overc\bxi)\bigr)\biggr|\notag\\
\le &\ C m_\d(\SX(\hat\bxi))\bigl[1+\b^{1-\d} r
\bigr] \SN^{(2, 2)}(a; L, \rho),
\end{align*}
Analyze the asymptotic term on the left-hand side.
By virtue of \eqref{phiplane:eq} for any $\bxi$ we have
\begin{equation*}
|b(\hat\bx, 0; \bxi) - b(\hat\bx, \Phi(\hat\bx); \bxi)|
\le \SN^{(1, 0)}(b; 1, \rho) r e^{-A} \vare(r e^{-A}),
\end{equation*}
so that by \eqref{mosch:eq},
\begin{align*}
\bigl|\sum_{\xi_d\in\SX(\hat\bxi)}\bigl(b(\hat\bx, 0; \bxi)
- b(\hat\bx, \Phi(\hat\bx); \bxi)\bigr)\bigr|
\le &\ \#\SX(\hat\bxi)\SN^{(1, 0)}(b; 1, \rho) r e^{-A}\vare(r e^{-A})\\[0.2cm]
\le &\ (2\rho)^\d\  m_\d(\SX(\hat\bxi)) r e^{-A} \vare(r e^{-A})
\SN^{(1, 0)}(b; 1, \rho).
\end{align*}
Assume that $A$ is so large that $r\log r e^{-A} \vare(r e^{-A})\le 1$.
Thus
\begin{equation}\label{outm:eq}
\begin{cases}
\biggl| \tr \CY_\b(\hat\bx, \hat\bxi) - \GA(g) \log r
\underset{\xi_d\in \SX(\hat\bxi)}\sum
b\bigl(\hat\bx, \Phi(\hat\bx); \overc\bxi, \Psi(\overc\bxi)\bigr)\biggr|
\le  C\CL_k(\hat\bxi; r),\\[0.3cm]
\CL_k(\hat\bxi; r) = m_\d(\SX(\hat\bxi))\bigl(1+ r^{(1-\d)k+1}
\bigr) \SN^{(2, 2)}(b; 1, \rho),
\end{cases}
\end{equation}
where we have replaced $\b$ with its value, i.e. $r^k$.

\underline{Step 3:  asymptotics of $\tr X_\a$. }
By Lemma \ref{trace:lem},
\[
\tr X_\a = \biggl(\frac{\a}{2\pi}\biggr)^{d-1}
\int  \s_{k, \bm}(\hat\bx)\tr \CY_\b(\hat\bx, \hat\bxi) \
d\hat\bx d\hat\bxi,
\]
and hence, by \eqref{outm:eq},
\begin{equation}
\biggl|\tr X_\a - \biggl(\frac{\a}{2\pi}\biggr)^{d-1}
\log r \ \GA(g) \GB_\k(\a, \bm; r )\biggr|
\le C \GR_k(\a, \bm; r, A),\label{tracea:eq}
\end{equation}
where the leading term is
\begin{equation}\label{leading:eq}
\GB_k(\a, \bm; r) =  \underset{\R^{d-1}}\int
\underset{\R^{d-1}}\int \s_{k, \bm}(\hat\bx)
\sum_{\xi_d\in\SX(\hat\bxi)}
b\bigl(\hat\bx, \Phi(\hat\bx); \overc\bxi, \Psi(\overc\bxi)\bigr)d\hat\bx d\hat\bxi,
\end{equation}
and the remainder $\GR_k(\a, \bm; r)$ is
\begin{align}\label{gremainder11:eq}
\GR_k(\a, \bm; r)
= &\ \a^{d-1}\underset{\R^{d-1}}\int\underset{\R^{d-1}}\int
\s_\bm(\a\hat\bx r^{-k-1}) \CL_k(\hat\bxi; r)
d\hat\bx d\hat\bxi\notag\\[0.2cm]
\le &\ C r^{(k+1)(d-1)}
\underset{\R^{d-1}}\int \CL_k(\hat\bxi; r) d\hat\bxi.
\end{align}
Here we have used the definition \eqref{sigmakm:eq}.
Let us calculate the leading term first.
According to Lemma \ref{area:lem},
\begin{gather*}
\GB_k(\a, \bm; r) =  \int_{\p\Om}\int_{\p\L}
\s_{k, \bm}(\hat\bx) b(\bx, \bxi)
\frac{|\bn_{\p\Om}(\bxi)\cdot \be_d|}{\sqrt{1+|\nabla\Phi(\hat\bx)|^2}}
dS_{\bx} dS_{\bxi},
\end{gather*}
where
\begin{equation*}
\bn_{\p\Om}(\bxi):= \frac{\bigl(-\p_{\xi_1} \Psi(\overc\bxi),
\ \dots,
 -\p_{\xi_{l-1}} \Psi(\overc\bxi), 1, -\p_{\xi_{l+1}} \Psi(\overc\bxi),
 \dots, -\p_{\xi_d} \Psi(\overc\bxi)\bigr)}
{\sqrt{1+|\nabla \Psi(\overc\bxi)|^2}},
\end{equation*}
is the unit normal to the surface $\xi_l = \Psi(\overc\bxi)$ at the point
$\bxi  = \bigl(\xi_1, \dots, \xi_{l-1}, \Psi(\overc\bxi),
\xi_{l+1}, \dots, \xi_d\bigr)$.
On the other hand, the unit normal to $\p\L$ at
the point $\bx = \bigl(\hat\bx, \Phi(\hat\bx)\bigr)$ is given by the vector
\[
\bn_{\p\L}(\bx) =
\frac{\bigl(-\nabla\Phi(\hat\bx), 1\bigr)}{\sqrt{1+|\nabla\Phi(\hat\bx)|^2}}.
\]
Thus
\[
\biggl| \bn_{\p\Om}(\bxi)\cdot\bn_{\p\L}(\bx) -
\frac{\bn_{\p\Om}(\bxi)\cdot\be_d}{\sqrt{1+|\nabla \Phi(\hat\bx)|^2}}\biggr|
\le |\nabla\Phi(\hat\bx)-\nabla\Phi(\hat\bzero)|.
\]
By \eqref{nabla_mod:eq} the right hand side does not exceed
\[
\vare(r^{k+1}\a^{-1})\le \vare(r e^{-A}),
\]
so the main term of the asymptotics satisfies the bound
\begin{align*}
|\GB_k(\a, \bm; r) - (2\pi)^{d-1} &\ \GW_1(\s_{k, \bm}b; \p\L, \p\Om)|\notag\\[0.2cm]
\le &\ C\vare(r e^{-A}) \max|b(\bx, \bxi)|
\a^{1-d} r^{(k+1)(d-1)},
\end{align*}
uniformly in $\bm\in\Z^{d-1}$. Since $\vare(re^{-A})\to 0$ as $A\to\infty$,
the above estimate can be rewritten as follows:
\begin{align}\label{leading1:eq}
|\GB_k(\a, \bm; r) - (2\pi)^{d-1} &\ \GW_1(\s_{k, \bm}b; \p\L, \p\Om)|\notag\\[0.2cm]
\le &\  \a^{1-d} r^{(k+1)(d-1)} w^{(1)}_k(r, A) \SN^{(2, 2)}(b; 1, \rho),
\end{align}
with some $\CW$-sequence $w^{(1)}_k(r, A)$.

\underline{Step 4: estimating the remainder \eqref{gremainder11:eq}.}
Let us show that the remainder \eqref{gremainder11:eq} defines a $\CW$-sequence.
By \eqref{gremainder11:eq} and \eqref{outm:eq},
\begin{equation}
\GR_k(\a, \bm; r)\le C r^{(k+1)(d-1)}
 \bigl(1+ r^{(1-\d)k+1}
\bigr)  \int_{|\hat\bxi|\le \rho}
m_\d(\SX(\hat\bxi)) d\hat\bxi \
\SN^{(2, 2)}(b; 1, \rho),\label{gremainder:eq}
\end{equation}
for any $\d\ge 1$, where $\b = r^k$.
By \eqref{reciprocal_om:eq},
\begin{equation*}
\int_{|\hat\bxi|\le \rho} m_\d\bigl(\SX(\hat\bxi)\bigr)d\hat\bxi
\le C\rho^{d-1-\d}(1+(\rho+\rho^2) \tilde M ),
\end{equation*}
for all $\d\in [1, 2)$.
Thus $\GR_k(\a, \bm; r)$ satisfies the bound
\begin{align}\label{grweak:eq}
\GR_k(\a, \bm; r)\le &\ C
r^{(k+1)(d-1)} w^{(2)}_k(r, A)
\SN^{(2, 2)}(b; 1, \rho),\notag\\[0.2cm]
w^{(2)}_k(r, A) = &\ 1+ r^{1+ (1-\d)k}.
\end{align}
Assume that $\d >1$.
In view of \eqref{k_upper_bound:eq},
\[
\frac{1}{\log\a}\sum_{k=0}^K \bigl(1 + r^{k(1-\d)+1}\bigr)
\le \frac{1}{\log r} + \frac{1}{\log \a}\  \frac{r}{1-r^{1-\d}}.
\]
Thus the triple limit $\lim_{r\to\infty} \limsup_{A\to\infty}\limsup_{\a\to\infty}$
of the above expression equals zero.  This proves that $w_k^{(2)}$ is a $\CW$-sequence.

\underline{Step 5: end of the proof.}
According to \eqref{tracea:eq}, \eqref{leading1:eq}
and \eqref{grweak:eq},
for each $k = 0, 1, \dots, K$ we have
\begin{align*}
\bigl|\tr X_\a - \biggl(\frac{\a}{2\pi}\biggr)^{d-1} &\ \GA(g) \log r\
\GW_1(\s_{k, \bm} b; \p\L, \p\Om)\bigr|\\[0.2cm]
\le &\ r^{(k+1)(d-1)} w_k(r, A)\SN^{(2,2)}(b; 1, \rho).
\end{align*}
with some $\CW$-sequence $w_k(r, A)$. The proof is complete.
\end{proof}

Now we can remove the condition $\nabla\Phi(\hat\bx_{k, \bm}) = 0$. Recall the notation
\eqref{mphi:eq}.

\begin{lem}\label{No2:lem}
Suppose that Conditions (1)-(4) of Lemma \ref{No1:lem} are satisfied.
We assume in addition that
\begin{equation}\label{gradphipsi:eq}
M_{\Phi}\le \frac{1}{2},\ M_{\Phi} M_{\Psi}\le \frac{1}{2},\ \
\textup{if}\ \ \ l = d.
\end{equation}
Then
\begin{align*}
\biggl|  \GT_\a(q_{k, \bm} b; \Pi_{k, \bm}, \Om; g)
-  \a^{d-1}\log r\ \GA(g) &\ \GW_1(\s_{k, \bm} b; \p\L, \p\Om)\biggr| \\[0.2cm]
\le &\ r^{(k+1)(d-1)}w_k(r, A)
\SN^{(2, 2)}(b, 1, \rho),
\end{align*}
with some $\CW$-sequence $w_k(r, A)$ independent of the symbol $b$
and of the point $\bm\in\Z^{d-1}$,
uniformly in $\rho\in [c, C]$ for arbitrary constants $0<c<C$.
\end{lem}

\begin{proof}
We reduce the problem to the one considered in Lemma \ref{No1:lem} using  a suitable
linear transformation.
Recall the definition of $\Pi_{\bm}$:
\begin{equation*}
\Pi_{k, \bm} = \{\bx: x_d> \Phi(\hat\bx_{k, \bm})
+ \hat\bb\cdot(\hat\bx-\hat\bx_{k, \bm})\},\
\hat\bb := \nabla\Phi(\hat\bx_{k, \bm}).
\end{equation*}
Now we use \eqref{orthogonal1:eq} and \eqref{chilambda:eq}
with $\boldk = (\hat\bx_{k, \bm}, \Phi(\hat\bx_{k, \bm}))$, $\boldk_1 = \bzero$, and
the non-degenerate transformation defined by
\begin{equation*}
\boldM\bx = (\hat\bx, x_d + \hat\bb\cdot\hat\bx),
\ \ \textup{so that}\ \  \boldM^T\bx = (\hat\bx + \hat\bb x_d, x_d).
\end{equation*}
Due to the unitary equivalence \eqref{orthogonal1:eq} and \eqref{chilambda:eq},
\begin{equation*}
\GT_\a(q_{k, \bm} b; \Pi_{k, \bm}, \Om; g)
= \GT_\a(\tilde q_{k} b_{\boldM, \boldk}; \Pi, \Om^T_{\boldM}; g),
\end{equation*}
where $\Pi = \{\bx: x_d >0\}$, $\Om^T_{\boldM} = \boldM^T\Om$,
\[
b_{\boldM, \boldk}(\bx, \bxi) = b(\boldM\bx + \boldk, (\boldM^T)^{-1} \bxi),
\]
and
\begin{gather*}
\tilde q_{k}(\hat\bx, x_d) =
\s_{k, \bzero}(\hat\bx)
\z_k\bigl(\a(x_d - \Phi_1(\hat\bx)) \bigr),\\[0.2cm]
\Phi_1(\hat\bx) = \Phi(\hat\bx+\hat\bx_{k, \bm})- \Phi(\hat\bx_{k, \bm})
- \hat\bb\cdot\hat\bx.
\end{gather*}
The function $\Phi_1(\hat\bx)$
satisfies the conditions of Lemma \ref{No1:lem}, and
in particular, the condition $\nabla\Phi_1(\hat\bzero) = 0$.
Furthermore, according to \eqref{omcircle:eq}, for $l\not= d$,
\begin{equation}\label{psi1:eq}
\Om^T_{\boldM} = \{\bxi: \xi_l > \Psi_1(\overc \bxi)\}\
\ \ \  \textup{or}\ \ \
 \Om^T_{\boldM} = \{\bxi: \xi_l < \Psi_1(\overc \bxi)\},
 \end{equation}
with
\begin{equation*}
\Psi_1(\overc\bxi) = \Psi(\tilde\bxi - \tilde\bb \xi_d, \xi_d)+ b_l \xi_d,\   \ \
\tilde\bxi = (\xi_1, \dots, \xi_{l-1}, \xi_{l+1}, \dots, \xi_{d-1}).
\end{equation*}
Note that all partial derivatives of $\Psi_1$ up to order $3$ are uniformly
bounded by a constant depending only on the parameters $M$ in \eqref{gradient:eq}
and $\tilde M$ in \eqref{gradient11:eq}.
If $l = d$, then in view of \eqref{omcircle:eq},
\[
(\boldM^T)^{-1}\Om^T_{\boldM} =
\{\bxi: \xi_d > \Psi(\hat\bxi)\}\ \
\textup{or}\ \ \ (\boldM^T)^{-1}\Om^T_{\boldM} = \{\bxi: \xi_d < \Psi(\hat\bxi)\}.
\]
In view of \eqref{gradphipsi:eq}, we can use Lemma \ref{linear:lem}
with $\boldB = (\boldM^T)^{-1}$, which ensures
that $\Om^T_{\boldM}$ is again given by \eqref{psi1:eq}
with a $\plainC3(\R^{d-1})$-function $\Psi_1$, whose partial derivatives
up to order $3$ are uniformly bounded by a constant
depending only on the parameters $M$ in \eqref{gradient:eq}
and  $\tilde M$ in \eqref{gradient11:eq}.

Thus for both $l <d$ and $l = d$ the domain
$\Om^T_{\boldM}$ satisfies \eqref{omcircle:eq} with some function $\Psi_1$.
Finally, due to the condition
$M_{\Phi}\le 1/2$,
the new symbol $b_{\boldM}(\bx, \bxi)$ vanishes if $|\bxi|\ge 2\rho$.
 Thus one can apply Lemma \ref{No1:lem} with $\bm = 0$ and
 with $2\rho$ instead of $\rho$. This leads to
the estimate
\begin{align*}
\biggl|\GT_\a(\tilde q_k b_{\boldM, \boldk}; \Pi, \Om^T_{\boldM}; g)
- \a^{d-1}\log r\  \GA(g)
&\ \GW_1(\s_{k, \bzero} b_{\boldM}; \p\L_{\boldM, \boldk},
\p\Om^T_{\boldM})\biggr|\\[0.2cm]
\le &\ r^{(k+1)(d-1)}w_k(r, A) \SN^{(2, 2)}(b_{\boldM}; 1, \rho),
\end{align*}
with some $\CW$-sequence $w_k$, and with
$\L_{\boldM, \boldk} = \boldM^{-1}(\L-\boldk)$.
 According to \eqref{gw_invariance:eq},
\[
\GW_1(\s_{k, \bzero} b_{\boldM, \boldk}; \p\L_{\boldM, \boldk}, \p\Om^T_{\boldM})
 = \GW_1(\s_{k, \bm} b; \p\L, \p\Om).
\]
Furthermore,
\[
\SN^{(2, 2)}(b_{\boldM, \boldk}; 1, \rho) \le C \SN^{(2, 2)}(b; 1, \rho),
\]
with a universal constant $C$. This leads to the proclaimed bound.
\end{proof}

\begin{cor} \label{No2:cor}
Suppose that the conditions
(1)-(4) of Lemma \ref{No1:lem} are satisfied,
and that \eqref{gradphipsi:eq} holds. Then
\begin{align}\label{no2:eq}
\biggl| \GT_\a(q_{k, \bm} b; \L, \Om; g) -
\a^{d-1}\log r \ \GA(g) &\ \GW_1(\s_{k, \bm} b; \p\L, \p\Om)\biggr| \notag\\[0.2cm]
\le &\ Cr^{(k+1)(d-1)}w_k(r, A)
\SN^{(d+1, d+2)}(b, 1, \rho),
\end{align}
with some $\CW$-sequence $w_k(r, A)$, independent of the symbol $b$
and of the point $\bm\in\Z^{d-1}$,
uniformly in $\rho\in [c, C]$ for arbitrary constants $0<c<C$.
\end{cor}

\begin{proof}
The estimate follows directly from Lemmas \ref{brick:lem} and \ref{No2:lem}.
\end{proof}

The next step is to obtain the appropriate asymptotic formulas for
$q^{\downarrow}$ and $q^{\uparrow}$ (see \eqref{arrows:eq})
instead of $q_{k, \bm}$.

\begin{lem} \label{alphaa_down:lem}
Suppose that
$g$ is as in \eqref{gp:eq}, that the domains $\L, \Om$ satisfy Condition
\ref{graph_type1:cond},
and that \eqref{gradphipsi:eq} holds.
Suppose also that $b$ is a symbol satisfying \eqref{bsupport:eq}. Then
\begin{equation}\label{alphaa_down:eq}
\lim_{r\to \infty}\limsup_{A\to\infty}\limsup_{\a\to\infty}\frac{1}{\a^{d-1}\log\a}
\GT_\a(q^{\downarrow} b; \L, \Om; g) = \GA(g) \GW_1( b; \p\L, \p\Om),
 \end{equation}
uniformly in $\rho\in [c, C]$ with arbitrary constants $0 < c < C$.
\end{lem}

\begin{proof}
Represent $q^{\downarrow}(\bx)$ in accordance with \eqref{qdown:eq}:
\begin{equation}\label{qdown1:eq}
q^{\downarrow}(\bx) =
\z_{-1}\bigl(\a(x_d-\Phi(\hat\bx))\bigr)
  + \sum_{k=0}^K\sum_{\bm\in\mathbb Z^{d-1}} q_{k, \bm}(\bx),
\end{equation}
and calculate  the contribution of each term to the sought trace.

First we consider $\z_{-1}$. Since the support of the symbol $b$ in the $\bx$-variable is
the ball $B(0, 1)$, the support of $\z_{-1} b$ is contained in the set
$$
\{\bx: |x_d - \Phi(\hat\bx)|\le C\a^{-1}, |\hat\bx|<1\}.
$$
Cover this set with open balls of radius $\a^{-1}$, and  denote by
$\phi_k$, $k = 1, 2, \dots, N$ the partition of unity associated with
this covering such that  $|\nabla^l \phi_k|\le C_l \a^{l}$, for all $l = 1, 2, \dots,$
uniformly in $k$. Clearly,
\[
\SN^{(n_1, n_2)}(\phi_k b; \a^{-1}, \rho)
\le C\SN^{(n_1, n_2)}(b; 1, \rho),
\]
for any $n_1, n_2$, uniformly in $k$. Now we can estimate, using \eqref{trace_symbol:eq}:
\[
\| \z_{-1} \op_{\a}(b) g(T(1; \L, \Om))\|_{\GS_1}
\le \sum_{k=1}^N \| \op_\a(\phi_k b)\|_{\GS_1}
\le C N \SN^{(d+1, d+1)}(b; 1, \rho).
\]
Clearly, the covering  can be chosen in such a way that
$N\le C\a^{d-1}$, which implies that
\begin{equation}\label{z-1:eq}
\| \z_{-1} \op_{\a}(b) g(T(1; \L))\|_{\GS_1}\le C\a^{d-1}
\SN^{(d+1, d+1)}(b; 1, \rho).
\end{equation}
Consider now the sum on the right-hand side of \eqref{qdown1:eq}.
For each trace
$
\GT_\a(q_{k, \bm}b; \L, \Om; g)
$
we use Corollary \ref{No2:cor}, and then
sum up the obtained inequalities over $\bm\in\mathbb Z^{d-1}$ and
$k\ge 0$.
Let us handle the asymptotic coefficient first:
\begin{align*}
Y(\a, r, A):=&\ \sum_{k=0}^K\sum_{\bm\in\mathbb Z^{d-1}}\a^{d-1} \log r \ \GA(g)
\GW_1(\s_{k, \bm}b; \p\L, \p\Om)\\[0.2cm]
= &\ \a^{d-1} \log r \ \GA(g)
\GW_1(b; \p\L, \p\Om) \sum_{k=0}^K 1,
\end{align*}
where we have used that fact that $\sum_{\bm}\s_{k, \bm} = 1$ for any $k=0, 1, \dots$.
Since
\begin{equation*}
 \sum_{k=0}^K 1
=  \frac{\log \a - A}{\log r} + O(1),
\end{equation*}
we have
\[
\lim_{\a\to\infty}\frac{1}{\a^{d-1}\log\a} Y(\a, r, A)
= \GA(g) \GW_1(b; \p\L, \p\Om),
\]
for any $A\in\R$ and $r > 0$.

Let us consider the remainder.
To estimate the sum up the right-hand sides
of \eqref{no2:eq} over different values of $k$ and $\bm$, we
observe that the summation over $\bm$ for each value of $k$, is restricted to
$|\bm|\le C\a r^{-(k+1)}$, since the support of the symbol $b$ in the $\bx$-variable
is contained in the unit ball.Thus
\begin{equation*}
Z (\a, r, A):= \sum_{k=0}^K \sum_{|\bm|\le C\a r^{-(k+1)}}
w_k( r, A) r^{(k+1)(d-1)}
\le  C\a^{d-1} \sum_{k=0}^K w_k( r, A).
\end{equation*}
By definition of the $\CW$-sequence  (see \eqref{limitzero:eq}),
\[
\lim_{r\to\infty} \limsup_{A\to\infty}\limsup_{\a\to\infty}\frac{1}{\a^{d-1}\log\a}
Z(\a, r, A) = 0.
\]
Together with \eqref{z-1:eq} this leads to \eqref{alphaa_down:eq}.
\end{proof}

\begin{lem}\label{alphaa_up:lem}
Suppose that the conditions of Lemma \ref{alphaa_down:lem}
are satisfied. Then
\begin{equation}\label{alphaa_up:eq}
\lim_{r\to\infty}\limsup_{A\to\infty}
\limsup_{\a\to\infty}\frac{1}{\a^{d-1}\log\a}
\GT_\a(q^{\uparrow} b; \L, \Om; g) = 0,
\end{equation}
uniformly in $\rho\in [c, C]$ with arbitrary constants $0 < c < C$.
\end{lem}

\begin{proof}
We use the partition of unity from Proposition \ref{hor:prop} associated with
the slowly-varying function defined in \eqref{dist_ell:eq}.
Let $\bx_j$ be the sequence constructed in Proposition \ref{hor:prop}, and let
$\ell_j = \ell(\bx_j)$. Denote by $\mathcal J$ the set of indices $j$ such that
$\psi_j q^{\uparrow}\chi_{B(0, 1)}\not \equiv 0$.
By  \eqref{q_up_down:eq}
we have $x_d - \Phi(\hat\bx)\ge 2 r^K \a^{-1}\ge 2 r^{-1} e^{-A}$
for all $\bx$ in the support of
$q^{\uparrow}$. Thus $\ell_j \ge cr^{-1}e^{-A}$
for all indices $j\in\mathcal J$, and hence
by Lemma \ref{in_ball:lem}, $B(\bx_j, 32\ell_j)\subset\L$,
and in particular, the support of $\psi_j$ is strictly inside $\L$.
Since $\SN^{(m)}(\psi_j; \ell_j)\le C_m$
for all $m$, it follows from Theorem \ref{localasymptotics:thm} that
\begin{align*}
\|\op_\a^l(\psi_j q^{\uparrow} b)g\bigl(T(1; \L, \Om)\bigr)\|_{\GS_1}
\le &\ \|\op_\a^l(\psi_j b)g\bigl(T(1; \L, \Om)\bigr)\|_{\GS_1}\\[0.2cm]
\le &\ C (\a\ell_j\rho)^{d-1} \SN^{(d+2, d+2)}(b\psi_j; \ell_j, \rho)\\[0.2cm]
\le &\ C (\a\ell_j\rho)^{d-1} \SN^{(d+2, d+2)}(b; 1, \rho),
\end{align*}
for all $j\in\mathcal J$. Consequently,
\begin{equation*}
\|\op_\a^l(q^{\uparrow }b) g\bigl(T(1; \L, \Om)\bigr)\|_{\GS_1}
\le C\a^{d-1} \SN^{(d+2, d+2)}(b; 1, \rho)
\underset{j\in\mathcal J}
\sum \ell_j^{d-1},
\end{equation*}
where we have estimated $\rho\le C$.
Using the finite intersection property, we can estimate:
\begin{equation*}
\underset{j\in\mathcal J}\sum \ell_j^{d-1}
\le C\underset{\substack{cr^{-1}e^{-A}<x_d-\Phi(\hat\bx)<C\\
|\hat\bx|\le C}}\int \ell(\bx)^{-1} d\bx
\le C \underset{|\hat\bx|< C}\int \ \ \
\underset{ cr^{-1}e^{-A}<t<C}\int t^{-1} dt d\hat\bx
\le C' (A + \log r).
\end{equation*}
Therefore
\[
\|\bigl(\op_\a^l(q^{\uparrow}b) g(T(1; \L, \Om)\bigr)\|_{\GS_1}
\le C( A + \log r)\a^{d-1} \SN^{(d+2, d+2)}(b; 1, \rho).
\]
Thus the triple limit in \eqref{alphaa_up:eq} equals zero, as claimed.
\end{proof}

Let us now combine Lemmas \ref{alphaa_down:lem} and \ref{alphaa_up:lem}.

\begin{cor}\label{alpha:cor}
Suppose that the conditions of Lemma \ref{alphaa_down:lem}
are satisfied. Then
\begin{equation}\label{alphaa:eq}
\lim_{\a\to\infty}\frac{1}{\a^{d-1}\log\a}
\tr  \op_\a^l(b) g(T(1; \L, \Om)) =
\GA(g) \GW_1( b; \p\L, \p\Om).
\end{equation}
\end{cor}

\begin{proof} To avoid cumbersome formulae, throughout the proof we write
$G_1\approx G_2$ for any two trace-class operators depending on $\a$, such that
\[
\lim_{\a\to\infty} \frac{1}{\a^{d-1}\log\a} \|G_1 - G_2\|_{\GS_1} = 0.
\]
For brevity write $T:=T(1; \L, \Om)$.
By \eqref{sandwich1:eq},
\[
\op_\a^l(b) g(T)\approx \chi_{\L} \op^l_\a(b) g(T).
\]
Rewrite the symbol $\chi_{\L}b$ in the form
\[
\chi_{\L}b = q^{\downarrow}\chi_{\L}b + q^{\uparrow}\chi_{\L}b.
\]
Again by \eqref{sandwich1:eq},
\[
\chi_{\L} q^{\downarrow} \op^l_\a(b) g(T)\approx  q^{\downarrow} \op^l_\a(b) g(T),
\]
and  $q^{\uparrow}\chi_{\L} =q^{\uparrow}$ by definition of $q^{\uparrow}$. Thus
\[
\op_\a^l(b) g(T) \approx  q^{\downarrow} \op^l_\a(b) g(T)
+ q^{\uparrow} \op^l_\a(b) g(T).
\]
Now apply  \eqref{alphaa_down:eq} and \eqref{alphaa_up:eq}.
Since the left-hand side of \eqref{alphaa:eq}
does not depend on $A$ or $r$, the claimed result follows.
\end{proof}

Next, from the formula \eqref{alphaa:eq} containing $T(1; \L, \Om)$
we deduce a similar asymptotics for the operator $T(a; \L, \Om)$.

\begin{thm}\label{last_step:thm}
Assume that
\begin{enumerate}
\item\label{10:item}
$\L, \Om$
are two graph-type domains satisfying Condition \ref{graph_type1:cond},
\item\label{grad:item}
\eqref{gradphipsi:eq} is satisfied,
\item
$a\in\BS^{(d+2, d+2)}$,
\item
$b$ is a symbol satisfying \eqref{bsupport:eq}.
 \end{enumerate}
Then for any $p = 1, 2, \dots $,
\begin{align}
\lim_{\a\to\infty}\frac{1}{\a^{d-1}\log\a}
\biggl(\tr\bigl(\op_\a^l(b) g_p(T(a; \L, \Om))\bigr)
- &\ \a^d \GW_0\bigl(b g_p(a); \L, \Om\bigr)\notag\\[0.2cm]
&\ - \a^{d-1}\log\a\  \GW_1\bigl(b \GA(g_p; a); \p\L, \p\Om\bigr)\biggr) = 0.
\label{last_step:eq}
\end{align}
\end{thm}

\begin{proof}
As in the proof of Corollary \ref{alpha:cor} we use the notation
$G_1\approx G_2$. For brevity we omit $\L$, $\Om$, $\p\L, \p\Om$
from notation.
%
%write $T(a)$ instead of $T(a; \L, \Om)$.
Due to Lemma \ref{commute_through2:lem},
\begin{equation}\label{ap:eq}
\op(b) g_p(T(a))\approx \op(b a^p) g_p(T(1)).
\end{equation}
Furthermore, by Lemma \ref{commute_through2:lem} again, with the notation
$g(t) = t^p-t$, we have:
\[
\op(b a^p) g_p(T(1)) = \op(ba^p) g(T(1)) + \op(ba^p) T(1)
\approx  \op(ba^p) g(T(1)) + T(ba^p).
\]
By Corollary \ref{alpha:cor},
\[
\lim_{\a\to\infty}\frac{1}{\a^{d-1}\log\a}\tr \op(ba^p) g(T(1))
= \GW_1(b\GA(g_p; a)),
\]
where we have taken into account that $\GA(g_p - g_1; a) = \GA(g_p; a)$.
By \eqref{sandwich1_dual:eq},
\[
T(ba^p)\approx \chi_\L \op(ba^p) P_{\Om, \a} \chi_{\L}.
\]
The trace of the operator on the right-hand side equals
$\a^d \GW_0(b g_p(a); \L, \Om)$. Thus
\begin{align*}
\lim_{\a\to\infty}\frac{1}{\a^{d-1}\log\a}
\biggl(\tr\bigl(\op_\a^l(ba^p) g_p(T(1))\bigr)
- &\ \a^d \GW_0(b g_p(a))\\[0.2cm]
&\ - \a^{d-1}\log\a \GW_1(b \GA(g_p; a))\biggr) = 0.
\end{align*}
The reference to \eqref{ap:eq} completes the proof of \eqref{last_step:eq}.
\end{proof}

%%%%%%%%%%%%%%%%%%%%%%%%%%%%%%%%%%%%%%%%%%%%%%%%%%%%%%%%%%%%%%%%%%%%%%%%%%%%%%%%%%%%

\section{Proof of Theorem \ref{main:thm}}
\label{proof:sect}

From now on we assume that the domains $\L, \Om$ satisfy the conditions of Theorem
\ref{main:thm}.
We begin the proof of Theorem \ref{main:thm} with removing conditions
\ref{10:item} and \ref{grad:item} from Theorem \ref{last_step:thm}.

\begin{thm}\label{metamain:thm}
Assume that $\L$ and $\Om$ are as in Theorem \ref{main:thm}. Suppose also that
$a, b\in\BS^{(d+2, d+2)}$ and the symbol $b$ is compactly supported. Then
the asymptotic formula \eqref{last_step:eq} holds.
\end{thm}

\begin{proof}

Our first step is to construct
suitable coverings of $\L$ and $\Om$ by open balls, and associated partitions of unity.

By Definition \ref{local_domains:defn}, we can cover $\overline \L$
by finitely many open balls $B(\bw_j, R)$, $j=1, 2, \dots$
of some radius $R$ in such a way
that if $\bw_j \notin\p\L$, then $B(\bw_j, R)\subset\L$.
Using Corollary \ref{new_coordinates1:cor}, and making $R$ smaller
if necessary, one can always assume that for each $\bw:=\bw_j\in\p\L$, in the ball
$B(\bw, R)$ the domain $\L$ is represented by a $\plainC1$-graph-type domain
$\G(\Phi; \boldO, \bw)$ with a function, satisfying \eqref{propertyphi:eq},
and the bound
\begin{equation}\label{phib:eq}
M_{\Phi} \le \frac{1}{8\sqrt d}.
\end{equation}
Denote by $\phi_j\in\plainC\infty_0(R^d)$, $j = 1, 2, \dots$,
a partition of unity subordinate to the above covering of $\L$.
Using the unitary equivalence \eqref{unitary:eq} we may assume that $R = 1$.

For each ball $B(\bw_j, 1)$ further steps depend on the
location of the point $\bw_j$.

\underline{Case 1: $\bw_j\in \p\L$.}
Translating $\bw_j$ to the point zero, and applying an appropriate orthogonal
transformation, we may assume that \eqref{local_graph:eq} holds with
$\boldO = \boldI, \boldk = \mathbf 0$, i.e.
\begin{equation*}
\L\cap B(\bzero, 1) = \{\bx: x_d > \Phi(\hat\bx)\}\cap B(\bzero, 1).
\end{equation*}
Now we construct an appropriate partition of unity for the domain $\Om$.
Just as for $\L$ above, by Definition \ref{local_domains:defn}
one can cover $\overline\Om$ by finitely many open balls
$B(\bxi_l, \rho),\ l = 1, 2, \dots$ of some radius $\rho>0$ in such a way that if
$\bxi_l\notin \p\Om$, then
$B(\bxi_l, \rho)\subset\Om$. Using Corollary \ref{new_coordinates2:cor},
and making $\rho$ smaller if necessary, one can always assume that for each
$\bxi = \bxi_l\in\p\Om$, in the ball $B(\bxi, \rho)$ the domain $\Om$
is represented by a graph-type domain described in \eqref{new_coordinates11:eq}
with a $\plainC3$-function $G$, satisfying
the bound $\|\nabla G\|_{\plainL\infty}\le 4\sqrt d$.
Redenoting $G$ by $\Psi$, and recalling \eqref{phib:eq} we conclude that
the condition \eqref{gradphipsi:eq} is satisfied.
Without loss of generality
we may assume that $\nabla^2 G$ and $\nabla^3 G$ are uniformly bounded on $\R^{d-1}$.
Denote by $\eta_j, j=1, 2, \dots$, the partition of unity subordinate to the constructed
covering.

\underline{Sub-case 1.1: $\bxi_l\in\p\Om$.} Applying an appropriate translation,
we may assume that $\bxi_l = \mathbf 0$. Thus by the
above construction, in the balls $B(\mathbf 0, 1)$ and
$B(\mathbf 0, \rho)$ resp. the domains $\L$ and $\Om$ resp.
are represented by the graph-type
domains satisfying Condition \ref{graph_type1:cond}.

The symbol $\tilde b(\bx, \bxi) = b(\bx, \bxi) \phi_j(\bx) \eta_l(\bxi)$
is supported on the domain $ B(\bold0, 1)\times B(\mathbf 0, \rho)$,
so that by Lemma \ref{localisation1:lem}
we may assume that $\L = \G(\Phi; \boldI, \bzero)$
and $\Om$ is as defined by \eqref{omcircle:eq}. Thus by
Theorem \ref{last_step:thm},
\begin{align}\label{xion:eq}
\lim_{\a\to\infty}\frac{1}{\a^{d-1}\log\a}
\biggl(\tr\bigl(\op_\a^l(b\phi_j \eta_l) &\ g_p(T(a))\bigr)
-  \a^d \GW_0\bigl(b\phi_j \eta_l g_p(a)\bigr)\notag\\[0.2cm]
&\ - \a^{d-1}\log\a\  \GW_1\bigl(b\phi_j \eta_l \GA(g_p; a)\bigr)\biggr) = 0.
\end{align}
Here and below for brevity we omit $\L, \Om, \p\L, \p\Om$ from notation.

\underline{Sub-case 1.2: $\bxi_l\in\Om$.}
Since the domain $\Om$ satisfies Condition
\ref{global:cond} (2), we can use Theorem \ref{localasymptotics:thm},
which gives the formula
\begin{equation}\label{xiin:eq}
\lim_{\a\to\infty}\frac{1}{\a^{d-1}\log\a}
\biggl(\tr\bigl(\op_\a^l(b\phi_j \eta_l) g_p(T(a))\bigr)
-  \a^d \GW_0\bigl(b\phi_j \eta_l g_p(a)\bigr)\biggr) = 0.
\end{equation}
 Summing up over all values of the index $l$, from \eqref{xion:eq}
 and \eqref{xiin:eq} we get
\begin{align}
\lim_{\a\to\infty}\frac{1}{\a^{d-1}\log\a}
\biggl(\tr\bigl(\op_\a^l(b\phi_j \eta) &\ g_p(T(a))\bigr)
-  \a^d \GW_0\bigl(b\phi_j g_p(a)\bigr)\notag\\[0.2cm]
&\ - \a^{d-1}\log\a\  \GW_1\bigl(b\phi_j \GA(g_p; a)\bigr)\biggr) = 0,
\label{subcase12:eq}
\end{align}
where we have denoted $\eta = \sum_l \eta_l$. Since $1-\eta$ is supported outside $\Om$,
we have
\[
\op_\a^l(b\phi_j (1-\eta)) g_p(T(a))
=
- \phi_j [\op_\a^l(b\eta), \chi_{\L}]
P_{\Om, \a}\op_\a^l(a) P_{\Om, \a} \chi_{\L}g_{p-1}(T(a)).
\]
Due to \eqref{sandwich2:eq},
adding this term to the left-hand side of \eqref{subcase12:eq}, does not change
the asymptotics. Thus
\begin{align}\label{case1on:eq}
\lim_{\a\to\infty}\frac{1}{\a^{d-1}\log\a}
\biggl(\tr\bigl( \op_\a^l(b\phi_j) &\ g_p(T(a))\bigr)
-  \a^d \GW_0\bigl(b\phi_j g_p(a)\bigr)\notag\\[0.2cm]
&\ - \a^{d-1}\log\a\  \GW_1\bigl(b\phi_j \GA(g_p; a)\bigr)\biggr) = 0,
\end{align}

\underline{Case 2: $\bw_j\in\L$.}
By Definition \ref{local_domains:defn}
we can cover $\overline\Om$ by finitely many open balls
$B(\bxi_l, \rho),\ l = 1, 2, \dots$ of some radius $\rho$ in such a way that if
$\bxi_l\notin \p\Om$, then
$B(\bxi_l, \rho)\subset\Om$.   Denote by $\eta_l, l= 1, 2, \dots$, the partition of unity
subordinate to this covering.
Since $B(\bw_j, 1)\subset\L$, the domain $\L$ satisfies Condition
\ref{global:cond}(\ref{space:subcond}),
and hence, by Theorem \ref{localasymptotics:thm},
\begin{equation*}
\lim_{\a\to\infty}\frac{1}{\a^{d-1}\log\a}
\biggl(\tr\bigl(\op_\a^l(b \phi_j \eta_l) g_p(T(a))\bigr)
-  \a^d \GW_0\bigl(b \phi_j \eta_l g_p(a)\bigr)\biggr) = 0.
\end{equation*}
Summing over $l$, we get
\begin{equation*}%\label{in:eq}
\lim_{\a\to\infty}\frac{1}{\a^{d-1}\log\a}
\biggl(\tr\bigl(\op_\a^l(b\phi_j \eta) g_p(T(a))\bigr)
-  \a^d \GW_0\bigl(b\phi_j  g_p(a)\bigr)\biggr) = 0,\
\end{equation*}
with $\eta = \sum_l\eta_l$. As in the previous case,
we can remove $\eta$ from this formula,
so that
\begin{equation}\label{case2in:eq}
\lim_{\a\to\infty}\frac{1}{\a^{d-1}\log\a}
\biggl(\tr\bigl( \op_\a^l(b\phi_j) g_p(T(a))\bigr)
-  \a^d \GW_0\bigl(b\phi_j  g_p(a)\bigr)\biggr) = 0.
\end{equation}

\underline{End of the proof.}
Summing up over all values of $j$, it follows from \eqref{case1on:eq}
and \eqref{case2in:eq} that
\begin{align}
\lim_{\a\to\infty}\frac{1}{\a^{d-1}\log\a}
\biggl(\tr\bigl( \op_\a^l(b\phi) &\ g_p(T(a))\bigr)
-  \a^d \GW_0\bigl( g_p(a)\bigr)\notag\notag\\[0.2cm]
&\ - \a^{d-1}\log\a\  \GW_1\bigl( \GA(g_p; a)\bigr)\biggr) = 0,
\label{end:eq}
\end{align}
where we have denoted $\phi = \sum_j \phi_j$.
By Corollary \ref{product:cor},
\begin{equation*}
\| \op_\a^l (b\phi) - \op_\a^l(b) \phi\|_{\GS_1}\le C\a^{d-1}.
\end{equation*}
Furthermore, $\phi\chi_\L = \chi_\L$, and hence
$\phi g_p(T(a)) = g_p(T(a))$.
 Thus one can
remove $\phi$ from the asymptotics \eqref{end:eq} altogether.
This leads to \eqref{last_step:eq}, as required.
\end{proof}

\begin{proof}[Proof of Theorem \ref{main:thm}]
Let us deduce \eqref{main:eq} from Theorem \ref{metamain:thm}.
Let $R>0$ be a number such that
$a(\bx, \bxi) = 0$ for $|\bx|^2+ |\bxi|^2\ge R^2$.
Let $b\in\plainC\infty_0(\R^{2d})$ be a function such that
$b(\bx, \bxi) = 1$ for $|\bx|^2+ |\bxi|^2 \le 2R^2$ and
$b(\bx, \bxi) = 0$ for $|\bx|^2+ |\bxi|^2 \ge 4R^2$. Write:
\begin{equation}\label{gpa:eq}
\tr g_p(T(a)) = \tr \bigl(\op_\a^l(b) g_p(T(a))\bigr)
+ \tr \bigl(\op_\a^l(1-b) g_p(T(a))\bigr).
\end{equation}
Note that
\begin{equation*}
\op_\a^l(1-b) T(a) = \op_\a^l(1-b) \op_\a^l(a) \chi_{\L} P_{\Om, \a} \chi_{\L}
- \op_\a^l(1-b) [\op_\a^l(a), \chi_{\L} P_{\Om, \a}] P_{\Om, \a} \chi_{\L}.
\end{equation*}
By Lemmas \ref{sandwich:lem} and \ref{sandwich_dual:lem},
\begin{equation*}
\| [\op_\a^l(a), \chi_{\L} P_{\Om, \a}]\|_{\GS_1}\le C\a^{d-1}.
\end{equation*}
Furthermore, by Corollary \ref{product:cor},
\begin{equation*}
\|\op_\a^l(1-b) \op_\a^l(a) \|_{\GS_1}
= \|\op_\a^l(1-b) \op_\a^l(a)  - \op_\a^l(a(1-b))\|_{\GS_1}\le C\a^{d-1}.
\end{equation*}
Therefore the second term in \eqref{gpa:eq} gives a contribution of order $O(\a^{d-1})$.
Applying Theorem \ref{metamain:thm} to the first term in \eqref{gpa:eq}, we obtain
\eqref{main:eq}.

It remains to justify the formula \eqref{main:eq} for the operator $S(a)$ with
$a$ replaced by $\re a$ on the right-hand side. Let us show that
\begin{equation}\label{tos:eq}
\|g_p\bigl(T(\re a)\bigr) - g_p\bigl(\re T(a)\bigr)\|_{\GS_1}\le C\a^{d-1}.
\end{equation}
%Here the dependence on the symbol $a$ is included in the constant $C$.
Rewrite the difference on the left-hand side as
\[
g_p\bigl(T(\re a)\bigr) - g_p\bigl(\re T(a)\bigr)
= \sum_{k=0}^{p-1} g_k\bigl(\re T(a)\bigr)\bigl(
T(\re a) - \re T(a)
\bigr)
g_{p-1-k}\bigl(T(\re a)\bigr),
\]
so that
\begin{align}
\|g_p\bigl(T(\re a)\bigr) - &\ g_p\bigl(\re T(a)\bigr)\|_{\GS_1}\notag\\[0.2cm]
\le &\ p\bigl(\|\op(a)\|^{p-1} + \|\op(\re a)\|^{p-1}
\bigr)\ \|\op(\re a) - \re \op(a)\|_{\GS_1}.\label{re:eq}
\end{align}
The operator $\re\op(a)$ is nothing but
$(\op^l(a) + \op^r(\overline{a}))/2$, and hence, by
\eqref{quantisation_symbol:eq},
\[
\|\op(\re a) - \re \op(a)\|_{\GS_1}
 = \frac{1}{2}\|\op^l(\overline{a}) -  \op^r(\overline{a})\|_{\GS_1}
 \le C \a^{d-1}.
\]
Using Lemma \ref{general_norm:lem}, we obtain \eqref{tos:eq} from
\eqref{re:eq}. Using \eqref{main:eq} for $T(\re a)$, we now obtain
\eqref{main:eq} for the operator $S(a)$.

The proof of Theorem \ref{main:thm} is complete.
\end{proof}

%%%%%%%%%%%%%%%%%%%%%%%%%%%%%%%%%%%%%%%%%%%%%%%%%%%%%%%%%%%%%%%%%%%%%%%%%%%%%%%%%%%%

\section{Closing the asymptotics: proof of Theorems \ref{main_anal:thm}
and \ref{main_s:thm}}\label{closure:sect}

The crucial point of the proof of Theorems \ref{main_anal:thm} and
\ref{main_s:thm} is the sharp estimates \eqref{gioev0:eq} and
\eqref{gioev:eq}. These estimates were essentially derived
in \cite{G1}, Chapter 3, see also \cite{G2}. We obtain them for more general symbols,
which require some additional functional calculus considerations.
For the sake of completeness we provide proofs even for the results borrowed from
\cite{G1}, \cite{G2}.

\subsection{Non-self-adjoint case}

The inequality \eqref{LSG:eq} in the next lemma was proved in \cite{G2}.

\begin{lem}\label{LSG:lem}
 Let $A$ be a trace class operator in a separable
Hilbert space $\GH$, and let $P$ be an orthogonal projection in $\GH$.
Suppose that
\begin{equation}\label{gseries:eq}
g(z) = \sum_{m=1}^\infty \om_m z^m,
\end{equation}
is a function analytic in a disk of radius $R>\|A\|$.
Denote
\begin{align}
g^{(1)}(t) = &\ \sum_{m=2}^\infty (m-1) |\om_m |t^{m-1},\label{g1:eq}\\[0.2cm]
g^{(2)}(t) = &\ \sum_{m=2}^\infty m(m-1) |\om_m |t^{m-2}.\label{g2:eq}
\end{align}
Then
\begin{equation}\label{LSG_trace:eq}
\|g(PAP) - Pg(A)P\|_{\GS_1}\le g^{(1)}\bigl(\|A\|\bigr)
\|PA(I-P)\|_{\GS_1},
\end{equation}
and
\begin{equation}\label{LSG:eq}
\|g(PAP) - Pg(A)P\|_{\GS_1}\le \frac{g^{(2)}\bigl(\|A\|\bigr)}{2}
\|PA(I-P)\|_{\GS_2} \ \|(I-P)AP\|_{\GS_2}.
\end{equation}
\end{lem}

\begin{proof} Denote $Q = I - P$. Then for $m\ge 2$ we write
\begin{align}\label{3.12:eq}
P A^m P = &\ PA(P + Q)A^{m-1} P = P(AP)A^{m-1} P + (PA) Q A^{m-1}P \notag\\
= &\ P(AP)A(P + Q)A^{m-2}P + (PA)Q A^{m-1}P\notag\\
= &\ P(AP)^2 A^{m-2} P + (PA)^2 Q A^{m-2} P + (PA)QA^{m-1} P\notag\\
= &\ P(AP)^m + \sum_{j=1}^{m-1} (PA)^{m-j} Q A^j P\notag\\
=&\ (PAP)^m + \sum_{j=1}^{m-1} (PA)^{m-j} Q A^j P.
\end{align}
Estimate the term under the sum:
\begin{equation*}
\|(PA)^{m-j} Q A^j P\|_{\GS_1}\le \|(PA)^{m-1-j}\|\  \|PAQ\|_{\GS_1}\
\|A^jP\|
\le \|A\|^{m-1} \|PAQ\|_{\GS_1},
\end{equation*}
so that
\begin{equation*}
\|PA^mP - (PAP)^m\|_{\GS_1}\le (m-1) \|A\|^{m-1} \|PAQ\|_{\GS_1}.
\end{equation*}
Now \eqref{LSG_trace:eq} follows.

On the other hand, for $j\ge2$,
\begin{align}\label{3.13:eq}
A^j P = &\ A^{j-1} (P + Q)A P = A^{j-1} PAP + A^{j-1}QAP\notag\\
= &\ A^{j-2} (P + Q)A PAP + A^{j-1} QAP \notag\\
= &\ A^{j-2} (PA)^2 P + A^{j-2}Q(AP)^2 + A^{j-1}QAP \notag\\
= &\ A(PA)^{j-1}P + \sum_{k=1}^{j-1} A^k Q(AP)^{j-k}P.
\end{align}
Substituting \eqref{3.13:eq} in \eqref{3.12:eq} we obtain
\begin{align*}
PA^mP-(PAP)^m = &\
\sum_{j=1}^{m-1} (PA)^{m-j}QA^jP\\
= &\ \sum_{j=1}^{m-1} (PA)^{m-j} QA(PA)^{j-1}P +
\sum_{j=2}^{m-1}(PA)^{m-j}Q \sum_{k=1}^{j-1}A^k Q(AP)^{j-k}P.
\end{align*}
This entails the bound
\begin{align*}
\|PA^m P-&\ (PAP)^m\|_{\GS_1}
\le (m-1)\|A\|^{m-2}\|PAQAP\|_{\GS_1}\\
&\ + \sum_{j=2}^{m-1} \|PA\|^{m-j-1} \|PAQ\|_{\GS_2} \sum_{k=1}^{j-1}
\|A\|^k \|QAP\|_{\GS_2}\|AP\|^{j-k-1}\\
\le &\ (m-1)\|A\|^{m-2}\|PAQAP\|_{\GS_1}
+ \|A\|^{m-2} \|PAQ\|_{\GS_2} \|QAP\|_{\GS_2} \sum_{j=2}^{m-1}\sum_{k=1}^{j-1} 1\\
\le &\ \frac{m(m-1)}{2} \|A\|^{m-2}\|PAQ\|_{\GS_2} \|QAP\|_{\GS_2}.
\end{align*}
This leads to \eqref{LSG:eq}.
\end{proof}

Thus, if $P, P_1$ are two orthogonal projections, then
\begin{align}\label{LSG1:eq}
\|g(PP_1AP_1P) - &\ PP_1g(A)P_1P\|_{\GS_1}\notag\\
\le
&\ \|g(PP_1AP_1P) - Pg(P_1AP_1)P\|_{\GS_1}  + \|g(P_1AP_1) - P_1g(A)P_1\|_{\GS_1}\notag\\
\le &\ \frac{g^{(2)}\bigl(\|A\|\bigr)}{2}\
\|PP_1AP_1(I-P)\|_{\GS_2} \ \|(I-P)P_1AP_1P\|_{\GS_2}\notag\\
&\hskip 3cm + g^{(1)}(\|A\|)\ \|P_1A(I-P_1)\|_{\GS_1}.
\end{align}
Now we apply the above estimates to pseudo-differential operators.
Let $A = \op_\a^l(a)$ with a compactly supported symbol $a$, and let
$D$ be the constant introduced in Lemma \ref{multiple:lem}. In particular,
by Lemma \ref{general_norm:lem},
\begin{equation}\label{t0:eq}
\|\op_\a^l(a)\|\le D \SN^{(d+2, d+2)}(a; \ell, \rho) =: t_0,\ \
\textup{and}\ \ \ \|T(a)\|\le t_0,
\end{equation}
under the condition $\a\ell\rho\ge 1$.

\begin{lem}\label{func_calc_anal:lem}
Suppose that $a\in\BS^{(d+2, d+2)}$ is a symbol such that
\[
\supp a\subset B(\bz, \ell)\times B(\bmu, \rho),
\]
and that $\a\ell\rho\ge 1$.
Let $g$ be given by the series
\eqref{gseries:eq} with the radius of convergence $R > t_0$.
Denote
\begin{equation*}
\tilde g(t) = \sum_{m=2}^\infty (m-1)^{2(d+2)}|\om_m| t^{m-1}.
\end{equation*}
Then
\begin{equation}\label{func_calc_anal:eq}
\| g(A) - \op_\a^l(g(a))\|_{\GS_1}\le  (\a\ell\rho)^{d-1}
\SN^{(d+2, d+2)}(a; \ell, \rho) \tilde g(t_0).
\end{equation}
\end{lem}

\begin{proof}
By \eqref{multiple:eq},
\[
\|\bigl(\op_\a^l(a)\bigr)^m - \op_\a^l(a^m)\|_{\GS_1}
\le  (m-1)^{2(d+2)}
(\a\ell\rho)^{d-1} \SN^{(d+2, d+2)}(a; \ell, \rho) t_0^{m-1},
\]
which leads to \eqref{func_calc_anal:eq}.
\end{proof}

In the next lemma we still need to remember the dependence of the parameter
$t_0$ on the symbol $a$, but we
do not specify the dependence of the \textsl{constants} in the estimates
on the symbol $a$ or the parameters $\ell, \rho$.

\begin{lem}\label{gioev0:lem}
Let the symbol $a$ be as in Lemma \ref{func_calc_anal:lem}.
Let $g$ be given by the series
\eqref{gseries:eq} with the radius of convergence $R > t_0$.
Then
\begin{equation}\label{gioev0:eq}
\limsup_{\a\to\infty} \frac{1}{\a^{d-1} \log\a}
\biggl|\tr g(T(a)) - \GW_0(g(a); \L, \Om)\biggr|\le C g^{(2)}(t_0).
\end{equation}
\end{lem}

\begin{proof}
Using an appropriate partition of unity,
Lemma \ref{omtoom0:lem} and estimates
\eqref{sandwich1:eq}, \eqref{sandwich1_dual:eq},
we can easily derive from the estimates
\eqref{sandwich1_dual:eq}
and  \eqref{HSestim:eq} that
\begin{align}
\|P_{\Om, \a}\op_\a(a) (I-P_{\Om, \a})\|_{\GS_1}\le &\ C\a^{d-1},\label{gi:eq}\\
\|\chi_{\L} P_{\Om, \a} \op_\a(a) P_{\Om, \a} (1-\chi_{\L})\|_{\GS_2}^2
\le &\ C \a^{d-1}\log\a,\notag
\end{align}
where $\op_\a(a)$ denotes any of the operators $\op_\a^l(a), \op_\a^r(a)$.
Moreover, by \eqref{t0:eq},
$\|\op_\a^l(a)\|\le t_0$.
Therefore it follows from \eqref{LSG1:eq} that
\begin{equation*}
\| g(T(a)) - \chi_{\L}P_{\Om, \a}g(\op_\a^l(a)))P_{\Om, \a}\chi_{\L}\|_{\GS_1}
\le C \a^{d-1}
\bigl[\log\a\ g^{(2)}(t_0) + g^{(1)}(t_0)\bigr].
\end{equation*}
Together with \eqref{func_calc_anal:eq} this gives
\begin{equation*}
\limsup_{\a\to\infty} \frac{1}{\a^{d-1} \log\a}
\|g(T(a)) - T(g(a))\|_{\GS_1}\le C g^{(2)}(t_0).
\end{equation*}
In order to find the trace of $T(g(a))$ we note
that in view of \eqref{gi:eq},
\[
\|(I-P_{\Om, \a})\op^l_\a(g(a))P_{\Om, \a}\|_{\GS_1}\le C\a^{d-1},
\]
so that
\[
\tr T(g(a)) = \tr \chi_{\L} \op^l_\a(g(a)) P_{\Om, \a}\chi_{\L} + O(\a^{d-1}).
\]
Writing out the kernel of the operator on the right-hand side and integrating,
we find that its trace equals $\a^d \GW_0(g(a); \L, \Om)$.
The proof is complete.
\end{proof}

\subsection{Proof of Theorem \ref{main_anal:thm}}
Represent the function $g$ in the form
\[
g(z) = g_p(z) + r_p(z), \ g_p(z) = \sum_{m=1}^p \om_m z^m,
r_p(z) = \sum_{m=p+1}^\infty \om_m z^m,
\]
so that
\[
\tr g(T(a)) = \tr g_p(T(a)) + \tr r_p(T(a)).
\]
For any $\vare >0$ one can find a number $p$ such that
\[
r_p^{(1)}(t_0) + r_p^{(2)}(t_0) < \vare,
\]
see definitions \eqref{g1:eq} and \eqref{g2:eq}.
By \eqref{GA_est_anal:eq},
\begin{equation*}
|\GA(r_p; b)|\le |b| r_p^{(1)}(|b|),
\end{equation*}
for any number $b\in\mathbb C, |b|<R$. Thus by definition of $\GW_1$ and $\GA$,
\[
|\GW_1(\GA(r_p, a); \p\L, \p\Om)|\le   C r_p^{(1)}(t_0) \le C\vare,
\]
so that according to Lemma \ref{gioev0:lem},
\begin{align*}
\limsup_{\a\to\infty}
\frac{1}{\a^{d-1}\log\a}\bigl|\tr r_p\bigl(T(a)\bigr)
- &\ \a^d \GW_0(r_p(a); \L, \Om)\\[0.2cm]
&\ - \a^{d-1} \log\a\  \GW_1(\GA(r_p; a); \p\L, \p\Om)
\bigr|\\[0.2cm]
\le &\ C \bigl(r_p^{(1)}(t_0) + r_p^{(2)}(t_0)\bigr) < C \vare.
\end{align*}
Adding this formula and \eqref{main:eq} for the trace $\tr g_p(T(a))$,
we obtain
\begin{align*}
\limsup_{\a\to\infty}
\frac{1}{\a^{d-1}\log\a}\bigl|\tr g\bigl(T(a)\bigr)
- &\ \a^d \GW_0(g(a); \L, \Om)\\[0.2cm]
&\ - \a^{d-1} \log\a\  \GW_1(\GA(g; a); \p\L, \p\Om)
\bigr|
\le C \vare.
\end{align*}
Since the parameter $\vare$ is arbitrary,
this proves \eqref{main_anal:eq}.
\qed

\subsection{Self-adjoint case}
A central role is played by the following abstract result, established in
\cite{LS} (see also \cite{LapSaf}), which we have slightly rephrased.

\begin{prop}\label{LapSaf:prop}
Let $A$ be a self-adjoint bounded operator in a separable Hilbert space $\GH$, and
let $P$ be an orthogonal projection in $\GH$ such that $PA\in\GS_2$.
Then for any function $\psi\in \plainC2(I)$, $I = [-\|A\|, \|A\|]$,
such that $\psi(0) = 0$,
the operators $P\psi(A)P$ and $\psi(PAP)$ are trace class, and
\begin{equation*}
\biggl|\tr \bigl(P\psi(A)P - \psi(PAP)\bigr)\biggr|
\le \frac{1}{2}\|\psi''\|_{\plainL\infty}(I) \|PA(1-P)\|^2_{\GS_2}.
\end{equation*}
\end{prop}

We also need a less elegant, but still useful estimate. The conditions on the
operators in the next lemma are certainly not optimal,
but they suffice for our purposes.

\begin{lem}\label{prim:lem}
Suppose that $A$ is a self-adjoint trace class operator in a Hilbert space $\GH$, and let
$P$ be an orthogonal projection in $\GH$. Then for any function $\psi$ such that
$t\hat\psi \in \plainL1(\R)$, we have
\begin{equation}\label{prim:eq}
\|P\psi(A)P - P\psi(PAP)P\|_{\GS_1}\le \frac{1}{\sqrt{2\pi}}\|PA(I-P)\|_{\GS_1}
\int |t|\ |\hat\psi(t)| dt.
\end{equation}

\end{lem}

\begin{proof}
For an arbitrary self-adjoint operator $B$ denote
\[
U(t) = U(t; B) = e^{iBt}, t\in\R,
\]
so that
\[
i \p_t U(t) + B U(t) = 0,\ U(0) = I.
\]
Compare the operators
\[
W_1(t) = P U(t; PAP)P \ \ \textup{and}\ \
W_2(t) = P U(t; A)P.
\]
It is clear, that
\begin{gather*}
i\p_t W_1(t) + PAP W_1(t) = 0,\\[0.2cm]
i\p_t W_2(t) + PAP W_2(t) + PA(I-P) U(t; A)P = 0.
\end{gather*}
Therefore $W(t):=W_2(t)-W_1(t)$ satisfies the equation
\begin{equation*}
i\p_t W(t) + PAP W(t) = - PA(I-P) U(t; A)P.
\end{equation*}
Thus
\begin{equation*}
W(t) = i\int_0^t U(t-s; PAP) PA(I-P) U(s; A)P ds,
\end{equation*}
and hence
\begin{equation}\label{prop_diff:eq}
\|P\bigl(U(t, A) - U(t; PAP)\bigr)P\|_{\GS_1}\le |t| \|PA(I-P) \|_{\GS_1},
\ t\in\R.
\end{equation}
 For any self-adjoint $B$ we have
 \[
 \psi(B) = \frac{1}{\sqrt{2\pi}}\int U(t; B) \hat\psi(t) dt,
 \]
 where $\hat\psi$ is the Fourier transform of $\psi$. Thus
 in view of \eqref{prop_diff:eq},
\begin{equation*}
\|P\psi(A)P - P\psi(PAP)P\|_{\GS_1}
\le \frac{1}{\sqrt{2\pi}}\|PA(I-P)\|_{\GS_1}
\int |t|\ |\hat\psi(t)| dt,
\end{equation*}
which is the required bound.
\end{proof}

Below we recall an elementary result of functional calculus for
pseudo-differential operators. For a function $\psi: \R\to\mathbb C$
we denote
\[
a_{\psi}(\bx, \by, \bxi)
 = \psi\biggl(\frac{a(\bx, \bxi)+\overline{a(\by, \bxi)}}{2}\biggr), \
 A_{\psi} = \op_\a^a(a_{\psi}),
\]
so that
\[
a_1(\bx, \by, \bxi) = \frac{a(\bx, \bxi)+\overline{a(\by, \bxi)}}{2},\
A_1 = \re \op_\a^l(a).
\]

\begin{lem}\label{func_calc:lem}
Suppose that $a\in\BS^{(d+2, d+2)}$ be a symbol such that
\[
\supp a\subset B(\bz, \ell)\times B(\bmu, \rho).
\]
Let $\psi\in \plainC\infty_0(\R)$. Then
\begin{equation}\label{func_calc:eq}
\| \psi(A_1) - A_{\psi}\|_{\GS_1}\le C(\a\ell\rho)^{d-1}
\bigr(\SN^{(d+2, d+2)}(a; \ell, \rho)\bigr)^{2d+3},
\end{equation}
with a constant depending only on the function $\psi$.
\end{lem}

\begin{proof}
Let
\[
w_t = w_t(\bx, \by, \bxi) = e^{i  t a_1(\bx, \by, \bxi)},\ W(t) = \op_\a^a(w_t),
\]
and $U(t) = U(t; A_1) = e^{i t A_1}$.
Let us show that
\begin{equation}\label{ampl_prop:eq}
\|U(t) - W(t)\|_{\GS_1}\le C(\a\ell\rho)^{d-1} t^{2d+5}
\bigr(\SN^{(d+2, d+2)}(a; \ell, \rho)\bigr)^{2d+3}.
\end{equation}
The operator $W(t)$ satisfies the equation
\[
i\p_t W(t) + \op_\a^a(a_1 w_t) = 0,
\]
so that
\[
i\p_t W(t) + \op_\a^a(a_1) W(t) = M(t), \ M(t) = \op_\a^a(a_1) W(t) - \op_\a^a(a_1 w_t).
\]
Since $A_1 = \op_\a^a(a_1)$, the difference
$E(t):=W(t) - U(t)$ satisfies the equation
\[
i\p_t E(t) + A_1 E(t) = M(t), E(0) = 0.
\]
Thus
 \begin{equation*}
E(t) = - i\int_0^t U(t-s; A_1) M(t) ds.
\end{equation*}
By Lemma \ref{quantisation_trace:lem} and Corollary \ref{product:cor},
\[
\|M(t)\|_{\GS_1}\le C(\a\ell\rho)^{d-1} t^{2d+4}
\bigr(\SN^{(d+2, d+2)}(a; \ell, \rho)\bigr)^{2d+3},
\]
so that \eqref{ampl_prop:eq} holds.

Since
$$
\psi(A_1) = \frac{1}{\sqrt{2\pi}}
\int U(t; A_1)\hat\psi(t) dt,
$$
it follows from \eqref{ampl_prop:eq} that
\begin{align*}
 \biggl\|\psi(A_1)  - &\ \frac{1}{\sqrt{2\pi}}
 \int W(t)\hat\psi(t) dt\biggr\|_{\GS_1}\\[0.2cm]
 \le &\
 C(\a\ell\rho)^{d-1} \int |t|^{2d+5} |\hat\psi(t)| dt
 \bigr(\SN^{(d+2, d+2)}(a; \ell, \rho)\bigr)^{2d+3}.
\end{align*}
By definition of $w_t$,
\[
\frac{1}{\sqrt{2\pi}}
 \int W(t)\hat\psi(t) dt = \op_\a^a(a_{\psi}).
\]
This completes the proof.
\end{proof}

Let us apply the estimates in Proposition \ref{LapSaf:prop},
and Lemmas \ref{prim:lem} and \ref{func_calc:lem} to the operator
$S(a) = \chi_{\L} P_{\Om, \a} \re \op^l_\a(a) P_{\Om, \a}\chi_{\L}$.
By Lemma \ref{general_norm:lem} the operator $A_1 = \re \op_\a^l(a)$ is bounded
by some constant $t_0 >0$ uniformly in $\a \ge 1$, so that $\| S(a)\|\le t_0$
for all $\a\ge 1$ as well.
In the lemma below the dependence on the symbol $a$ is included
in the constants.

\begin{lem}\label{gioev:lem}
Let $\L$ and $\Om$ be bounded $\plainC1$-domains, and let
$a=a(\bx, \bxi)$ be a compactly supported symbol from $\BS^{(d+2, d+2)}$.
If $\psi\in\plainC\infty$
and $\psi(0) = 0$, then
\begin{equation}\label{gioev:eq}
\limsup_{\a\to\infty}
\frac{1}{\a^{d-1}\log\a}
\bigl|\tr \psi\bigl(S(a)\bigr) - \a^d \GW_0(\psi(\re a); \L, \Om)\bigr|
\le C \|\psi''\|_{\plainL\infty(-t_0, t_0)}.
\end{equation}
\end{lem}

\begin{proof}
As in the previous lemma, we use the notation
$A_1 = \re\op_\a^l(a)$.
Using an appropriate partition of unity,
Lemma \ref{omtoom0:lem} and estimates
\eqref{sandwich1:eq}, \eqref{sandwich1_dual:eq},
we can easily derive from the estimates
\eqref{sandwich1_dual:eq}
and  \eqref{HSestim:eq} that
\begin{align*}
\|P_{\Om, \a}A_1 (I-P_{\Om, \a})\|_{\GS_1}\le &\ C\a^{d-1},\\
\|\chi_{\L} P_{\Om, \a} A_1 P_{\Om, \a} (1-\chi_{\L})\|_{\GS_2}^2
\le &\ C \a^{d-1}\log\a.
\end{align*}
Remembering that $\|S(a)\|\le t_0$,
together with Proposition \ref{LapSaf:prop} the second estimate gives:
\[
\biggl|\tr\psi\bigl(\chi_{\L}P_{\Om, \a}A_1 P_{\Om, \a}\chi_{\L}\bigr)
- \tr \biggl(\chi_{\L}\psi(P_{\Om, \a}A_1 P_{\Om, \a})\chi_{\L}\biggr)\biggr|
\le C\|\psi''\|_{\plainL\infty(-t_0, t_0)}\a^{d-1}\log\a.
\]
Now, according to Lemma \ref{func_calc:lem} and
Lemma \ref{prim:lem} with $A = A_1=\re \op_\a^l(a)$, $P = P_{\Om, \a}$,
we have
\[
\|\chi_{\L}\psi(P_{\Om, \a} A_1  P_{\Om, \a})\chi_{\L}
- \chi_{\L} P_{\Om, \a} A_{\psi} P_{\Om, \a} \chi_{\L}\|_{\GS_1}
\le C\a^{d-1},\ A_{\psi} = \op_\a^a(a_{\psi}).
\]
The last two estimates together imply that
\[
\bigl|\tr\psi\bigl(\chi_{\L}P_{\Om, \a}A_1 P_{\Om, \a}\chi_{\L}\bigr)
-\tr \chi_{\L} P_{\Om, \a} A_{\psi} P_{\Om, \a} \chi_{\L}\bigr|
\le C\|\psi''\|_{\plainL\infty(-t_0, t_0)} \a^{d-1} \log\a + C\a^{d-1}.
\]
Arguing as in the proof of Lemma \ref{gioev0:lem}, we find that
the trace of the operator
\begin{equation*}
\chi_{\L} P_{\Om, \a} A_{\psi} P_{\Om, \a}\chi_{\L},
\end{equation*}
equals $\a^d \GW_0(\psi(\re a); \L, \Om) + O(\a^{d-1})$.
Now \eqref{gioev:eq} follows.
\end{proof}

\subsection{Proof of Theorem \ref{main_s:thm}}
Represent the function $g$ in the form
\[
g(t) = \b t + \psi(t), \b = g'(0),
\]
with a function $\psi\in\plainC\infty$ such
that $\psi(0) = \psi'(0) = 0$. Therefore
\[
g(t) = \b t + \int_0^{t} \int_{0}^{\tau} \psi''(s) ds d\tau.
\]
For any $\vare >0$ one can find a polynomial $z=z(t)$
of a sufficiently large degree $p-2$ such that
\[
\max_{|t|\le t_0}|z(t) - \psi''(t)|\le \vare.
\]
Therefore,
\[
g(t) = g_p(t) + \phi(t),
\]
where
\[
g_p(t) = \b t + \int_0^t\int_0^\tau z(s) ds d\tau
\]
is a polynomial of degree $p$, and
\[
\phi(t) = \int_0^{t} \int_{0}^{\tau} \bigl(\psi''(s) - z(s)\bigr)ds d\tau.
\]
Thus
\[
\tr g(S(a)) = \tr g_p(S(a)) + \tr \phi(S(a)).
\]
By construction,
$\phi(0) = \phi'(0) = 0$ and
\[
\max_{|t|\le t_0}|\phi''(t)| < \vare, \ \ \max_{|t|\le t_0}|\phi'(t)|<\vare t_0,
\ \ \ \max_{|t|\le t_0}|\phi(t)|<\vare t_0^2.
\]
By \eqref{GA_est_smooth:eq},
\begin{equation*}
|\GA(\phi; b)|\le 4 |b|\max_{|t|\le |b|} |\phi'(t)|,
\end{equation*}
for any number $b\in\R$. Thus by definition of $\GW_1$ and $\GA$,
\[
|\GW_1(\GA(\phi, \re a); \p\L, \p\Om)|\le C\vare,
\]
so that according to Lemma \ref{gioev:lem},
\begin{align*}
\limsup_{\a\to\infty}
\frac{1}{\a^{d-1}\log\a}\bigl|\tr \phi\bigl(S(a)\bigr)
- &\ \a^d \GW_0(\phi(\re a); \L, \Om)\\[0.2cm]
&\ - \a^{d-1} \log\a\  \GW_1(\GA(\phi; \re a); \p\L, \p\Om)
\bigr|
\le C \vare.
\end{align*}
Adding this formula and \eqref{main:eq} for the trace $\tr g_p(S(a))$,
we obtain
\begin{align*}
\limsup_{\a\to\infty}
\frac{1}{\a^{d-1}\log\a}\bigl|\tr g\bigl(S(a)\bigr)
- &\ \a^d \GW_0(g(\re a); \L, \Om)\\[0.2cm]
&\ - \a^{d-1} \log\a\  \GW_1(\GA(g; \re a); \p\L, \p\Om)
\bigr|
\le C \vare.
\end{align*}
Since the parameter $\vare$ is arbitrary,
this proves \eqref{main_s:eq}.
\qed

%%%%%%%%%%%%%%%%%%%%%%%%%%%%%%%%%%%%%%%%%%%%%%%%%%%%%%%%%%%%%%%%%%%%%%%%%%%%%%%%%%%%

\section{Appendix 1: A lemma by H. Widom}
\label{widom90:sect}

In this Appendix we establish a variant of Lemma A from H. Widom's paper
\cite{Widom2}.

\subsection{Function $m_{\d}$} In this section we use the notation
\eqref{mempty:eq} and \eqref{rhok:eq} introduced in Subsect.
\ref{multiple:subsect} for an arbitrary finite set
$\SX\subset (-2\rho, 2\rho)$, where $\rho >0$
is a fixed number. Let $\CC^{(n)}_{\rho}$ be the $n$-dimensional cube defined
in \eqref{cube:eq}.
Let $\Phi\in\plainC3\bigl(\overline{\CC^{(d-1)}_{\rho}}\bigr)$ be a real-valued function.
For each $\hat\bx\in \CC^{(d-1)}_{\rho}$ and each $l = 1, 2, \dots, d$ define the sets
$\SX_l, \SX_l^{(+)}, \SX_l^{(-)}$ in the following way.
First, we set
\begin{equation}\label{sxl:eq}
\SX_l(\hat\bx) = \SX_l(\hat\bx; \Phi)
= \{x_d \in (-2\rho, 2\rho): \Phi(\overc \bx) = x_l\}.
\end{equation}
Using the Area Formula (see \cite{EG}, Theorem 1, Section 3.3.2), one can show that
for almost all $\hat\bx$ the set $\SX_l(\hat\bx)$ is finite.
Thus the function
$m_\d(\SX_l(\hat\bx))$ (see \eqref{rhok:eq}) is well-defined a.a. $\hat\bx$.
Denote
\[
\SL^{(\pm)}_l(\hat\bx) = \{x_d\in (-2\rho, 2\rho): \pm x_l > \pm\Phi(\overc \bx)\}.
\]
For each $\hat\bx\in\R^{d-1}$ this set is either empty or it is
a countable union of open intervals.
Denote by $\SX_l^{(\pm)}(\hat\bx) = \SX_l^{(\pm)}(\hat\bx; \Phi)$ the set of all
endpoints of these intervals which are strictly inside $(-2\rho, 2\rho)$.
Clearly, $\SX_l^{(\pm)}(\hat\bx)\subset \SX_l(\hat\bx)$.

\begin{thm}\label{bigdim:thm}
Let $\Phi\in\plainC3\bigl(\overline{\CC^{(d-1)}_{\rho}}\bigr)$ with some $\rho >0$.
Then for any $l = 1, 2, \dots, d$ the following
statements hold:
\begin{enumerate}
\item
For almost all $\hat\bx\in \CC^{(d-1)}_{\rho}$:
\begin{enumerate}
\item
the set $\SX_l(\hat\bx)$ is finite,
\item
$\SX_l(\hat\bx) = \SX_l^{(\pm)}(\hat\bx)$,
\item
the disjoint open intervals forming $\SL^{\pm}_l(\hat\bx)$ have \underline{distinct}
endpoints.
\end{enumerate}
\item
For any $\d\in (0, 2)$ the
function $m_{\d}\bigl(\SX_l(\hat\bx)\bigr)$ satisfies the bound
\begin{equation}\label{reciprocal:eq}
\int_{\hat\bx\in \CC^{(d-1)}_{\rho}} m_\d(\SX_l(\hat\bx)) d\hat\bx\le C \rho^{d-1-\d}
\bigl(1+\rho\|\nabla^2\Phi\|_{\plainL\infty} +
\rho^2\|\nabla^3\Phi\|_{\plainL\infty}\bigr).
\end{equation}
The constant
$C$ depends only on $\d$ and dimension $d$.
\end{enumerate}
\end{thm}

\begin{rem}\label{area:rem}
As indicated above, we can immediately see that the set $\SX_l(\hat\bx)$ is finite
a.e. $\hat\bx$. Indeed, if $ l = d$, then obviously this set consists of one point.
Suppose that $l < d$.
Denote by $\Xi:\CC_\rho^{d-1}\to\CC_\rho^{d-1}$ the mapping
\begin{equation*}
\Xi(\overset{\circ}\bx) =
\bigl(x_1, \dots, x_{l-1}, \Phi(\overc\bx), x_{l+1}, \dots, x_{d-1}\bigr).
\end{equation*}
Then by the area formula (see, e.g. Theorem 1, p.96 in \cite{EG}),
\begin{equation*}
\underset{\CC_\rho^{d-1}}\int J_{\Xi}(\overset{\circ}\bx)
 d\overset{\circ}\bx
 = \underset{\CC_\rho^{d-1}}\int \sum_{\overset{\circ}\bx\in \Xi^{-1}(\hat\bx)}
1 \  d\hat\bx,
\end{equation*}
where $J_{\Xi}$ is the Jacobian of the map $\Xi$, i.e. $|\p_{x_d} \Phi|$.
The integral on the left-hand side is finite, and hence the integrand on the right-hand side
is finite a.e. $\hat\bx$, which implies that $\SX_l(\hat\bx)$ is finite a.e. $\hat\bx$.
\end{rem}

In the remaining proof of Theorem \ref{bigdim:thm}
we rely on H. Widom's paper \cite{Widom2}, where the above theorem
was proved for $\d = 1$. We begin with the study of the $2$-dimensional case.

\subsection{Special case $d = 2$}  Let $\phi\in \plainC3([-2\rho, 2\rho])$.
The definition of $\SX_l(s; \phi)$ takes the form
\begin{gather*}
\SX_1(s) = \SX_1(s; \phi) = \{t\in (-2\rho, 2\rho): \phi(t) = s\};\\[0.2cm]
\SX_2(s) = \SX_2(s; \phi) = \{t\in (-2\rho, 2\rho): \phi(s) = t\}.
\end{gather*}
The lemma below is a variant of Sublemma 1 from \cite{Widom2}.

\begin{lem}\label{widom2:lem}
Let $\phi$ be as defined above. Then for $l = 1, 2$ the following
statements hold:
\begin{enumerate}
\item
For almost all $s\in (-2\rho, 2\rho)$:
\begin{enumerate}
\item
the set $\SX_l(s)$ is finite,
\item
$\SX_l(s) = \SX_l^{(\pm)}(s)$,
\item
the disjoint open intervals forming $\SL^{(\pm)}_l(s)$ have  distinct endpoints.
\end{enumerate}
\item
For any $\d\in (0, 2)$ the function $m_{\d}\bigl(\SX_l(s)\bigr)$ satisfies the
bound
\begin{equation*}
\int_{-2\rho}^{2\rho} m_{\d}\bigl(\SX_l(s)\bigr) ds  \le
\begin{cases}
C_\d \rho^{1-\d}(1 + \rho\|\phi''\|_{\plainL\infty}),\ \ 0 < \d \le 1;\\[0.4cm]
C_\d \rho^{1-\d}(1 + \rho\|\phi''\|_{\plainL\infty}
+ \rho^{2}\|\phi'''\|_{\plainL\infty}),\ \ 1 < \d <2.
\end{cases}
\end{equation*}
The constant $C$ does not depend on the function $\phi$ or the parameter $\rho$.
\end{enumerate}
\end{lem}

\begin{proof} For $l=2$ the Lemma follows trivially. In particular,
$\#\SX_2(s)\le 1$ for all $s\in (-2\rho, 2\rho)$,
so that $m_\d(X_2(s))\le \rho^{-\d}$  (see definition \eqref{rhok:eq}),
and hence the required inequality follows immediately.

Let $l=1$, and let
$s\in (-2\rho, 2\rho)$ be a non-critical point of $\phi$, i.e. if $\phi(x) = s$, then
$\phi'(x)\not = 0$.
For any such $s$ the sets $\SX_1^{(+)}(s), \SX_1^{(-)}(s)$
trivially coincide with $\SX_1(s)$. Also, if the set $\SL_1^{(\pm)}(s)$
is not empty, then the constituent open intervals do not have common endpoints.
By Sard's Theorem the set of non-critical points
is a set of full measure, so that in combination
with Remark \ref{area:rem} this proves Part (1) of the Lemma.

Proof of Part (2). Decompose the open set
\[
E = \{x\in (-2\rho, 2\rho): \phi'(x)\not = 0\},
\]
into the union of disjoint open intervals. Let $s\in (-2\rho, 2\rho)$
be a non-critical point
of $\phi$, and denote for brevity $m (s)= m_\d\bigl(\SX_1(s)\bigr)$.
 Denote by
$\CL_{\pm}(s)$ the subset of those open intervals forming $\SL^{(\pm)}_1(s)$
which do not have
$-2\rho$ or $2\rho$ as their endpoints, and denote
\[
m^{(\pm)}(s) =
\begin{cases}
(2\rho)^{-\d},\  \ \textup{if}\ \ \CL_{\pm}(s) = \varnothing;\\[0.2cm]
\underset{J\in\CL_{\pm}(s)}\sum |J|^{-\d},\ \ \textup{otherwise}.
\end{cases}
\]
Then it is clear that
\begin{equation}\label{plusminus:eq}
m(s)\le m^{(+)}(s) + m^{(-)}(s).
\end{equation}
Estimate separately $m^{(+)}(s)$ and $m^{(-)}(s)$.

Consider an interval $J\in \CL_+(s)$.
Since $s$ is a non-critical value of $\phi$, the left endpoint of $J$
falls in some interval $K\subset E$,
on which $\phi'<0$,  and
the right endpoint falls in some interval $I\subset E$,
on which $\phi'>0$.  Thus $J=(x_K(s), x_I(s))$,
where $x_K(s)$ and $x_I(s)$ are unique solutions of the equation
$\phi(x) = s$ on the intervals $K$ and $I$ respectively.
Writing
\[
I = (\b_I^{(-)}, \b_I^{(+)}),
\]
we immediately conclude that
\[
|J|\ge x_I(s) - \b_I^{(-)},
\]
and hence
\begin{equation}\label{mr:eq}
m^{(+)}(s)\le \sum_I \bigl(x_I(s)-\b_I^{(-)}\bigr)^{-\d},
\end{equation}
where the summation is taken over all intervals $I\subset E$ such that
$\phi'(x) >0, x\in I,$ and $s\in \phi(I)$.
Remembering the set of critical points of $\phi$
has measure zero,  we can use
\eqref{mr:eq} to estimate:
\[
\int_{-2\rho}^{2\rho} m^{(+)}(s) ds\le
C\rho^{1-\d} + \sum_I\int_{\phi(I)} \bigl(x_I(s)-\b_I^{(-)}\bigr)^{-\d}ds,
\]
where the summation is taken over all intervals $I\subset E$
on which $\phi'(x) >0$.
Estimate the integral for each $I$ individually, denoting for brevity
$\b^{(\pm)} = \b_I^{(\pm)}$.

\underline{Case 1: $0<\d < 1$.} Write:
\begin{equation*}
\int_{\phi(I)} \bigl(x_I(s)-\b^{(-)}\bigr)^{-\d}ds
=  \int_{\b^{(-)}}^{\b^{(+)}}
\phi'(x) (x-\b^{(-)})^{-\d} dx.
\end{equation*}
Since $\phi'(\b^{(-)}) = 0$, we have
\[
|\phi'(x)|\le \|\phi''\|_{\plainL\infty}|I|,\ x\in I,
\]
and hence the integral is bounded from above by
\[
\|\phi''\|_{\plainL\infty} |I|^{2-\d} \int_0^1 t^{-\d}dt
\le C \|\phi''\|_{\plainL\infty} |I|^{2-\d}.
\]
Therefore
\begin{align}\label{case1:eq}
\int_{-2\rho}^{2\rho} m^{(+)}(s) ds\le &\
C \rho^{1-\d}+ C\|\phi''\|_{\plainL\infty}\sum_{I\subset E} |I|^{2-\d}\notag\\[0.2cm]
\le &\ C\rho^{1-\d} (1+ \rho\|\phi''\|_{\plainL\infty}),
\end{align}
where we have used that $\sum_{I\subset E} |I| \le 4\rho$.

\underline{ Case 2: $\d = 1$.} Write:
\begin{equation*}
\int_{\phi(I)} \bigl(x_I(s)-\b^{(-)}\bigr)^{-1}ds
= \int_{\b^{(-)}}^{\b^{(+)}}
\phi'(x) (x-\b^{(-)})^{-1} dx
= \int_{\b^{(-)}}^{\b^{(+)}} \phi''(x) \log\frac{\b^{(+)}-\b^{(-)}}{x-\b^{(-)}}dx,
\end{equation*}
where we have used the fact that $\phi'(\b^{(-)}) = 0$. The last integral
does not exceed
\[
\|\phi''\|_{\plainL\infty}(\b^{(+)} - \b^{(-)})
\int_1^1 \log \frac{1}{t} dt\le C\|\phi''\|_{\plainL\infty}|I|,
\]
so that
\begin{equation}\label{case2:eq}
\int_{-2\rho}^{2\rho} m^{(+)}(s) ds\le C
+ C\|\phi''\|_{\plainL\infty}\sum_{I\subset E} |I|
\le C(1+\rho\|\phi''\|_{\plainL\infty}).
\end{equation}

\underline{Case 3: $1<\d <2$.} Write:
\begin{align}\label{dash:eq}
\int_{\phi(I)} \bigl(x_I(s)-\b^{(-)}\bigr)^{-\d}ds
= &\ \int_{\b^{(-)}}^{\b^{(+)}}
\phi'(x) (x-\b^{(-)})^{-\d} dx\notag\\[0.2cm]
= &\ \frac{1}{\d-1}\int_{\b^{(-)}}^{\b^{(+)}} \phi''(x)
\bigl[(x-\b^{(-)})^{1-\d} - (\b^{(+)}-\b^{(-)})^{1-\d}\bigr]dx.
\end{align}
Here we have used again the fact that $\phi'(\b^{(-)}) = 0$.
The interval $I$ can be of one of the following two types:
\begin{itemize}
\item $\b^{(+)} = 2\rho$, in which case $E$ contains only one interval  of this type;
\item $\b^{(+)}$ is a a critical point of $\phi$, i.e. $\phi'(\b^{(+)}) = 0$.
\end{itemize}
If $\b^{(+)} = 2\rho$, then the integral \eqref{dash:eq} does not exceed
\[
C_\d \|\phi''\|_{\plainL\infty} (\b^{(+)} - \b^{(-)})^{2-\d}\int_0^1(t^{1-\d}-1)dt
\le C_\d \|\phi''\|_{\plainL\infty} \rho^{2-\d}.
\]
If $\phi'(\b^{(+)}) = 0$, then there exists a point $x_0\in I$ such that
$\phi''(x_0) = 0$, so that
\[
|\phi''(x)|\le \|\phi'''\|_{\plainL\infty} |I|, \ x\in I.
\]
Thus the integral \eqref{dash:eq} does not exceed
\[
C\|\phi'''\|_{\plainL\infty} |I|
\int_{\b^{(-)}}^{\b^{(+)}} (x-\b^{(-)})^{1-\d} dx
\le C_{\d}
\|\phi'''\|_{\plainL\infty} |I|^{3-\d}.
\]
Therefore
\begin{align}\label{case3:eq}
\int_{-2\rho}^{2\rho} m^{(+)}(s) ds\le &\
C \rho^{1-\d}+ C\rho^{2-\d}\|\phi''\|_{\plainL\infty}
+ C \|\phi'''\|_{\plainL\infty}\sum_{I\subset E} |I|^{3-\d}\notag\\[0.2cm]
\le &\ C\rho^{1-\d} (1+ \rho\|\phi''\|_{\plainL\infty}
+ \rho^2 \|\phi'''\|).
\end{align}
Estimates \eqref{case1:eq}, \eqref{case2:eq} and \eqref{case3:eq} for $m^{(-)}(s)$
are proved in the same way. A reference to \eqref{plusminus:eq}
completes the proof.
\end{proof}

\subsection{Proof of Theorem \ref{bigdim:thm}}
If $l = d$, then
the Theorem follows trivially. In particular,
$\#\SX_d(\hat\bx)\le 1$
for all $\hat\bx\in \CC^{(d-1)}_{\rho}$,
so that $m_\d(\SX_d(\hat\bx))\le \rho^{-\d}$,
and hence
the required inequality \eqref{reciprocal:eq} follows immediately.

Suppose that $l \not= d$. In this case for each
$\tilde\bx = (x_1, x_2, \dots, x_{l-1}, x_{l+1}, \dots, x_{d-1})$ define
the auxiliary function
\[
\phi(t) = \phi_{\tilde\bx}(t) = \left.\Phi(\overc\bx)\right|_{x_d = t},
\]
so that $\SX_l(\hat\bx; \Phi) = \SX_1(x_l; \phi_{\tilde\bx})$,
and $\SX^{(\pm)}_l(\hat\bx; \Phi) = \SX^{(\pm)}_1(x_l; \phi_{\tilde\bx})$.
Therefore Part (1) of the Theorem follows from Part (1) of Lemma \ref{widom2:lem}.
Moreover,
\[
\int_{\CC^{(d-1)}_{\rho}} m_\d\bigl(\SX_l(\hat\bx; \Phi)\bigr)d\hat\bx
= \int_{\CC^{(d-2)}_{\rho}} \int_{-2\rho}^{2\rho}
m_\d\bigl(\SX_1(s; \phi_{\tilde\bx})\bigr) ds d\tilde\bx.
\]
The estimate \eqref{reciprocal:eq} follows now from Part (2) of Lemma
\ref{widom2:lem}.
\qed

%%%%%%%%%%%%%%%%%%%%%%%%%%%%%%%%%%%%%%%%%%%%%%%%%%%%%%%%%%%%%%%%%%%%%%%%%%%%%%%%%%%%

\section{Appendix 2: change of variables}

In this section we provide some elementary, but useful information from the
multi-variable calculus.

\subsection{Change of variables: surfaces}
First we prove a few lemmas describing surfaces in different coordinates.

\begin{lem}\label{new_coordinates1:lem}
Let $S\subset\R^d$ be a set.
Suppose that in a neighbourhood of a point $\bw\in S$,
the set $S$ is locally described by the function $\Phi$, i.e.
\begin{equation}\label{E1:eq}
S\cap B(\bw, R) = \{\bx: x_d = \Phi(\hat\bx)\}\cap B(\bw, R),
\end{equation}
for some $R>0$,
where $\hat\bx = (x_1, x_2, \dots, x_{d-1})$,
and $\Phi\in\plainC{m}(\R^{d-1}), m\ge 1$.
Then there exists a number $R_1$,
an orthogonal transformation $\boldO$,
and a function $F\in\plainC{m}(\R^{d-1})$,
such that for $\boldE = (\boldO, \bw)\in E(d)$,
\begin{gather*}
\boldE^{-1} S\cap B(\mathbf 0, R_1) =
\{ \bx: x_d = F(\hat\bx)\} \cap B(\mathbf 0, R_1),\\[0.2cm]
F(\hat{\mathbf 0})=0,\ \nabla F(\hat{\mathbf 0}) = 0.
\end{gather*}
\end{lem}

\begin{proof}

Since the set $S_1 = S - \bw$ satisfies the condition
\eqref{E1:eq} with $\bw = 0$, it suffices to prove
the lemma for $\bw = \mathbf 0$.
Thus we may assume that $\Phi(\hat{\mathbf 0}) = 0$.
Rewrite the equation $x_d = \Phi(\hat\bx)$ as follows:
\[
x_d = \hat\bb \cdot \hat\bx + \tilde\Phi(\hat\bx),\ \
\hat\bb = \nabla\Phi(\hat{\mathbf 0}),\
\tilde\Phi(\hat\bx) = \Phi(\hat\bx) - \hat\bb\cdot\hat\bx.
\]
Since $\Phi\in\plainC{m}, m\ge 1$, we have
\begin{equation}\label{fderiv:eq}
\max_{|\hat\bx|\le s}|\nabla\tilde\Phi(\hat\bx)|=:\d(s)\to 0,\ s\to 0,
\end{equation}
and hence
\begin{equation}\label{f:eq}
\max_{|\hat\bx|\le s}|\tilde\Phi(\hat\bx)|\le \d(s) s.
\end{equation}
Denote $\bb = (-\hat\bb, 1)$ and rewrite the equation again:
\[
\bn\cdot\bx = |\bb|^{-1}\tilde\Phi(\hat\bx),\ \bn = \bb|\bb|^{-1}.
\]
Let $\boldO$ be an orthogonal matrix in which the last column equals $\bn$, and denote
$\by = \boldO^{T}\bx$, so that $\bn\cdot\bx = y_d$, and the above equation takes the form
\begin{equation}\label{contraction1:eq}
y_d = W(y_d; \hat\by),\ W(y_d, \hat\by):= |\bb|^{-1} \tilde\Phi(\widehat{\boldO\by}).
\end{equation}
By construction, for all $\by\in B(\mathbf 0, s)$ we have
\[
y_d\in I(s):=[-\d(s) s, \d(s)s].
\]
From now on we assume that $s\le R$ is so small that $\d (s)\le 1/4$. Thus the cylinder
\begin{equation*}
\CC = \{\hat\by\in\R^{d-1}: |\hat\by|< s/4\}\times I(s)
\end{equation*}
 belongs to the ball $B(\mathbf 0, s)\subset B({\mathbf 0}, R)$.
 Therefore for all $ \hat\by : |\hat\by|< s/4$,
 the function $ W(\ \cdot\ ; \hat\by)$ maps the interval $I(s)$ into itself. Moreover,
 in view of the condition $\d(s)\le 1/4$,
 \[
| \p_{y_d} W(y_d, \hat\by)|\le \max_{|\hat\bx|\le s}|\nabla \tilde\Phi(\hat\bx)|
\le \d(s)\le \frac{1}{4},
 \]
 for all $\by\in\CC$, so that by the Contraction Mapping Theorem,
 for each $\hat\by$ the equation
 \eqref{contraction1:eq} has a unique solution $y_d\in I(s)$.
 Denote this solution by $F(\hat\by)$. Clearly, $F(\hat{\mathbf 0}) = 0$.
Using the Implicit Function Theorem one shows that this solution is a $\plainC{m}$-
function of $\hat\by$.
Moreover,
\[
\nabla_{\hat\by} F(\hat\by) = \left.
\frac{\nabla_{\hat\by} W( t, \hat\by)}{1-\p_{t} W(t, \hat\by)}
\right|_{t = F(\hat\by)},
\]
so that $\nabla F(\hat{\mathbf 0}) = 0$ as required. Now extend $F$ to the entire space
$\R^{d-1}$ as a $\plainC{m}$-function, and take $R_1 = s/8$.
\end{proof}

This lemma immediately yields a useful transformation of domains:

\begin{cor}\label{new_coordinates1:cor}
Let $\L\subset\R^d$ be a domain and let $\bw\in\p\L$. Suppose that in the
ball $B(\bw, R)$ the domain $\L$ is represented
by a $\plainC{m}$-graph-type domain $\G = \G(\Phi; \boldO, \bw)$, where $m\ge 1$.
Then there exists a number $R_1$, an orthogonal transformation $\tilde\boldO$,
and a function $F\in\plainC{m}(\R^{d-1})$, satisfying \eqref{propertyphi:eq}
and  $\nabla F(\hat\bzero) = 0$, such that
in the ball $B(\bw, R_1)$ the domain $\L$ is represented by
the graph-type domain $\G(F; \tilde\boldO, \bw)$.
\end{cor}

Now we establish the legitimacy of assumption
\eqref{omcircle:eq}. More precisely, we show that
any $\plainC{m}$-domain can be made to satisfy the condition
\eqref{omcircle:eq} locally.

\begin{lem}\label{new_coordinates2:lem}
Let $S$ be a set such that
for some Euclidean isometry $\boldE = (\boldO, \boldk)$,
\begin{equation}\label{E2:eq}
\boldE^{-1} S\cap B(\mathbf 0, R) = \{\bx: x_d = \Phi(\hat\bx)\} \cap B(\mathbf 0, R),
\end{equation}
for some $R>0$, and
with some function $\Phi\in\plainC{m}(\R^{d-1})$, such
that $\Phi(\hat{\mathbf 0}) = 0$.
Then there exists a number $R_1>0$ such that
one can find a number $l = 1, 2, \dots, d$
and a function $G\in\plainC{m}(\R^{d-1})$
such that $G(\overc{\mathbf 0}) = 0$,
\begin{gather}\label{E3:eq}
S \cap B(\boldk, R_1) =
\{ \bx : x_l = k_l + G(\overset{\circ} \bx
- \overset{\circ} \boldk)\} \cap B(\boldk, R_1),\\
\ \overset{\circ}\bx = (x_1, \dots, x_{l-1}, x_{l+1}, \dots, x_d),\notag
\end{gather}
and $\|\nabla G\|_{\plainL\infty}\le 4\sqrt d$.
\end{lem}

\begin{proof} Since the set $S_1 = S - \boldk$ satisfies the condition
\eqref{E2:eq} with $\boldE = (\boldO, \mathbf 0)$, it suffices to prove
the lemma for $\boldk = \mathbf 0$.

In view of Lemma \ref{new_coordinates1:lem} we may assume
without loss of generality that the transformation $\boldE$ is chosen in such a way
that $\nabla \Phi(\hat{\mathbf 0})$ = 0.
Denote, as in the proof of Lemma \ref{new_coordinates1:lem},
\[
\d(s) = \max_{|\hat\bx|\le s}|\nabla \Phi(\hat\bx)|, \ s\le R.
\]
Rewrite the equation $x_d = \Phi(\hat\bx)$ for the variable
$\bt: \bx = \boldE^{-1}\bt = \boldO^T \bt$:
\[
\bigl(\boldO^T\bt\bigr)_d - \Phi\bigl(\widehat{\boldO^T\bt}\bigr) = 0,
\]
so that
\[
\bn\cdot\bt - \Phi\bigl(\widehat{\boldO^T\bt}\bigr) = 0,
\]
with some unit vector $\bn\in\R^d$. Let $l$ be such that
$|n_l| = \max_j |n_j|\ge d^{-1/2}$.
Rewrite the equation in the form:
\begin{equation}\label{contraction:eq}
t_l = W(t_l; \overc\bt),\ \
W(t_l; \overc\bt) = -\frac{1}{n_l}
\overc\bn\cdot\overc\bt
+ \frac{1}{n_l} \Phi\bigl(\widehat{\boldO^T\bt}\bigr).
\end{equation}
Since $|\Phi(\hat\bx)|\le \d(s)s$ for all $|\hat\bx|\le s$,
the equation \eqref{contraction:eq} implies that
for all $\bt\in B(\mathbf 0, s)$ we have
\begin{equation*}
t_l\in I(\overc\bt; s):=\biggl[ -\frac{1}{n_l}
\overc\bn\cdot\overc\bt - \sqrt d \d(s) s,
 -\frac{1}{n_l}
\overc\bn\cdot\overc\bt + \sqrt d\d(s) s
\biggr].
\end{equation*}
From now on we assume that $s\le R$ is so small that $\sqrt d \d(s)\le 1/4$.
Thus the set
\begin{equation*}
\CC(s):=\{\bt: |\overset{\circ}\bt|\le (4\sqrt d)^{-1}s,\
t_l\in I(\overset{\circ} \bt; s)\},
\end{equation*}
is contained in $B(\mathbf 0, s)\subset B(\mathbf 0, R)$. Therefore for
all $\overc\bt: |\overset{\circ}\bt|\le (4\sqrt d)^{-1}s$,
the function $W(\ \cdot\ , \overc\bt)$
maps the interval $I(\overc\bt; s)$ into itself. Moreover, in view
of the condition $\sqrt d \d(s)\le 1/4$, we have
\[
|\p_{t_l} W(t_l; \overc\bt)|\le \sqrt d\d(s)\le \frac{1}{4},
\]
 for all $\bt\in \CC(s)$, so that by the Contraction Mapping Theorem,
 for each $\overc\bt$
 the equation \eqref{contraction:eq} has a unique solution
 $t_l\in I(\overset{\circ}\bt; s)$. Denote this solution by
 $G(\overc\bt)$. Clearly, $G(\overc{\mathbf 0}) = 0$.
 By the Implicit Function Theorem,
 $G$ is a $\plainC{m}$-function on $\{\overc\bt: |\overc\bt|\le s(4\sqrt d)^{-1}\}$.
 In particular,
\[
\nabla_{\overc\bt} G(\overc\bt) = \left.
\frac{\nabla_{\overc\bt} W( t, \overc\bt)}{1-\p_{t} W(t, \overc\bt)}
\right|_{t = G(\overc\bt)},
\]
so that
\begin{equation*}
\sup_{|\overc\bt|< s(4\sqrt d)^{-1}}|\nabla_{\overc\bt} G(\overc\bt)|
\le \frac{4\sqrt d}{3} \bigl(1+ \max_{|\hat\bx|\le s}|\nabla\Phi(\hat\bx)|\bigr)
\le \frac{4\sqrt d}{3}(1+\d(s))\le \frac{4\sqrt d + 1}{3}.
\end{equation*}
Now take $R_1 = s(12\sqrt d)^{-1}$ and extend $G$ to $\R^{d-1}$ in such a way that
 $\|\nabla G\|_{\plainL\infty}\le 4\sqrt d$.

\end{proof}

\begin{cor} \label{new_coordinates2:cor}
Let $\L\subset\R^d$ be a domain and let $\bw\in\p\L$.
Suppose that in the ball $B(\bw, R)$ the domain $\L$ is represented by
the $\plainC{m}$-graph-type domain $\G = \G(\Phi; \boldO, \bw)$, $m\ge 1$.
Then there exist a real number $R_1>0$ and an integer $l = 1, 2, \dots, d$,
a a function $G\in\plainC{m}(\R^{d-1})$, such that $G(\overc{\mathbf 0}) = 0$,
\begin{equation*}
\|\nabla G\|_{\plainL\infty}\le 4\sqrt d,
\end{equation*}
and
\begin{equation}\label{new_coordinates11:eq}
\L\cap B(\bw, R_1)
= \begin{cases}
\{\bx: x_l > w_l + G(\overc\bx - \overc\bw)\}\cap B(\bw, R_1),\\[0.2cm]
\ \textup{or}\ \
\{\bx: x_l < w_l + G(\overc\bx - \overc\bw)\}\cap B(\bw, R_1).
\end{cases}
\end{equation}
\end{cor}

\begin{proof}
Denote $\boldE = (\boldO, \boldk)$.
By Definition \ref{domains:defn} (1),
\[
\boldE^{-1}\L \cap B(\mathbf 0, R) =
\{\bx: x_d > \Phi(\hat\bx)\}
\cap B(\mathbf 0, R).
\]
Thus the boundary $S:=\p\L$ satisfies \eqref{E2:eq}.
Due to Lemma \ref{new_coordinates2:lem} the boundary $S$ also satisfies
\eqref{E3:eq} with a function $G$, which satisfies all required properties.
Thus the domain $\L$ is given by \eqref{new_coordinates11:eq}.
\end{proof}

Finally, we need one more technical result in which we use
a linear transformation of a very specific form.

\begin{lem}\label{linear:lem}
Let $\boldB$ be the linear transformation
\begin{equation*}
\boldB\bx = (\hat\bx + \hat\bb x_d, x_d),
\end{equation*}
with some vector $\hat\bb\in\R^{d-1}$.
Suppose that the surface $S$ is described by
\begin{equation}\label{bsurface:eq}
\boldB S = \{\bx: x_d = \Phi(\hat\bx)\},
\end{equation}
with some function $\Phi\in\plainC{m}(\R^{d-1}), m\ge 1$, such that
\begin{equation}\label{phim:eq}
 \|\nabla\Phi\|_{\plainL\infty}\le M,\
\sum_{k=1}^m \|\nabla^k\Phi\|_{\plainL\infty}\le C,
\end{equation}
with some number $M >0$.
Suppose that  $|\hat\bb|\le (2M)^{-1}$.
Then there is a function $F\in \plainC{m}(\R^{d-1})$ such that
\begin{equation*}
S = \{\bx: x_d = F(\hat\bx)\},
\end{equation*}
and
\begin{equation*}
\sum_{k=1}^m \|\nabla^k F\|_{\plainL\infty}\le C,
\end{equation*}
uniformly in $\Phi$, satisfying \eqref{phim:eq}.
\end{lem}

\begin{proof}
Since the surface $\boldB S - \Phi(\hat{\mathbf 0}) \be_d$ satisfies
\eqref{bsurface:eq} with the function $\Phi(\hat\bx) - \Phi(\hat{\mathbf 0})$,
it suffices to prove the lemma assuming that $\Phi(\hat{\mathbf 0}) = 0$.
By definition
each $\bt\in S$ satisfies the equation
\begin{equation}\label{linear:eq}
(\boldB\bt)_d = \Phi(\widehat{\boldB\bt}), \ \ \textup{i.e.}\ \
t_d = W(t_d; \hat\bt):=\Phi(\hat\bt + \hat\bb t_d).
\end{equation}
Denote $b:=|\hat\bb|\le (2M)^{-1}$.
Since $\Phi(\hat{\mathbf 0})=0$ and $Mb\le 2^{-1}$, we have
\begin{equation}\label{phimap:eq}
|\Phi(\hat\bt + \bb t_d)|
\le M(|\hat\bt| + b|t_d|)
\le M|\hat\bt| + \frac{1}{2}|t_d|
\end{equation}
for all $\bt$.
Therefore a solution of the equation
\eqref{linear:eq} satisfies
\[
|t_d|\le 2M|\hat\bt|.
\]
Thus
for each $\hat\bt\in\R^{d-1}$
the function $W(\ \cdot\ ; \hat\bt)$ maps the interval
\[
I(\hat\bt) = \bigl[- 2M|\hat\bt|,\ 2M|\hat\bt|\bigr]
\]
into itself. Furthermore, in view of the condition $bM\le 1/2$,
we have
\[
\p_{t_d}W(t_d; \hat\bt)\le Mb\le \frac{1}{2}.
\]
Thus by the Contraction Mapping Theorem, for each
$\hat\bt\in \R^{d-1}$ the equation \eqref{linear:lem} has a unique solution
$t_d\in I(\hat\bt)$. We denote this solution
by $F(\hat\bt)$. Moreover,
by the Implicit Function Theorem, $F$ is a $\plainC{m}$-function, with
derivatives bounded uniformly in $\Phi$, satisfying \eqref{phim:eq}.
\end{proof}

\subsection{Change of variables: integration}
We need a simple version of the change of variables in the
area formula (see, e.g. Theorem 1, p.96 in \cite{EG}).
We use again the notation
\begin{equation*}
\hat\bx = (x_1, x_2, \dots, x_{d-1}),
\overc \bx = (x_1, \dots, x_{l-1}, x_{l+1}, \dots, x_d),
\end{equation*}
for a fixed $l = 1, 2, \dots, d$.
Let $F\in\plainC1(\R^{d-1})$. For each $\hat\bx$ let
$\SX(\hat\bx)\subset\R$ be the set
\begin{equation*}
\SX(\hat\bx) = \{x_d\in\R: F(\overc \bx) = x_l\}.
\end{equation*}
Recall also that the vector
\begin{equation}\label{unit_normal:eq}
\bn_S(\bx):= \frac{\bigl(-\p_{x_1} F(\overc\bx),
\ \dots,
 -\p_{x_{l-1}} F(\overc\bx), 1, -\p_{x_{l+1}} F(\overc\bx),
 \dots, -\p_{x_d} F(\overc\bx)\bigr)}
{\sqrt{1+|\nabla F(\overc\bx)|^2}},
\end{equation}
defines the (continuous) unit normal to the surface
\begin{equation}\label{surface:eq}
S = \{\bx: x_l = F(\overc\bx)\},
\end{equation}
at the point
$\bx = \bigl(x_1, \dots, x_{l-1}, F(\overc\bx), x_{l+1}, \dots, x_d\bigr)$.

Below for brevity we use the notation
$(\overc\bx, F(\overc\bx))$ for the vector
$\bx$ with $x_l = F(\overc\bx)$.

\begin{lem}\label{area:lem}
Let
$f\in\plainC\infty_0(\R^d), F\in\plainC1(\R^{d-1})$ be some functions, and let
$S\subset\R^d$ be the surface \eqref{surface:eq}.
Then
\begin{equation*}
\underset{\R^{d-1}}\int  \sum_{x_d\in \SX(\hat\bx)}
f\bigl(\overc\bx, F(\overc\bx)\bigr) d \hat\bx
= \int_{S} |\bn_S(\bx)\cdot\be_d|\ f(\bx)d S_{\bx}.
\end{equation*}
\end{lem}

\begin{proof} If $l = d$, then $\SX(\hat\bx) = \{F(\hat\bx)\}$, so that
the integral equals
\begin{equation*}
\underset{\R^{d-1}}\int
f\bigl(\hat\bx, F(\hat\bx)\bigr) d \hat\bx
=
\underset{\R^{d-1}}\int \frac{f\bigl(\hat\bx, F(\hat\bx)\bigr)}
{\sqrt{1+|\nabla F(\hat\bx)|^2}}
\sqrt{1+|\nabla F(\hat\bx)|^2} d\hat\bx
= \underset{S}\int |\bn_S(\bx)\cdot\be_d| f(\bx)d S_{\bx},
\end{equation*}
as required.

Suppose that $l < d$.
Denote by $\Xi:\R^{d-1}\to\R^{d-1}$ the mapping
\begin{equation*}
\Xi(\overset{\circ}\bx) =
\bigl(x_1, \dots, x_{l-1}, F(\overc\bx), x_{l+1}, \dots, x_{d-1}\bigr).
\end{equation*}
Then by the Change of Variables formula (see, e.g. Theorem 2, p.99 in \cite{EG}),
\begin{equation*}
\underset{\R^{d-1}}\int \sum_{x_d\in \SX(\hat\bx)}
f\bigl(\overset{\circ}\bx, F(\overset{\circ}\bx)\bigr) d \hat\bx
= \underset{\R^{d-1}}\int \sum_{\overset{\circ}\bx\in \Xi^{-1}(\hat\bx)}
f\bigl(\overset{\circ}\bx, F(\overset{\circ}\bx)\bigr) d\hat\bx
= \underset{\R^{d-1}}\int J_{\Xi}(\overset{\circ}\bx)
f\bigl(\overset{\circ}\bx, F(\overset{\circ}\bx)\bigr) d\overset{\circ}\bx,
\end{equation*}
where $J_{\Xi}$ is the Jacobian of the map $\Xi$. A straightforward
computation shows that $J_{\Xi} = |\p_{x_d} F|$. Now, in view of
\eqref{unit_normal:eq}, the last integral
can be rewritten as
\[
\underset{\R^{d-1}}\int \frac{|\p_{x_d}F(\overc\bx)|}
{\sqrt{1+|\nabla F(\overc\bx)|^2}}
f\bigl(\overc\bx, F(\overc\bx)\bigr)
\sqrt{1+|\nabla F(\overc\bx)|^2} d\overc\bx
= \int_{S} |\bn_S(\bx)\cdot\be_d| f(\bx)d S_{\bx},\
\]
as required.
\end{proof}

%%%%%%%%%%%%%%%%%%%%%%%%%%%%%%%%%%%%%%%%%%%%%%%%%%%%%%%%%%%%%%%%%%%%%%%%%%%%%%%%%%%%

 \section{Appendix 3: A trace-class formula}

 Let $\GH$ be a separable Hilbert space. Consider in $\plainL2(\R^n, \GH)$, $n\ge 1$
 the pseudo-differential operator $T$
 with an $\GH$-valued symbol $\gt(\bx, \bxi)$, i.e.
 \begin{equation*}
 (T\phi)(\bx) = (\op_1^l(\gt) \phi)(\bx)
 = \frac{1}{(2\pi)^n}\int\int e^{i(\bx-\by)\cdot\bxi}
 \gt(\bx, \bxi) \phi(\by) d\by d\bxi,
 \end{equation*}
for any  $\phi\in\plainL2(\R^n, \GH)$.
Our objective is to give a proof of the standard formula for
the trace of $T$ under the assumption that both $\gt$ and $T$ are trace-class.
Note that we do not provide conditions which ensure that $T\in\GS_1$.

 \begin{lem}\label{trace:lem}
 Suppose that the operator
 $\gt(\bx, \bxi)$ is trace class a.a. $\bx$ and $\bxi$, that
 $T = \op_1^l(\gt)$ is trace-class, and that
 \begin{equation}\label{gt_trace:eq}
 \|\gt(\ \cdot\ ,\  \cdot\ )\|_{\GS_1}\in \plainL1(\R^n\times\R^n)
 \cap\plainL2(\R^n\times\R^n).
 \end{equation}
 Then
 \begin{equation*}
 \tr T  = \frac{1}{(2\pi)^n} \int \tr \gt(\bx, \bxi) d\bx d\bxi.
 \end{equation*}
 \end{lem}

 \begin{proof}
 Since $T\in\GS_1$,
 for any family of bounded operators $K_s$ in $\plainL2(\R^n, \GH)$,
 strongly converging to $I$
as $s\to\infty$, we have
\[
\|T - T K_s \|_{\GS_1}\to 0, \ s\to\infty.
\]
Thus it suffices to consider instead of $T$
the operator $B = \op_1^a(\gb)$
with the amplitude $\gb(\bx, \by, \bxi) = \gt(\bx, \bxi)\eta(\bxi)\psi(\by)$,
where $\eta, \psi\in\plainC\infty_0(\R^n)$, and prove that
\begin{equation}\label{trace1:eq}
\tr B = \frac{1}{(2\pi)^n} \int\int \psi(\bx)\eta(\bxi)\tr\gt(\bx, \bxi) d\bx d\bxi.
\end{equation}
Since $B$ is trace-class, for any orthonormal
basis (o.n.b.) $F_j$ in $\plainL2(\R^n, \GH)$
we have
\[
\tr B = \sum_{j}(B F_j, F_j),
\]
and the series converges absolutely, see \cite{BS}, Ch. 11, Theorem 2.7.
We choose the o.n.b. labeled by two indices.
Let $f_j, j\in\mathbb N$, be an o.n.b.
of $\GH$, and let $g_s, s = \in\mathbb N,$ be an o.n.b. of $\plainL2(\R^n)$, so that
$f_j\otimes g_s$ is an o.n.b. of $\plainL2(\R^n, \GH)$.
Consider the finite sum
\begin{align*}
S_{N, M} = &\ \sum_{s=1}^N \sum_{j = 1}^M(B f_j\otimes g_s, f_j\otimes g_s)\\[0.2cm]
= &\ \frac{1}{(2\pi)^n} \sum_{s=1}^N
\int e^{i\bxi\cdot(\bx-\by)}\eta(\bxi)
\CT_M(\bx, \bxi) g_s(\by) \psi(\by)\overline{g_s(\bx)} d\by d\bxi d\bx,
\end{align*}
with
\begin{equation*}
\CT_M(\bx, \bxi) = \sum_{j=1}^M(\gt (\bx, \bxi) f_j, f_j)_{\GH}.
\end{equation*}
Since
\[
|\CT_M(\bx, \bxi)|\le \|\gt(\bx, \bxi)\|_{\GS_1},
\]
and $\CT_M(\bx, \bxi)\to \tr \gt(\bx, \bxi)=:\CT(\bx, \bxi), M\to\infty$ pointwise,
in view of \eqref{gt_trace:eq},
by the dominated convergence theorem we conclude that the double sum $S_{N, M}$
converges as $M\to\infty$, to
\begin{align}
S_N = &\ \frac{1}{(2\pi)^n}
\int
e^{-i\bxi\cdot\by} \sum_{s=1}^N \psi(\by)
\eta(\bxi)\hat \CT^{(s)}(\bxi)g_s(\by)  d\by d\bxi,
\label{qf1:eq}\\[0.2cm]
&\ \hat \CT^{(s)}(\bxi) = \int e^{i\bxi\cdot\bx}\CT(\bx,\bxi) \overline {g_s(\bx)} d\bx.
\notag
\end{align}
As $N\to\infty$, the integral
\[
I_N(\bxi) = \int  e^{-i\bxi\cdot\by} \psi(\by)
\sum_{s=1}^N\hat \CT^{(s)}(\bxi)g_s(\by)d\by
\]
for almost all $\bxi$ converges to
\[
\int \CT(\by, \bxi) \psi(\by) d\by
\]
by Plancherel's Theorem. By Bessel inequality $I_N(\bxi)$ is bounded from above by
\[
\biggl(\int |\psi(\by)|^2 d\by\biggr)^{\frac{1}{2}}
\biggl(\sum_{s=1}^N |\hat \CT^{(s)}(\bxi)|^2\biggr)^{\frac{1}{2}}
\le \|\psi\|_{\plainL2} \biggl(\int  |\CT(\bx, \bxi)|^2 d\bx\biggr)^{\frac{1}{2}}
\]
uniformly in $N$. By \eqref{gt_trace:eq}
the right hand side is an $\plainL2$-function and
$\eta$ has compact support, so that
by the Dominated Convergence Theorem, in
\eqref{qf1:eq} one can pass to the limit as $N\to\infty$:
\begin{equation*}
\lim_{N\to\infty} S_N
= \frac{1}{(2\pi)^n} \int \CT(\by, \bxi) \eta(\bxi)\phi(\by) d\by d\bxi.
\end{equation*}
This coincides with \eqref{trace1:eq}, which completes the proof of the Lemma.
 \end{proof}

%%%%%%%%%%%%%%%%%%%%%%%%%%%%%%%%%%%%%%%%%%%%%%%%%%%%%%%%%%%%%%%%%%%%%%%%%%%%%%%%%%%%

 \section{Appendix 4:
 Invariance with respect to the affine change of variables}

Our aim is to show that the coefficient \eqref{w1:eq} does not change under the
affine change of variables.
Let $\boldM$ be a non-degenerate linear transformation of $\R^d$,
let $\boldk, \boldk_1$ be vectors in $\R^d$,
and let $a = a(\bx, \bxi)$
be a continuous function on $\R^d\times\R^d$ with a compact support in both variables.
As in \eqref{orthogonal1:eq} denote
\[
a_{\boldM, \boldk, \boldk_1}(\bx, \bxi)
= a\bigl(\boldM\bx+\boldk, (\boldM^{T})^{-1}\bxi+\boldk_1\bigr).
\]
For a set $S\subset\R^d$ denote
$S_{\boldM, \boldk} = \boldM^{-1}(S-\boldk)$, $S^T_{\boldM, \boldk} = \boldM^T (S-\boldk)$.

\begin{lem}\label{linear_invariance:lem}
Let $F, G\in\plainC1(\R^d)$ be two real-valued functions such that
$\nabla G(\bx)\not =0$, $\nabla G(\bxi) \not = 0$, and
let $S, P$ be two surfaces, defined by the
equations $F(\bx) = 0$ and $G(\bxi) = 0$ respectively.
Suppose that the function $a$ is as defined above. Then
for any non-degenerate linear transformation $\boldM$ and any pair of vectors
$\boldk, \boldk_1\in\R^d$, one has
\begin{equation*}
\GW_1(a_{\boldM, \boldk, \boldk_1};
S_{\boldM, \boldk}, P^T_{\boldM, \boldk_1}) = \GW_1(a; S, P).
\end{equation*}
\end{lem}

\begin{proof}
The surfaces $S_{\boldM, \boldk}$ and $P^T_{\boldM, \boldk_1}$ are defined by the
equations
\[
F_{\boldM, \boldk}(\bx) = 0 \ \ \textup{and}\ \
G_{\boldM, \boldk_1}(\bxi) = 0
\]
respectively. Define
\begin{equation*}
Z(t, s; \boldM, \boldk, \boldk_1) = \underset{F_{\boldM, \boldk}(\bx) > t} \int\ \
\underset{G_{\boldM, \boldk_1}(\bxi)>s}
\int
a_{\boldM, \boldk, \boldk_1}
(\bx, \bxi) \bigl|\nabla F_{\boldM, \boldk}(\bx)
\cdot\nabla G_{\boldM, \boldk_1}(\bxi)\bigr|
d\bxi d\bx.
\end{equation*}
A straightforward change of variables gives the equality
\begin{equation}\label{zmts:eq}
Z(t, s; \boldM, \boldk, \boldk_1) = Z(t, s; \boldI, \bzero, \bzero).
\end{equation}
By Proposition 3, Ch 3.4, \cite{EG},
\begin{align*}
\frac{\p}{\p t} \frac{\p}{\p s}
Z(t, s; \boldM, &\
\boldk, \boldk_1) \biggr|_{t=s=0}\\[0.2cm]
= &\ \underset{F_{\boldM, \boldk}(\bx)=0} \int\ \
\underset{G_{\boldM, \boldk_1}(\bxi)=0}
\int
a_{\boldM, \boldk, \boldk_1}(\bx, \bxi)
\frac{\bigl|\nabla F_{\boldM, \boldk}(\bx)
\cdot\nabla G_{\boldM, \boldk_1}(\bxi)\bigr|}
{|\nabla F_{\boldM, \boldk}(\bx)|\ |\nabla G_{\boldM, \boldk_1}(\bxi)|}
dS_{\bxi} dS_{\bx}\\[0.2cm]
= &\ \GW_1(a_{\boldM, \boldk, \boldk_1};
S_{\boldM, \boldk}, P^T_{\boldM, \boldk_1}),
\end{align*}
where we have used that the vectors
\begin{equation*}
\frac{\nabla F_{\boldM, \boldk}(\bx)}
{|\nabla F_{\boldM, \boldk}(\bx)|}
\ \ \textup{and}\ \
\frac{\nabla G_{\boldM, \boldk_1}(\bxi)}
{|\nabla G_{\boldM, \boldk_1}(\bxi)|}
\end{equation*}
define unit normals to the surfaces $S_{\boldM, \boldk}$ and $P^T_{\boldM, \boldk_1}$
at the points $\bx$ and $\bxi$ respectively.
Now the proclaimed equality follows from \eqref{zmts:eq}.
\end{proof}

%%%%%%%%%%%%%%%%%%%%%%%%%%%%%%%%%%%%%%%%%%%%%%%%%%%%%%%%%%%%%%%%%%%%%%%%%%%%%%%%%%%%

\bibliographystyle{amsplain}
\bibliography{bibmaster}

\begin{thebibliography}{99}

\bibitem{Arsu}
G. Arsu, \emph{On Schatten-von Neumann
class properties of pseudodifferential operators. The Cordes-Kato method},
J. Operator Theory \textbf{59} (2008), Issue 1, pp. 81-114.




\bibitem{BS_U}
M. \v S. Birman and M. Z. Solomyak,
\emph{Estimates of singular numbers of integral operators
}Uspekhi Mat. Nauk \textbf{32} (1977), no.1, 17--84,
Engl. transl. in: Russian Math. Surveys \textbf{32}(1977), no. 1, 15--89.
1987.

\bibitem{BS}
M.\v S. Birman and M. Z. Solomyak,
\emph{Spectral theory of self-adjoint operators in Hilbert space}, Reidel,
1987.

\bibitem{BKS}
M.Sh. Birman, G.E. Karadzhov, M.Z. Solomyak,
\emph{
Boundedness conditions and spectrum estimates for the operators
$b(X)a(D)$ and their analogs},
Estimates and asymptotics for discrete spectra of integral and differential equations
(Leningrad, 1989--90),
85--106,
Adv. Soviet Math., 7, Amer. Math. Soc., Providence, RI, 1991.


\bibitem{BotSil}
A. B\"ottcher, B. Silbermann, \emph{
Analysis of Toeplitz operators},
Second edition,
Springer Monographs in Mathematics, Springer-Verlag, Berlin, 2006.



\bibitem{BCT} L. Brandolini, L. Colzani, G. Travaglini,
\emph{Average decay of Fourier transforms and integer points in polyhedra},
Ark. Mat. \textbf{35} (1997), 253--275.

\bibitem{ConToft}
F. Concetti, J. Toft, \emph{
Schatten-von Neumann properties
for Fourier integral operators with non-smooth symbols. I},
Ark. Mat. \textbf{47} (2009), no. 2, 295--312.


\bibitem{DIK} P. Deift, A. Its, I. Krasovsky,
\emph{Asymptotics of Toeplitz, Hankel, and Toeplitz+Hankel determinants
with Fisher-Hartwig singularities}, arXiv:0905.0443v2  [math.FA].

\bibitem{Ehr} T. Ehrhardt, \emph{
A status report on the asymptotic behavior of
Toeplitz determinants with Fisher-Hartwig singularities},
Recent advances in operator theory (Groningen, 1998), 217--241,
Oper. Theory Adv. Appl., \textbf{124}, Birkh\"auser, Basel, 2001.


\bibitem{EG} L.C. Evans, R.F. Gariepy,
\emph{Measure Theory and Fine Properties of Functions},
CRC Press, 1992.


\bibitem{G1} D. Gioev, \emph{Generalizations of
Szeg\H o Limit Theorem: Higher Order Terms
and Discontinuous Symbols,} PhD Thesis, Dept. of Mathematics,
Royal Inst. of Technology (KTH), Stockholm, 2001.

\bibitem{G2} D. Gioev, \emph{Szeg\H o Limit Theorem for operators
with discontinuous symbols and applications to entanglement entropy},
(2006) IMRN, article ID 95181, 23 pages.

\bibitem{GiKl} D. Gioev, I. Klich,
\emph{Entanglement Entropy of fermions in any dimension and the Widom Conjecture},
Phys. Rev. Lett. \textbf{96} (2006), no. 10, 100503, 4pp.


\bibitem{GK} I. C.Gohberg, M. G. Krein,
\emph{Introduction to the theory of linear non-self-adjoint operators},
Translations of Mathematical Monographs, Vol. 18, American Mathematical Society,
Providence, R.I., 1969.


\bibitem{GrSz} U. Grenander, G. Szeg\H o,
\emph{Toeplitz forms and their applications},
Berkeley- Los Angeles, U of California press, 1958.


\bibitem{HLS}R.C. Helling, H. Leschke, W.L. Spitzer,
\emph{
A special case of a conjecture by widom with implications
to fermionic entanglement entropy}, Int Math Res Notices,
Advance Access published on June 21, 2010; doi:10.1093/imrn/rnq085.
%
%
%arXiv:0906.4946v1 [math-ph].
%
%

\bibitem{H} L. H\"ormander, \emph{The Analysis of Linear Partial Differential
Operators, I}, Grundlehren Math. Wiss. \textbf{256}, Springer-Verlag, Berlin, 1983.

\bibitem{H3} L. H\"ormander, \emph{The Analysis of Linear Partial Differential
Operators, III}, Grundlehren Math. Wiss. \textbf{274}, Springer-Verlag, Berlin, 1985.



\bibitem{Hw} I.L. Hwang, \emph{The $L_2$-boundedness of pseudo-differential
operators}, Trans. AMS \textbf{302} (1987), pp. 55-76.

\bibitem{Land_Wid}
H. Landau, H. Widom,
\emph{Eigenvalue distribution of time and frequency limiting},
J. Math. Analysis Appl. \textbf{77} (1980), 469--481.


\bibitem{LS}
A. Laptev, Yu. Safarov, \emph{
A generalization of the Berezin-Lieb inequality},
Contemporary Mathematical Physics, 69--79,
Amer. Math. Soc. Transl. Ser. 2, 175, Amer. Math. Soc., Providence, RI, 1996.



\bibitem{LapSaf}
A. Laptev and Yu. Safarov, \emph{Szeg\H o type limit theorems},
J. Funct. Anal., 138 (1996), 544-559.

\bibitem{Lerner}
N. Lerner, \emph{Some facts about the Wick calculus. Pseudo-differential operators},
135--174,
Lecture Notes in Math., 1949, Springer, Berlin, 2008.


\bibitem{Linnik} I. J. Linnik,
\emph{A multidimensional analog of a limit theorem of G. Szeg\H o},
Mathematics USSR-Izvestiya \textbf{9}(1975), no. 6, 1323--1332.


\bibitem{Nikol} N.K.Nikolski, \emph{
Operators, functions, and systems: an easy reading. Vol.1.
Hardy, Hankel, and Toeplitz}, Mathematical Surveys and Monographs, \textbf{92},
American Mathematical Society, Providence, RI, 2002.


\bibitem{Roc}	R. Roccaforte,
\emph{Asymptotic expansions of traces for certain convolution operators},
Trans. Amer. Math. Soc. \textbf{285} (1984), no. 2, 581--602.

\bibitem{Rond}
C. Rondeaux, \emph{
Classes de Schatten d'op\'erateurs pseudo-diff\'erentiels},
Ann. Sci. \'Ecole Norm. Sup. (\textbf{4}) \textbf{17} (1984), no. 1, 67--81.


\bibitem{Robert}D. Robert, \emph{Autour de l'approximation semi-classique},
Progress in Mathematics, 68. Birkh\"auser Boston, Inc., Boston, MA, 1987.


\bibitem{Roz} G.Rozenblum, \emph{On some analytical
index formulas related to operator-valued symbols}, Electron. J.
Differential Equations 2002, No. 17, 31 pp. (electronic).

\bibitem{ShT} M. A. Shubin, V. N. Tulovskii,
\emph{The asymptotic distribution of the
eigenvalues of pseudodifferential operators in $\mathbb R^{n}$}, (Russian)
Mat. Sb. (N.S.) \textbf{92(134)} (1973), 571--588.


\bibitem{Sh} M. A. Shubin, \emph{
Pseudodifferential operators and spectral theory},
Springer Series in Soviet Mathematics, Springer-Verlag, Berlin, 1987.

\bibitem{Simon}
B. Simon, \emph{
Trace ideals and their applications},
Second edition, Mathematical Surveys and Monographs, 120,
American Mathematical Society, Providence, RI, 2005.

\bibitem{Slep}
D. Slepian, \emph{
Prolate spheroidal wave functions,
Fourier analysis and uncertainity. IV.
Extensions to many dimensions; generalized prolate spheroidal functions},
Bell System Tech. J. \textbf{43} (1964), 3009--3057.

\bibitem{Slep1}
D Slepian,
\emph{Analytic solution of two apodization problems},
J. Optical Soc. of America, \textbf{55} (1965), no 9, 1110--1115.

\bibitem{Sz} G. Szeg\H o, \emph{On certain Hermitian forms associated with
the Fourier series of a positive function},
Festskrift Marcel Riesz, Lund, 1952, pp. 228--238.

\bibitem{Toft} J. Toft,
\emph{Schatten-von Neumann properties in the Weyl calculus
and calculus of metrics on symplectic vectore spaces},
Ann. Global Anal. Geom. \textbf{30} (2006), no. 2, 169--209.

\bibitem{Widom4} H. Widom,
\emph{A theorem on translation kernels in $n$ dimensions},
 Trans. AMS \textbf{94}
(1960),  no. 1, 170--180.



\bibitem{Widom3} H. Widom,
\emph{Szeg\H o's limit theorem, the higher-dimensional matrix case},
 J. Funct. Anal.  \textbf{39}
(1980), 182--198.

\bibitem{Widom1} H. Widom,
\emph{On a class of integral
operators with discontinuous symbol},
Toeplitz centennial (Tel Aviv, 1981), pp. 477--500,
Operator Theory: Adv. Appl., 4, Birkh\"auser, Basel-Boston, Mass., 1982.

\bibitem{Widom2} H. Widom,
\emph{On a class of integral operators on a half-space
with discontinuous symbol},  J. Funct. Anal.  \textbf{88}
(1990),  no. 1, 166--193.
\end{thebibliography}

\end{document}